\documentclass[reqno]{amsart}
\usepackage{graphicx, tikz-cd}
\graphicspath{{images/}}
\usepackage{mathrsfs}
\usepackage{url}
\usepackage[normalem]{ulem}
\usepackage{mathtools}
\usepackage{amssymb}
\usepackage[utf8]{inputenc}
\usepackage[toc,page]{appendix}
\usepackage{enumitem}
\usepackage{array}
\usepackage{verbatim}
\usepackage{hyperref}
\hypersetup{
    colorlinks=true,
    linkcolor=blue,
    filecolor=magenta,      
    urlcolor=cyan,
}
\usepackage{stmaryrd}

\usepackage{chngcntr}
\counterwithin{figure}{section}

\usetikzlibrary{decorations.pathmorphing,shapes}

\newcounter{sarrow}

\newtheorem{theorem}{Theorem}[section]
\newtheorem{lemma}[theorem]{Lemma}
\newtheorem{proposition}[theorem]{Proposition}
\newtheorem{conjecture}[theorem]{Conjecture}

\theoremstyle{definition}
\newtheorem{definition}[theorem]{Definition}
\newtheorem{example}[theorem]{Example}

\newtheorem{construction}[theorem]{Construction}
\newtheorem{corollary}[theorem]{Corollary}

\newtheorem{convention}[theorem]{Convention}

\newtheorem{construction--theorem}[theorem]{Construction--Theorem}

\theoremstyle{remark}
\newtheorem{remark}[theorem]{Remark}
\newtheorem{warning}[theorem]{Warning}

\numberwithin{equation}{section}

\DeclareMathOperator{\End}{End}

\newcommand{\bbR}{\mathbb{R}}

\newcommand{\Z}{\mathbb{Z}}

\newcommand{\C}{\mathbb{C}}

\newcommand{\N}{\mathbb{N}}
\newcommand{\p}{\partial}
\newcommand{\fr}{\mathfrak}

\newcommand{\cM}{\mathcal{M}}
\newcommand{\cA}{\mathcal{A}}
\newcommand{\cL}{\mathcal{L}}
\newcommand{\bbC}{\mathbb{C}}
\newcommand*{\Scale}[2][4]{\scalebox{#1}{$#2$}}

\newcommand{\fg}{\mathfrak{g}}
\newcommand{\ft}{\mathfrak{t}}

\DeclareMathOperator{\Ker}{Ker}

\DeclareMathOperator{\Hom}{Hom}

\DeclareMathOperator{\Aut}{Aut}

\DeclareMathOperator{\id}{Id}

\DeclareMathOperator{\acts}{\rotatebox[origin=c]{-90}{$\circlearrowright$}}

\newcommand{\R}{\mathbb R}
\newcommand{\bbZ}{\mathbb Z}
\newcommand{\bbN}{\mathbb N}

\newcommand{\bbP}{\mathbb P}

\newcommand{\fu}{\mathfrak u}

\newcommand{\cC}{\mathcal C}

\newcommand{\cE}{\mathcal E}

\newcommand{\cH}{\mathcal H}

\newcommand{\cJ}{\mathcal J}
\newcommand{\cK}{\mathcal K}
\newcommand{\cN}{\mathcal N}

\newcommand{\cS}{\mathcal S}

\newcommand{\cW}{\mathcal W}

\DeclareFontFamily{U}{BOONDOX-calo}{\skewchar\font=45 }
\DeclareFontShape{U}{BOONDOX-calo}{m}{n}{<-> s*[1.05] BOONDOX-r-calo}{}
\DeclareFontShape{U}{BOONDOX-calo}{b}{n}{<-> s*[1.05] BOONDOX-b-calo}{}
\DeclareMathAlphabet{\mathcalboondox}{U}{BOONDOX-calo}{m}{n}
\newcommand{\bbA}{\mathbb A}
\newcommand{\bbG}{\mathbb G}
\makeatletter
\let\save@mathaccent\mathaccent
\newcommand*\if@single[3]{%
	\setbox0\hbox{${\mathaccent"0362{#1}}^H$}%
	\setbox2\hbox{${\mathaccent"0362{\kern0pt#1}}^H$}%
	\ifdim\ht0=\ht2 #3\else #2\fi
}
\newcommand*\rel@kern[1]{\kern#1\dimexpr\macc@kerna}
\newcommand*\widebar[1]{\@ifnextchar^{{\wide@bar{#1}{0}}}{\wide@bar{#1}{1}}}
\newcommand*\wide@bar[2]{\if@single{#1}{\wide@bar@{#1}{#2}{1}}{\wide@bar@{#1}{#2}{2}}}
\newcommand*\wide@bar@[3]{%
	\begingroup
	\def\mathaccent##1##2{%
		\let\mathaccent\save@mathaccent
		\if#32 \let\macc@nucleus\first@char \fi
		\setbox\z@\hbox{$\macc@style{\macc@nucleus}_{}$}%
		\setbox\tw@\hbox{$\macc@style{\macc@nucleus}{}_{}$}%
		\dimen@\wd\tw@
		\advance\dimen@-\wd\z@
		\divide\dimen@ 3
		\@tempdima\wd\tw@
		\advance\@tempdima-\scriptspace
		\divide\@tempdima 10
		\advance\dimen@-\@tempdima
		\ifdim\dimen@>\z@ \dimen@0pt\fi
		\rel@kern{0.6}\kern-\dimen@
		\if#31
		\overline{\rel@kern{-0.6}\kern\dimen@\macc@nucleus\rel@kern{0.4}\kern\dimen@}%
		\advance\dimen@0.4\dimexpr\macc@kerna
		\let\final@kern#2%
		\ifdim\dimen@<\z@ \let\final@kern1\fi
		\if\final@kern1 \kern-\dimen@\fi
		\else
		\overline{\rel@kern{-0.6}\kern\dimen@#1}%
		\fi
	}%
	\macc@depth\@ne
	\let\math@bgroup\@empty \let\math@egroup\macc@set@skewchar
	\mathsurround\z@ \frozen@everymath{\mathgroup\macc@group\relax}%
	\macc@set@skewchar\relax
	\let\mathaccentV\macc@nested@a
	\if#31
	\macc@nested@a\relax111{#1}%
	\else
	\def\gobble@till@marker##1\endmarker{}%
	\futurelet\first@char\gobble@till@marker#1\endmarker
	\ifcat\noexpand\first@char A\else
	\def\first@char{}%
	\fi
	\macc@nested@a\relax111{\first@char}%
	\fi
	\endgroup
}
\makeatother

\usetikzlibrary{decorations.markings} %arrows for open immersions and closed immersions
\tikzset{
  closed/.style = {decoration = {markings, mark = at position 0.5 with { \node[transform shape, xscale = .8, yscale=.4] {/}; } }, postaction = {decorate} },
  open/.style = {decoration = {markings, mark = at position 0.5 with { \node[transform shape, scale = .7] {$\circ$}; } }, postaction = {decorate} }
}

\DeclareMathOperator{\ad}{ad}
\DeclareMathOperator{\Ad}{Ad}

\DeclareMathOperator{\Coker}{Coker}

\DeclareMathOperator{\Conv}{Conv}

\DeclareMathOperator{\Fun}{Fun}

\DeclareMathOperator{\Loc}{Loc}
\DeclareMathOperator{\LocVect}{LocVect}

\DeclareMathOperator{\Map}{Map}

\DeclareMathOperator{\Span}{Span}

\DeclareMathOperator{\nonab}{nonab}

\DeclareMathOperator{\pt}{pt}

\DeclareMathOperator{\reglue}{reglue}

\begin{document}

\title{Spectral Networks and Non-Abelianization}

\author{Matei Ionita \and Benedict Morrissey}
\maketitle

\begin{abstract}
    We generalize the non-abelianization of \cite{gaiotto2013spectral} from the case of $SL(n)$ and $GL(n)$ to arbitrary reductive algebraic groups.  This gives a map between a moduli space of certain $N$-shifted weakly $W$-equivariant $T$-local systems on an open subset of a cameral cover $\tilde{X}\rightarrow X$ to the moduli space of $G$-local systems on a punctured Riemann surface $X$.  For classical groups, we give interpretations of these moduli spaces using spectral covers.
    
    Non-abelianization uses a set of lines on the Riemann surface $X$ called a spectral network, defined using a point in the Hitchin base.  We show that these lines are related to trajectories of quadratic differentials on quotients of $\tilde{X}$.  We use this to describe some of the generic behaviour of lines in a spectral network.
\end{abstract}

\tableofcontents

\section{Introduction}

%\todo{1.2 (outline) is done, all other sections are not done}
%\todo{Only possible change is now to fully adopt in introduction the $X^{\circ}$ notation, or completely excise it.}

Non-abelianization was introduced in \cite{gaiotto2013spectral} as a way to describe conjecturally holomorphic symplectic\footnote{On each holomorphic symplectic leaf, using the modification that specifies certain reductions of structure to a Borel as outlined in Section \ref{subsubsec: Borel structures arbitrary G}.} ``co-ordinate charts'' on $\Loc_{G}(X)$, here roughly meaning complex algebraic maps $Y\rightarrow \Loc_{G}(X)$, where $G=SL(n)$ or $G=GL(n)$, $X$ is a non-compact Riemann surface, and $Y$ is complex analytically a $Z(G)$-gerbe over $(\bbC^{*})^{m}$.   The precise spaces $Y$ used are moduli space of $\bbG_{m}$-local systems on a connected finite cover $\overline{X}\rightarrow X$, with punctures, and other additional data, where $\bbG_{m}$ denotes the multiplicative group.  The Riemann--Hilbert correspondence gives that these moduli spaces are complex analytically isomorphic to a $Z(G)$ gerbe over $(\bbC^{*})^{m}$ for some $m$.  This paper generalizes this construction to arbitrary reductive algebraic groups $G$.

Non-abelianization for $GL(n)$ can be seen as a two step procedure.  In the first of these steps we have an $n:1$ cover $\pi: \overline{X}\rightarrow X$ with ramification locus $R_{\rho}\subset \overline{X}$ and branch locus $P\subset X$.  We  pushforward the restriction to $\overline{X}\backslash \pi^{-1}(P)$ of a local system of one dimensional vector spaces on $\overline{X}\backslash R_{\rho}$ to gain a local system $\mathbf{E}$ of $n$ dimensional vector spaces on $X\backslash P$.  The covers $\overline{X}\rightarrow X$ used are the \emph{spectral curves} associated to a point in the Hitchin base.

In the second of these steps we modify this local system on $X\backslash P$ along a set of real codimension one loci $\cW\subset X$ (called a \emph{spectral network}), to produce a new local system on $X\backslash P$, which in fact extends to a local system on $X$.  More precisely, the modification works by specifying a locally constant section of $A_{\cW}\in \Aut(\mathbf{E}|_{\cW\backslash J})$ (where $J\subset \cW$ are the intersections points of multiple loci of $\cW$).  We then ``cut'' $\mathbf{E}$ along $\cW$, and re-identify both sides using the automorphism $A_{\cW}$, as is rigorously described in Definition \ref{defn: regluing map}.  Consider the following example:

\begin{example}[Modifying local systems on $S^{1}$.]
\label{ex: modifying local systems on S1}
Assume we are given an $n$-dimensional local system $\mathbf{E}$ on $S^{1}$ (the analogue of $X$ in this example), a point $s\in S^{1}$ (the analogue of $\cW$), and an automorphism $\rho\in \Aut(\mathbf{E}|_{s})$. The point $s \in S^1$ determines an isomorphism $[0,1] / \{\{0\}\sim \{1\}\} \cong S^1$, which identifies $\{0\}$ with $s$. Denote by $\pi$ the composition of this isomorphism with the natural quotient map $[0,1]\rightarrow [0,1]/\{\{0\}\sim \{1\}\}$. We can produce a new local system by  taking $\pi^{*}\mathbf{E}$ (informally we referred to this as ``cutting'' along $s$) and then identifying (``gluing'') $\pi^{*}\mathbf{E}|_{0}$, and $\pi^{*}\mathbf{E}|_{1}$ by the map $\rho:\pi^{*}\mathbf{E}|_{0}\rightarrow \pi^{*}\mathbf{E}|_{1}$. This gives a new local system \[\pi^{*}\mathbf{E}/\{\pi^{*}\mathbf{E}|_{0} \sim_{\rho}\pi^{*}\mathbf{E}|_{1}\}\rightarrow [0,1]/\{\{0\}\sim \{1\}\} \cong S^{1}.\]

An example in this setting of the types of local systems on $S^{1}$ we will consider follows.  Let $p: S^{1}\rightarrow S^{1}$ be a double cover.  Let $\cL\rightarrow S^{1}$ be a one dimensional local system on $S^{1}$.  We then have a rank two local system $\mathbf{E}:=p_{*}\cL$ on $S^{1}$. 

Let $a_{1}, a_{2}$ be the two preimages of $s\in S^{1}$ under $p$.  A basis of $\mathbf{E}|_{s}=\cL|_{a_{1}}\oplus \cL|_{a_{2}}$ given by a basis of $\cL|_{a_{1}}$ and a basis of $\cL|_{a_{2}}$ identifies the monodromy of $\mathbf{E}$ with an off diagonal matrix.  Hence the monodromy of $\mathbf{E}$ is not trivial and hence $\mathbf{E}$ does not extend to a local system on the disc $D$ (with $\partial D=S^{1}$).  We can use the type of modification above to modify $\mathbf{E}$ to produce a local system that does extend to the disc $D$.
\end{example}

%\todo{The above example is useful, but I suggest that we take the opportunity to also explain why the gluing is necessary. We should add some version of the following example.}

%\begin{example}
%\label{eg:local_monodromy}
%To get a feeling for how the pushforward looks locally around a branch point, consider the local model $\pi : \C \to \C$, $\pi(z) = z^2$, and let $\cL$ be a trivial local system of 1-dimensional vector spaces on $\C$. Then ${\pi|_{\C^\times}}_* \cL$ is a rank 2 local system on $\C^\times$, with monodromy around 0 given by the permutation matrix which exchanges the 2 branches of the covering:
%\[ 
%\left( \begin{array}{cc}
%   0  & 1 \\
%   1  & 0
%\end{array} \right)
%\]
%As such, ${\pi|_{\C^\times}}_* \cL$ does not extend to a local system on $\C$.
%\end{example}

Doing this for families of one dimensional local systems on a sufficiently nice spectral curve $\overline{X}\backslash R_{\rho}$ (with some conditions on the monodromy around $R_{\rho}$) gives a map between a moduli space of such local systems and the moduli space of $GL(n)$ local systems on $X$. A minor modification of the above allows one to work with $SL(n)$ local systems.  In the $SL(2)$ case (or the $SL(2,\bbR)$ case \cite{fenyes2015dynamical}) these coordinates\footnote{There are multiple results along these lines.  In \cite{fenyes2015dynamical} Fenyes considers twisted $SL(2,\bbR)$ local systems, in the sense of local systems on the unit tangent bundle of a \emph{compact} Riemann surface with monodromy $-Id$ around the unit circle in each tangent fiber.  These specify $PGL(2,\bbR)$ local systems on the Riemann surface.  Fenyes identifies the abelianization coordinates, that is to say the monodromy of the $\bbG_{m}(\bbR)$-local systems, in this setting with Thurston's shear--bend coordinates.  In \cite{hollands2016spectral} (building on \cite{gaiotto2013wallHitchin}) it is shown that for appropriate spectral networks on a \emph{non-compact} Riemann surface the abelianization coordinates for $SL(2,\bbC)$ corresponding to a certain set of paths on the spectral cover give the pullback of the Fock--Goncharov coordinates (\cite{fock2006moduli}) for the moduli of $PGL(2,\bbC)$-local systems to the moduli of $SL(2,\bbC)$ local systems.  These paths do not in general generate the abelianization of the fundamental group of the spectral cover, and hence they will generally give strictly less information than the $SL(2,\bbC)$ Fock--Goncharov coordinates.   The paper \cite{hollands2016spectral} uses a version of nonabelianization that assigns a reduction of structure to a Borel at the boundary.  See Section \ref{subsubsec: Borel structures arbitrary G} for how to assign these reductions of structure to a Borel for certain types of spectral network and arbitrary reductive algebraic groups $G$.} recover Fock--Goncharov coordinates (respectively Thurston's shear--bend coordinates) \cite{gaiotto2013wallHitchin, fenyes2015dynamical, hollands2016spectral, fock2006moduli}. Rigorous details for the construction in the $SL(2)$ case can be found in any of \cite{fenyes2015dynamical, hollands2016spectral, nikolaev19abelianization}.   

The \emph{spectral network} has a definition\footnote{At the physical level of rigour.} in terms of a 2d-4d wall crossing problem, which we review in Section \ref{subsec: introduction 2d-4d wall crossing and non-abelianization}. However, we instead use a combinatorial definition as in \cite{gaiotto2013spectral} of what we call a \emph{basic abstract spectral network}, where we impose some conditions on the behaviour of lines in the network.  We also use an iterative construction\footnote{It is \emph{not} currently clear that this agrees with networks defined via the 2d-4d wall crossing problem beyond the rank two case, and simple examples.  Even in these cases at least some of this agreement remains at a physical level of rigour.  In particular the 2d-4d wall crossing problem has not yet been defined rigorously.} as in \cite{gaiotto2013spectral} that associates a basic spectral network to some points in the Hitchin base.  We exclude points that give us a dense network $\cW\subset X$, (in the sense of Definition \ref{defin:basic_abstract} (1)), or fail to satisfy certain other properties as explained in Sections \ref{subsec: abstract cameral networks} and \ref{subsec:wkb_cameral}. In the $SL(2,\mathbb{R})$ case, Fenyes \cite{fenyes2015dynamical} explains how to deal with dense networks.  Spectral networks have previously appeared in the WKB analysis of differential equations under the name of \emph{Stokes Graphs}, see e.g. \cite{iwaki2015exact, iwaki2014exact, bnr1982new, takei2017wkb, aoki2008virtual, aoki2005virtual, honda2015virtual, aoki2001exact}.   The construction of $\cW$ proceeds iteratively, at each stage adding trajectories of vector fields which are locally non-canonically\footnote{The precise statement is that these trajectories are the images of trajectories on a cameral cover $\tilde{X}$ which are canonically associated to a root.} associated to roots of $GL(n)$. The added trajectories start at points where the trajectories of previous iteration intersect as shown in Figure \ref{fig:new_stokes}.  We can perform non-abelianization both with the iteratively constructed networks and with the basic abstract spectral networks of Definition \ref{def:spectral_net}, of which the networks of \cite{gaiotto2014spectralsnake} are an important example (see Remark \ref{rem:snakes}) which produce some of the Fock--Goncharov coordinates for $G=SL(n)$.

We now explain why these new curves starting at intersection points are necessary.  Consider two curves in $\cW$, and assume that we associated to them automorphisms $A_{\cW,1}, A_{\cW,2}$. Assume, moreover, that the curves intersect in a point $J$. If $A_{\cW,1}$ and $A_{\cW,2}$ do not commute, gluing as in Example \ref{ex: modifying local systems on S1} would introduce monodromy around the point $J$. To address this problem, we add additional lines\footnote{For $GL(n)$ or $SL(n)$ we add a single new line.} starting at $J$, together with automorphisms associated to these lines, such that the monodromy of the reglued system around $J$ is the identity, as originally done in \cite{bnr1982new}.  In the setting of \emph{Stokes Graphs} the presence of these new lines is explained in terms of deforming integration contours in \cite{aoki2001exact}.  We consider more general possible intersections in Section \ref{sec: Stokes Data}.

%Check Citation is correct

%\todo{I'm no longer sure we need the following remark -- I seem to have absorbed it} \textcolor{green}{I agree}

%\begin{remark}
%From the point of view of non-abelianization, condition \ref{item:mon_branch} of definition \ref{defin:basic_abstract} is the reason for introducing cameral networks in the first place. After descending to a spectral network on $X$, the product of pre-Stokes factors for lines incident to a branch point will compensate for the monodromy of local systems pushed forward via $\pi:\tilde X\setminus \pi^{-1}(P) \to X\setminus P$, as in \ref{eg:local_monodromy}. As a result, the local system will extend to the branch point.

%The non-commutativity of pre-Stokes factors could introduce monodromy whenever Stokes curves happen to meet away from ramification points. Condition \ref{item:mon_other} ensures that this doesn't happen.
%\end{remark}

In this paper we generalize non-abelianization to an arbitrary reductive algebraic group $G$.  Table \ref{tab: GL(n) vs G} shows some of the replacements we use (these are explained in subsequent sections of the introduction).  Many of these analogues essentially come from the literature on Hitchin systems.  We provide further details in the remainder of this introduction.

We also prove some results about the behaviour of spectral networks arising from points in the Hitchin base, that are not only novel in our setting but also novel for $G=GL(n)$ and $G=SL(n)$.  We review these in Section \ref{subsec: Intro spectral and cameral networks}.  In Section \ref{subsec: introduction 2d-4d wall crossing and non-abelianization} of the introduction we provide a brief review of the physical origin of spectral networks and the non-abelianization map.  This link to BPS states provides another motivation for this work.

\begin{table}
\label{tab: GL(n) vs G}
\begin{tabular}
{ m{6cm} | m{6cm} }
  $GL(n)$ & Arbitrary reductive algebraic group $G$
  \\ \hline
  Spectral cover $\overline{X}\rightarrow X$, an $n:1$ cover. 
  & 
  Cameral cover $\tilde{X}\rightarrow X$; in the cases we consider, a $W$-cover over a dense open $U\subset X$.  
  \\ \hline
  $n$-dimensional local system on $X$ 
  & 
  $G$-local system on $X$  
  \\ \hline
  Hitchin base for $GL(n)$,  
  & 
  Hitchin base for $G$,\\ $\cA:=\oplus_{i=1}^{n}\Gamma(X, \cL^{\otimes i})$ & $\cA:=\Gamma(X, (\cL\times_{\bbG_{m}}\mathfrak{t})\sslash W)$
  \\ \hline
  One dimensional local systems on $\overline{X}\backslash R_{\rho}$, with a condition on the monodromy.
  & 
  $N$-local system $\cE_{N}\rightarrow X\backslash P$, such that $\cE_{N}/T\cong \tilde{X}|_{X\backslash P}$, with a condition on the monodromy.
\end{tabular}
\caption{$GL(n)$ and arbitrary reductive algebraic groups $G$.}
\end{table}

\subsection{Cameral Covers and $N$-local systems}

\emph{Cameral covers} were introduced by \cite{donagi1993decomposition, scognamillo1998elememtary, faltings1993stable} for considering the Hitchin system for arbitrary reductive algebraic groups. Given a reduced unramified $GL(n)$ spectral cover $\overline{X}\rightarrow X$, we can produce a cameral cover as corresponding to the sheaf of sections $\tilde{X}:=\mathcal{H}om^{invertible}_{/X}(X\times \{1,...,n\}, \overline{X})$, that is to say the sheaf of invertible homomorphisms between the sheaves of sets over $X$ given by $\overline{X}$ and $X\times \{1,...,n\}$. That is the sheets of the cameral cover parameterize labellings of the branches of the spectral cover by $1,...,n$. Away from the branch locus the cameral cover is equivalent to a principal $S_{n}$-bundle, with the right $S_{n}$ action coming from precomposition.  This relationship between (unramified) spectral and cameral covers is entirely analogous to the relationship between vector bundles and principal $GL(n)$-bundles.  There is a notion of cameral cover \cite{faltings1993stable, donagi1993decomposition, scognamillo1998elememtary} $\tilde{X}\rightarrow X$ for any reductive algebraic group $G$, and away from the ramification points these correspond to principal $W$-bundles for $W$ the Weyl group of $G$.  See Section \ref{subsec: Hitchin base} for more details.

A one dimensional local system on $\overline{X}\backslash R_{\rho}$ (the complement of the ramification divisor in the spectral cover) can be seen as describing a $N:=N_{GL(n)}(T)$-local system on $X\backslash P$ (where $P$ is the branch locus of $\overline{X}\rightarrow X$), such that the associated $W$-local system $\cE_{N}/T$ is isomorphic to the cameral cover associated to $\overline{X}\backslash \pi^{-1}(P)$ (removing the ramification locus has the effect of restricting the cameral cover to $X\backslash P$).  There is a restriction on which local systems can be described in this way.  This together with one additional restriction can be described via a condition on the monodromy around $P$ which we call the $S$-monodromy condition, and make precise in Section \ref{subsec: alpha monodromy}. This is analogous to the relationship between the spectral and cameral description of Hitchin fibers considered in \cite{donagi1993decomposition, donagi1995spectral, donagi2002gerbe}. The precise link between the S-monodromy condition, and the $SL(n)$ analogue of one dimensional local systems on the spectral cover, with ramification points removed, is provided in Proposition \ref{proposition : Reason for alpha monodromy condition}.

Following \cite{donagi2002gerbe} (which treats bundles, rather than local systems) we can also describe these $N$-local systems in terms of certain $T$-local systems on $\tilde{X}\backslash R$ ($R$ the ramification locus of $\tilde{X}\rightarrow X$) with some additional structure, which we call $N$-shifted weakly $W$-equivariant $T$-local systems.  A straightforward modification of the unramified case of \cite{donagi2002gerbe}, together with consideration of the $S$-monodromy condition shows:

%\todo{Modify statement of the below theorem}

\begin{theorem}[Theorems \ref{thm:flat_DG}, \ref{theorem: connected component of LocT-1}.  The part of this not involving the $S$-monodromy condition is a minor modification of the unramified case of results of Donagi--Gaitsgory in \cite{donagi2002gerbe}.]
\label{thm: flat DG intro}
Consider the unramified map $\pi : \tilde X\backslash R \to X \backslash P$, where we have removed the branch points and preimages of branch points from a cameral cover $
\tilde{X}\rightarrow X$. Then there is an equivalence of categories between:
\begin{enumerate}
\item Weakly $W$-equivariant, $N$-shifted $T$-local systems $\cL$ on $\tilde X \setminus R$ (see Definition \ref{defn: N shifted Weakly W equivariant});
\item $N$-local systems $\cE$ on $X \backslash P$, equipped with an isomorphism $\cE/T \cong \tilde X \setminus R$. of $W$-bundles on $X\backslash P$.
\end{enumerate}
This gives an isomorphism of the corresponding moduli functors (see Definitions \ref{defn: Moduli N shifted weakly W equivariant}, \ref{defin: N local systems corresponding to a cameral cover}):
\begin{equation}
    \label{eq: Map of Moduli spaces_intro}
   \Loc_{N}^{\tilde{X}\backslash R}(X\backslash P) \cong \Loc_{T}^{N}(\tilde{X}\backslash R)
\end{equation}

For $X$ non-compact this fits into a commutative diagram
\[
\begin{tikzcd}
\Loc_{T}^{N, S}(\tilde{X}\backslash R) \arrow{r} \arrow{d}{\cong} &  \Loc_{T}^{N}(\tilde{X}\backslash R) \arrow{d}{\cong}\\
\Loc_{N}^{\tilde{X}, S}(X \backslash P) \arrow{r} & \Loc^{\tilde{X}}_{N}(X\backslash P)
\end{tikzcd}
\]
where $\Loc_{T}^{N, S}(\tilde{X}\backslash R)$ and $\Loc_{N}^{\tilde{X}, S}(X\backslash P)$ are the moduli spaces of $N$-shifted weakly $W$-equivariant $T$-local systems on $\tilde{X}\backslash R$, and of $N$-local systems on $X\backslash P$ corresponding to $\tilde{X}\backslash R$ (respectively), which satisfy the $S$-monodromy condition (see Definitions \ref{alpha monodromy T bundle side}, \ref{moduli alpha monodromy T bundle side}, \ref{def:alpha_monodromy_N}, and \ref{definition: Moduli N bundles S monodromy}).  This condition corresponds to a restriction on the monodromy around either $P\subset X$.
\end{theorem}

We use the equivalence of Equation \ref{eq: Map of Moduli spaces_intro} to see that $\Loc_{T}^{N}(\tilde{X}\backslash R)$ is a stack. 

%We also provide a version of this theorem in which the conditions on the monodromy we need to perform non-abelianization are imposed, see Theorem \ref{thm: Flat DG with alpha monodromy}.

\subsection{Spectral and Cameral Networks}
\label{subsec: Intro spectral and cameral networks}

In this section we revisit the construction of a spectral network (see Construction \ref{constr:wkb_cameral}) from a point in the Hitchin base.  Our construction restricts to that of of Longhi and Park \cite{longhi2016ade} in the cases of $SL(n)$, $SO(2n)$, $E6$ and $E7$, and agrees with that of Gaiotto--Moore--Neitzke for $SL(n)$ or $GL(n)$ \cite{gaiotto2013spectral}, upon choices of local trivializations of the cameral/spectral cover as discussed in Remark \ref{rem:root_noncanonical}.  A point $a$ in the Hitchin base for a compact Riemann surface $X^{c}$ corresponds to a map $X^{c}\xrightarrow{a} \ft_{\cL}\sslash W$ for a line bundle $\cL$ on $X^{c}$ (see Section \ref{subsec: Hitchin base}). The cameral cover associated to $a$ is then defined as the pullback 

\[
\begin{tikzcd}
\tilde{X}^{c} \arrow{r}\arrow{d} & \ft_{\cL}\arrow{d}\\
X^{c} \arrow{r}{a}  & \ft_{\cL}\sslash W.
\end{tikzcd}
\]

Let $X=X^{c}\backslash D$ for a reduced divisor $D$, and $\cL=K_{X^{c}}(D)$.  For each root $\alpha$ we get a real oriented projective vector field $V_{\alpha}$ on $\tilde{X}\backslash R :=\tilde{X}^{c}\times_{X^{c}}X\backslash P$, by taking the vectors that pair with the covector field $\tilde{X}^{c}\rightarrow \ft_{K_{X^{c}}(D)}\xrightarrow{\alpha} K_{X^{c}}(D)$ to produce a positive real number. Furthermore, this assignment of real projective vector fields is $W$-equivariant with respect to the action of $W$ on both $\tilde{X}$ and on roots, in the sense that $w_{*}V_{\alpha}=V_{w(\alpha)}$.

Picking an appropriate point in the Hitchin base, we then iteratively define the \emph{cameral network} $\tilde{\cW}\subset \tilde{X}$ (See Construction \ref{constr:wkb_cameral} and Definition \ref{def:wkb_cameral}) by first drawing trajectories of some of these vector fields $V_{\alpha}$ starting at ramification points.  We call these trajectories Stokes curves or Stokes lines. We then draw new Stokes lines coming from the intersection points of previously drawn trajectories.  More precisely when two or more lines intersect, providing that some conditions are satisfied, we draw trajectories of the vector field $V_{\alpha}$, starting at the intersection point, for each root that is a non-negative integral linear combination of the labels of the intersecting lines.  A new phenomenon that occurs beyond the simply laced case is that we can add multiple new lines at an intersection point of two lines. 

The subset $\tilde{\cW}\subset \tilde{X}$ is invariant under the action of $W$, and hence descends to a subset $\cW\subset X$ which we call a \emph{basic spectral network} (see Definition \ref{def:spectral_net}) if it satisfies certain conditions.  It is not clear which points in the Hitchin base $\tilde{\cW}$ satisfy these conditions, and as such provide a basic spectral network\footnote{In the sense we get a basic abstract cameral network in the sense of Definition \ref{defin:basic_abstract}, and hence an abstract spectral network as in Definition \ref{def:spectral_net}.}; see the discussion around Conjecture \ref{conjecture: open dense set}.  However there are many examples of such networks \cite{gaiotto2013spectral, longhi2016ade, longhiloom}, and some partial results in this direction describing the generic behaviour of individual lines in spectral networks.   We call the networks arising from a point in the Hitchin base WKB networks.

We now outline some results we do have about the behaviour of spectral networks.  We can describe the trajectories on $\tilde{X}$ corresponding to the root $\alpha$ as the pullback of the trajectories of a quadratic differential $\omega_{\alpha}$ (see Construction \ref{construction: Quadratic form alpha}) on $\tilde{X}/s_{\alpha}$.  This allows us to use classical results about trajectories of quadratic differentials to understand individual lines in the cameral network and in the spectral network.

We introduce a restriction on the residues of $a\in \Gamma(X^{c}, \ft_{K_{X^{c}}(D)})$ at the divisor $D$ which we call condition R (Definition \ref{definition: Condition R}).  This corresponds to restricting to an open dense set in the Hitchin base.  We prove the following proposition about the behaviour of lines in the spectral/cameral network near $D$ by reducing it to a result about the trajectories of a quadratic differentials near poles.

\begin{proposition}[Proposition \ref{prop: open sets around D}]
Assume that we start with a point $a$ in the Hitchin base corresponding to a cameral cover $\tilde{X}_{a}\xrightarrow{\pi} X$, which satisfies condition R. Then, for each $d\in D$, there exists an open disc $B_{d}$, with $d\in B_d \subset X^{c}$, such that any line $\ell$ produced in the iterative construction of the cameral network $\tilde{\cW}$ which enters $\pi^{-1}(B_d)$, or is produced as a new Stokes curve inside $\pi^{-1}(B_d)$, 
%associated with any root $\alpha$ (as specified in Construction \ref{constr:wkb_cameral})
will not leave $\pi^{-1}(B_d)$.
\end{proposition}

This means that an analogous statement holds for lines $\ell$ in the associated spectral network $\cW$.

This allows us to consider the ``cutting and gluing'' procedure on $X\backslash \cup_{d\in D}B_{d}$ rather than on $X$.  This allows us to deal with a complication arising where infinitely many new Stokes lines are produced\footnote{However we do not know in what generality we can avoid the complication of there being infinitely many Stokes lines in $X\backslash \cup_{d\in D}B_{d}$.} in $B_{d}$ as shown in Figure \ref{fig:joints_accumulate} and discussed in Remark \ref{remark: infinitely many new Stokes curves}.

%Under a condition we call condition R (definition \ref{definition: Condition R}) we show in proposition \ref{prop: open sets around D} that we can essentially ignore the behaviour of the network sufficiently close to the divisor $D$, where we can have infinitely many new Stokes lines generated as in figure \ref{fig:joints_accumulate}. 

The second result we prove by this method is Proposition \ref{prop: intro SF locus} (Corollary \ref{corollary: behaviour on SF locus}) below; which is a classification of the lines appearing in a cameral network (and hence also the associated spectral network) associated to a point in an open dense\footnote{When $\text{deg}(D)>2$.} subset $\cA_{SF,0}^{\diamondsuit}$ of the Hitchin base (defined in Definition \ref{def: saddle free}).  We do this by reducing it to a result classifying trajectories of quadratic differentials which we recall in Definition \ref{definition: types trajectories}, Proposition \ref{prop: no recurrent or closed} and Proposition \ref{prop: SF dense in each S1 orbit}.  This classification ensures that the individual trajectories of a spectral network associated to $a\in \cA_{SF,0}^{\diamondsuit}$ behave well.

\begin{proposition}[Corollary \ref{corollary: behaviour on SF locus}]
\label{prop: intro SF locus}
For $a\in \cA_{SF,0}^{\diamondsuit}$ (which is a dense open set of the Hitchin base for a connected compact Riemann surface $X^{c}$, the line bundle $K_{X^{c}}(D)$, and $\text{deg}(D)>2$) the image in $X^{c}$ of a Stokes curve labelled by $\alpha$ (for any root $\alpha$) in the WKB construction (\ref{constr:wkb_cameral}) starting with $a$ has one of the following behaviours:
\begin{itemize}
    \item It is a primary Stokes curve, starts at a branch point (of $\tilde{X}_{a}\rightarrow X$), and tends to a point in $D$.
    \item It starts at a joint (an intersection of Stokes curves) and tends to a point in $D$.
    \item It starts at a joint and tends to a branch point.
\end{itemize}
\end{proposition}

Note that if $a$ generates a spectral network, the spectral network consists of the images in $X^{c}$ of Stokes curves on $\tilde{X}^{c}$.

\subsection{Non-abelianization}

In this section we complete our overview of the non-abelianization map.  To complete our ``cutting'' and ``gluing'' procedure along a basic spectral network  (Construction \ref{thm:bulk_map_exists}), we need to define the automorphisms with which we reglue. These are defined iteratively in Construction \ref{construction of swkb}.

 For a fixed local system $\cE_{N}\rightarrow X\backslash P$ we define the automorphisms for the primary Stokes lines so that the monodromy around each point $p\in P$ will be set to zero after the cutting and regluing procedure.  We then iteratively define the automorphisms for all Stokes lines, by requiring that the monodromy around each joint $J\in \cW\subset X$ will be set to zero by the cutting and regluing procedure.  In the iterative step, at an intersection of Stokes lines we need a map from the automorphisms associated to the incoming lines, to the automorphisms associated to the outgoing lines.  The following Theorem (Theorem \ref{thm:assignment_stokes_intersection}) essentially provides such a map:
 
 \begin{theorem}[Theorem \ref{thm:assignment_stokes_intersection}]
Consider an intersection of distinct lines each labelled by a root, such that the labels form a convex set $C_{in}$ (see Definition \ref{def: convex set of roots}) and no two lines are labelled by the same root.  Let $U_{\gamma}$ denote the root subgroup associated to a root $\gamma$.  There is a unique morphism of schemes:
\begin{align*}
\prod_{\gamma \in C_{in}} U_{\gamma} 
&\xrightarrow{s} 
\prod_{\gamma \in C_{out}} U_{\gamma}
\\
(u_\gamma)_{\gamma \in C_{in}} 
&\xmapsto{s}
(u'_\gamma)_{\gamma \in C_{out}}
\end{align*}
where $C_{out}$ is the set of roots that are non-negative integral combinations of $C_{in}$, such that for any $(u_\gamma)_{\gamma \in C_{in}} $ the product of all Stokes factors ($u_{\gamma}^{\pm 1}$ and $(u_{\gamma}')^{\pm 1}$) on all lines intersected as one moves in a circle around the intersection point is the identity (see Theorem \ref{thm:assignment_stokes_intersection} for a more precise statement, which includes how the orientation determines if one uses $u_{\gamma}$ or $u_{\gamma}^{-1}$).
%Furthermore for any decoration $(u_\gamma)_{\gamma \in C_{in}}$, $s((u_\gamma)_{\gamma \in C_{in}})$ is the unique set of Stokes factors for the outgoing rays such that $u_{C_{out}} = \id$, using the notation of equation \ref{eq:u_cout}, and the specified decorations on the incoming rays.
\end{theorem}

The map $s$ gives a map from the automorphisms on the incoming intersecting Stokes lines (identified via a trivialization with elements of some root subgroup $U_{\gamma}\subset G$) to the automorphisms on the Stokes lines leaving the intersection.  We show this map does not depend on the trivialization chosen in Lemma \ref{lem:assignment_stokes_newlines_equivariant}.

This gives our main result: the definition of a nonabelianization map from appropriate $N$-local systems on $X\backslash P$, to $G$-local systems on $X$:

\begin{construction}
[Construction \ref{thm:bulk_map_exists}]
The data of a basic abstract spectral network\footnote{See Definition \ref{def:spectral_net}.  In particular this includes the case of a basic WKB spectral network, as in Definitions \ref{def:wkb_cameral} and \ref{def:spectral_net}, associated to a point of the Hitchin base.} determines a morphism of algebraic stacks:
\[
\begin{tikzcd}
\Loc_N^{\tilde X \backslash R,S}(X\backslash P)
\arrow{r}{\nonab} &
\Loc_G(X).
\end{tikzcd}
\]
Here $\Loc_N^{\tilde X\backslash R ,S}(X\backslash P)$ is the moduli space of $N$-local systems on  $X\backslash P$ corresponding to a given cameral cover $\tilde{X}\backslash R$, together with a restriction on the monodromy around the points $P$ (see Section \ref{subsec: alpha monodromy}).  These local systems are an analogue of $\bbG_{m}$-local systems on the spectral cover, as described in Table \ref{tab: GL(n) vs G}.
\end{construction}

Composing this morphism with the identification of Theorem \ref{thm: flat DG intro} gives a map

\[
\Loc_{T}^{N, S}(\tilde{X}\backslash R)\rightarrow \Loc_G(X)
\]
from the moduli space of $N$-shifted weakly $W$-equivariant $T$-local systems on $\tilde{X}\backslash P$ satisfying the $S$-monodromy condition to the moduli of $G$-local systems on $X$.  This allows definition of complex analytic coordinates on $\Loc_{G}(X)$ by taking the monodromy of the $T$-local systems.

This map agrees with that defined by Gaiotto, Moore and Neitzke in \cite{gaiotto2013spectral} for $G=GL(n)$ or $G=SL(n)$,  as will be discussed in Section \ref{subsec: intro Spectral decomposition for representations}.

Proving whether or not this map has the properties conjectured in \cite{gaiotto2013spectral} for the case $G=SL(n)$ remains open for $n>2$.  In particular \cite{gaiotto2013spectral} conjectures that for $G=SL(n)$, when restricted to (holomorphic) symplectic leaves, a minor modification\footnote{This modification is that of assigning Borel structures near $D$, as is outlined in Section \ref{subsubsec: Borel structures arbitrary G} under appropriate conditions on the spectral network.} of this map is symplectic, and is an isomorphism on the zeroth cohomology of the tangent complex.

%Unfortunately we can not yet prove that it has the conjectured properties (essentially conjectured for $SL(n)$, in \cite{gaiotto2013spectral, gaiotto2013wallHitchin, gaiotto2010four}) of its restriction to holomorphic symplectic leaves, and under some additional assumptions, being hyperkahler, and an isomorphism on the zeroth degree cohomology of the tangent complex.

\subsection{Spectral Decomposition for Representations}
\label{subsec: intro Spectral decomposition for representations}

Given a representation $\rho:G\rightarrow GL(V)$ we can associate to a point in the Hitchin base (for $G$), a spectral cover $\overline{X}_{\rho}\rightarrow X$ \cite{donagi1993decomposition}.  There is a version of the non-abelianization construction in the case where $\rho$ is minuscule, and $G$ is simply laced via the path detour rules of \cite{gaiotto2013spectral, longhi2016ade}. 

We also provide explicit descriptions of the moduli space $\Loc_N^{\tilde X,S}(X)$ in terms of one dimensional local systems on the spectral cover $\overline{X}_{\rho}$ in the case where $\rho$ is the defining representation of one of the classical groups $GL(n)$, $SL(n)$, $Sp(2n)$, $SO(2n)$, or $SO(2n+1)$. 

\subsubsection{Minuscule Representations}

In the case where $\rho$ is minuscule, and $G$ is simply laced, a description of the image under $\rho$ of the Stokes factors is given by the path detour rules of \cite{gaiotto2013spectral, longhi2016ade}. The path detour rules describe the image in terms of parallel transport in the $\bbG_{m}$-local system on the spectral cover associated to an appropriate $N$-local system by $\rho$. In particular, there is a set of lifts $\overline{\gamma_{i}}$, of a loop around $p\in X$, such that the Stokes factors of initial\footnote{Those Stokes lines starting at a branch point of the cover $\tilde{X}\rightarrow X$.} Stokes lines can be identified as the exponential of the sums of the parallel transport operators along these paths, as is described in Construction \ref{construction: exponential path rule non abelianization}.   Alternatively the Stokes factors can be described as the sum of a automorphisms given by parallel transport along certain paths.

As this exponential can be defined for arbitrary local systems on the spectral cover associated to $\rho$ (and a fixed cameral cover), the above association of Stokes data gives a variant of non-abelianization (Construction \ref{construction: exponential path rule non abelianization})

\[\Loc_{\bbG_{m}}^{-1}(\overline{X}_{\rho}^{ne}\backslash R_{\rho})\xrightarrow{nonab_{PD}}  \Loc_{GL(V)}(X),\]

where $\overline{X}^{ne}_{\rho}\rightarrow X$ is a modification of the spectral cover called the non-embedded spectral cover (Definition \ref{defn: non-embedded spectral cover}), $R_{\rho}$ is it's ramification locus, and $\Loc_{\bbG_{m}}^{-1}(\overline{X}^{ne}_{\rho}\backslash R_{\rho})$ is the moduli space of $\bbG_{m}$-local systems on $\overline{X}^{ne}_{\rho}\backslash R_{\rho}$ with the property that they have monodromy $-1$ around each ramification point $r\in R_{\rho}$.

This version of non-abelianization is compatible with that of in Construction \ref{thm:bulk_map_exists}, as shown by the following theorem:

\begin{theorem}[Theorem \ref{conj: exp path rules minuscule}]
Let $G$ be a simply laced reductive algebraic group and $\rho:G\rightarrow GL(V)$ a minuscule representation.  The following diagram commutes:
\begin{equation}
\label{eq:nonab_pathrule_commute_intro}
\begin{tikzcd}
\Loc_{N}^{\tilde{X}\backslash R, S}(X\backslash P) \arrow{r}{Sp^{S}_{\rho}} \arrow{d}{\nonab} &  \Loc_{\bbG_{m}}^{-1}(\overline{X}_{\rho}\backslash R_{\rho}) \arrow{d}{\nonab_{PD}}\\
\Loc_{G}(X)\arrow{r}{\rho} & \Loc_{GL(V)}(X),
\end{tikzcd}
\end{equation}
where $Sp^{S}_{\rho}$ is a modification of the map of Construction \ref{construction: associated local system of vector bundles on spectral}.

%where $\Loc_{N}^{\tilde{X}\backslash R, S}(X\backslash P)$ is the moduli space of $N$-local systems on  $X\backslash P$ corresponding to a given cameral cover, together with a restriction on the monodromy around the points $P$, $\Loc_{\bbG_{m}}^{-1}(\overline{X}_{\rho}\backslash R_{\rho})$ is a moduli of one dimensional local systems (equivalently $\bbG_{m}$-local systems), satisfying a restriction on their monodromy, on a version of the spectral cover\footnote{Called the non-embedded spectral cover in Definition \ref{defn: non-embedded spectral cover}.} associated to $\tilde{X}$, and the map $\nonab_{PD}$ is the version of the non-abelianization map where we glue by the Stokes factors prescribed by the monodromy around the branch points, as in Construction \ref{construction: exponential path rule non abelianization}. 
\end{theorem}

For $G=GL(n)$, $G=SL(n)$, or $G=SO(2n)$ and the defining representation we give a more precise relationship between the maps $nonab$ and $nonab_{PD}$ in Section \ref{subsec: equivalence for GL(n) and SL(n)}.  For $GL(n)$ this statement is that there is a commutative diagram (see Equation \ref{eq:nonab_pathrule_identify_GL(n)}):

\begin{equation}
\begin{tikzcd}
\Loc_{N}^{\tilde{X}\backslash R, S}(X\backslash P) \arrow{rr}{\cong} \arrow{rd}{nonab} & & \Loc_{\bbG_{m}}^{-1}(\overline{X}_{\rho}^{ne}\backslash R_{\rho}) \arrow{ld}{nonab_{PD}}\\
 & \Loc_{GL(n)}(X). &
\end{tikzcd}
\end{equation}

%In type A, for basic WKB spectral networks, this method of determining Stokes factors agrees with the path detour rule of \cite{gaiotto2013spectral}. \todo{There is a note saying : This is slightly nonsenseical -- be clear what the lemma says (?)}

\subsubsection{Spectral Data and Classical Groups}

%In the case of arbitrary reductive algebraic groups $G$, and a representation $\rho$, such that its restriction to the normalizer of a torus, $\rho|_{N}$ is faithful we can describe the aforementioned moduli of $N$-local systems corresponding to a given cameral cover in terms of one dimensional local systems on the (non-embedded) spectral cover $\overline{X}_{\rho}^{ne}$ (which is birational to the usual spectral cover).

%To $\rho$, $\tilde{X}$, and some additional choices we can associate another cover $\overline{X}_{K}$, equipped with a one dimensional local system $\cL_{K}$, with the property:

%\todo{Needs to be changed}

%\begin{theorem}[Theorem \ref{Claim: Description of N bundles in spectral terms}]
%For $\rho$, $\overline{X}_{K}$, and some additional choices as above (see section \ref{subsec: faithful N representations} for specifics), and $\tilde{X}\rightarrow X$ unramified there is an isomorphism of stacks:
%\[
%\Loc_{N}^{\tilde{X}}(X)
%\xrightarrow{\cong} 
%\Loc_{\bbG_{m}}(\overline{X}_{\rho}^{ne})\times_{\Loc_{\bbG_{m}}'(\overline{X}_{K})}\{\cL_{K}\},
%\]
%where $\cL_{K}$ denotes a particular point of $\Loc_{\bbG_{m}}'(\overline{X}_{K})$ which is a modification of the moduli of $\bbG_{m}$-local systems on $\overline{X}_{K}$ defined in definition \ref{definition: Loc'}.
%\end{theorem}

%In the case where $\tilde{X}\rightarrow X$ is ramified, there is a version of this theorem which includes the imposition of the necessary monodromy conditions.

% \todo{Include the appropriate version of this}

  For the classical groups $GL(n)$, $SL(n)$, $SO(2n)$, $Sp(2n)$, and $SO(2n+1)$, we define in Section \ref{subsec: explicit descrtions for defining representations of classical groups} moduli spaces of $\bbG_{m}$-local systems (or equivalently local systems of one dimensional vector spaces) on the spectral curve, together with additional data, which are isomorphic to the moduli spaces of $N$-local systems $\Loc_{N}^{\tilde{X}\backslash R, S}(X\backslash P)$ considered in non-abelianization.

For $GL(n)$ and $SL(n)$, these are given in the following proposition:

\begin{proposition}[Spectral descriptions, see Propositions \ref{prop: spectral cameral GL(n)}, \ref{prop: SL(n) without alpha momodromy}, \ref{proposition : Reason for alpha monodromy condition}, \ref{prop: alpha monodromy GL(n)}]
For the groups $GL(n)$ and $SL(n)$, the spaces $\Loc_{N}^{\tilde{X}\backslash R}(X\backslash P)$ and $\Loc_{N}^{\tilde{X}\backslash R, S}(X\backslash P)$ are isomorphic to those given in table \ref{tab: Classical Groups Spectral descriptions}.

Here $\pi_{\rho}:\overline{X}\rightarrow X$ is the associated spectral cover, and $R_{\rho}$ and $P$ are the ramification and branch divisors respectively.  The map $\Loc_{\bbG_{m}}(\overline{X}\backslash R_{\rho})\rightarrow (\bbG_{m}/\bbG_{m})^{\# R_{\rho}}$ comes from restricting to loops around points in $R_{\rho}$.

\begin{table}[h]
\tiny
    \centering
    \begin{tabular}{c|c|c}
    Group     &  $\Loc_{N}^{\tilde{X}\backslash R}(X\backslash P)$ &  $\Loc_{N}^{\tilde{X}\backslash R, S}(X\backslash P)$ \\ \hline
       $GL(n)$  & $\Loc_{\bbG_{m}}\big(\overline{X}\backslash \pi_{\rho}^{-1}(P)\big)$ & $\Loc_{\bbG_{m}}(\overline{X}\backslash R_{\rho})\times_{(\bbG_{m}/\bbG_{m})^{\# R_{\rho}}}(\{-1\}/\bbG_{m})^{\# R_{\rho}}$\\
       $SL(n)$ & $\Loc_{\bbG_{m}}\big(\overline{X}\backslash \pi_{\rho}^{-1}(P)\big)\times_{\Loc_{\bbG_{m}}(X\backslash P)}\{\underline{\bbC}_{X\backslash P}\}$ & $\Loc_{\bbG_{m}}(\overline{X}\backslash R_{\rho})\times_{\Loc_{\bbG_{m}}(X\backslash P)}\{\underline{\bbC}_{X\backslash P}\}$ \\
    \end{tabular}
        \caption{Spectral descriptions of moduli spaces}
    \label{tab: Classical Groups Spectral descriptions}
\end{table}

\end{proposition}

\subsection{2d-4d Wall crossing and Non-abelianization}
\label{subsec: introduction 2d-4d wall crossing and non-abelianization}

This section describes our understanding of how spectral networks arose in the series of works \cite{gaiotto2012wall2d4d, gaiotto2013spectral, gaiotto2013wallHitchin, gaiotto2010four, gaiotto2013framed}.  It is not required for the rest of the paper.  We stress that much of this work remains at the physical level of rigour, and we recommend that the interested reader consults the original references, or the references for mathematicians \cite{neitzke2014notes, neitzkemetric}.

The paper \cite{gaiotto2010four} considers the wall crossing of BPS states in an $\mathcal{N}=2$, $d=4$ field theory. In this context they interpret the Kontsevich--Soibelman wall crossing formula\footnote{Strictly speaking the wall crossing formula of \cite{gaiotto2010four} is slightly different to the wall crossing formula of Kontsevich--Soibelman for the formal torus algebra.  Mathematically this corresponds to how \cite{filippini2017stability} adds an extra parameter to the isomonodromy problem of \cite{bridgeland2012stability}, and that while Kontsevich--Soibelman work with a formal torus algebra, \cite{gaiotto2010four} does not work formally.  See the papers \cite{bridgeland2019geometry, bridgeland2019riemann, barbieri2019quantized} for more about working non-formally.} of \cite{kontsevich2010cohomological, kontsevich2008stability, joyce2012theory} as a consistency condition for a set of ``coordinates" $\cM_{\zeta} \supset U \xrightarrow{\chi_{\zeta, a}} (\mathbb{C}^{\times})^{n}$,  where $\cM_{\zeta}$ is the fiber over $\zeta\in \bbP^{1}$ of the twistor space associated to the hyperk\"{a}hler 3d Coulomb branch of the field theory compactified on $S^{1}$.  These coordinate charts are associated to generic points $a$ in the 4d Coulomb branch $\cA$, and a twistor parameter $\zeta\in \bbP^{1}\backslash \{0.\infty\}$.  We recall that there is the structure of an integrable system on the map  $\cM\rightarrow \cA$ \cite{seiberg1994monopoles, seiberg1994electric} (see also \cite{donagi1996supersymmetric}).  The coordinate charts are meant to respect the holomorphic symplectic structure of the 3d Coulomb branch, in the sense explained in the second and third to last paragraphs of this section.  Recall that the Coulomb branch has a hyperk{\"a}hler structure because at the low energy limit the theory produced by compactification on $S^{1}$ is a $3d$ $\cN=4$ $\sigma$-model with the Coulomb branch as target.  Analysis in \cite{alvarez1981geometrical} shows that the targets of a 3d $\cN=4$ $\sigma$-models must be hyperk{\"a}hler.  There are differential equations describing how these coordinate charts vary as we move in the base $\cA$, and as we vary $\zeta$.  

Consider the case of a field theory of class $S$ produced by compactifying the $6d$ $(0,2)$ superconformal field theory on the Riemann surface $X$ \cite{gaiotto2012n}.  Then $\cM\rightarrow \cA$ is the Hitchin fibration for the curve $X$ and a group $G$ of type $ADE$ \cite{gaiotto2013wallHitchin}.  The twistor space is then the Deligne twistor space (e.g. \cite[\S 4]{simpson1996hodge}) which is the moduli space of $\zeta$-connections.

The equations describing variation of the coordinate charts\footnote{These coordinate charts are coordinate charts for a coarse or rigidified moduli space.} are analogous to the $tt^{*}$-equations of \cite{dubrovin1993geometry}, in particular having an irregular singularity and thus exhibiting Stokes phenomena around $\zeta=0$.  If one ignores potential mathematical difficulties coming from the infinite dimensionality of various spaces one can construct \cite{gaiotto2010four} a family of solutions via an integrodifferential equation, which is analogous to the sectorial solutions of the $tt^{*}$-equations in \cite{dubrovin1993geometry}.  These solutions will be defined on sectors of the form $\{\zeta=re^{i\theta}| \theta_{1} < \theta < \theta_{2}\}$, and the solutions on two adjacent sectors differ on the common boundary by the action of a Stokes factor which is a bimeromorphic map $(\bbC^{\times})^{n}\dashrightarrow (\bbC^{\times})^{n}$.  These Stokes phenomena provide one reason that the coordinate charts for $\zeta\in \bbP^{1}\backslash 0$ do not provide a map on twistor sections, and thus these coordinate charts are not hyperk\"{a}hler.  However \cite{gaiotto2010four} describes how to reconstruct the hyperk\"{a}hler metric using these coordinates, and furthermore describes this metric as a modification of the semiflat metric on the Hitchin moduli space by ``quantum corrections."  This leads to predictions about the asymptotic (in the Hitchin base) difference between the semiflat metric and the hyperk\"{a}hler metric, versions of which are proved in \cite{mazzeo2019asymptotic, dumas2019asymptotics, fredrickson2020asymptotic, fredrickson2020exponential}.  Furthermore for $G=SL(2)$ the paper \cite{ott2020higgs} finds an asymptotic equivalence between semiflat coordinate charts and the coordinate charts described above, described as shear--bend coordinates.  The paper \cite{dumas2020opers} contains some numerical evidence for the conjecture that one can construct the hyperk\"{a}hler structure on the Hitchin moduli space in this way, and the conjectures regarding the holomorphic symplectic nature of these coordinates for $SL(2)$ and $SL(3)$.

Interpreting the Kontsevich--Soibelman wall crossing formula in terms of Stokes factors of an isomonodromic connection appears in a mathematically rigorous fashion in \cite{bridgeland2012stability}, see also \cite{filippini2017stability, bridgeland2019riemann, bridgeland2019geometry, barbieri2019quantized}.  In particular \cite{bridgeland2019riemann, bridgeland2019geometry, barbieri2019quantized} consider isomonodromy problems that are non-formal, in the sense they work with the torus algebra\footnote{Quantized in the case of \cite{barbieri2019quantized}.} rather than the formal torus algebra.  The papers \cite{bridgeland2020monodromy, allegretti2019stability} are particularly relevant as they consider the isomonodromy problems related to $PGL_{2}$-local systems.

The paper \cite{gaiotto2012wall2d4d} considers the situation where we have a 4d field theory as above, together with a family of 2d surface operators $S_{m}$, for $m \in M.$  This produces a family of vector bundles $E \rightarrow M\times \cM_{\zeta}$.  Restricting to $\zeta=1$ the resultant vector bundle $E$ has a flat connection relative to $\cM_{1}$.

In this case wall crossing occurs for families of sections $\cM\rightarrow E_{m}$ associated to points of $(a,m)\in \cA\times M$.  Again \cite{gaiotto2012wall2d4d} interprets this in terms of an isomonodromy problem which describes such families of sections.  

Consider now theories of class S.  For a given surface operator of the $6d$ theory, and a point $x\in X$ we get a surface operator in the 4d theory on a manifold $N$ by considering the theory on $N\times X$,  with a surface operator $S\times \{x\}\subset N\times X$ for $S\subset N$ a surface of dimension 2.  In type $A$ there is a canonical surface defect of the 6d theory \cite{alday2010loop, gaiotto2012n}, giving a family of surface defects of the 4d theory parametrized by $X$ such that the associated family of vector bundles on $\cM_{\zeta}\times X$ is the tautological bundle \cite[\S 7]{gaiotto2012wall2d4d}.  See \cite{longhi2016ade, longhi2017ade} for analysis of canonical surface defects in types $D$ and $E$. 

This gives a new way to construct charts $Y \rightarrow \cM_{1}=\Loc_{G}(X)$, for $Y$ complex analytically a $Z(G)$ gerbe over $(\bbC^{\times})^{m}$.  Namely for $a\in \cA$ we consider the 2d-4d wall crossing restricted to $\{a\}\times X\hookrightarrow \cA\times X$, and to $\zeta=1$.  To describe a map $Y \rightarrow \Loc_{G}(X)$ it is equivalent by the universal property to describe a $G$-bundle $E\rightarrow X\times Y$ which is a local system in the $X$-directions.  Let $\cW_{a}\subset \{a\}\times X$ be the restriction of the 2d-4d walls to $\{a\}\times X$.  We can describe the local system $E\rightarrow (X\backslash \cW_{a})\times Y$ as an $N=N_{G}(T)$-local system\footnote{With fixed associated $W$-bundle.}, which is ``glued'' together along $\cW$ to give a $G$-local system.  For $SL(n)$ this procedure gives the chart as a map from a moduli space of certain $N$-local systems (interpreted as a moduli of certain one dimensional local systems on a connected spectral cover $\overline{X}_{a}$ and playing the role of $Y$ in the above discussion) to $\Loc_{G}(X)$.  In \cite{gaiotto2013spectral} a network $\cW$ on $X$ with which to perform the above ``regluing'' procedure is defined iteratively; this procedure is conjectured to give the restriction of the 2d-4d wall crossing problem's walls to $\{a\}\times X\subset \cA\times X$.  A modification of these maps is conjectured to be holomorphic symplectic\footnote{Namely the modification that also gives a reduction of structure to a Borel around each point $d\in D\subset X^{c}$.  See Section \ref{subsubsec: Borel structures arbitrary G} for how to perform this modification.} when restricted to holomorphic symplectic leaves.  As mentioned earlier this is known in the case of $G=SL(2)$ by the identification with Fock--Goncharov coordinates \cite{gaiotto2013wallHitchin, fenyes2015dynamical, hollands2016spectral}. 

These are related to the original coordinates of \cite{gaiotto2010four} by the conjecture that, roughly, $\chi_{1,a}$ is left inverse to the coordinate chart constructed above.  A precise conjecture is that the composition of the above map with $\chi_{1,a}$ is equal to the rigidification $Y\rightarrow (\bbC^{\times})^{m}$.

In the case $G=SL(2)$ this procedure is considered without using the framework of 2d-4d wall crossing in \cite{gaiotto2013wallHitchin} by considering the asymptotics of $\zeta$-connections as $\zeta \rightarrow 0$ at all points $x\in X$.

%\todo{Now throw some shade on how the network isn't really defined properly}

%\todo{Give the three isomonodromy problems perspective}

%\todo{also in this section (or otherwise) give an overview of the conjectural properties of these coordinates, the predictions on the metric that have been verified, and the relation to Fock--Goncharov coordinates.}

\subsection{Outline}

In Section \ref{sec: Moduli spaces background} we review the moduli of local systems, spectral covers and cameral covers. We also provide a brief overview of the Hitchin moduli space, even though this is strictly speaking unnecessary for the sequel.

In Section \ref{sect:prelim} we derive various identities in Lie algebras and Lie groups that are necessary for the paper, and describe 2d scattering diagrams (Definition \ref{def:scattering_diagram}).  The 2d scattering diagrams describe locally what happens where Stokes lines in cameral or spectral networks intersect, including how to assign Stokes factors to the new Stokes lines in Theorem \ref{thm:assignment_stokes_intersection}.

In Section \ref{sec: Cameral Networks} we introduce basic spectral and cameral networks for arbitrary reductive algebraic groups $G$.  In contrast to preceding work, we do not use trivializations of the spectral or cameral cover.  We show that each lines in the spectral network is the image of a trajectory of a quadratic differential on a cover $\tilde{X}/s_{\alpha}\rightarrow X$ (for some root $\alpha$).  We use this to describe what WKB spectral networks look like near $d\in D$ in Section \ref{subsec: trajectories near D} for spectral networks corresponding to a point in an open dense subset of the Hitchin base that we denote $\cA^{\diamondsuit}_{R}$.  We also show that for spectral networks corresponding to a point in an open dense subset $\cA_{SF,0}^{\diamondsuit}$ of the Hitchin base every line in the spectral network is of one of three types specified in Corollary \ref{corollary: behaviour on SF locus}.  We note in particular that each of these types of line has a finite set of limit points.

In Section \ref{sec: flat Donagi Gaitsgory} we introduce various moduli spaces of $N$-local systems.  We also describe these in terms of $T$-local systems with additional structure on the cameral cover.  We define the non-abelianization map from $N$-local systems on $X\backslash P$, with fixed associated $W$ local system, and satisfying the $S$-monodromy condition to $G$-local systems on $X$.  From the definition it is manifest that the nonabelianization map is an algebraic map between the relevant moduli spaces.

In Section \ref{sec: Spectral Decomposition} we reinterpret our results on nonabelianization in terms of spectral covers associated to various representations.  In Section \ref{sec: exponential path rules} we explain for groups with simply laced Lie algebras and minuscule representations the relation between the non-abelianization map and the path detour rules of \cite{gaiotto2013spectral, longhi2016ade}. %\todo{May need modification}For arbitrary representations that are faithful when restricted to $N\subset G$ we describe the moduli spaces of $N$-local systems in terms of $\bbG_{m}$ local systems on the associated spectral cover in \ref{subsec: faithful N representations}.  
For the classical groups we provide explicit descriptions in \ref{subsec: explicit descrtions for defining representations of classical groups} of the various  moduli space of $N$-local systems introduced in Section \ref{sec: flat Donagi Gaitsgory} in terms of moduli spaces of one dimensional local systems on the spectral cover associated to the defining representation.

In appendix \ref{sect:planar} we provide explicit computations for the Stokes factors assigned to outgoing Stokes lines from an intersection where Stokes lines labelled\footnote{When one fixes a branch of the cameral cover.} by roots in a real 2-dimensional subspace of $\ft$ intersect.  In the $GL(n)$ case this corresponds to the Cecotti--Vafa wall crossing formula of \cite{cecotti1993classification}.

\subsection{Acknowledgements}
%\todo{To be written}

We would like to thank R. Donagi and T. Pantev for extensive discussions.  We would like to thank A. Fenyes, J. Hilburn, S. Lee, A. Neitzke for discussions on this subject, and P. Longhi for answering some questions on \cite{longhi2016ade}.

We would like to thank A. Neitzke for comments on Section \ref{subsec: introduction 2d-4d wall crossing and non-abelianization}, and R. Donagi for reviewing the results of the paper with us.

%B. Morrissey would like to thank S. Tavener and N. Eckardt.
The authors were partially supported by NSF grants DMS 2001673, and 1901876, and by Simons HMS Collaboration grant \#347070, and \#390287.

Some of this material also appears in the authors' theses \cite{ionita2020spectral, morrissey2020nonabelianization}.

\subsection{Conventions, and Notation}

%\todo{Is this right -- I think we're fine with complex algebraic?} All geometric spaces are to be considered as complex analytic, unless specified otherwise.
 We use a non-standard convention for Chevalley bases introduced in Definition \ref{defin:chevalley_basis}.

When working with the moduli of local systems we work in the setting of derived Artin stacks over the complex numbers. In particular all pushforwards, fiber products, and mapping stacks should be understood to be in this setting.  Having said that we do not use this in a particularly serious way, see Remark \ref{remark: where derived geometry is used} for a precise overview of where this is necessary.

While we are mostly working within the setting of algebraic geometry, we frequently consider the Hitchin base $\cA$ and the curve $X$ as complex manifolds, particularly in Section \ref{sec: Cameral Networks}.  Sometimes we treat the curve $X$ as a purely topological object.  We hope it is clear from context when this is happening.

 Our notion of a basic spectral network does not cover all cases of interest, as noted in Remark \ref{remark: What we miss}.  As there is not a rigorous definition covering all networks of interest, when we use the words spectral network we are either being informal, referencing basic abstract spectral networks, or referencing basic WKB spectral networks.  All rigorous statements involving spectral networks, will be about basic abstract spectral networks, or basic WKB spectral networks in the sense of Definition \ref{def:spectral_net}.
 
\subsubsection{Riemann Surfaces, Covers and Spectral networks}

\begin{itemize}
    \item $X$ -- a non-compact Riemann surface, $X=X^{c}\backslash D$, for $X^{c}$ a compact Riemann surface, and $D\subset X^{c}$ a reduced divisor.
    \item $X^{c}$ -- Compact Riemann Surface.
    \item $\ft_{\cL}:=\cL\times_{\bbG_{m}}\ft$, for $\cL$ a line bundle on $X^{c}$ or $X$.
    \item $\cA$ denotes the Hitchin base, see Definition \ref{def: Hitchin base}.
    \item $\cA^{\diamondsuit}\subset \cA$ is the subset of the Hitchin base corresponding to points which define a smooth cameral cover, see Definition \ref{definition:  cA diamondsuit} and Proposition \ref{prop: diamond suit implies smooth}.
    \item $\cA_{R}^{\diamondsuit}$ is the subset of $\cA^{\diamondsuit}$ corresponding to points satisfying condition R, see Definition \ref{definition: Condition R}.
    \item $\cA_{SF}^{\diamondsuit}$ and $\cA_{SF,0}^{\diamondsuit}$ are subsets $\cA^{\diamondsuit}$ of the Hitchin base consisting of saddle-free points, see Definition \ref{def: saddle free}.
   \item $\pi: \tilde{X}\rightarrow X$ -- a cameral cover of $X$ which is the restriction of a cameral cover $\tilde{X}^{c}\rightarrow X^{c}$, see Definition \ref{def: cameral cover}.
    \item For $\pi: \tilde{X}\rightarrow X$ a cameral cover, we denote by $R\subset \tilde{X}$ the ramification points, $P\subset X$ the branch points.
    \item $\chi_{\alpha}$ is a meromorphic one form on $\tilde{X}^{c}$ defined in Construction \ref{def:form_root}.
    \item $V_{\alpha}$ is an oriented projective vector field on $\tilde{X}\backslash R$ defined in Construction \ref{constr:wkb_cameral} part \ref{item: initialization}.
    \item $\omega_{\alpha}$ is a meromorphic quadratic form on $\tilde{X}_{\alpha}^{c}:=\tilde{X}^{c}\sslash <s_{\alpha}>$ defined in Construction \ref{construction: Quadratic form alpha}.
    \item $\pi: \overline{X}\rightarrow X$ is a spectral cover/curve, see Definition \ref{defn: spectral cover}.
    \item $X^{\circ}:= ReBl_{P}(X)$ is the oriented real blow up of $X$ along $P$.
    \item $\tilde{X}^{\circ}\rightarrow X^{\circ}$, $\tilde{X}^{\circ}$ is the oriented real blow up of $\tilde{X}$ along $R$.
    \item $X^{\circ '}=X^{\circ}\backslash \cup_{d\in D} B_{d}$, where $d\in B_{d}\subset X^{\circ}$ is an open set contracting to $\{d\}$ defined in Section \ref{subsec: trajectories near D}.
    \item $\tilde{X}^{\circ '}:=\tilde{X^{\circ}}\times_{X^{\circ}}X^{\circ '}$.
    \item $S^{1}_{p}\subset X^{\circ}$ is the preimage of $p\in P$ under $X^{\circ}\rightarrow X$.
    \item $x_{p}\in S^{1}_{p}$ is a choice of point. %Used A LOT! especially in S6.2
    \item $\gamma_{p}$ is a path in $S^{1}_{p}$ corresponding to a generator of its fundamental group, in the direction opposite to that specified by the orientation of $S^{1}_{p}$ (which is induced by the orientation of $X$).
    \item $\Lambda_{p}$ -- the set of roots $\alpha$ such that some trivialization of $\tilde{X}^{\circ}$ at $x_{p}$ identifies the monodromy of $\tilde{X}^{\circ}$ around $\gamma_{p}$ with $s_{\alpha}$.
    \item $\pi: \overline{X}^{ne}_{\rho}\rightarrow X$ the non-embedded spectral cover of $X$ associated to a representation $\rho$ of $G$, see Definition \ref{defn: non-embedded spectral cover}.  There are also variants $\pi: \overline{X}^{ne, \circ_{P}}_{\rho}\rightarrow X^{\circ}$, and $\pi: \overline{X}^{ne, \circ_{R}}_{\rho}\rightarrow X^{\circ}$ (see Definition \ref{definition: blown up spectral covers}).% and  $\pi: \overline{X}^{ne,r, \circ_{R}}_{\rho}\rightarrow X^{\circ}$
    \item $\overline{X}^{ne, \circ '}_{\rho}:=\overline{X}^{ne, \circ}_{\rho}\times_{X^{\circ}}X^{\circ '}$. %similarly  $\overline{X}^{ne, r, \circ '}_{\rho}:=\overline{X}^{ne,r, \circ}_{\rho}\times_{X^{\circ}}X^{\circ '}$.
    \item $\tilde{\cW}\subset \tilde{X}^{\circ'}$ is a basic cameral network, see Definition \ref{defin:basic_abstract} and Construction \ref{constr:wkb_cameral}.
    \item $\cW\subset X^{\circ'}$ is a basic spectral network as in Definition \ref{def:spectral_net}.
    %\item $\cW'=\cW\cap X^{\circ '}$.\todo{define these in main text}
        \item $\tilde{\cJ}$ is the set of joints of a cameral network $\cW$.  Defined inside Definition \ref{defin:basic_abstract}.
    \item $\cJ$ is the set of joints of a spectral network $\cW$.  Defined inside Construction \ref{defin:basic_abstract}.
   % \item $\cJ':=\cJ\cap X^{\circ '}$.
    \item $\ell \subset \cW$ -- is a Stokes line (or Stokes curve) in either a cameral or a spectral network.
    \item $c \subset \cW\backslash \cJ$ -- is a connected component of $\cW\backslash \cJ$.
\end{itemize}

\subsubsection{Algebraic Groups}
\begin{itemize}
    \item $G$ is a reductive algebraic group over $\C$.
    \item $T\subset G$ is a maximal torus.
    \item $T_{\alpha}\hookrightarrow T$ is a subgroup of the maximal torus introduced  immediately after Equation \ref{equation: Orthogonal decomposition of t}.
    \item $N=N_{G}(T)$ is the normalizer of a maximal torus.
    \item $n_{\alpha}\in N$ is an element of $N_{G}(T)$ defined in Equation \ref{eq:preferred_lift}.
    \item $W$ is the Weyl group of a group $G$.
    \item $1_{W}$ is the identity of $W$.
    \item $q:N\rightarrow W$ is the map given by quotienting by $T$.
    \item $\fg=\text{Lie}(G)$, $\ft=\text{Lie}(T)$ are the Lie algebras of $G$ and $T$ respectively.  We sometimes denote $\ft$ by $\ft_{G}$.
    \item $\ft^{\vee}$ or $\ft_{G}^{\vee}$ -- denote the dual of $\ft$.
    \item $\Phi\subset \ft^{\vee}$ is the root system of $G$.
    \item $\Lambda_{char}$ or $\Lambda_{char, G}$ denote the character lattice of $G$.
    \item $\bbG_{m}$ - is the multiplicative group.
    \item $i_{\alpha}:\mathfrak{sl}(2):\rightarrow \fg$, $I_{\alpha}:SL(2)\rightarrow G$, denote the $\mathfrak{sl}(2)$-triple (and $SL(2)$-triple) associated to $\alpha$ see discussion around Diagram \ref{diag:exp}.
    %\item $C_{\alpha_{1},...,\alpha_{j}}:=\sum_{i=1}^{j}\bbR_{\geq 0 }\alpha_{i}\subset \ft^{*}$, the cone of a set of roots see \ref{def:cone_roots}.
    \item $\fu_{\alpha_{1},...,\alpha_{j}}\subset \fg$ is a Lie subalgebra defined in Definition \ref{definition: poyhedral nilpotent lie algebra}.
    \item $\fu_{\Delta}$ -- is a  Lie algebra defined in Definition \ref{definition: poyhedral nilpotent lie algebra}.
    \item $\Conv^\N_{\gamma_1, \dots, \gamma_k}$ -- is the restricted convex hull of the set of roots $\{\gamma_1, \dots, \gamma_k\}$.  This consists of roots which can be written as linear combinations, with coefficients in $\N\cup \{0\}$, of the set $\{\gamma_1, \dots, \gamma_k\}$. Defined in Definition \ref{def:convex_hulls}.
    \item $\overrightarrow{\prod}_{s \in S}$, $\overleftarrow{\prod}_{s \in S}$ -- For a totally ordered index set $S$, these denote the products in which factors appear in increasing or decreasing order of their index, respectively.
    \item $[X/G]$ denotes the stack quotient of a space $X$ by $G$.
    \item $X\sslash G$ denotes the GIT quotient of a space $X$ by $G$.
\end{itemize}

\subsubsection{Moduli of Local systems}

\begin{itemize}
    \item $\Loc_{G}(X)$ -- Moduli stack of Betti $G$-local systems on $X$.
    \item $\Loc_{N}^{\tilde{X}^{\circ}}(X^{\circ})$ -- defined as $\Loc_{N}(X^{\circ})\times_{\Loc_{W}(X^{\circ})}\{\tilde{X}^{\circ}\}$ in Definition \ref{defin: N local systems corresponding to a cameral cover}, the moduli of $N$-local systems on $X^{\circ}$ corresponding to the cameral cover $\tilde{X}^{\circ}\rightarrow X^{\circ}$.
    \item $\Loc_{N}^{\tilde{X}^{\circ}, S}(X^{\circ})$.  The moduli stack of $N$-local systems on $X^{\circ}$, corresponding to the cameral cover $\tilde{X}^{\circ}$, satisfying the S-monodromy condition.  See Definitions \ref{def:alpha_monodromy_N}, \ref{definition: Moduli N bundles S monodromy}.  The space $\Loc_{N}^{\tilde{X}^{\circ'}, S}(X^{\circ'})$ is isomorphic to this.
    \item $\Loc_{T}^{N}(\tilde{X}^{\circ})$ is the moduli functor of $N$-shifted weakly $W$-equivariant $T$-local systems on $\tilde{X}^{\circ}$.  See Definitions \ref{defn: N shifted Weakly W equivariant}, \ref{defn: Moduli N shifted weakly W equivariant}.
    %\item $\Loc_{T}^{N, -1}(\tilde{X}^{\circ})$ see definition \ref{def: locT-1}.
    \item $\Loc_{T}^{N,S}(\tilde{X}^{\circ})$ is the moduli stack of $N$-shifted weakly $W$-equivariant local systems on $\tilde{X}^{\circ}$ satisfying the S-monodromy condition.  See Definitions \ref{alpha monodromy T bundle side}, \ref{moduli alpha monodromy T bundle side}.
    \item $\Aut_{\cW, G}(X^{\circ '})$ is the moduli stack of $G$-local systems on $X^{\circ '}$, together with a locally constant section of the sheaf of automorphisms on each connected component of $\cW\backslash \cJ$, see Definition \ref{defn: stack G bundles with automorphisms}.
    \item $\Aut_{\cW, N, G}^{\tilde{X}^{\circ'}}(X^{\circ'}):=\Loc_{N}^{\tilde{X}^{\circ'}}(X^{\circ'})\times_{\Loc_{G}(X^{\circ'})}\Aut_{\cW, G}(X^{\circ'})$ is defined in Definition \ref{def: Regluing stack with N structure}.
    %\item $\Loc^{enh}_{A_{K}}(\overline{X}^{ne,r}_{\rho})$,  $\Loc^{enh}_{A_{K}}(\overline{X}^{ne,r,\circ'}_{\rho})$ -- see definition \ref{def: enhanced moduli space}.  
    %\item $\Loc^{enh,S}_{A_{K}}(\overline{X}_{\rho}^{ne,r, \circ})$ see definition \ref{definition: enhanced moduli space with alpha}.
    \item $reglue: \Aut_{\cW, G}(X^{\circ '})\rightarrow \Loc_{G}(X^{\circ'} \backslash \cJ)$ is the map defined in Definition \ref{defn: regluing map}.
    \item $s^{WKB}: \Loc_{N}^{\tilde{X}^{\circ '}, S}(X^{\circ '})\rightarrow \Aut_{\cW, G}(X^{\circ '})$ is defined in Construction \ref{construction of swkb}
    \item $s_{PD}: \Loc_{\bbG_{m}}(\overline{X}_{\rho}^{ne,\circ_{R}})\rightarrow \Aut_{\cW, GL(V)}(X^{\circ '})$  is defined inside Construction \ref{construction: exponential path rule non abelianization}.
    \item $nonab: \Loc_{N}^{\tilde{X}^{\circ}, S}(X^{\circ})\rightarrow \Loc_{G}(X)$ is the map defined in Construction \ref{thm:bulk_map_exists}.
    \item $\Loc_{\bbG_{m}}^{-1}(\overline{X}_{\rho}^{ne,\circ})$ is a moduli space of $\bbG_{m}$-local systems on the spectral cover $\overline{X}_{\rho}^{ne,\circ}$ satisfying a restriction on their monodromy, as defined inside Construction \ref{construction: exponential path rule non abelianization}.
    \item $nonab_{PD}: \Loc_{\bbG_{m}}^{-1}(\overline{X}_{\rho}^{ne,\circ})\rightarrow \Loc_{GL(V)}(X)$ is defined in Construction \ref{construction: exponential path rule non abelianization}.
\end{itemize}

%For a ramified covering map $\pi : \tilde X \to X$, we denote by $P\subset X$ the branch points, and by $R \subset \tilde X$ the ramification points. $\tilde X$ denotes a cameral cover of $X$, while $\bar X$ denotes a spectral cover.

%We reserve the notation $\tilde \cW \subset \tilde X$ for a cameral network, $\overline \cW \subset \bar X$ for its projection to a spectral cover, and $\cW \subset X$ for the spectral network as in \cite{gaiotto2013spectral}, which is the projection of $\tilde{\cW}$ to the base $X$.

%$G$ denotes a connected, reductive algebraic group over $\C$. $T$ is a choice of maximal torus, $N=N_{G}(T)$ its normalizer in $G$, $W$ the Weyl group. $\fr g = \text{Lie}(G)$, and we denote by $\Phi$ the set of roots of $\fr g$.

%$Loc_{G}(X)$

%\subsubsection{Other}
%\begin{itemize}
%    \item For $V$ a vector space over a field $F$, and $S\subset V$ a subset, $\langle S \rangle_{F}$ denotes the span of all elements in $S$.\todo{where is this used!  In Lie preliminaries -- note that this clashes with notation used with nilpotent Lie algebras}
%\end{itemize}

\section{The Moduli Space of Local Systems, and the Hitchin Base}
\label{sec: Moduli spaces background}

In this section we briefly review the moduli space of local systems, the Hitchin base, and the Hitchin moduli space (the third of these is not strictly speaking necessary for the remainder of the paper though it is important motivationally and for the underlying physics outlined in section \ref{subsec: introduction 2d-4d wall crossing and non-abelianization}).

\subsection{Moduli of Local Systems}

Let $X$ be a (not necessarily compact) Riemann surface, such that $X=X^{c}\backslash D$, for a compact Riemann surface $X^{c}$ and a reduced divisor $D$.  We only stipulate that $D$ is reduced because this is the situation we will consider, in order to simplify the discussion in Sections \ref{subsec: trajectories near D} and \ref{subsection: denseness of individual trajectories}.

\begin{definition}[Local System]
A \textbf{$G$-local system} on $X$ is a locally constant sheaf $E$, of schemes\footnote{Or otherwise depending on the category you are working in.} with $G$-action, such that $E_{x}$ is a $G$-torsor for all $x\in X$.
\end{definition}

A closely related object is a representations of the fundamental group.

Let $E$ be a $G$-local system on $X$.  If we fix a point $x\in X$, and a trivialization of $E_{x}$, the holonomy of the local system gives a representation 
\[\pi_{1}(X,x)\xrightarrow{\rho_{E}} \Aut(E_{x})\cong G.\]
  Varying the point $x\in X$, and the trivialization $E_{x}\cong G$ changes the representation by conjugation.

Conversely given a representation $\pi_{1}(X,x)\xrightarrow{\rho} G$, we can define a local system $E_{\rho}$, by taking the trivial local system $G\times X^{un}\rightarrow X^{un}$ on the universal cover $X^{un}\rightarrow X$ and quotienting by the diagonal action of $\pi_{1}(X,x)$ to get a local system
\[E_{\rho}:=G\times X^{un}/\pi_{1}(X,x)\rightarrow X^{un}/\pi_{1}(X,x)\cong X.\]
This local system is equipped with a framing at $x$.

The relation between local systems and representations of the fundamental group gives a simple way to describe the underived moduli stack of local systems $\Loc_{G}(X)_{cl}$.

Let $g$ be the genus of $X^{c}$ and $d=\text{deg}(D)$. There is then a choice of generators (canonical cycles) $\alpha_{1}, \beta_{1},...,\alpha_{g},\beta_{g}, \gamma_{1},...,\gamma_{d}$ of $\pi_1(X)$, such that we have an isomorphism:
\begin{equation}
\label{equation: Relation fundamental group}
\pi_{1}(X)=
\left\langle 
\alpha_{1}, \beta_{1},...,\alpha_{g},\beta_{g}, \gamma_{1},...,\gamma_{d}
\left|
\prod_{i=1}^{g}[\alpha_{i},\beta_{i}]\prod_{i=1}^{d}\gamma_{i}=Id
\right.\right\rangle.
\end{equation}

We then have that group homomorphisms from $\pi_{1}(X)$ to $G$ are parameterized by an affine algebraic variety, namely
\[\Hom^{grp}(\pi_{1}(X), G)\hookrightarrow G^{2g+d}\ni (x_{\alpha_{1}},x_{\beta_{1}},...,x_{\alpha_{g}}, x_{\beta_{g}}, x_{\gamma_{1}},..,x_{\gamma_{d}})\]
is cut out by the equation $\prod_{i=1}^{g}[x_{\alpha_{i}},x_{\beta_{i}}]\prod_{j=1}^{d}x_{\gamma_{j}}=Id$.

This motivates the following definition:

\begin{definition}
The \textbf{underived moduli stack of local systems} $\Loc_{G}(X)_{cl}$, for $G$ an affine algebraic group, is given by the stacky quotient:
\[\Loc_{G}(X)_{cl}:=[\Hom^{grp}(\pi_{1}(X), G)/G],\]
where $G$ acts by conjugation.  
\end{definition}

\begin{remark}
Taking the GIT quotient rather than the stacky quotient gives the character variety.
\end{remark}

\begin{example}
Let $G=T$ be commutative.  We then have that 
\[\Loc_{T}(X)_{cl}=[T^{2g+d}/T]=BT\times T^{2g+d}.\] 
\end{example}

\begin{example}
Let $W$ be a finite group (for example a Weyl group).  We then have that 
\[\Loc_{W}(X)_{cl}=[\Hom(\pi_{1}(X),W)/W]\] 
has a discrete set of connected components.  Each component is of the form $B\Gamma$, for $\Gamma\subset W$ the finite group that preserves each element in the image of $\pi_{1}(X)$ under conjugation.
\end{example}

\begin{definition}
The \textbf{Betti stack} associated to $X$ is the derived stack $X_{B}\in \Fun(dSch^{op}, \cS)$ (here $\cS$ is the $\infty$-category of spaces or Kan complexes), which is the constant functor taking any derived scheme to a Kan complex corresponding to the topological space $X$ in the classical topology.
\end{definition}

\begin{definition}
For $G$ an affine algebraic group, the \textbf{derived moduli stack of local systems} is given by 
\[\Loc_{G}(X):=\Hom_{dSt}(X_{B}, BG),\]
where $\Hom_{dSt}(-,-)$ is the derived mapping stack  (See \S19.1 of \cite{lurie20spectral}, and \S 2.2.6.3 of \cite{toen2008homotopical}).
\end{definition}

See e.g. \cite{toen2004hag, pantev2018poisson} for more details.
 
\begin{remark}
Recall that there is a truncation map $t_{0}:dSt\rightarrow St$ from derived stacks, to stacks in the underived sense.  It is shown in Theorem 5.1 of \cite{toen2004hag} that \[t_{0}(\Loc_{G}(X))=\Loc_{G}(X)_{cl}.\]
\end{remark}

\begin{example}
A simple example of the difference between derived and underived moduli spaces of local systems is given by considering the moduli space of local systems on the sphere:
We have that $\Loc_{G}(\bbP^{1})_{cl}\cong BG$, while $\Loc_{G}(\bbP^{1})\cong Spec( Sym^{\bullet}(\fg^{\vee}[1]))/G$.
\end{example}

\begin{remark}[On the use of Derived Geometry]
\label{remark: where derived geometry is used}
We stress that we do not seriously interact with derived geometry, and the reader should feel free to initially ignore this subtlety.  In particular large parts of this paper do not require derived geometry, as summarized below.

We use the derived viewpoint on the moduli of local systems in the definition of the moduli stack of $N$-shifted weakly $W$-equivariant $T$-local systems on the cameral cover (Definition \ref{defn: Moduli N shifted weakly W equivariant}), and hence also in the moduli stack of these which satisfy the S-monodromy condition (Definition \ref{moduli alpha monodromy T bundle side}).  We also use derived geometry to define the stacks $\Aut_{\cW, G}(X^{\circ '})$ (Definition \ref{defn: stack G bundles with automorphisms}) and $\Aut_{\cW, N, G}^{\tilde{X}^{\circ'}}(X^{\circ'})$ (Definition \ref{def: Regluing stack with N structure}).  Our construction of the non-abelianization map factors through the stack $\Aut_{\cW, G}(X^{\circ '})$, however we described the map in this way for convenience, and one could describe the map without using derived geometry.  The description we use makes it manifest that non-abelianization gives a map of stacks.  The derived viewpoint will be needed in future work.  %Finally in section \ref{subsubsec: enhanced moduuli space unramified spectral cover} \todo{may need changing with 6.3} we define a modification $\Loc'_{\bbG_{m}}(\overline{X}_{K})$ of the moduli of local systems on an unramified spectral cover $\overline{X}_{K}$.  This is used in the spectral descriptions of the the certain moduli of $N$-bundles in Theorems \ref{Claim: Description of N bundles in spectral terms} and \ref{theorem: spectral description of mN bundles with S monodromy}.  It is not necessary to use it in section \ref{subsec: explicit descrtions for defining representations of classical groups}.
\end{remark}

\subsection{Hitchin Base}
\label{subsec: Hitchin base}

In this section we review well known facts about the Hitchin base following e.g. \cite{hitchin1987stable, ngo2010lemme, donagi1993decomposition}.  In this section we use the compact Riemann surface $X^{c}$.  We fix a principal $\bbG_{m}$-bundle $\cL \rightarrow X^{c}$. 

Denote $\ft_{\cL}:=\cL\times_{\bbG_{m}}\ft$.

\begin{definition}[Hitchin base]
\label{def: Hitchin base}
We define the \textbf{Hitchin base} $\cA_{G}$ (which we denote by $\cA$ if no confusion will arise) associated to a reductive algebraic group $G$, a compact Riemann surface $X^{c}$, and a line bundle $\cL$ as;
\[
\cA_G:=\Gamma(X^{c}, \ft_{\cL}\sslash W),
\]
where $\ft_{\cL}\sslash W$ refers to the GIT quotient.
\end{definition}

For classical groups we can give more explicit descriptions of the above:

\begin{example}
For $G=GL(n)$, we have an identification $\bbC[\ft]^{W}\cong \bbC[\bbA^{n}]^{S_{n}}$. This algebra is generated by the elementary symmetric polynomials in $n$ variables.  We hence have an identification:
\[\cA_{GL(n)}\cong \oplus_{i=1}^{n}\Gamma(X, \cL^{\otimes i}).\]
\end{example}

\begin{definition}[Cameral Cover]
\label{def: cameral cover}
Let $a\in \cA$ be a point of the Hitchin base.  We define the \textbf{Cameral cover} of $X$ associated to the point $a\in \cA$ as the pullback scheme:
\begin{equation}
\label{eq:def_cameral}
\begin{tikzcd}
\tilde{X}^{c}_{a} \arrow{r}{\tilde{a}}\arrow{d}{\pi} & \ft_{\cL}\arrow{d}\\
X^{c} \arrow{r}{a} & \ft_{\cL}\sslash W
\end{tikzcd}
\end{equation}

We denote the map $\tilde{X}_{a}^{c}\rightarrow \ft_{\cL}$ by $\tilde{a}$.
\end{definition}

The $W$-action on $\ft_{\cL}$ (as a scheme over $\ft_{\cL}\sslash W$) induces a $W$-action on $\tilde{X}^{c}_{a}$, which preserves the fibers of $\tilde{X}_{a}^{c}\rightarrow X^{c}$.  

\begin{remark}
Let $\ft^{reg}\subset \ft$ be the regular elements of $\ft$.  This is the complement of the root hyperplanes in $\ft$.  As this is preserved under the $\bbG_{m}$ and $W$ actions we can define $\ft^{reg}_{\cL}:=\cL\times_{\bbG_{m}}\ft^{reg}$, and it's GIT quotient by $W$.

Let $P\subset X^{c}$ be the complement of the preimage of the subset $\ft^{reg}_{\cL}\sslash W\hookrightarrow \ft_{\cL}\sslash W$.  Then $P$ is the branch locus of the cover $\tilde{X}^{c}\rightarrow X^{c}$, and $\tilde{X}^{c}\backslash \pi^{-1}(P)\rightarrow X^{c}\backslash P$ is a principal $W$-bundle.
\end{remark}

Cameral covers were introduced in \cite{donagi2002gerbe, faltings1993stable, scognamillo1998elememtary} in order to describe the fibers of the Hitchin moduli space for arbitrary reductive algebraic groups $G$.  For classical groups we can instead use a simpler curve called the spectral cover.

\begin{definition}[Spectral Curve/Cover]
\label{defn: spectral cover}
Let $G\hookrightarrow GL(V)$ be a group together with an embedding into a linear group $GL(V)$.

This induces a map $\ft_{G}\rightarrow \ft_{GL(V)}$, a map $\ft_{G}\sslash W_{G}\rightarrow \ft_{GL(V)}\sslash W_{GL(V)}$
and hence a map 
\[\cA_{G}\xrightarrow{i} \cA_{GL(V)}\cong \oplus_{i=1}^{\text{dim}(V)}\Gamma(X, \cL^{\otimes i})\]
by postcomposing with the map $\ft_{G}\sslash W_{G}\rightarrow \ft_{GL(V)}\sslash W_{GL(V)}$.

We define the \textbf{spectral curve} $\overline{X}_{a}^{c}$ corresponding to a point:
\[
a=(a_{1}(x),...,a_{\text{dim}(V)}(x))\in \oplus_{i=1}^{\text{dim}(V)}\Gamma(X, \cL^{\otimes i})\cong \cA_{GL(V)}
\]
as the zero locus (ie. intersection with the zero section) of the polynomial
\[\lambda^{\text{dim}(V)}+a_{1}(x)\lambda^{\text{dim}(V)-1}+a_{2}(x)\lambda^{\text{dim}(V)-2}+...+a_{dim(V)}(x)\in \Gamma(\text{Tot}(\cL), \pi^{*}\cL),\]
where $\pi:\text{Tot}(\cL)\rightarrow X^{c}$ is the projection, and $\lambda\in \Gamma(\text{Tot}(\cL), \pi^{*}\cL)$ is the tautological section.

For $a\in \cA_{G}$ we define the spectral curve $\overline{X}_{a}^{c}:=\overline{X}_{i(a)}^{c}$.
\end{definition}

For $G=GL(n)$, we can informally see the sheets of the cameral cover as corresponding to labellings of the sheets of the spectral cover. More formally, pick a morphism $S_{n-1}\hookrightarrow S_{n}$, corresponding to permutations which fix a chosen object. We then have an identification: 
\[\overline{X}_{a}^{c}\cong \tilde{X}_{a}^{c}/S_{n-1}.\]

For unramified\footnote{See \S 9 of \cite{donagi2002gerbe} for a general version.} spectral curves we can reverse this as follows:

\[\tilde{X}_{a}^{c}\hookrightarrow \Hom_{/X^{c}}(\{1,...,n\}\times X^{c}, \overline{X}_{a}^{c}),\]

where we are considering both sides as sheaves of sets over $X^{c}$.  The image corresponds to the subsheaf of morphisms such that at each $\bbC$-point $x\hookrightarrow X^{c}$ the map $\{1,...,n\}\rightarrow \{x\}\times_{X}X^{c}_{a}$ is an isomorphism. This map gives an isomorphism between $\tilde{X}_{a}^{c}$ and this subsheaf.

\begin{figure}[h]
    \centering
    \includegraphics[width = \textwidth]{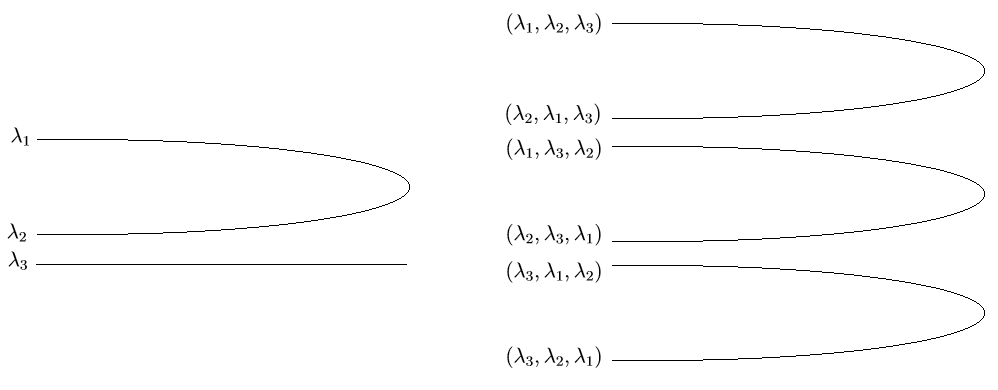}
    \caption{Real points of the spectral cover (left) and the cameral cover (right) in a neighborhood of a ramification point of order 2, for $G = GL(3)$.  Note that the set of real points in a neighbourhood of is not canonically defined, and depends on additional choices.}
    \label{fig:spectral_vs_cameral}
\end{figure}

\begin{definition}
We define the divisor $D_{H}\subset \ft_{\cL}\sslash W$ by $D_{H}:=\left(\cL\times_{\bbG_{m}}\left(\cup H_{\alpha}\right)\right)\sslash W$, where we take the union over all root hyperplanes $H_{\alpha}$.
\end{definition}

\begin{definition}
\label{definition:  cA diamondsuit}
We define the Zariski open $\cA^{\diamondsuit}\subset \cA$ as being the set of sections $a\in \Gamma(X^{c}, \ft_{\cL}\sslash W)$ that intersect the $D_{H}$ transversely, that is to say $a$ intersects the non-singular part of $D_{H}$ transversely.
\end{definition}

\begin{proposition}
\label{proposition: A diamond non empty}
Let $\cL=K_{X^{c}}(D)$ for a reduced divisor $D$, $deg(D)> 2$ .  Then $\cA^{\diamondsuit}\subset \cA$ is non-empty.
\end{proposition}

This is a special case of e.g. Proposition 4.6.1 of \cite{ngo2010lemme}.  The requirement  on degree comes down to needing $\cL$ to be very ample.

%\todo{No -- the above is problematic -- the proof of Ngo uses that the line bundle is very ample.  However this is not true for the canonical bundle!  (The $g\geq 2$ is for $K_{X}$ to be ample!) Email Ron to ask this question +++.  Scognamillo uses the base point freeness of $K_{X}$ to prove if for $K_{X}$.  Note:  We should hence be able to check that Scognamillo's argument is correct -- and then the statement follows simply for $deg(D)\geq 2$}

\begin{proposition}[See \cite{faltings1993stable} or Lemme 4.6.3 of \cite{ngo2010lemme} for a proof.]
\label{prop: diamond suit implies smooth}
If $a\in \cA^{\diamondsuit}$, then the associated cameral cover $\tilde{X}_{a}$ is smooth.  Conversely if $\tilde{X}_{a}$ is smooth then $a\in \cA^{\diamondsuit}$.
\end{proposition}

\begin{remark}
The spectral cover can be smooth without the cameral cover being smooth.  An example is given by a spectral curve mapping to the curve $X$ that locally looks like the projection from the curve $y^{3}=x$ to the $x$-coordinate.
\end{remark}

By definition we have that if $a\in \cA^{\diamondsuit}$, then $\tilde{a}(\tilde{X})\cap (\cL\times_{\bbG_{m}}(H_{\alpha}\cap H_{\beta}))=\emptyset$ for $\alpha\neq \beta$.  

\subsection{Hitchin Moduli Space}

Strictly speaking, this section is completely unnecessary for this paper.  However, it is important motivationally, and to some of the conjectures \cite{gaiotto2013wallHitchin, gaiotto2012wall2d4d} in this subject.

In this section we have a compact Riemann surface $X^{c}$ (which we consider as an algebraic curve), a reductive algebraic group $G$ and a line bundle $\cL$ .

\begin{definition}
An ($\cL$-twisted) Higgs $G$-bundle on $X^{c}$ is a pair $(E, \varphi)$, where $E$ is a principal $G$-bundle on $X$, and $\varphi\in \Gamma(X^{c}, \ad(E)\otimes \cL)$ is called a Higgs field.
\end{definition}

\begin{example}
\label{example:  GL(n) Higgs Fields}
For $G=GL(n)$ this is equivalent to a pair $(\mathbf{E}, \varphi)$, where $\mathbf{E}\rightarrow X^{c}$ is a \emph{vector bundle} on $X^{c}$, and $\varphi\in \Gamma(X^{c}, \End(\mathbf{E})\otimes \cL)$.
\end{example}

The Hitchin moduli space is the (Artin) stack which represents the moduli problem for Higgs bundles on $X$. More formally:

\begin{definition}[Hitchin Moduli Space]
The Hitchin moduli space is the moduli space of sections, or equivalently the  mapping stacks:
\[ \cM_{H}(X,G,\cL):= \Gamma(X, [\fg_{\cL}/G])=\Map_{/X}(X, [\fg_{\cL}/G]).\]

Here $[\fg_{\cL}/G]$ refers to the stack quotient, and $\fg_{\cL}:=\cL\times_{\bbG_{m}}\fg$.
\end{definition}

\begin{remark}
For $\cL\cong K_{X^{c}}$, Serre duality or coherent duality, together with an identification $\fg\cong \fg^{\vee}$, gives an identification $\cM_{H}(X,G,L)\cong T^{*}Bun_{G}(X)$, of the Hitchin moduli space with the cotangent bundle\footnote{In the case where $G$ is not semisimple, we interpret the cotangent bundle and the mapping stack in the derived sense. Alternatively, one could rigidify with respect to $Z(G)$, or use the framework of good and very good stacks as in \cite{beilinson1991quantization}.} of the moduli space of $G$-bundles on $X$.
\end{remark}

This moduli space is endowed with the structure of an integrable system \cite{hitchin1987self} (see also chapter 2 of \cite{beilinson1991quantization} ).

\begin{definition}[Hitchin Fibration]
The \textbf{Hitchin fibration} is the map induced by post-composition with the map $[\fg_{\cL}/G]\rightarrow \fg_{\cL}\sslash G\cong \ft\sslash W$ from the stacky quotient to the GIT quotient:
\[\cM_{H}=\Gamma(X, [\fg_{\cL}/G])\xrightarrow{h} \Gamma(X, \fg_{\cL}\sslash G)\cong \Gamma(X, \ft_{\cL}\sslash W)=: \cA\]
\end{definition}

\begin{example}[Hitchin fibration for $GL(n)$]
In the case of $GL(n)$, using the description of $GL(n)$-Higgs bundles of Example \ref{example: GL(n) Higgs Fields} we have that the Hitchin fibration takes a Higgs $GL(n)$-bundle $(\mathbf{E}, \varphi)$ to the characteristic polynomial of $\varphi$.

The coefficients of the characteristic polynomial are the elementary symmetric functions.
\end{example}

We can describe the fibers of the Hitchin fibration using the machinery of spectral and cameral covers.  For $G=GL(n)$ informally we can see the Hitchin base as describing the eigenvalues of the Higgs field, so the fiber must describe the eigenlines (or coeigenlines).  A description of the (co)eigenlines should correspond to an object like a line bundle on the spectral cover, which we can now see as describing the eigenvalues of a Higgs field.

More formally:

\begin{proposition}[Abelianization of $GL(n)$ Hitchin fibers \cite{hitchin1987stable, beauville1989spectral}]
Let $a\in \cA_{GL(n)}$ be a point of the Hitchin base such that the associated spectral cover $\overline{X}_{a}$ is smooth

There is an isomorphism 
\[h^{-1}(a)\cong Pic(\overline{X}_{a}),\]
where $Pic(X_{a})$ is the Picard stack.
\end{proposition}

The identification is given by the following pair of maps, which are inverses to each other:
\[h^{-1}(a)\rightarrow Pic(\overline{X}_{a})\]
\[(E, \varphi)\mapsto \Coker(\pi^{*}E\xrightarrow{\lambda - \varphi} \pi^{*}E)\otimes (\pi^{*}\cL)^{-1}\]
and 
\[Pic(\overline{X}_{a})\rightarrow h^{-1}(a)\]
\[\cL\mapsto (\pi_{*}\cL, \pi_{*}\lambda).\]

\begin{remark}
This can be generalized to reduced, and ultimately to arbitrary spectral curves \cite{beauville1989spectral, schaub1998courbes}.  This requires replacing line bundles with rank one torsion free sheaves (and in the non-reduced case generalizing the notion of rank).
\end{remark}

\begin{remark}
By \cite{donagi2002gerbe} there is a similar statement for other groups $G$.  Namely for a smooth cameral cover there is an isomorphism between the Hitchin fiber, and the moduli of certain $T$-bundles on the cameral cover equipped with additional structure.  This motivates the definitions of the moduli spaces we consider in section \ref{sec: flat Donagi Gaitsgory}.
\end{remark}

%The non-abelianization of \cite{gaiotto2013spectral} uses $\bbC^{*}$-local systems on an open subset of a spectral cover $\overline{X}$.  This can be seen as an analogue of the line bundles used in the abelianization of Hitchin fibers.  It is hence natural to ask for an analogue of the above statement for arbitrary reductive algebraic groups $G$.  This is given in \cite{donagi2002gerbe}, we are not going to give the precise statement, but roughly it identifies Higgs bundles which correspond to $b\in \cB$ (where $\tilde{X}_{b}$ is smooth) with $T$-bundles on $\tilde{X}_{b}$ which satisfy a modification of the $W$-equivariance, together with some additional discrete data\footnote{Denoted by $\beta_{i}$ in \emph{loc cit.}}. 

%Finally in \cite{gaiotto2013spectral} it is conjectured that the non-abelianization map is holomorphic symplectic and in fact K\"{a}hler.  The K\"{a}hler structure this refers to is described in \cite{hitchin1987self} (see also \cite{simpson1996hodge} for the twistor space point of view following Deligne).

%\input{Old_Hitchin_Moduli.tex}

\section{Lie Theory and Stokes Data}
\label{sect:prelim}

This section contains some Lie theoretic notation and constructions over the complex numbers, and describes the local picture of cameral networks occurring when Stokes lines intersect. Section \ref{subsect:chevalley} contains background material with no novel ideas, and it can be safely skipped until one encounters back-references to it.  Warning \ref{warn:sl2_triples} outlines a non-standard sign convention for Chevalley bases, which we use throughout the paper. Section \ref{sec: Stokes Data} introduces 2D scattering diagrams, and contains results about how to assign unipotent automorphisms to lines in these diagrams.  These 2D scattering diagrams describe the local picture of cameral networks, and also the local picture of spectral networks up to a subtlety about the labels of lines (see Section \ref{sec: Cameral Networks}).  

\subsection{Chevalley bases}
\label{subsect:chevalley}

We begin by fixing a basis of $\fr g$ with respect to which the structure constants are particularly well behaved.

\begin{definition}
\label{defin:chevalley_basis}
Let $\fr g$ be a complex semisimple Lie algebra. Fix a Cartan $\fr t \subset \fr g$, which determines a root space decomposition $\fr g = \fr t \oplus \bigoplus_{\alpha \in \Phi} \fr g_\alpha$ (with $\Phi$ denoting the set of roots); also fix a set of simple roots.
Consider a basis\footnote{We mean a basis of the underlying vector space.} of $\fr g$ compatible with this decomposition, that is to say:
\begin{itemize}
    \item a basis $\{h_\alpha\}_{\alpha \text{ simple}}$ for $\fr t$;
    \item a basis $\{e_\alpha\}_{\alpha\in \Phi}$ for each of the 1-dimensional root spaces $\fr g_\alpha$.
\end{itemize}
If $\gamma \in \Phi$ is not simple, write it as a linear combination of simple roots $\gamma=\sum_{i=1}^{n}a_{i}\alpha_i$, and define $h_\gamma$ as the corresponding linear combination $h_{\gamma}:=\sum_{i=1}^{n}a_{i}h_{\alpha_i}$.

We call the basis a \textbf{Chevalley basis}\footnote{See warning \ref{warn:sl2_triples}.} if it satisfies the relations:
\begin{align}
    [h_\alpha, e_\gamma] &= 2 \frac{(\alpha, \gamma)}{(\alpha,\alpha)} e_\gamma,
    \label{eq:chevalley_sl2}
    \\
    [e_{\alpha}, e_{-\alpha}] &= -h_\alpha ,
    \label{eq:chevalley_cmm}
    \\
    \label{eq:chevalley_p}
    [e_\alpha, e_\gamma] &= \left\{
    \begin{array}{ll}
    0 & \text{if } \alpha + \gamma \text{ not a root,} \\
    \pm (p_{\alpha,\gamma}+1) e_{\alpha + \gamma} & \text{if } \alpha + \gamma \text{ is a root.}
    \end{array}
     \right.
\end{align}
In Equation \ref{eq:chevalley_p}, $p_{\alpha,\gamma}$ is the largest integer such that $\alpha - p_{\alpha,\gamma} \gamma$ is a root. 
\end{definition}

\begin{remark}
\label{rem:describe_p}
If we want to fix numbers $C_{\alpha, \gamma}$ such that the system of equations:
\[
[e_\alpha, e_\gamma] = C_{\alpha,\gamma}\cdot e_{\alpha + \gamma}
\]
in the variables $\{e_{\gamma}\}_{\gamma \in \Phi}$ is consistent, it's not necessarily possible to take all $C_{\alpha, \gamma} = \pm 1$. Indeed, one can prove that $C_{\alpha, \gamma} C_{-\alpha, -\gamma} = (p_{\alpha, \gamma} + 1)^2$.

A Chevalley basis is appealing because it has $C_{\alpha, \gamma} = C_{-\alpha, -\gamma} = \pm (p_{\alpha, \gamma} + 1)$, which are small integers:
\begin{itemize}
    \item For $\fr g$ of type ADE, all $p_{\alpha, \gamma} = 0$.
    \item For $\fr g$ of type BCF, $p_{\alpha, \gamma} = 1$ if $\alpha, \gamma$ are both short roots, and $p_{\alpha, \gamma} = 0$ otherwise.
    \item For $\fr g$ the Lie algebra of the group $G2$, $p_{\alpha, \gamma} \in \{0,1,2\}$, depending on the angle between $\alpha$ and $\gamma$.
\end{itemize}
\end{remark}

\begin{warning}
\label{warn:sl2_triples}
According to our definition, for any root $\alpha$, $\{h_\alpha, e_\alpha, -e_{-\alpha}\}$ is an $\fr{sl}_2$ triple. However, the standard definition of a Chevalley basis does not include the minus sign in Equation \ref{eq:chevalley_cmm}, which would make $\{h_\alpha, e_\alpha, e_{-\alpha}\}$ an $\fr{sl}_2$ triple. We chose a non-standard sign convention in order to simplify later constructions.
\end{warning}

Due to Remark \ref{rem:describe_p} and Warning \ref{warn:sl2_triples}, we can informally see a Chevalley basis as a choice of $\fr{sl}_2$ triples for all roots of $\fr g$, with Lie brackets as simple as possible.

\begin{proposition}
Every complex semisimple Lie algebra $\fr g$ admits a Chevalley basis.
\end{proposition}

This result was proved by Chevalley in \cite{chevalley1955}, and an exposition can be found in the blog post \cite{tao_blog_chevalley}.

\begin{comment}
This is the proof sketch, keeping it around in case we decide to use it.

\begin{enumerate}
    \item Use the Killing form to identify $\fr h$ with $\fr h^\vee$, and define $h_{\alpha}$ as the co-roots with the normalization:
    \[  h_{\alpha} = \frac{2}{(\alpha,\alpha)} \alpha^\vee.   \]
    This ensures that \ref{eq:chevalley_sl2} holds.
    \item Take arbitrary generators for the root spaces and rescale them such that \ref{eq:chevalley_cmm} holds.
    \item We can still rescale $e_\alpha$ by some $A_\alpha$, as long as we rescale $e_{-\alpha}$ by $A_\alpha^{-1}$, so that \ref{eq:chevalley_cmm} is preserved. We can use this freedom to ensure that $C_{\alpha,\beta} = C_{-\alpha,-\beta}$. For this, we use the automorphism $\Phi$ of $\fr g$ which acts by $-1$ on $\fr h$, and on the root space by $\alpha \mapsto - \alpha$. Then $\Phi(e_\alpha) = b_\alpha e_\alpha$ for some $b_\alpha$, and we take $A_{\alpha}$ such that $A_\alpha = b_\alpha A_{-\alpha}$.
\end{enumerate}
\end{comment}

\begin{example}
\label{eg:chevalley_basis}
Let $\fr g = \fr{sl}_3$, and $\alpha, \beta$ denote a choice of simple roots. We construct a Chevalley basis from the following basis of the Cartan: 
\[
h_{\alpha} = \left(
\begin{array} {ccc}
1 & 0 & 0 \\
0 & -1 & 0 \\
0 & 0 & 0
\end{array}
\right)
\hspace{1cm}
h_{\beta} = \left(
\begin{array} {ccc}
0 & 0 & 0 \\
0 & 1 & 0 \\
0 & 0 & -1
\end{array}
\right)
\]
and the following basis vectors for the root spaces:
\[
e_{\alpha} = \left(
\begin{array} {ccc}
0 & 1 & 0 \\
0 & 0 & 0 \\
0 & 0 & 0
\end{array}
\right)
\hspace{1cm}
e_{\beta} = \left(
\begin{array} {ccc}
0 & 0 & 0 \\
0 & 0 & 1 \\
0 & 0 & 0
\end{array}
\right)
\hspace{1cm}
e_{\alpha + \beta} = \left(
\begin{array} {ccc}
0 & 0 & 1 \\
0 & 0 & 0 \\
0 & 0 & 0
\end{array}
\right)
\]
\[
e_{-\alpha} = \left(
\begin{array} {ccc}
0 & 0 & 0 \\
-1 & 0 & 0 \\
0 & 0 & 0
\end{array}
\right)
\hspace{7mm}
e_{-\beta} = \left(
\begin{array} {ccc}
0 & 0 & 0 \\
0 & 0 & 0 \\
0 & -1 & 0
\end{array}
\right)
\hspace{7mm}
e_{-\alpha - \beta} = \left(
\begin{array} {ccc}
0 & 0 & 0 \\
0 & 0 & 0 \\
-1 & 0 & 0
\end{array}
\right)
\]
There are other Chevalley bases for $\fr{sl}_3$.
For any re-scaling of $e_\alpha, e_\beta$ by $A, B \in \C^\times$, it is possible to re-scale $e_{\alpha + \beta}$ by $AB$, and $e_{-\alpha}, e_{-\beta}$, and  $e_{-\alpha - \beta}$ by $A^{-1}, B^{-1}$, and $A^{-1}B^{-1}$ respectively, so that relations \ref{eq:chevalley_sl2} - \ref{eq:chevalley_p} are unchanged.

\begin{comment}
\begin{figure}[h]
\caption{Rescaling a Chevalley basis for $\fr{sl}_3$}
\label{fig:chevalley_mult}
\includegraphics[scale = 0.2]{images/chevalley.jpg}
\end{figure}
\end{comment}

\end{example}

The existence of Chevalley bases implies the following fact, which will be used repeatedly.

\begin{lemma}
\label{lem:roots_commutators_vanish}
Let $\fr g$ be a semisimple Lie algebra, and consider a root space decomposition $\fr g = \fr t \oplus \bigoplus_{\alpha \in \Phi} \fr g_\alpha$, as in Definition \ref{defin:chevalley_basis}. Then for all $\alpha, \beta \in \Phi$ such that $\alpha+\beta \neq 0$, $[\fr g_\alpha, \fr g_\beta] \subset \fr g_{\alpha + \beta}$. Moreover, $[\fr g_\alpha, \fr g_\beta] = 0$ if and only if $\alpha + \beta \not \in \Phi$.
\end{lemma}

\begin{proof}
The root spaces $\fr g_\alpha$ are 1-dimensional, so it suffices to consider the bracket $[e_\alpha, e_\beta]$ for some $e_\alpha \in \fr g_\alpha$ and $e_\beta \in \fr g_\beta$. The Jacobi identity implies that, for all $h \in \fr t$:
\begin{align*}
\big[h, [e_\alpha,e_\beta] \big] &= \big[ [h,e_\alpha], e_\beta \big] + \big[ e_\alpha, [h,e_\beta] \big]
\\
&= \big[ \alpha(h)\cdot e_\alpha, e_\beta\big] + \big[e_\alpha, \beta(h)\cdot e_\beta\big]
\\
&= (\alpha + \beta)(h) \cdot [e_\alpha, e_\beta].
\end{align*}
Hence $[\fr g_\alpha, \fr g_\beta] \subset \fr g_{\alpha+\beta}$.

It's clear then that if $\alpha+\beta \not \in \Phi$, then $[\fr g_\alpha, \fr g_\beta] = 0$. The converse follows from the existence of a Chevalley basis: according to relation \ref{eq:chevalley_p}, if $\alpha + \beta$ is a root, then $[e_\alpha, e_\beta]$ is a nonzero scalar multiple of $e_{\alpha+\beta}$.
\footnote{For a more elementary proof of the converse, see e.g. Theorem 6.44 in \cite{kirillov_introduction}.}
\end{proof}

\begin{comment}
In general, we have:

\begin{remark}
Fix a consistent choice of signs in Equation \ref{eq:chevalley_p}. The set of Chevalley bases satisfying \ref{eq:chevalley_sl2} - \ref{eq:chevalley_p} with this choice of signs is a torsor for the maximal torus $T$. Indeed, the action is by re-scaling the basis elements $e_{\alpha_i}$, for $\{\alpha_i\}$ a set of simple roots.
\end{remark}
\end{comment}

We now let $\fr g$ be a reductive Lie algebra. We fix a Chevalley basis for the semisimple part of $\fr g$.%\footnote{Note that the roots of $\fg$ are the roots of its semisimple part.}
In particular, this determines a map:
\begin{align*}
    \Phi & \to \{\fr{sl}_2 \text{-triples in } \fr g \},
    \\
    \alpha &\mapsto (h_\alpha, e_\alpha, -e_{-\alpha}).
\end{align*}
In other words, each $\alpha$ determines a map $i_\alpha : \fr{sl}_2 \to \fr g$. Because $SL(2)$ is simply connected, Lie's theorems provide unique group homomorphisms $I_\alpha$, making the following diagram commute:

\begin{equation}
\label{diag:exp}
\begin{tikzcd}
\fr{sl}_2 \arrow{r}{i_\alpha}\arrow{d}{\exp} & \fr g \arrow{d}{\exp} \\
SL(2)\arrow{r}{I_\alpha} & G 
\end{tikzcd}
\end{equation}

Let $s_\alpha \in W$ be the reflection about the root hyperplane orthogonal to $\alpha$, and recall the short exact sequence
\begin{equation}
\label{eq:weyl_exact_seq}
\begin{tikzcd}
1 \arrow{r} & T \arrow{r} & N \arrow{r}{q} & W \arrow{r} & 1.
\end{tikzcd}
\end{equation}

Any $\fr{sl}_2$ triple, in particular the one coming from our Chevalley basis, determines an $n_\alpha \in N$ such that $q(n_\alpha) = s_\alpha$, via the formula:

\begin{equation}
\label{eq:preferred_lift}
n_{\alpha}:= I_{\alpha}\left(\left(\begin{matrix}0 & 1\\ -1 & 0\end{matrix}\right)\right)
%n_\alpha := \exp \frac{\pi}{2} (e_\alpha + e_{-\alpha}).
\end{equation}

These group elements these were first introduced by Tits in \cite{tits1966normalisateurs}.

\begin{warning}
The lifts $n_\alpha$ do \emph{not}, in general provide a group homomorphism $W\rightarrow N$. For instance, $n_\alpha^2 \neq \id$ in the setting of Example \ref{ex:lift_sl2}. Providing such a homomorphism is impossible in general, because the sequence \ref{eq:weyl_exact_seq} does not split in general \cite{curtis1974normalizers} (e.g. it does not split for $SL(2,\bbC)$).  See \cite{curtis1974normalizers} for discussion of when such a homomorphism exists.  Alternatively, as shown in \cite{tits1966normalisateurs}, it is possible to find a subgroup of $N$ which is a finite cover of the Weyl group $W$.
\end{warning}

\begin{remark}
\label{ex:lift_sl2}
Working analytically for $\fr g = \fr{sl}_2$, a Chevalley basis is:
\[
h_\alpha = \left( \begin{array}{cc}
  1   & 0 \\
  0   & -1
\end{array}
\right)
\hspace{1cm}
e_\alpha = \left( \begin{array}{cc}
  0   & 1 \\
  0   & 0
\end{array}
\right)
\hspace{1cm}
e_{-\alpha} = \left( \begin{array}{cc}
  0   & 0 \\
  -1   & 0
\end{array}
\right).
\]
We then have that %The corresponding lift is a rotation by $\pi/2$:
\[
n_\alpha  = \left( \begin{array}{cc}
  0   & 1 \\
  -1   & 0
\end{array}
\right) = \exp \left(\frac{\pi}{2} \left( \begin{array}{cc}
  0   & 1 \\
  -1   & 0
\end{array}
\right)\right).
\]

Hence for arbitrary $G$, when working complex analytically we have that 
\begin{equation}
\label{eq: nalpha as exponential}n_{\alpha}=I_{\alpha}\left( \exp \left(\frac{\pi}{2} \left(\left( \begin{array}{cc}
  0   & 1 \\
  0   & 0
\end{array}
\right) + \left( \begin{array}{cc}
  0   & 0 \\
  -1   & 0
\end{array}
\right)\right)\right)\right) = \exp \left(\frac{\pi}{2}(e_{\alpha}+e_{-\alpha})\right)\end{equation}
\end{remark}

\begin{lemma}
\label{lem:w}
The element $n_\alpha$ is a lift to $N$ of the simple reflection $s_\alpha$.
\end{lemma}

\begin{proof}
We prove this in the setting of Lie groups, the result then immediately follows for reductive algebraic groups over $\bbC$.

For every root $\alpha$, the Killing form of $\fr g$ determines an orthogonal decomposition:
\begin{equation}
\label{eq:decomp_lie_algebra}
 \fr{t} = \fr t_{\alpha} \oplus \fr t_{\Ker(\alpha)} . 
\end{equation}  
Here $\fr{t}_\alpha = \Span( h_\alpha)$; by construction, $h_\alpha$ is a scalar multiple of the co-root $\alpha^\vee$, the dual of $\alpha$ under the Killing form.  We define $t_{\Ker(\alpha)}:=\Ker(\alpha)$.

We prove that conjugation by $n_\alpha$ fixes $\fr t_{\Ker(\alpha)}$, and its restriction to $\fr t_{\alpha}$ is reflection about the hyperplane $\Ker \alpha$.

%\todo{Should be able to replace (1) with compatibility of the adjoint action with the map $SL(2)\rightarrow G$}
\begin{enumerate}
    \item We show that $\ad_{n_\alpha} (h_\alpha) = - h_\alpha$. For $G = SL(2)$, this is a trivial computation. The general case is reduced to the $SL(2)$ case using diagram \ref{diag:exp} (we here use $e$, $h$ and $-f$ to denote a Chevalley basis for $\mathfrak{sl}_{2}$):
    \begin{align*}
        \ad_{n_\alpha} (h_\alpha)
        &= \left.\frac{d}{dt}\right|_{t=0}
        \exp(i_\alpha(\pi(e-f)/2))
        \exp(i_{\alpha}(t h))
         \exp(i_\alpha(-\pi(e-f)/2))  \\
        &= \left.\frac{d}{dt}\right|_{t=0}
        I_\alpha  \left(  \exp(\pi(e-f)/2) 
        \exp(t h)
         \exp(-\pi(e-f)/2)\right) \\
         &= I_\alpha \left(  \exp(\pi(e-f)/2) 
        h
         \exp(-\pi(e-f)/2)\right) \\
         &= -h_\alpha .
    \end{align*}
    \item If $\alpha(h) = 0$, then:
    \[   
    [h, e_{\pm\alpha}] =\pm \alpha(h) e_{\pm\alpha} = 0.
    \]
    It follows that $n_\alpha h n_\alpha^{-1} = h$.

\end{enumerate}
\end{proof}

The next result is the main input for Lemma \ref{lemma: monodromy to Stokes factors}, where we interpret the left hand side (LHS) of the equation as a product of Stokes factors, and the RHS as the monodromy of an $N$-local system.

\begin{lemma}
\label{lem:triple_product}
We have the following identity in $G$:
\[   \exp{(e_\alpha)} \exp{(e_{-\alpha})}\exp{(e_\alpha)} = n_\alpha .  \]
\end{lemma}

This makes sense in the algebraic setting, as the exponent of nilpotent elements of the Lie algebra can be defined algebraically.

\begin{proof}
Again, it suffices to prove this for $G=SL(2)$, because Diagram \ref{diag:exp} means the general case follows from this one. For $SL(2)$ this is the straightforward identity:
\[   
\left( \begin{array}{cc}
    1 & 1  \\
    0 & 1
\end{array}
\right)
\left( \begin{array}{cc}
    1 & 0  \\
    -1 & 1
\end{array}
\right)
\left( \begin{array}{cc}
    1 & 1  \\
    0 & 1
\end{array}
\right)
=
\left( \begin{array}{cc}
    0 & 1  \\
    -1 & 0
\end{array}
\right).
\]
\end{proof}

Consider the orthogonal decomposition, as used in the proof of Lemma \ref{lem:w}:
\begin{equation}
\label{equation: Orthogonal decomposition of t}
 \fr{t} = \fr t_{\alpha} \oplus \fr t_{\Ker(\alpha)} . 
\end{equation}  

Let $T_{\alpha} = \exp (\fr t_{\alpha})$ and $T_{\Ker(\alpha)} = \exp (\fr t_{\Ker \alpha})$. We can equivalently describe $T_{\alpha}$ as the image $ I_{\alpha}(T_{SL(2)})$, under the homomorphism $I_{\alpha}$ from diagram \ref{diag:exp}, of the maximal torus of $SL(2)$.

\begin{lemma}
\label{lem:torus_decomposition}
The multiplication homomorphism:
\[
T_{\alpha} \times T_{\Ker(\alpha)} \to T
\]
is surjective, and its kernel is contained in the subgroup:
\[
\big\{(\id, \id), \big(I_\alpha(-\id_{SL(2)}), I_\alpha(-\id_{SL(2)})\big)\big\}.
\]
\end{lemma}

\begin{proof}
Due to the orthogonal decomposition \ref{eq:decomp_lie_algebra} at the Lie algebra level, there is a surjective homomorphism $T_{\alpha} \times T_{\Ker \alpha} \to T$. Its finite kernel is the intersection $T_{\alpha} \cap T_{\Ker \alpha}$ in $G$, which embeds antidiagonally in $T_{\alpha} \times T_{\Ker \alpha}$ via the map $t\mapsto (t,t^{-1})$. Since $T_{\alpha} = I_\alpha(T_{SL(2)})$, and $T_{\Ker \alpha} \subset \Ker(\exp(\alpha))$, we need only analyze the diagram:
\[
\begin{tikzcd}
T_{SL(2)} \arrow{r}{I_{\alpha}} & T \arrow{r}{\exp(\alpha)} & \C^\times .
\end{tikzcd}
\]
It follows that:
\[    T_{\alpha} \cap T_{\Ker \alpha} \subset  I_{\alpha}\big(\Ker (exp(\alpha)\circ I_\alpha)\big). \]

We have the containment $Ker(I_{\alpha})\subset \{ \pm \id_{SL(2)} \}$.  Hence as $exp(\alpha)|_{T_{\alpha}}$ has trivial kernel the result follows.
\end{proof}

\begin{lemma}
\label{lem:commutation_t_alpha}
For every $t\in T_\alpha$, conjugation by $n_\alpha$ gives $t^{-1}$:
\[
n_\alpha t n_\alpha^{-1} =  t^{-1}.
\]
\end{lemma}

\begin{proof}
The relation is the image under $I_\alpha$ of the $SL(2)$ relation:
\[
\left( \begin{array} {cc} 
0 & 1 \\ -1 & 0
\end{array} \right)
\left( \begin{array} {cc} 
a & 0 \\ 0 & a^{-1}
\end{array} \right)
\left( \begin{array} {cc} 
0 & -1 \\ 1 & 0
\end{array} \right)
=
\left( \begin{array} {cc} 
a^{-1} & 0 \\ 0 & a
\end{array} \right) .
\]
\end{proof}

\begin{lemma}
\label{lem:t_ker_comm}
The adjoint action of $T_{\Ker(\alpha)}$ on $\fr g$ fixes $e_{\pm \alpha}$. Consequently, the adjoint action of $T_{\Ker(\alpha)}$ on $N$ fixes $n_\alpha$.
\end{lemma}

\begin{proof}
If $h \in \fr t_{\Ker \alpha}$, then $[h,e_\alpha] = \alpha(h) e_\alpha = 0$. It follows that $\ad_{\exp(h)}(e_\alpha) = e_\alpha$.
\end{proof}

\begin{lemma}
\label{lem:adj_scalar}
For all $t \in T$, $\ad_t(e_\alpha)$ is a scalar multiple of $e_\alpha$. Moreover, all scalar multiples of $e_\alpha$ arise in this way.\footnote{Due to Lemma \ref{lem:t_ker_comm}, we could just as well restrict to $t\in T_\alpha$.} 
\end{lemma}

\begin{proof}
Consider relation \ref{eq:chevalley_sl2} from the Definition \ref{defin:chevalley_basis} of a Chevalley basis:
\[
[h_\gamma, e_\alpha] = 2 \frac{(\gamma, \alpha)}{(\gamma, \gamma)} e_\alpha .
\]
Using that the adjoint action preserves the Lie bracket, and the fact that $\ad_{t}(h_\gamma) = h_\gamma$, we obtain:
\[
[h_\gamma,\ad_t(e_\alpha)] = 2 \frac{(\gamma, \alpha)}{(\gamma, \gamma)} \ad_t(e_\alpha) .
\]
Thus, $\ad_t(e_\alpha)$ belongs to the same root eigenspace as $e_\alpha$. Since the root eigenspaces are 1-dimensional, there exists $\lambda \in \C^\times$ such that:
\begin{equation}
\label{eq:adj_scalar}
\ad_t(e_\alpha)= \lambda e_\alpha.
\end{equation}

For the converse statement, it suffices to find, for every $\lambda \in \C^\times$, some $t \in T_\alpha$ satisfying Equation \ref{eq:adj_scalar}. Looking at the image of $T_{SL(2)}$ under $I_\alpha : SL(2) \to G$ reduces this to a simple computation in $SL(2)$. Choose a square root of $\lambda$, and then:
\[
\left(
\begin{array}{cc}
    \lambda^{1/2} & 0 \\
    0 & \lambda^{-1/2}
\end{array}
\right)
\left(
\begin{array}{cc}
    0 & 1 \\
    0 & 0
\end{array}
\right)
\left(
\begin{array}{cc}
    \lambda^{-1/2} & 0 \\
    0 & \lambda^{1/2}
\end{array}
\right)
=
\left(
\begin{array}{cc}
    0 & \lambda \\
    0 & 0
\end{array}
\right).
\]
\end{proof}

\begin{lemma}
\label{lem:commutation_n}
Let $n \in N$, and $q(n) \in W$ be it's image in the Weyl group. Define $\alpha' = q(n)(\alpha)$. Then there exists some $t \in T_{\alpha'}$ such that:
\begin{itemize}
    \item $\ad_n(e_{\pm\alpha}) = \ad_t (e_{\pm \alpha'})$;
    \item $\Ad_{n} (n_\alpha) = \Ad_{t} (n_{\alpha'})$.
\end{itemize}
\end{lemma}

\begin{proof}
%\todo{MI: edited this proof 1/16; review. BM: Reviewed: 1/17}  
For all $\gamma \in \Phi$ and $h_\gamma$ as in Definition \ref{defin:chevalley_basis}, let $\gamma' = q(n)(\gamma)$. Due to Lemma \ref{lem:w}, $\ad_n(h_\gamma) = h_{\gamma'}$. The fact that the adjoint action preserves the Lie bracket, and relation \ref{eq:chevalley_sl2} from Definition \ref{defin:chevalley_basis} imply:
\[
[h_{\gamma'} , \ad_n(e_{\pm \alpha})] = \pm 2 \frac{(\gamma, \alpha)}{(\alpha, \alpha)} \ad_n(e_{\pm \alpha}).
\]
Similarly,
\[
[h_{\gamma'} , e_{\pm \alpha'}] = \pm 2 \frac{(\gamma', \alpha')}{(\alpha', \alpha')} e_{\pm \alpha'}.
\]

The bilinear pairing between roots is Weyl-invariant, so for all roots $\gamma$:
\[
\frac{(\gamma, \alpha)}{(\alpha, \alpha)} = \frac{(\gamma', \alpha')}{(\alpha', \alpha')}.
\]
Then $\ad_n(e_{\pm \alpha})$ and $e_{\pm \alpha'}$ belong to the same eigenspace for the adjoint action of $\fr t$.
These eigenspaces are one dimensional, which means $\ad_n(e_{\pm \alpha})$ and $e_{\pm \alpha'}$ are scalar multiples of each other.
Now, using Lemma \ref{lem:adj_scalar}, there exist $t_{\pm} \in T$ such that $\ad_n(e_{\pm \alpha}) = \ad_{t_{\pm}}(e_{\pm \alpha'})$. Moreover, due to Lemma \ref{lem:t_ker_comm}, we can assume that $t_{\pm} \in T_{\alpha'}$. To conclude, it remains to show that we can take $t_+ = t_-$.

It suffices to show that $\ad_{t_+}(e_{-\alpha'}) = \ad_{t_-}(e_{-\alpha'})$.
Using that the Lie bracket is preserved under the adjoint action we apply $\ad_n$ to Relation \ref{eq:chevalley_cmm} from Definition \ref{defin:chevalley_basis} to obtain:
\[
[\ad_{t_+}(e_{\alpha'}), \ad_{t_-}(e_{-\alpha'})] = - h_{\alpha'}.
\]
On the other hand, applying $\ad_{t_+}$ (to Relation \ref{eq:chevalley_cmm} for the root $\alpha'$) we obtain:
\[
[\ad_{t_+}(e_{\alpha'}), \ad_{t_+}(e_{-\alpha'})] = - h_{\alpha'}.
\]
Since $\ad_{t_+}(e_{-\alpha'})$ is a scalar multiple of $\ad_{t_-}(e_{-\alpha'})$, these two relations imply they are equal. So there exists $t \in T_{\alpha'}$, such that:
\[
\ad_n(e_{\pm \alpha}) = \ad_t(e_{\pm \alpha'}).
\]

Finally, using Equation \ref{eq: nalpha as exponential}, we obtain $\Ad_n(n_\alpha) = \Ad_t(n_{\alpha'})$.
\end{proof}

\subsection{Stokes Data}
\label{sec: Stokes Data}

In non-abelianization we will modify $G$-local systems by unipotent automorphisms called Stokes factors along a set of curves on $X$ called a spectral network. When these curves intersect, we may need to add additional curves, and assign Stokes factors to them as well. In this section, we define 2D scattering diagrams (Definition \ref{def:scattering_diagram}), which describe the local structure of the intersecting curves around an intersection point\footnote{Strictly speaking this describes the local structure of a basic cameral network.  There is a slight difference between the local structure of a cameral network and that of a spectral network, due to the labels of the lines being different, as is explained in section \ref{sec: Cameral Networks}.}. Then we show (Theorem \ref{thm:assignment_stokes_intersection}) that the Stokes factors for the outgoing curves are uniquely determined by the Stokes factors for the incoming ones.

We begin with some results about unipotent subgroups of $G$ associated to convex sets of roots.

%Consider an intersection of Stokes rays labelled by the roots $\alpha_{1},\alpha_{2},...,\alpha_{j}$ of a Lie algebra $\fg$, with root system $\Phi\subset \ft^{*}$.

%\begin{definition}[Cone of a set of roots]
%\label{def:cone_roots}
%The \emph{Cone} of a set of roots $\alpha_{1},...,\alpha_{j}$ is the set $C_{\alpha_{1},...,\alpha_{j}}:=\Sigma_{i=1}^{j}\bbR_{\geq 0}\alpha_{i}\subset \ft^{*}$.
%\end{definition}

%We will restrict to the case where the set of roots is convex in the following sense:

\begin{definition}[Convex set of roots]
\label{def: convex set of roots}
We say that a set of roots $C \subset \Phi$ is \emph{convex} if 
there exists a polarization $\Phi = \Phi_+ \coprod \Phi_-$ such that $C \subset \Phi_+$.
%the cone $C_{\alpha_{1},...,\alpha_{j}}\subset \ft^{*}$ is strongly convex, in that it does not contain a 1d real vector space.
\end{definition}

\begin{definition}
\label{definition: poyhedral nilpotent lie algebra}
We define the following Lie algebras:
\begin{itemize}
    \item For $\alpha \in \Phi$, let $\fu_\alpha$ denote the root space $\fg_\alpha$.
    \item For $\{\alpha_1, \dots, \alpha_j\}$ a convex set of roots, let $\langle \fu_{\alpha_1} , \dots, \fu_{\alpha_n} \rangle$ denote the Lie subalgebra of $\fg$ generated by $\fu_{\alpha_1}, \dots, \fu_{\alpha_n}$.
    \item For $C \subset \Phi$ a convex subset closed under addition, define:
    \[
    \fu_{C}
    :=
    \bigoplus_{\gamma \in C}\fu_{\gamma}.
    \]
    This is a Lie algebra by Lemma \ref{lem:roots_commutators_vanish}.  We will abuse notation by denoting $\fu_{\{\gamma_{1},...,\gamma_{k}\}}$ as $\fu_{\gamma_1, \dots, \gamma_k}$.
\end{itemize}
\end{definition}

\begin{lemma}
All Lie algebras from Definition \ref{definition: poyhedral nilpotent lie algebra} are nilpotent.
\end{lemma}

\begin{proof}
Due to the convexity assumption, there exists a polarization $\Phi = \Phi_+ \coprod \Phi_-$ of the root system, such that $\{\alpha_1, \dots, \alpha_j\} \subset \Phi_+$. Recall that the Lie algebra:
\[
\fr{n}_+ = \bigoplus_{\alpha \in \Phi_+} \fg_\alpha
\]
is nilpotent. 
The Lie algebras in Definition \ref{definition: poyhedral nilpotent lie algebra} are Lie subalgebras of $\fr n_+$. Lie subalgebras of nilpotent Lie algebras are nilpotent.
\end{proof}

\begin{definition}
For any set of convex roots $\{\gamma_1, \dots, \gamma_k\}$ closed under addition, let $U_{\gamma_1, \dots, \gamma_k} := \exp ( \fu_{\gamma_1, \dots, \gamma_k})$.
\end{definition}

\begin{remark}
For any nilpotent Lie algebra $\fu$, the exponential map $\exp : \fu \to U$ is algebraic (The Taylor series is finite).  Therefore, all constructions in this section that involve the exponential map makes sense in the setting of algebraic groups. Moreover, for $\fu$ nilpotent, $\exp :\fu \to U$ is an isomorphism of schemes.
%\todo{reference}
\end{remark}

\begin{definition}
\label{def:convex_hulls}
Let $\{\alpha_1, \dots, \alpha_j\}$ be a convex set of roots. 
We define their \textbf{restricted convex hull} as the subset:
\[
\Conv^\N_{\alpha_1, \dots, \alpha_j} 
:=
\left\{ 
\gamma \in \Phi 
| 
\gamma = \sum_{i=1}^j n_i \alpha_i, 
\; 
n_i \in \N \cup\{0\} 
\right\} .
\]
\end{definition}

The restricted convex hull is motivated by the following reformulation of Lemma \ref{lem:roots_commutators_vanish}.

\begin{lemma}
\label{lem:iterated_brackets_roots}
Let $\alpha, \beta \in \Phi$, such that $\alpha \neq \pm \beta$. Then $\langle \fu_{\alpha},\fu_{\beta}\rangle = \fu_{\Conv^\N_{\alpha,\beta}}$, using the notation of Definition \ref{definition: poyhedral nilpotent lie algebra}.
\end{lemma}

\begin{proof}
By Lemma \ref{lem:roots_commutators_vanish}, $[\fu_\alpha, \fu_\beta] = \fu_{\alpha + \beta}$ if $\alpha + \beta \in \Phi$, and $[\fu_\alpha, \fu_\beta] = 0$ otherwise. By recursive application of this result, we obtain that $\langle \fu_{\alpha}, \fu_{\beta}\rangle$ contains $\fu_{\gamma}$ if and only if $\gamma \in \Conv_{\alpha, \beta}^\N$.
\end{proof}

\begin{example}
\label{eg:convex_hull_typeA}
In root systems of type ADE, for every convex pair $\{\alpha_1, \alpha_2\}$, $\Conv^\N_{\alpha_1, \alpha_2}=\{\alpha_{1},\alpha_{2}\}$. This is because the restriction of the root system to the plane spanned by $\alpha_1$ and $\alpha_2$ is a root system of type $A_1 \times A_1$ or $A_2$. In both cases, the claim is obvious. (See Figures \ref{fig:a1root} and \ref{fig:a2root}.)
\end{example}

\begin{figure}[h]
    \centering
    \begin{minipage}{.5\textwidth}
      \centering
      \includegraphics[width = 0.7 \textwidth]{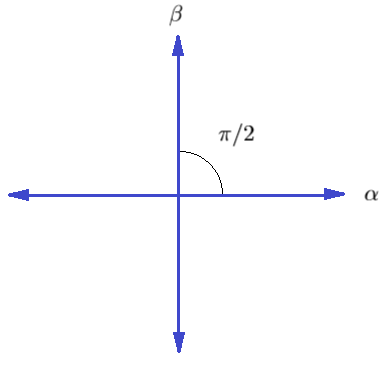}
      \caption{The root system $A_1 \times A_1$.}
      \label{fig:a1root}
    \end{minipage}%
    \begin{minipage}{.5\textwidth}
      \centering
      \includegraphics[width = 0.8 \textwidth]{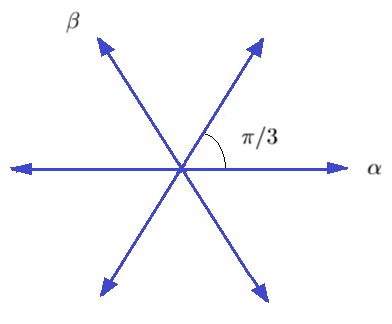}
      \caption{The root system $A_2$.}
      \label{fig:a2root}
    \end{minipage} 
\end{figure}

\begin{example}
In a root system of type B2 (see Figure \ref{fig:b2root}), let $\alpha_1, \alpha_2$ be orthogonal long roots. Then:
\[
\Conv^\N_{\alpha_1, \alpha_2} = \{\alpha_1,  \alpha_2\}.
\]

Because of the restriction that the coefficients must lie in $\bbN\cup \{0\}$ we have that $(\alpha_1 + \alpha_2)/2\notin \Conv^\N_{\alpha_1, \alpha_2}$.
\end{example}

\begin{figure}[h]
    \centering
    \begin{minipage}{.5\textwidth}
      \centering
      \includegraphics[width = 0.75 \textwidth]{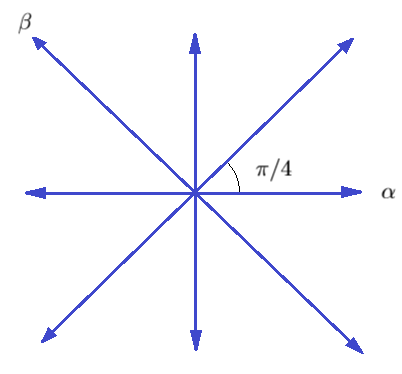}
      \caption{The root system $B_2$.}
      \label{fig:b2root}
    \end{minipage}%
    \begin{minipage}{.5\textwidth}
      \centering
      \includegraphics[width = 0.67\textwidth]{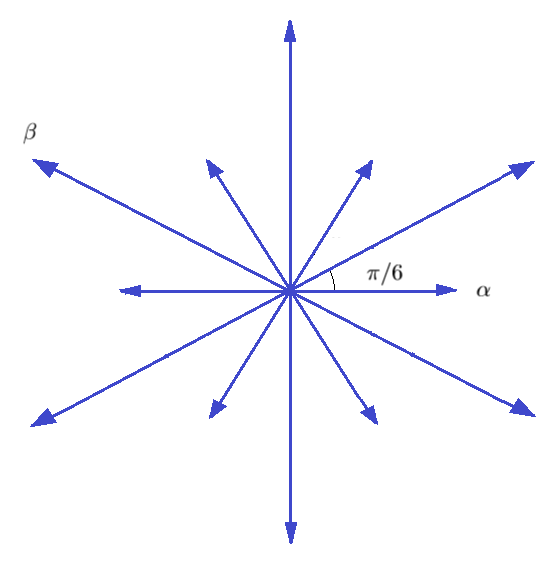}
      \caption{The root system $G_2$.}
      \label{fig:g2root}
    \end{minipage}
\end{figure}

See Appendix \ref{sect:planar} for other explicit examples, in the case of the planar root systems of Figures \ref{fig:a1root}-\ref{fig:g2root}.

\begin{definition}[Height of a root]
\label{definition: height of a root}
Given a polarization $\Phi = \Phi_+ \coprod \Phi_-$, we let $\{\alpha_1, \dots, \alpha_d\}$ denote the simple roots with respect to this polarization.  Recall that the simple roots are a basis for the root system. Then all $\gamma \in \Phi^{+}$ can be written uniquely as:
\[
\gamma = \sum_{i=1}^d n_i \alpha_i ,\; n_i \in \N \cup\{0\}.
\]
We define the \textbf{height} of $\gamma\in \Phi^{+}$ as $ht(\gamma) = \sum_{i=1}^d n_i$.
\end{definition}

Note that, $ht(\gamma_1) + ht(\gamma_2) = ht(\gamma_1 + \gamma_2)$. 

\begin{lemma}
\label{lemma on height of commutator}
Let $X_{\beta}\in \fu_{\beta}$, nd $X_{\gamma}\in \fu_{\gamma}$.  If $0\neq [X_{\beta},X_{\delta}] \in \fu_{\gamma}$, then $ht(\beta), ht(\delta) < ht(\gamma)$.
\end{lemma}

\begin{proof}
This follows immediately from Lemma \ref{lem:roots_commutators_vanish}.
\end{proof}

\begin{proposition}
\label{prop:mult_iso}
Let $C \subset \Phi$ be a convex subset such that $C = \Conv_C^\N$. Then multiplication gives an isomorphism of schemes:
\[
\prod_{\gamma \in C} U_{\gamma}
\overset{M}{\longrightarrow}
U_C ,
\]
for any ordering of the product on the left hand side.
\end{proposition}

\begin{warning}
The morphism in Proposition \ref{prop:mult_iso} does \emph{not} in general respect the group structure.
\end{warning}

\begin{proof}
We use the Baker-Campbell-Hausdorff formula:
\[
\exp(X) \exp(Y) 
= 
\exp \left( X + Y + \frac{1}{2}[X,Y] + \dots \right) ,
\]
where the dots indicate higher order iterated Lie brackets of $X$ and $Y$. We only need this formula for the case when each of $X$, $Y$ spans a root space of $\fg$. Due to Lemma \ref{lem:roots_commutators_vanish} and the convexity assumption, there are only finitely many nonzero iterated Lie brackets in this case.

For all $\gamma \in C$, let $X_\gamma \in \fu_\gamma$. Applying the Baker-Campbell-Hausdorff formula iteratively, we obtain:
\[
\prod_{\gamma \in C} \exp(X_\gamma) 
= 
\exp \left( \sum_{\gamma \in C} (X_\gamma + \text{brack}_\gamma) \right) ,
\]
where $\text{brack}_\gamma$ is the sum of all iterated Lie brackets appearing in the resulting exponential which belong to the root space $\fu_\gamma$.  Note that $\text{brack}_\gamma$ depends on the set $\{X_{\gamma}\}_{\gamma\in C}$.

It follows that we have a commutative diagram of schemes:
\[
\begin{tikzcd}
\prod_{\gamma \in C} U_{\gamma} 
\arrow{r}{M}
&
U_C
\\
\bigoplus_{\gamma \in C} \fu_\gamma
\arrow{u}{\prod_{\gamma \in C} \exp }
\arrow{r}{m}
&
\fu_C ,
\arrow[swap]{u}{\exp}
\end{tikzcd}
\]
where $m$ is the map:
\begin{equation}
\label{eq:mult_bijective}
(X_\gamma)_{\gamma \in C} 
\longmapsto
\sum_{\gamma \in C} X_\gamma + \text{brack}_\gamma .
\end{equation}
Since the vertical arrows are isomorphisms of schemes (because the algebraic groups are unipotent), it suffices to prove that $m$ is invertible. Because we can compose $m$ with the projections $\fu_C \to \fu_\gamma$, invertibility means recovering the input tuple $(X_{\gamma})_{\gamma \in C}$ from the output tuple $(X_\gamma + \text{brack}_\gamma )_{\gamma \in C}$.
We will argue by induction on the height of $\gamma \in C$.

Since $C$ is convex, there exists a polarization $\Phi = \Phi_+ \coprod \Phi_-$ such that $C \subset \Phi_+$.  We consider the height of a root (Definition \ref{definition: height of a root}) with respect to such a polarization.

By Lemma \ref{lemma on height of commutator} if a non-zero iterated Lie bracket involving $X_{\beta}$ belongs to $\fu_\gamma$, then $ht(\beta) < ht(\gamma)$. In other words, $\text{brack}_\gamma$ only depends on those $X_{\beta}$ with $ht(\beta) < ht(\gamma)$.

We now argue that $m$ is invertible by inductively constructing an inverse, where we are inducting on the height of the element we are applying $m$ to.  The base case is given by all $\gamma$ of minimal height: for these, $\text{brack}_\gamma = 0$, and the composition of $m$ with the projection $\fu_C \to \fu_\gamma$ recovers $X_\gamma$.
For the inductive step, assume we know $X_\beta$ for all $\beta \in C$ such that $ht(\beta) < ht(\gamma)$. These determine $\text{brack}_\gamma$,
so we can recover $X_\gamma$ uniquely from the output of Equation \ref{eq:mult_bijective}.
\end{proof}

\begin{example}
For $\fg = \fr{sl}_3$, choose a polarization so that the positive root spaces correspond to upper-triangular matrices. We have the explicit formula:
\begin{align*}
\left( \begin{array}{ccc}
    1 & a & 0 \\
    0 & 1 & 0 \\
    0 & 0 & 1
\end{array}
\right)
\left( \begin{array}{ccc}
    1 & 0 & c \\
    0 & 1 & 0 \\
    0 & 0 & 1
\end{array}
\right)
\left( \begin{array}{ccc}
    1 & 0 & 0 \\
    0 & 1 & b \\
    0 & 0 & 1
\end{array}
\right)
&=
\left( \begin{array}{ccc}
    1 & a & c+ab \\
    0 & 1 & b \\
    0 & 0 & 1
\end{array}
\right)
\\
= \exp 
&
\left( \begin{array}{ccc}
    0 & a & c + ab/2 \\
    0 & 0 & b \\
    0 & 0 & 0
\end{array}
\right)
\end{align*}
The map $(a,b,c) \mapsto (a, b , c+ ab/2)$ is clearly invertible.
\end{example}

We remark that Proposition \ref{prop:mult_iso} is really a statement about root spaces. If we use a basis for $\fg$ that is not compatible with the root space decomposition, then the result need not be true as Example \ref{example: really needed root spaces} shows.

\begin{example}
\label{example: really needed root spaces}
Let $\fg = \fr{sl}_3$, and choose the following basis for the Lie subalgebra of strictly upper triangular matrices:
\[
\left( \begin{array}{ccc}
    0 & 1 & 1 \\
    0 & 0 & 0 \\
    0 & 0 & 0
\end{array}
\right),
\hspace{1cm}
\left( \begin{array}{ccc}
    0 & 0 & 0 \\
    0 & 0 & 1 \\
    0 & 0 & 0
\end{array}
\right),
\hspace{1cm}
\left( \begin{array}{ccc}
    0 & 1 & -1 \\
    0 & 0 & 0 \\
    0 & 0 & 0
\end{array}
\right).
\]
Then:
\begin{align*}
\left( \begin{array}{ccc}
    1 & a & a \\
    0 & 1 & 0 \\
    0 & 0 & 1
\end{array}
\right)
\left( \begin{array}{ccc}
    1 & 0 & 0 \\
    0 & 1 & b \\
    0 & 0 & 1
\end{array}
\right)
\left( \begin{array}{ccc}
    1 & c & -c \\
    0 & 1 & 0 \\
    0 & 0 & 1
\end{array}
\right)
=&
\left( \begin{array}{ccc}
    1 & a+c & a-c + ab \\
    0 & 1 & b \\
    0 & 0 & 1
\end{array}
\right).
\end{align*}
Elements of the form:
\[
\left( \begin{array}{ccc}
    1 & x & z \\
    0 & 1 & y \\
    0 & 0 & 1
\end{array}
\right)
\]
with $y = -2$ and $x \neq -z$ form a codimension 1 locus not in the image of the multiplication map.
\end{example}

We will obtain an immediate corollary (\ref{cor:stokes_sector}) from Proposition \ref{prop:mult_iso}. To state it, we need some additional definitions.

\begin{definition}
\label{def:scattering_diagram}
Let $C_{in} \subset \Phi$ be a convex set, and set $C_{out} = \Conv^\N_{C_{in}}$.
An \textbf{undecorated 2D scattering diagram} is a finite collection of distinct oriented rays in $\bbR^{2}$, starting or ending at $\{0\}\in \bbR^{2}$, together with the data of:
\begin{itemize}
    \item a bijection between the set of incoming rays and $C_{in}$ (we say that incoming half-lines are labeled by elements of $C_{in}$);
    \item a bijection between the set of outgoing rays and $C_{out}$;
    
\end{itemize}
A \textbf{decorated 2D scattering diagram} is an undecorated 2d scattering diagram together with:
\begin{itemize}
    \item  For every ray $\ell$ with label $\alpha$, an element $u_{\ell} \in U_\alpha$ called the \textbf{Stokes factor},
\end{itemize}
such that the product taken over both incoming and outgoing Stokes factors, in clockwise order around the intersection point, is the identity:
\begin{equation}
\label{eq:ordered_prod}
\overrightarrow{\prod}_{\ell \in C_{in} \coprod C_{out}} u_{\ell}^{\pm 1} = \id .
\end{equation}
Here $\overrightarrow{\prod}$ denotes a clockwise-ordered product\footnote{This condition is independent of the ray at which we start the clockwise-ordered product.},
and the exponent accounts for orientation: it is $-1$ for incoming rays, and $+1$ for outgoing rays.
\end{definition}

%\todo{Do we want a remark about how we would like to remove the implicit assumptions of lines associated to different roots being distinct, and all incoming labels being a convex set?}

\begin{definition}
A \textbf{solution to an (undecorated) 2D scattering diagram} is informally a way to assign Stokes factors $u_\gamma \in U_\gamma$ to the outgoing half-lines, given arbitrary Stokes factors on the incoming rays, such that the result is a decorated 2D scattering diagram. Formally, it is a morphism of schemes:
\[
\prod_{\alpha \in C_{in}} U_\alpha
\to 
\prod_{\gamma \in C_{out}} U_\gamma,
\]
such that the product in Equation \ref{eq:ordered_prod}, taken over the inputs and outputs of the morphism, is the identity.
\end{definition}

\begin{remark}
A solution to an (undecorated) 2D scattering diagram does not depend on the angles between the Stokes rays.  In fact it only depends on the cyclic order of the rays.
\end{remark}

See Figures \ref{fig:outgoing_sector}, \ref{fig:A3_joint} for examples. This definition is motivated by the local structure of cameral networks, as defined in Section \ref{sec: Cameral Networks}, around intersections of Stokes curves.

\begin{corollary}
\label{cor:stokes_sector}
Consider a 2D scattering diagram such that the incoming rays are constrained to a sector of the plane with central angle $< \pi$, and the outgoing rays are constrained to the opposite sector as in Figure \ref{fig:outgoing_sector}. Then the scattering diagram has a unique solution.

\begin{figure}[h]
    \centering
    \includegraphics[width = 0.7 \textwidth]{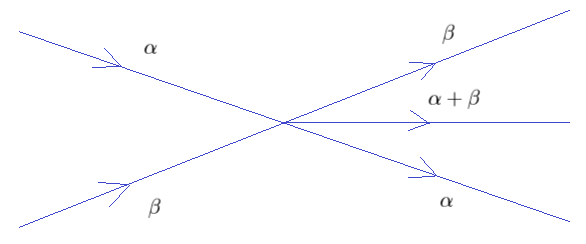}
    \caption{A 2D scattering diagram in which incoming and outgoing curves are restricted to opposite sectors.}
    \label{fig:outgoing_sector}
\end{figure}

\end{corollary}

\begin{proof}
Due to the assumption about separation of incoming and outgoing rays, equation \ref{eq:ordered_prod} has the form:
\[
\overrightarrow{\prod}_{\alpha \in C_{in}} u_\alpha^{-1} \
\overrightarrow{\prod}_{\gamma \in C_{out}} u_\gamma = \id .
\]
All factors in the first product are known, and all factors in the second product must be determined.
Let:
\begin{align*}
m_{in} &: \overleftarrow{\prod}_{\alpha \in C_{in}} U_\alpha \to U_{C_{out}} ,
\\
m_{out} &: \overrightarrow{\prod}_{\gamma \in C_{out}} U_\gamma \to U_{C_{out}} .
\end{align*}
be the multiplication maps, where $\overrightarrow{\prod}$ denotes a clockwise-ordered product, and $\overleftarrow{\prod}$ a counterclockwise-ordered product. Then the solution to the 2D scattering diagram is the composition:
\begin{equation}
\label{equation: Equation of corollary 3.32}
\begin{tikzcd}
\overleftarrow{\prod}_{\alpha \in C_{in}} U_\alpha
\arrow{r}{m_{in}}
&
U_{C_{out}}
\arrow{r}{m_{out}^{-1}}
& 
\overrightarrow{\prod}_{\gamma \in C_{out}} U_\gamma .
\end{tikzcd}
\end{equation}
The map $m_{out}^{-1}$ exists by Proposition \ref{prop:mult_iso}.
\end{proof}

In Appendix \ref{sect:planar}, we give explicit formulas for the composition $m_{out}^{-1} \circ m_{in}$, in the case of planar root systems.

\subsubsection{Rays Not Separated}

In the remainder of the section, we generalize Corollary \ref{cor:stokes_sector} by allowing incoming and outgoing rays in the 2D scattering diagram to be interspersed, as in Example \ref{eg:A3_joint}.

The subset $\sum_{\gamma \in C_{out}} \R_+ \cdot \gamma \subset \ft^*$ is a cone with vertex at the origin.  By a \emph{face} of $C_{out}$ we refer to the intersection of $C_{out}$ with a face of the cone $\sum_{\gamma \in C_{out}} \R_+ \cdot \gamma$.  Recall that by a face of the cone $\sum_{\gamma \in C_{out}} \R_+ \cdot \gamma$ we mean an intersection of this cone with a (real) half-space of $\ft$, such that no interior points of the cone are included in this intersection. Throughout, we use the notation $\Delta \subset C_{out}$ to denote a face of $C_{out}$ of unspecified dimension.

%The restriction of the faces of this cone to the finite set $C_{out}$ provides a rational cone structure on $C_{out}$. \todo{Is ``rational'' the correct adjective?} 

\begin{lemma}[Faces and projection]
\label{lem:faces_projection}
Let $\Delta_{f}\subset \Delta$ be a face. The projection $p_f: \fu_{\Delta}\rightarrow \fu_{\Delta_{f}}$ with kernel $\oplus_{\gamma\in \Delta \setminus \Delta_{f}}\fu_{\gamma}$ is a morphism of Lie algebras.
\end{lemma}

\begin{proof}
We need to prove that $p_f\big([X_1,X_2]\big) = \big[p_f(X_1),p_f(X_2)\big]$, for all $X_1, X_2 \in \fr u_{\Delta}$. We analyze two cases.
\begin{itemize}
    \item If $X_1, X_2 \in \fu_{\Delta_f}$, then $[X_1, X_2] \in \fu_{\Delta_f}$, due to Lemma \ref{lem:roots_commutators_vanish}. Hence $p_f$ leaves each of $X_1, X_2, [X_1,X_2]$ unchanged.
    \item If at least one of $X_1, X_2$ is not in $\fu_{\Delta_f}$, then, using Lemma \ref{lem:roots_commutators_vanish} and the assumption that $\Delta_f$ is a face, either $[X_1, X_2] = 0$ or $[X_1, X_2] \not \in \fu_{\Delta_f}$. Thus, the equality $p_f\big([X_1,X_2]\big) = \big[p_f(X_1),p_f(X_2)\big]$ holds with both sides equal to 0.
\end{itemize}
\end{proof}

As a consequence:

\begin{lemma}
\label{lem:projection_to_face}
The morphism of algebraic groups $U_{\Delta}\to U_{\Delta_{f}}$ corresponding to the projection of Lemma \ref{lem:faces_projection}, acts as the identity on $U_{\Delta_f}$ and sends every element of the form $\exp(X_\gamma)$ ($X_{\gamma}\in \fu_{\gamma}$) for $\gamma \not \in \Delta_f$ to $\id_{U_{\Delta_f}}$.
\end{lemma}

The following easy example demonstrates our strategy for assigning Stokes factors to outgoing rays in 2D scattering diagrams, making use of Lemma \ref{lem:projection_to_face}. We then generalize and formalize this strategy in Theorem \ref{thm:assignment_stokes_intersection}.

\begin{example}
\label{eg:A3_joint}
Let $\fg = \fr{sl}_4$, and consider the polarization such that the positive roots correspond to strictly upper triangular matrices. Let $\alpha, \beta, \gamma$ denote the simple roots. Then the positive roots and their root spaces are:
\[
\left( \begin{array}{cccc}
    0 & \alpha &\alpha + \beta & \alpha + \beta + \gamma \\
    0 & 0 & \beta & \beta + \gamma \\
    0 & 0 & 0 & \gamma \\
    0 & 0 & 0 & 0
\end{array}
\right)
\]
We consider the 2D scattering diagram pictured in Figure \ref{fig:A3_joint}\footnote{This example was not obtained from a WKB construction.}, with three incoming rays, labeled by the simple roots $\alpha, \beta, \gamma$. Then $C_{out} = \Conv^\N_{\alpha,\beta,\gamma} = \Phi_+$.

\begin{figure}[h]
    \centering
    \includegraphics[width = 0.7 \textwidth]{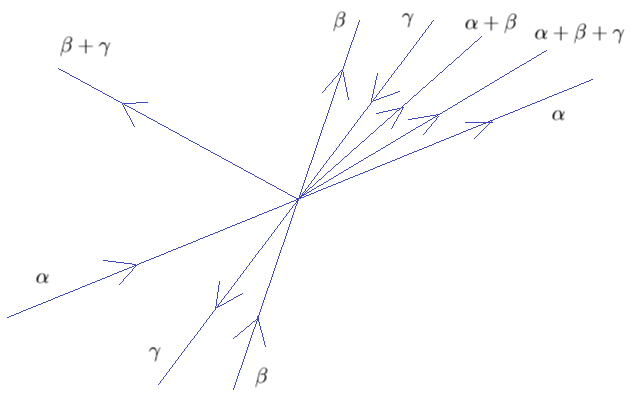}
    \caption{A 2D scattering diagram labeled by positive roots of $\fr{sl}_4$.}
    \label{fig:A3_joint}
\end{figure}

Equation \ref{eq:ordered_prod}, which expresses the fact that the clockwise-ordered product of Stokes factors around the intersection is the identity, has the explicit form:
\begin{equation}
\label{eq:stokes_A3_bulk}
u_\alpha^{-1} u'_\gamma u_\beta^{-1} u'_\alpha u'_{\alpha + \beta + \gamma} u'_{\alpha + \beta} u_\gamma^{-1} u'_{\beta} u'_{\beta + \gamma} 
= \id ,
\end{equation}
where $u_\alpha, u_\beta, u_\gamma$ are the known incoming Stokes factors, and the elements with a prime are the outgoing Stokes factors which need to be determined.

\begin{figure}[h]
    \centering
    \includegraphics[width = 0.5 \textwidth]{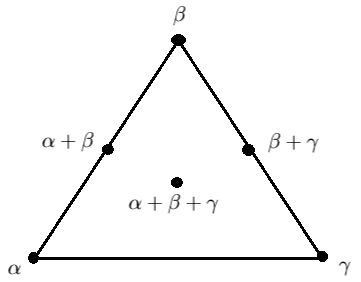}
    \caption{A cross-section of the cone over $\Phi_+$ in the root system A3.}
    \label{fig:A3_cone_section}
\end{figure}

Figure \ref{fig:A3_cone_section} shows a cross-section of the cone structure on $C_{out}$, with labels indicating 1-dimensional faces. Let $\Delta$ denote the 2D face spanned by $\alpha$ and $\beta$. The image of equation \ref{eq:stokes_A3_bulk} under the projection $U_{C_{out}} \to U_{\Delta}$ is:
\begin{equation}
\label{eq:stokes_A3_face}
u_\alpha^{-1}u_\beta^{-1}u'_\alpha u'_{\alpha + \beta} u'_\beta = \id .
\end{equation}

Projecting further to the 1-dimensional face spanned by $\alpha$, we obtain $u_\alpha = u'_\alpha$. Analogously, we obtain $u_\beta = u'_\beta$.
Substituting these into \ref{eq:stokes_A3_face} gives:
\begin{equation}
\label{eq:u'_comm}
u'_{\alpha + \beta} = u_\alpha^{-1}u_\beta u_\alpha u_\beta^{-1}  .
\end{equation}
We can confirm that $u'_{\alpha + \beta} \in U_{\alpha +\beta}$ as follows:
The RHS of equation \ref{eq:u'_comm} is in the kernel of both projections $U_{\Delta} \to U_\alpha$ and $U_{\Delta} \to U_\beta$. The intersection of the two kernels is precisely $U_{\alpha + \beta}$.

The elements $u'_\gamma$, $u'_{\beta + \gamma}$ are determined analogously. Then $u'_{\alpha + \beta + \gamma}$ is the only remaining unknown in \ref{eq:stokes_A3_bulk}, so it is uniquely determined:

\begin{equation}
u'_{\alpha + \beta + \gamma} 
=
u_\alpha^{-1}u_\beta u_\gamma^{-1}u_\alpha 
(u_\beta^{-1}u_\gamma u_\beta u_\gamma^{-1}) u_\beta^{-1}u_{\gamma} (u_\beta u_\alpha^{-1}u_\beta^{-1} u_\alpha) .
\end{equation}
Again, this element is in the intersection of the kernels of all face projections, which is $U_{\alpha + \beta + \gamma}$.
\end{example}

Example \ref{eg:A3_joint} was made easy by the fact that, at every stage, there is at most a single root which doesn't lie on one of the faces. Clearly, we cannot expect this simplification in general: Figures \ref{fig:b2root} and \ref{fig:g2root} show that this fails for the root systems B2 and G2, respectively. In fact it fails for the simply laced root systems D4, as the following example shows.

\begin{example}
\label{eg:D4}
\begin{figure}[h]
    \centering
    \includegraphics[width = 0.25 \textwidth]{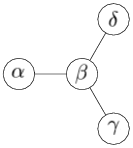}
    \caption{Dynkin diagram for D4, with simple roots labeled.}
    \label{fig:dynkin_D4}
\end{figure}
Consider the root system D4, with simple roots $\{\alpha, \beta, \gamma, \delta\}$ as labeled on the Dynkin diagram in Figure \ref{fig:dynkin_D4}. The positive roots, ordered by the dimension of the smallest-dimensional face they lie on, are:
\begin{align*}
&1:\alpha, \beta, \delta, \gamma ,
\\
&2:\alpha + \beta, \beta + \gamma, \beta + \delta,
\\
&3:\alpha + \beta + \gamma , \alpha + \beta + \delta , \beta + \gamma + \delta ,
\\
&4:\alpha + \beta + \gamma + \delta , \alpha + 2\beta + \gamma + \delta .
\end{align*}
There are two roots in the interior of the 4-dimensional cone.
\end{example}

The following result generalizes the strategy of Example \ref{eg:A3_joint} appropriately. For convenience, we introduce some notation for an ordered product of unipotent elements.

Let $u_{\gamma}^{(')}$ refer to $u_{\gamma}'$ for outgoing lines, and to $u_{\gamma}$ for incoming lines.  Let $\Delta$ be a face of $C_{out}$ of arbitrary dimension. Then we define:
\begin{equation}
\label{eq:u_cout}
u_\Delta := \overrightarrow{\prod}_{\gamma \in (C_{in}\cap \Delta) \coprod (C_{out}\cap \Delta)} (u_\gamma^{(')})^{\pm 1}
\end{equation}
as a product over all Stokes factors corresponding to roots $\gamma \in \Delta$, taken in clockwise order around the intersection, where the exponent is $+1$ for outgoing curves, and $-1$ for incoming curves.  In order to define this we need to choose a ray at which we start the clockwise ordering.  Different choices will result in conjugate values of $u_{\Delta}$.  Equation \ref{eq:ordered_prod}, expressing the constraint that a solution to a 2D scattering diagram must satisfy, can be rewritten as $u_{C_{out}} = \id$.

\begin{theorem}
\label{thm:assignment_stokes_intersection}
Every 2D scattering diagram has a unique solution. Concretely, this means that there is a unique morphism of schemes:
\begin{align*}
\prod_{\gamma \in C_{in}} U_{\gamma} 
&\xrightarrow{s} 
\prod_{\gamma \in C_{out}} U_{\gamma}
\\
(u_\gamma)_{\gamma \in C_{in}} 
&\xmapsto{s}
(u'_\gamma)_{\gamma \in C_{out}}
\end{align*}
such that $u_{C_{out}}: = \overrightarrow{\prod}_{\gamma \in C_{in} \cup C_{out}} (u_\gamma^{(')})^{\pm 1} = \id$, where we are using the notation and sign conventions of equation \ref{eq:u_cout}.

Furthermore for any decoration $(u_\gamma)_{\gamma \in C_{in}}$, $s((u_\gamma)_{\gamma \in C_{in}})$ is the unique set of Stokes factors for the outgoing rays such that $u_{C_{out}} = \id$, using the notation of equation \ref{eq:u_cout}, and the specified decorations on the incoming rays.
\end{theorem}

\begin{remark}
The proof we provide gives a construction of this morphism.
\end{remark}

\begin{remark}
The second uniqueness statement is not implied by the fact that the 2D scattering diagram has a solution, because a priori for a particular decoration there might be multiple assigments of Stokes factors to outgoing lines satisfying the equation $u_{C_{out}} = \id$.  The uniqueness of the solution only implies that exactly one of them would extend to a solution of the scattering diagram.
\end{remark}

\begin{proof}
The proof is by induction on the dimension of the faces $\Delta$ of $C_{out}$. Our inductive hypothesis is that there are unique $\{u_\gamma'\}_{\gamma \in \Delta}$ such that $u_\Delta = \id$.  Recall the definition of $u_{\Delta}$ from Equation \ref{eq:u_cout} and note that while this depended on a choice of ray, the condition $u_\Delta = \id$ is independent of this choice.

For the base step, if $\Delta$ is a 1-dimensional face, then $u_{\Delta} = u_\gamma' u_\gamma^{-1}$. The unique solution to the equation $u_\Delta = \id$ is $u_\gamma' = u_\gamma$.

For the inductive step, we assume the inductive hypothesis for all faces $\Delta_f$ of $\Delta$, together with the additional assumption that if $\gamma$ lies in multiple faces $\Delta_{f}$, the $u_{\gamma}'$ assigned to it by the hypothesis for each face is independent of the face $\Delta_{f}$ considered. Applying the morphism in Lemma \ref{lem:projection_to_face} to Equation \ref{eq:u_cout} gives the version of Equation \ref{eq:u_cout} for $\Delta_{f}$. Applying the inductive hypothesis determines $u'_{\gamma}$ uniquely for all $\gamma$ belonging to some face $\Delta_f$ of $\Delta$.
It suffices to determine $u'_\gamma$ for $\gamma \in \Delta^{int} := \Delta \setminus \cup_f \Delta_f$, in such a way that $u_{\Delta}$ is the identity, and show that there is a unique way to do this.  In particular the inductive construction will satisfy the additional hypothesis that if $\gamma$ lies in multiple faces $\Delta_{f}$, the $u_{\gamma}'$ assigned to it by the hypothesis for each face in independent of the face $\Delta_{f}$ considered.  This is because in every case $u_{\gamma}'$ will be equal to the Stokes factor assigned to the outgoing ray labelled by $\gamma$ by the lowest dimensional face containing $\gamma$.

\begin{enumerate}
    \item 
    \label{case:curves_together}
    First, we analyze the situation in which the outgoing rays for $\gamma \in \Delta^{int}$ are not separated by any curves labeled by roots lying on a face of $\Delta$. Then $u_\Delta = \id$ gives an equation:
    \begin{equation}
    \label{eq:prod_int}
    \overrightarrow{\prod}_{\gamma \in \Delta^{int}} u'_{\gamma} = u_{known},
    \end{equation}
    where the LHS is an ordered product over all elements that need to be determined, and the RHS is a function of
    quantities that are already known, namely $u_\gamma$ or $u'_\gamma$ for $\gamma$ on one of the faces of $\Delta$. Moreover, the projection of $u_{known}$ to each $U_{\Delta_f}$ is $u_{\Delta_f}$, which is the identity due to the inductive hypothesis. Thus:
    \[
    u_{known} \in \Ker(U_\Delta \longrightarrow \prod_f U_{\Delta_f}) = U_{\Delta^{int}} .
    \]
    By Proposition \ref{prop:mult_iso}, the multiplication map:
    \begin{equation}
    \label{eq:mult_delta_int}
    \overrightarrow{\prod}_{\gamma \in \Delta^{int}} U_\gamma \overset{M}{\longrightarrow} U_{\Delta^{int}}
    \end{equation}
    is an isomorphism of schemes. Therefore, $M^{-1}(u_{known})$ is the unique tuple satisfying Equation \ref{eq:prod_int}, thus ensuring that $u_{\Delta} = \id$.
    
    \item 
    \label{case:two_curves}
    Second, we analyze the situation in which there are are two curves labeled by $\gamma_1, \gamma_2 \in \Delta^{int}$, which are separated by a number of curves labeled by $\beta_1, \dots , \beta_k \in \cup_f \Delta_f$, as in Figure \ref{fig:exchange_cone}.
    
    \begin{figure}[ht]
        \centering
        \includegraphics[width = 0.2 \textwidth]{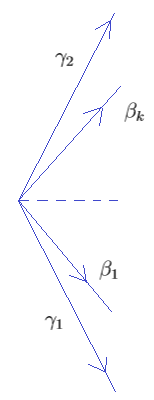}
        \caption{Two curves labeled by roots in $\Delta^{int}$, separated by $k$ curves labeled by roots not in $\Delta^{int}$. (The curves outside the cone spanned by $\gamma_1, \gamma_2$ are not drawn.)}
        \label{fig:exchange_cone}
    \end{figure}
    
    We reduce this to case \ref{case:curves_together} by repeated applications of Proposition \ref{prop:mult_iso}. Concretely, to exchange the curves labeled by $\gamma_1$ and $\beta_1$, we use the composition of Equation \ref{equation: Equation of corollary 3.32}:
    \begin{equation}
    \label{eq: twisted transposition}
    \begin{tikzcd}
    U_{\gamma_1} \times U_{\beta_1} 
    \arrow{r}
    &
    U_{\Conv^\N_{\gamma_1, \beta_1}}
    \arrow{r}
    &
    U_{\beta_1} \times \prod_{\delta \in \Conv^\N_{\gamma_1,\beta_1}\setminus \{\gamma_1, \beta_1\}} U_\delta \times U_{\gamma_1}
    \end{tikzcd}
    \end{equation}
    Explicit formulas for these maps are given in Lemma \ref{lem:swap2}.  We call the application of this map a twisted transposition.
    Since $\gamma_1 \in \Delta^{int}$, $\Conv^\N_{\gamma_1,\beta_1} \setminus \{\gamma_1, \beta_1\} \subset \Delta^{int}$. So the new factors $U_\delta$ that we introduce are subgroups of $U_{\Delta^{int}}$. In the next step, we need to exchange $U_{\gamma_1}$ and all $U_\delta$ past $U_{\beta_2}$. Continuing in this way for all $\beta_j$, we are almost in the setting of case \ref{case:curves_together}.  However unlike case \ref{case:curves_together} we may have multiple rays labelled by the same root, due to the new rays created by the twisted transpositions.  We leave dealing with this subtlety to case \ref{case: no holds barred}.

    \item \label{case: no holds barred}Finally, when there are more than two roots in $\Delta^{int}$, and the corresponding curves are not next to each other, we choose one of the curves to keep fixed, and apply the procedure in case \ref{case:two_curves} to all of the others, to move them next to the chosen curve. This procedure involves repeated twisted transpositions of two elements $u_{\gamma_i}u_{\beta_j}$, in order to obtain a product $u_{\beta_j}\cdot \prod_{\delta \in \Conv^\N_{\gamma_i,\beta_j} \setminus \{\gamma_i, \beta_j\}}u_\delta \cdot u_{\gamma_i}$ as in Equation \ref{eq: twisted transposition}.
    We do this to neighboring factors in the equation $u_\Delta = \id$ until factors involving roots in $\Delta^{int}$ are no longer separated. This yields:
    \begin{equation}
    \label{eq:multiset}
    \prod_{\delta \in S} \tilde u_\delta = u_{known},
    \end{equation}
    where:
    \begin{itemize}
        \item $u_{known} \in U_{\Delta^{int}}$ is a product of factors corresponding to roots in $\cup_f \Delta_f$, just like in case \ref{case:curves_together}.
        \item $S$ is a multiset containing elements of $\Delta^{int}$, possibly with multiplicity. The reason we obtain a multiset is that the same $\delta \in \Phi$ may appear in $\Conv^\N_{\gamma_i,\beta_j}$ and $\Conv^\N_{\gamma_k,\beta_l}$, for $i\neq k$ and/or $j\neq l$.
        \item $\tilde u_{\delta}$ is either $u'_\delta$ for some $\delta \in \Delta^{int}$ (i.e. one of the unknowns which must be determined), or $u_\delta^{\beta_i,\gamma_j}$, defined as the image of $u_{\beta_j}^{-1}u_{\gamma_i}u_{\beta_j}u_{\gamma_i}^{-1} \in U_{\Conv^\N_{\gamma_i,\beta_j} \setminus \{\gamma_i, \beta_j\}}$ under the composition:
        \begin{equation}
        \label{eq:mult_comm_proj}
        \begin{tikzcd}
        U_{\Conv^\N_{\gamma_i,\beta_j} \setminus \{\gamma_i, \beta_j\}}
        \arrow{r}{M^{-1}}
        &
        \prod_{\nu \in \Conv^\N_{\gamma_i,\beta_j} \setminus \{\gamma_i, \beta_j\}} U_\nu
        \arrow{r}{\text{proj}_{U_\delta}}
        &
        U_{\delta},
        \end{tikzcd}
        \end{equation}
        for some $\beta_{i}$ and $\gamma_{j}$.
    \end{itemize}
    
    To get rid of the multiset $S$ in Equation \ref{eq:multiset}, we again apply twisted transpositions to modify the factors on the LHS of this equation.  We proceed starting with root $\gamma$ of minimal height\footnote{For a polarization, with $\Phi^{+}\supset \Delta$.} in this multiset.  For these curves we use twisted transposition to bring them next to each other, and then take the product of the factors $u$ on these lines.  We then proceed to the next lowest height.  By Lemma \ref{lemma on height of commutator} we only introduce new factors of a height greater than the height of the root we are dealing with, and as such this process terminates with Equation \ref{eq:multiset} modified to the form:
    \begin{equation}
    \label{eq:thing o' beauty}
    \overrightarrow{\prod}_{\gamma \in \Delta^{int}} (u'_\gamma \cdot u_\gamma^{\text{brack}}) = u_{known},
    \end{equation}
    where, $u_\gamma^{\text{brack}}$ is the product of all factors $u_{\gamma}^{\delta_1,\delta_2} \in U_\gamma$ obtained by passing a commutator $u_{\delta_1}^{-1}u_{\delta_2}u_{\delta_1}u_{\delta_2}^{-1}$ through Equation \ref{eq:mult_comm_proj}.
    
    Just like in case \ref{case:curves_together}, each of the factors $(u'_\gamma \cdot u_\gamma^{\text{brack}})_{\gamma \in \Delta^{int}}$ from the LHS of Equation \ref{eq:thing o' beauty} is uniquely determined, by taking the preimage under the multiplication map of Equation \ref{eq:mult_delta_int}. It remains to show that we can uniquely determine the tuple $(u'_\gamma)_{\gamma \in \Delta^{int}}$ from the tuple $(u'_\gamma \cdot u_\gamma^{\text{brack}})_{\gamma \in \Delta^{int}}$. This can be proved by induction on the height of the roots $\gamma \in \Delta^{int}$. The argument is analogous to that in the proof of Proposition \ref{prop:mult_iso}, and we do not repeat it here.
\end{enumerate}
\end{proof}

\section{Cameral and Spectral Networks}
\label{sec: Cameral Networks}

The non-abelianization map of Section \ref{sec: flat Donagi Gaitsgory} takes certain $N$-local systems on $X\backslash P$ and produces a $G$ local system on $X$ by modifying (``cutting and regluing'') along a certain real one dimensional locus.  The \emph{spectral network} is the data of this locus, together with some additional data\footnote{These are the labels of Definition \ref{def:spectral_net}.} necessary for doing the modification.  In this section we introduce spectral networks on $X$ via the closely related notion of a \emph{cameral network} on the cameral cover $\tilde{X}\rightarrow X$. This differs from the notions/constructions in \cite{gaiotto2013spectral, longhi2016ade} in that:
\begin{itemize}
    \item we do not use trivializations of the spectral or cameral cover\footnote{See Remark \ref{rem:root_noncanonical} for the precise relation to the constructions of \cite{gaiotto2013spectral,longhi2016ade}.};
    \item we give a construction for all reductive algebraic groups $G$, whereas \cite{gaiotto2013spectral} treats the case of $G = GL(n)$ and $SL(n)$, and \cite{longhi2016ade} treats groups of type ADE.
\end{itemize}

Before defining cameral networks, we specify the setup and introduce some notation.

\begin{convention}
\label{convention:cameral}
Let $X^{c}$ denote a smooth, connected, compact Riemann surface and $D$ a reduced divisor, with $\text{deg}(D)>2$. We will work with the non-compact Riemann surface $X = X^{c} \setminus D$, and consider meromorphic 1-forms on $X^c$ which are sections of $K_{X^{c}}(D)$. The restriction of these 1-forms to $X$ is holomorphic. We only consider cameral covers $\pi : \tilde X \to X$ which:
\begin{itemize}
    \item Arise from points in the Hitchin base for the divisor $D$, i.e. from sections $a : X^c \to \ft_{K_{X^{c}}(D)}\sslash W$.  We are here abusing notation by denoting by $K_{X^{c}}(D)$ the associated principal $\bbG_{m}$-bundle.
    \item Come from $a\in \cA^{\diamondsuit}$ (see Definition \ref{definition:  cA diamondsuit}), which means that $\tilde{X}$ is smooth (Proposition \ref{prop: diamond suit implies smooth}), and that all ramification points of $\pi$ have order 2 (Proposition \ref{proposition: A diamond non empty}).
    \item The compact cameral cover $\tilde{X}^{c}\rightarrow X^{c}$ is not ramified over $D$.  See Definition \ref{definition: Condition 0} for a description of the corresponding locus in the Hitchin base. 
\end{itemize}
When we wish to emphasize the point in the base giving rise to the cameral cover we will denote the cameral covder by $\tilde{X}_{a}$.  We will denote by $\tilde{X}^{c}$ or $\tilde{X}^{c}_{a}$ the cameral cover of $X^{c}$ giving rise to the cameral cover of $X$.

Many of the constructions, and some of the examples do not require either that $\text{deg}(D)>2$ condition, or that $D$ is reduced.
\end{convention}

%\begin{definition}
%\label{definition: Hitchin base smooth locus}
%We denote by $\cA^{\diamondsuit}\subset \cA$ the subset of the Hitchin base (for the line bundle $K_{X^{c}}(D)$) such that the cameral cover is smooth.
%\end{definition}
%\todo{Move into section on background on Hitchin moduli and local systems.}

\begin{remark}
\label{remark: What we miss}
The assumption that $\text{deg}(D)>2$ means that we only consider cameral networks and non-abelianization for non-compact surfaces. 
The reason is that Stokes curves on a compact surface $X^{c}$ necessarily either become dense\footnote{In the sense of Definition \ref{defin:basic_abstract} (1).} somewhere on $X^{c}$ or include trajectories between different branch points of the cameral cover (see \cite{strebel1984quadratic, strebel1978density} for the $SL(2,\bbC)$ case).  We don't develop techniques to deal with these situation in this paper. See \cite{fenyes2015dynamical} for a treatment of non-abelianization for $SL(2,\R)$ local systems on compact surfaces. See also \cite{hollands2016spectral} for spectral networks of Fenchel--Nielsen type which consist of trajectories between branch points, on a possibly compact Riemann surface, in the case of $SL(2,\C)$, and \cite{hollands_kidwai} for the case of $SL(n,\C)$.  Even in the non-compact case these are unfortunately excluded from consideration in this paper by our assumption that we do not have trajectories between branch points.  Some of these papers works in slightly modified settings.  

Because we are excluding several interesting cases of spectral networks, we call the spectral networks that we do consider \emph{basic spectral networks}.
\end{remark}

\begin{convention}
\label{convention:blowup}
Let $\pi :\tilde X \to X$ denote a cameral cover as specified in Convention \ref{convention:cameral}. We denote by $P \subset X$ the branch points, and $R \subset \tilde X$ the ramification points.  Let $X^\circ$ be the marked, oriented real blowup of $X$ at $P$. Specifically, for all $p \in P$:
\begin{itemize}
    \item $p$ is replaced by a boundary circle $S^1_p$;
    \item the orientation of $X$ induces an orientation of $S^1_p$;
    \item we choose a marked point $x_p \in S^1_p$.
\end{itemize}

Let $\tilde X^\circ$ be the oriented real blowup of $\tilde X$ along $R$, and denote by $S^1_r$ the boundary circle which is the preimage of $r \in R$ under the blowup map.  There is a map $\pi^{\circ}: \tilde{X}^{\circ}\rightarrow X^{\circ}$, which exhibits $\tilde X^\circ$ as a principal $W$-bundle over $X^\circ$.
\end{convention}

\subsection{Abstract Cameral Networks}
\label{subsec: abstract cameral networks}

In this section we define basic abstract cameral networks.  In later sections we will then consider a special case of these called basic WKB cameral networks which we will define in Definition \ref{def:wkb_cameral}.

\begin{definition}
%\label{def:comb_cameral}
\label{defin:basic_abstract}
Let $\pi : \tilde X \to X$ be a cameral cover satisfying convention \ref{convention:cameral}, and $\pi^\circ : \tilde X^\circ \to X^\circ$ the induced map on real blowups.
A \textbf{basic abstract cameral network} on $\tilde X^\circ$ is the data of:
\begin{itemize}
    \item A subset $X^{\circ '}\subset X^{\circ}$ such that this inclusion induces a homotopy equivalence. We then define $\tilde X^{\circ '} = \tilde X^\circ \times_{X^\circ} X^{\circ '}$.
    \item A finite set $\tilde \cW$ of oriented, semi-infinite, real curves $\ell \subset \tilde X^{\circ '}$, called \textbf{Stokes curves} (or Stokes lines) that intersect in a finite set $\tilde \cJ\subset \tilde X^{\circ '}$. By ``semi-infinite'' we mean that $\ell$ is homotopic to $[0,1)$.  We furthermore require that if we compactify $\tilde X^{\circ '}$ with a boundary circle $S^1_d$ for each $d \in D$, then for any $\ell \in \tilde \cW$, the set $\overline{\ell}\backslash \ell$ is a single point on the boundary circles $S^1_d$, for some $d \in D$.
    \item for each $\ell \in \tilde \cW$, a root $\alpha_\ell \in \Phi$ called the \textbf{label} of the Stokes curve.
\end{itemize}

This data must satisfy the following conditions:

\begin{enumerate}
    \item
    \label{item:non-denseness}
    \textbf{Non-denseness:}
    Let $f : [0,1] \to \tilde X^{\circ'}$ such that, the embedded line segment is transversal to all lines in $\tilde{\cW}$ passing through $f(1/2)$. Then there is a neighbourhood $1/2\in U\subset [0,1]$ such that $f\big(U\big) \cap \tilde \cW$ is discrete. (The point is to rule out curves becoming dense, as in Figure \ref{fig:dense_cylinder}.)
    
    \begin{figure}[h]
        \centering
        \includegraphics[width = 0.8\textwidth]{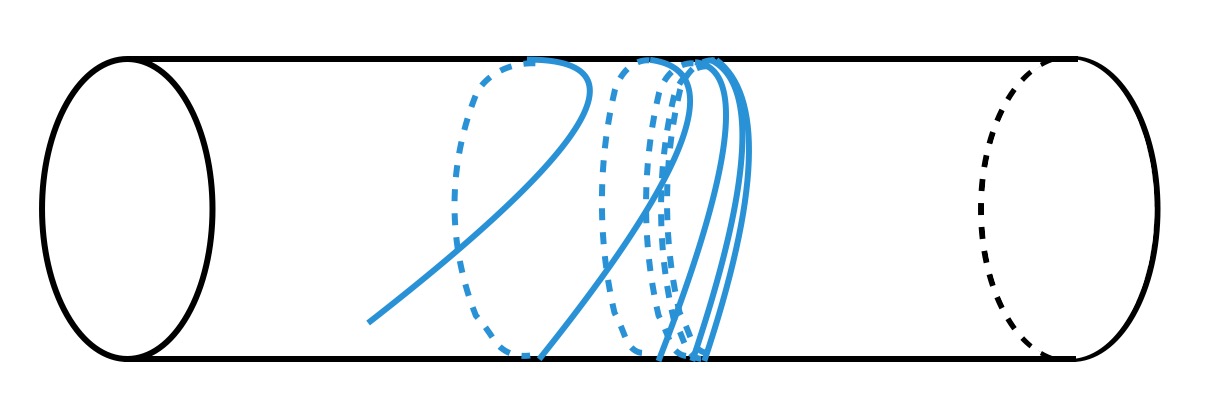}
        \caption{Dense curve wrapping around a cylinder.}
        \label{fig:dense_cylinder}
    \end{figure}

    \item
    \label{item:equivariance}
    \textbf{Equivariance:}
    For any Stokes curve $\ell\in \tilde \cW$ and Weyl group element $w\in W$, $w(\ell) \in \tilde \cW$. Moreover, $\alpha_{w(\ell)} = w\big(\alpha_{\ell}\big)$.
    
    \item
    \label{item:local_ramification}
    \textbf{Local picture around ramification points:}
    For each $r\in R$,
    there are six Stokes curves originating from the boundary circle $S^1_r$. For each $r$, let $\alpha_r \in \Phi$ be the root, unique up to $\pm$, such that $\tilde a(r)$ lies on the root hyperplane $H_{\alpha_r}$. Then the six curves are labeled by $\pm \alpha_r$ in alternating fashion, as shown in Figure \ref{fig:ramif1}.
    
    \begin{figure}[h]
    \centering
    \begin{minipage}{.5\textwidth}
      \centering
      \def\svgwidth{150pt}
      %% Creator: Inkscape 1.0.1 (0767f8302a, 2020-10-17), www.inkscape.org
%% PDF/EPS/PS + LaTeX output extension by Johan Engelen, 2010
%% Accompanies image file 'Diagram4_55.pdf' (pdf, eps, ps)
%%
%% To include the image in your LaTeX document, write
%%   \input{<filename>.pdf_tex}
%%  instead of
%%   \includegraphics{<filename>.pdf}
%% To scale the image, write
%%   \def\svgwidth{<desired width>}
%%   \input{<filename>.pdf_tex}
%%  instead of
%%   \includegraphics[width=<desired width>]{<filename>.pdf}
%%
%% Images with a different path to the parent latex file can
%% be accessed with the `import' package (which may need to be
%% installed) using
%%   \usepackage{import}
%% in the preamble, and then including the image with
%%   \import{<path to file>}{<filename>.pdf_tex}
%% Alternatively, one can specify
%%   \graphicspath{{<path to file>/}}
%% 
%% For more information, please see info/svg-inkscape on CTAN:
%%   http://tug.ctan.org/tex-archive/info/svg-inkscape
%%
\begingroup%
  \makeatletter%
  \providecommand\color[2][]{%
    \errmessage{(Inkscape) Color is used for the text in Inkscape, but the package 'color.sty' is not loaded}%
    \renewcommand\color[2][]{}%
  }%
  \providecommand\transparent[1]{%
    \errmessage{(Inkscape) Transparency is used (non-zero) for the text in Inkscape, but the package 'transparent.sty' is not loaded}%
    \renewcommand\transparent[1]{}%
  }%
  \providecommand\rotatebox[2]{#2}%
  \newcommand*\fsize{\dimexpr\f@size pt\relax}%
  \newcommand*\lineheight[1]{\fontsize{\fsize}{#1\fsize}\selectfont}%
  \ifx\svgwidth\undefined%
    \setlength{\unitlength}{595.27559055bp}%
    \ifx\svgscale\undefined%
      \relax%
    \else%
      \setlength{\unitlength}{\unitlength * \real{\svgscale}}%
    \fi%
  \else%
    \setlength{\unitlength}{\svgwidth}%
  \fi%
  \global\let\svgwidth\undefined%
  \global\let\svgscale\undefined%
  \makeatother%
  \begin{picture}(1,1)%
    \lineheight{1}%
    \setlength\tabcolsep{0pt}%
    \put(0,0){\includegraphics[width=\unitlength,page=1]{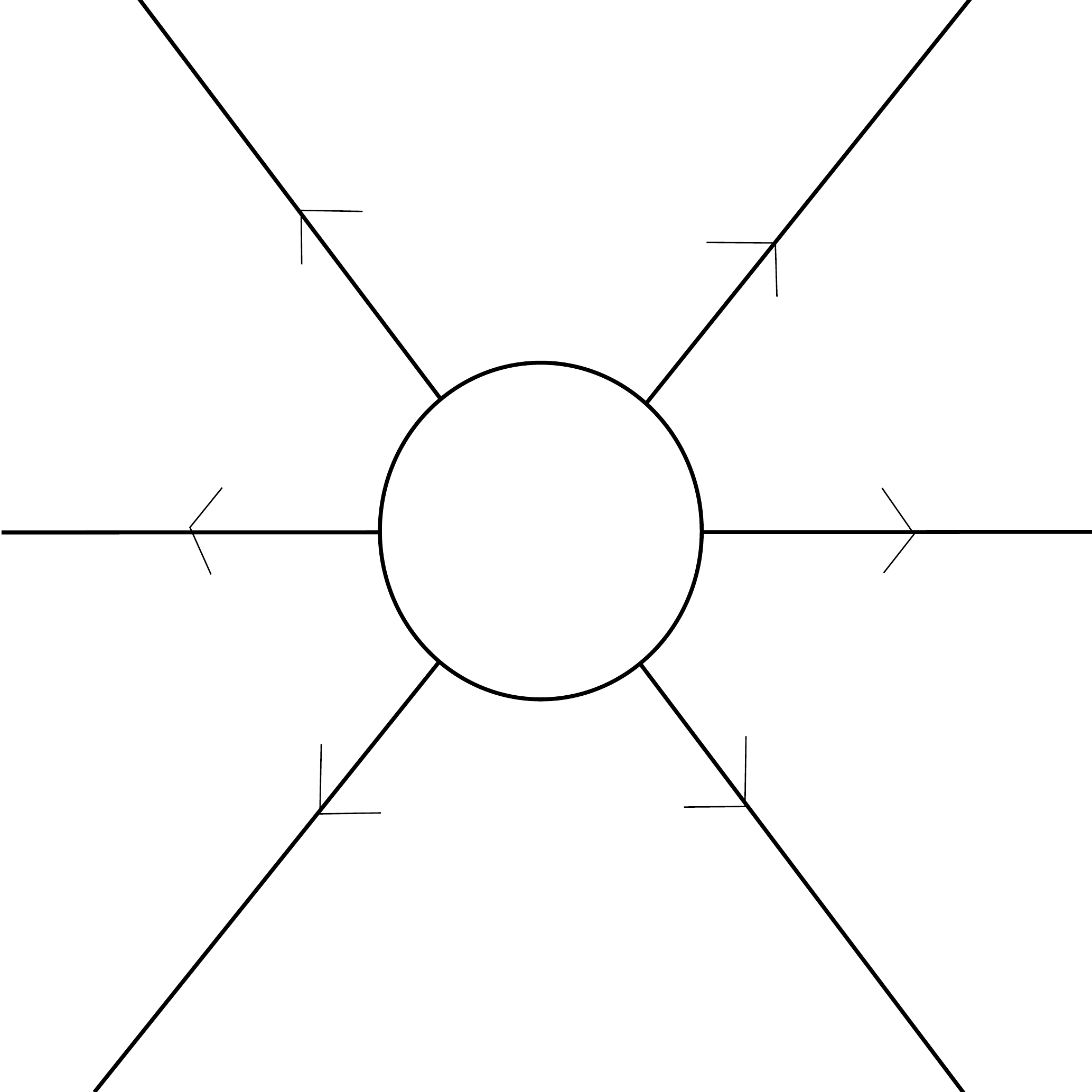}}%
    \put(0.44909966,0.55188315){\color[rgb]{0,0,0}\makebox(0,0)[lt]{\lineheight{1.25}\smash{\begin{tabular}[t]{l}$S_{r}^{1}$\end{tabular}}}}%
    \put(0.82665027,0.78901038){\color[rgb]{0,0,0}\makebox(0,0)[lt]{\lineheight{1.25}\smash{\begin{tabular}[t]{l}$\alpha_{r}$\end{tabular}}}}%
    \put(0.8372762,0.39575671){\color[rgb]{0,0,0}\makebox(0,0)[lt]{\lineheight{1.25}\smash{\begin{tabular}[t]{l}$-\alpha_{r}$\end{tabular}}}}%
    \put(0.8154222,0.14732753){\color[rgb]{0,0,0}\makebox(0,0)[lt]{\lineheight{1.25}\smash{\begin{tabular}[t]{l}$\alpha_{r}$\end{tabular}}}}%
    \put(0.26497178,0.12429083){\color[rgb]{0,0,0}\makebox(0,0)[lt]{\lineheight{1.25}\smash{\begin{tabular}[t]{l}$-\alpha_{r}$\end{tabular}}}}%
    \put(0.0940034,0.41421432){\color[rgb]{0,0,0}\makebox(0,0)[lt]{\lineheight{1.25}\smash{\begin{tabular}[t]{l}$\alpha_{r}$\end{tabular}}}}%
    \put(0.2678843,0.88268609){\color[rgb]{0,0,0}\makebox(0,0)[lt]{\lineheight{1.25}\smash{\begin{tabular}[t]{l}$-\alpha_{r}$\end{tabular}}}}%
  \end{picture}%
\endgroup%

      \caption{Stokes curves around a ramification point $r$.}
      \label{fig:ramif1}
    \end{minipage}%
    \begin{minipage}{.5\textwidth}
      \centering
      \includegraphics[scale = 0.5]{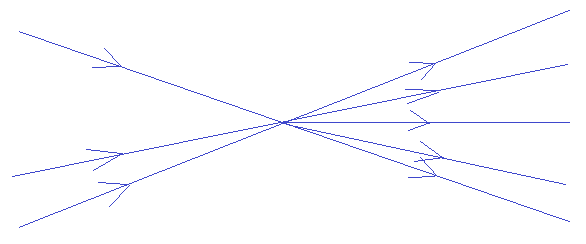}
      \caption{Stokes curves around a joint.}
      \label{fig:stokes_joint}
    \end{minipage}
\end{figure}

    \item
    \label{item:local_joint}
    \textbf{Local picture around joints:}
    The orientation on the Stokes curves specifies positive and negative parts of the tangent spaces of these curves, corresponding to tangent vectors in the same, and in the opposite direction of the orientation respectively.
    
    Let $x\in \cJ$ be a joint.  Consider the rays in $T_{x}X$ given by the negative part of the tangent spaces at $x$ of the incoming Stokes curves, and the positive part of the tangent spaces at $x$ to the outgoing Stokes curves.  Label these rays by the roots labelling the correspond Stokes curves.  We require that these form an undecorated 2D scattering diagram in the sense of Definition \ref{def:scattering_diagram}.
    %Let $x \in \tilde \cJ$ be an intersection point of Stokes curves, and let $\{\ell_1, \dots, \ell_j\}$ be the outgoing Stokes curves at $x$. 
    %For each $\alpha \in \Phi$, let $H_\alpha^+ \subset \fr t$ denote the real codimension 1 half-space:
    %\[
    %\{ X \in \ft | \alpha(X) \in \R^+ \} .
    %\]
    %There exists a $\C$-linear map $A_x : T_x\tilde X \to \fr t$, and a subset $\{\alpha_1, \dots, \alpha_j\} \in \Phi$, such that, for all $1 \leq i \leq j$:
    %\begin{itemize}
     %   \item The label of $\ell_i$ is $\alpha_i$;
      %  \item The real span of the tangent vector to $\ell_i$ at $x$ equals $A_x^{-1}(H_{\alpha_i}^+)$.\footnote{In the case of WKB networks, the linear map $A_x$ is dual to $\tilde a(x) \in K_{\tilde X,x} \otimes \fr t$; see Remark \ref{rem:tangent_preimage}.}
    %\end{itemize}
    %The incoming Stokes curves at $x$ are a subset of $\{\ell_1, \dots, \ell_j\}$.
   %See Figure \ref{fig:stokes_joint} for an example.
    
    \item 
    \label{item:acyclic}
    \textbf{Acyclicity:}
    %Let $\tilde \cJ \subset \tilde X^{\circ '}$ be the set of intersections of Stokes curves.
    Regard $\tilde \cW$ as a directed graph, with one vertex for each boundary circle of $\tilde X^{\circ '}$ and each intersection of Stokes curves $x \in \tilde \cJ$, and one directed edge for each connected component of $\tilde \cW \setminus \tilde \cJ$. This graph is required to be acyclic.
\end{enumerate}

We say that the Stokes curves originating at a boundary circle are \textbf{primary Stokes curves}, those originating at an intersection point $x \in \tilde \cJ$ are \textbf{new Stokes curves}, and the intersection points $x \in \tilde \cJ$ are the \textbf{joints} of the cameral network.
\end{definition}

In Definition \ref{defin:basic_abstract}, the equivariance condition ensures that basic abstract cameral networks on $\tilde X$ descend to objects on $X$ called basic spectral networks; see Definition \ref{def:spectral_net}. The acyclicity condition implies the existence of a total order on $\cJ:=\tilde{\cJ}/W$ (see Lemma \ref{lem: top ordering J}), which is necessary for associating Stokes factors to curves in an iterative fashion in Construction \ref{construction of swkb}.

We will sometimes omit the term ``abstract'' when discussing these networks.

%\begin{remark}
%\label{rem:filtration_segments}
%Lemma \ref{lem:topological_ordering} provides a total order on $R \cup \tilde \cJ$. This implies the existence of a filtration $F_\bullet(\tilde \cW \setminus \tilde \cJ)$ such that:
%\begin{itemize}
%    \item $F_{i} (\tilde \cW \setminus \tilde \cJ) = F_{i-1} (\tilde \cW \setminus \tilde \cJ) \coprod \{l_{i}\}$, where $l_{i}$ is a single connected component of $(\tilde \cW \setminus \tilde \cJ)$;
%    \item For all $i$, let $l_i$ be the unique Stokes curve segment in $F_i(\tilde \cW \setminus \tilde \cJ) \setminus F_{i-1}(\tilde \cW \setminus \tilde \cJ)$. If $l_i$ is composable with $l_j$, in the sense that the endpoint of $l_j$ is the starting point of $l_i$, then $j < i$.
%\end{itemize}
%\end{remark}

\begin{remark}
\label{remark: generalizing our definition spectral network}
There are possible generalizations of Definition \ref{defin:basic_abstract} that we don't consider in this article:
\begin{itemize}
    \item The acyclicity condition implies that no Stokes curves run into a ramification point. In particular, we don't deal with double lines such as those in Figure \ref{fig:double_wall}. These have been studied in \cite{gaiotto2013spectral} with respect to describing BPS states, \cite{takei2008sato} in the setting of WKB analysis, and in \cite{hollands2016spectral}, \cite{hollands_kidwai} in the context of Fenchel-Nielsen coordinates for $G=SL(2)$ and $G = SL(n)$, respectively.
    
    \item The acyclicity condition could be relaxed, to allow certain cycles that don't preclude the inductive definition of Stokes factors in Construction \ref{construction of swkb}.

    \item Condition \ref{item:local_ramification} is motivated by the behaviour of WKB networks locally around ramification points of order 2. Since we only study smooth cameral covers, all ramification points have order 2. Working with more general ramification patterns would require modifications to condition \ref{item:local_ramification}. The paper \cite{longhi2016ade} contains some examples of higher order ramification for groups of type ADE. Some more general types of ramification can be treated by degenerating smooth cameral curves (together with their cameral networks) to a singular cameral curve. See Figure 36 of \cite{gaiotto2013spectral}, and the paper \cite{williams2016toda} for examples for $SL(n)$ and $GL(n)$.  Extra work is necessary to extend the non-abelianization procedure of Section \ref{sec: flat Donagi Gaitsgory} to these general ramification patterns; we don't carry this out in the present article.
    \item Condition \ref{item:local_joint} is needed to iteratively assign decorations (unipotent automorphisms) to lines in the cameral network.  Basic WKB cameral networks for $GL(n)$ or $SL(n)$ that do not satisfy the criterion that the roots labelling incoming Stokes lines are a convex set of roots will also violate the acyclicity condition, as shown in Lemmas 3.2.6 and 4.1.20 of the first authors thesis \cite{ionita2020spectral}.  We believe it should be possible to make some relaxations to the assumption that the tangent rays are all distinct, in particular it should be possible to have Stokes rays labelled by multiple roots.
    %\todo{Make a comment about Matei's 4.1.19, 4.1.20! + saying that we need this to inductively give the requirements.  +++ saying that these are the type of joint that occur in the WKB construction }
\end{itemize}
We expect that the phenomena enumerated above, and disallowed in Definition \ref{defin:basic_abstract}, are non-generic for the WKB cameral networks defined subsequently, when a sufficiently high degree divisor $D$ is considered.  See Conjecture \ref{conjecture: open dense set} for further discussion.

\end{remark}

\subsection{Abstract Spectral Networks}

In Section \ref{sec: flat Donagi Gaitsgory}, we will need to cut and glue local systems on $X^\circ$. In view of this, we define basic spectral networks on $X^\circ$, which are informally pushforwards of basic cameral networks via $\pi^\circ : \tilde X^\circ \to X^\circ$. 

\begin{definition}
\label{def:spectral_net}
Let $\tilde \cW$ be a basic abstract Cameral network. The associated \textbf{basic abstract spectral network} $\cW$ is the following union of oriented, semi-infinite, real curves on $X^{\circ '}$:
\begin{itemize}
    \item As oriented half-lines, $\cW = \pi_* \tilde \cW$.
    \item The label of $\ell \in \cW$ is a
    locally constant section $\psi_\ell$
    of $\Hom_W(\tilde X^\circ|_{\ell}, \Phi)$.
    It maps any preimage $\tilde \ell$ of $\ell$ to the root which labels $\tilde \ell$.
\end{itemize}

We call the lines $\ell$ Stokes curves or Stokes lines, and we call the images of joints in the cameral network joints.  We denote the set of these joints by $\cJ=\tilde{\cJ}/W$.

If $\tilde \cW$ is a basic WKB cameral network (as will be defined in Definition \ref{def:wkb_cameral}), then we say that $\cW$ is a \emph{basic WKB spectral network}.
\end{definition}

\begin{remark}
\label{rem:three_lines}
In Definition \ref{defin:basic_abstract}, we require that six Stokes curves start at each ramification point. Since $\pi : \tilde X \to X$ has degree 2 around each branch point, the $W$-equivariance of cameral networks ensures that the corresponding spectral network has three Stokes curves starting at each branch point. This is in agreement with the constructions of spectral networks in \cite{gaiotto2013spectral}, \cite{longhi2016ade}.
\end{remark}

\begin{remark}
\label{rem:root_noncanonical}
Any trivialization $\phi : \tilde X^{\circ'}|_{\ell} \cong \ell \times W$, gives a mapping:
\begin{align}
\begin{split}
\label{eq:root_noncanonical}
\Hom_W(\tilde X^{\circ'}|_{\ell}, \Phi) &\to \Phi ,
\\
\psi_\ell &\mapsto \psi_\ell \big(\phi^{-1}(1_W)\big).
\end{split}
\end{align}
Thus, each choice of trivialization gives a way to label curves in the spectral network by elements of $\Phi$.

In the earlier works \cite{gaiotto2013spectral}, \cite{longhi2016ade}, the covering $\pi : \tilde X \to X$ is trivialized away from a system of branch cuts, and curves in the spectral network are labeled by roots of $\fr g$. The map in Equation \ref{eq:root_noncanonical} gives the relationship between our labels and their labels. This identifies our spectral networks with those of \emph{loc. cit.}
\end{remark}

%In the non-abelianization construction of Section \ref{sec: flat Donagi Gaitsgory}, we assume that cameral networks satisfy an additional constraint, described below. 

%\begin{definition}
%\label{defin:basic_abstract}
%A \textbf{basic abstract cameral network} is an abstract cameral network $\tilde \cW$ such that:
%\begin{itemize}
  %  \item at every joint $x \in \tilde \cJ$, the tangents of the incoming and outgoing Stokes curve at $x$ together with their labels form an undecorated 2d scattering diagram in the sense of definition \ref{def:scattering_diagram}.%labels of the incoming Stokes curves form a convex subset $C_{in} \subset \Phi_+$, and the labels of the outgoing Stokes curves are $C_{out} = \Conv_{C_{in}}^\N$. (See \ref{def: convex set of roots} for the definition of a convex set of roots, and \ref{def:convex_hulls} for the definition of the restricted convex hull.)
%\end{itemize}
%\end{definition}

\begin{remark}
\label{rem:snakes}
An example of basic abstract spectral networks that are not known to be basic WKB spectral networks come from the paper \cite{gaiotto2014spectralsnake} which considers ``minimal spectral networks'' for $G = SL(n)$, and shows they are related to Fock-Goncharov coordinates. Minimal spectral networks are not, a priori, basic WKB spectral networks (Definition \ref{def:spectral_net}); Section 4.2 of \emph{loc. cit.} finds some examples which are basic WKB spectral networks in the cases where $n\leq 5$.
Minimal spectral networks are examples of basic abstract spectral networks.
\end{remark}

\subsubsection{Ordering of Joints}

\begin{lemma}
\label{lem: top ordering J}
Let $\cJ=\tilde{\cJ}/W$ be the set of joints of a basic abstract spectral network.  Then there is a total order on $\cJ$, such that for the graph formed by Stokes curves and joints of the basic spectral network, for any edge $e$, $\text{source}(e) < \text{target}(e)$.
\end{lemma}

\begin{proof}
First we note that the directed graph associated to the basic spectral network is acyclic: Suppose otherwise that this has a cycle.  Then the associated basic cameral network has a cycle spanned by concatenating lift(s) of this cycle to the directed graph associated to the basic cameral network, contradicting property \ref{item:acyclic} of Definition \ref{defin:basic_abstract}.

Thus by Lemma \ref{lem:topological_ordering} below we have an order on $\cJ$ with the desired properties.
\end{proof}

\begin{lemma}
\label{lem:topological_ordering}
Let $G = (V,E)$ be a directed acyclic graph, with finite vertex set $V$ and finite edge set $E$. Then there exists a total order on $V$, such that for all $e \in E$, $\text{source}(e) < \text{target}(e)$.
\end{lemma}

\begin{proof}
This is a standard result in graph theory literature, called a topological ordering. We include the following explicit construction for completeness. For each iteration $k \geq 1$:

\begin{enumerate}
    \item Since the graph is acyclic, the subset $V_k \subset V$ of vertices with in-degree equal to 0 is nonempty. Label the vertices in $V_k$ by $n_k + 1, \dots, n_k + |V_k|$, where $n_k$ is the number of vertices already labeled before iteration $k$.
    \item Remove all vertices in $V_k$, and all edges $E_k = \{e \in E | \text{source}(e) \in V_k\}$ from the graph. I.e. update $V := V \setminus V_k$, $E := E \setminus E_k$. If $V$ is nonempty, return to step 1 with $k := k+1$.
\end{enumerate}
Since $|V_k| \geq 1$ for each iteration $k$, and $\sum_{k} |V_k| = |V|$, the number of iterations is bounded above by $|V|$. Hence the construction finishes after finitely many steps.
\end{proof}

\subsection{WKB Cameral Networks}
\label{subsec:wkb_cameral}

The most important examples of basic abstract cameral networks are basic WKB cameral networks, constructed in Construction \ref{constr:wkb_cameral} as collections of integral curves of certain vector fields on $\tilde X$. The spectral networks associated to these cameral networks generalize trajectories of quadratic differentials (see e.g. \cite{strebel1984quadratic} and our discussion in Example \ref{example: SL2 and quadratic differentials}).

Start with a cameral cover $\tilde{X}^{c}\xrightarrow{\pi} X^{c}$ corresponding to a point in $\cA^{\diamondsuit}$.  Note that we can restrict the resulting cameral cover of $X^{c}$ to a cover $\tilde{X}\rightarrow X\subset X^{c}$.

Recall that this cameral cover is equipped with a map:
\[
\begin{tikzcd}
\tilde X^{c} \arrow{r}{\tilde a} & \fr{t}_{ K_{X^{c}}(D)},
\end{tikzcd}
\]
clearly this gives a map $\tilde{a}'$
\[\tilde{X}\xrightarrow{\tilde a'} \fr{t}_{\pi^{*}(K_{X^{c}}(D))},\]
by noting that $\fr{t}_{ \pi^{*}(K_{X^{c}}(D))}=\fr{t}_{K_{X^{c}}(D)}\times_{X^{c}}\tilde{X^{c}}$.

Each root $\alpha \in \fr t^\vee$ can be evaluated on $\fr t$, giving a section of $\pi^*\big( K_{X^{c}}(D)\big)$.

\begin{construction}
\label{def:form_root}
To each root $\alpha$ of $\fr g$ we associate the meromorphic 1-form $\chi_\alpha$ on $\tilde X^{c}$, defined as the composition:
\[
\begin{tikzcd}
\tilde X^{c} \arrow{r}{\tilde a'} & \fr{t}_{\pi^{*}(K_{X^{c}}(D))} \arrow{r}{1\times \alpha} & \pi^{*}(K_{X^{c}}(D))\arrow{r} &  K_{\tilde{X}^{c}}(\pi^{*}(D)).
\end{tikzcd}
\]
We will abuse notation by also denoting $\chi_{\alpha}|_{\tilde{X}}$ as $\chi_{\alpha}$.
\end{construction}

The 1-form $\chi_{\alpha}$ is non-vanishing away from the ramification locus.  It vanishes to order two on the subset of the ramification locus where $\tilde{a}(\tilde{X})$ intersects $ K_{X^{c}}(D)\times_{\bbG_{m}}H_{\alpha}$, because $(1\times \alpha)\circ \tilde{a}'$ vanishes to order one, and $K_{\tilde{X}^{c}}(\pi^{*}(D))=\pi^{*}(K_{X^{c}}(D))(R)$.

In an analogous way $\chi_{\alpha}$ vanishes to order one at ramification points of $\pi$ away from where $\tilde{a}(\tilde{X})$ intersects $K_{X^{c}}(D)\times_{\bbG_{m}} H_{\alpha}$.

We also define quadratic differentials associated to $\alpha$ on $\tilde{X}_{\alpha}^{c}:=\tilde{X}^{c}\sslash <s_{\alpha}>$.  We will use these to analyze the trajectories of $\chi_{\alpha}$.  We write $\tilde{X}_{\alpha}:=\tilde{X}\sslash <s_{\alpha}>$ and we decompose the projection $\pi$ as:
\begin{equation}
\label{equation: factoring pi}
\begin{tikzcd}
\tilde X^{c}
\arrow[swap]{r}{p_\alpha}
\arrow[bend left]{rr}{\pi} 
&
\tilde X^{c}_{\alpha}
\arrow[swap]{r}{\pi_\alpha}
&
X^{c} .
\end{tikzcd}
\end{equation}
For $a\in \cA^{\diamondsuit}$, the $W$-equivariant map $\tilde{X}^{c}\rightarrow \ft_{K_{X^{c}}(D)}$ descends to a map 
\[\tilde{X}_{\alpha}^{c}=\tilde{X}^{c}\sslash <s_{\alpha}>\rightarrow \ft_{K_{X^{c}}(D)}\sslash <s_{\alpha}>.\]
As above we can enhance this to a map 
\[\tilde{X}_{\alpha}^{c}\rightarrow \ft_{\pi_{\alpha}^{*}(K_{X^{c}}(D))}\sslash <s_{\alpha}>.\]

\begin{construction}
\label{construction: Quadratic form alpha}
We define the meromorphic two form $\omega_{\alpha}\in \Gamma(\tilde{X}^{c}_{\alpha}, (K_{\tilde{X}^{c}_{\alpha}}(\pi_{\alpha}^{*}D))^{\otimes 2})$ as being the two form corresponding to the section 
\[\tilde{X}^{c}_{\alpha}\rightarrow \ft_{\pi_{\alpha}^{*}(K_{X^{c}}(D))}\sslash <s_{\alpha}> \xrightarrow{\alpha} \pi_{\alpha}^{*}(K_{X^{c}}(D))\sslash \bbZ_{2}\rightarrow K_{\tilde{X}_{\alpha}^{c}}(\pi_{\alpha}^{*}(D))\sslash \bbZ_{2}\cong (K_{\tilde{X}_{\alpha}^{c}}(\pi_{\alpha}^{*}(D)))^{\otimes 2}.\]
We will sometimes denote this by $\omega_{\alpha, a}$ if we want to emphasize the point of the Hitchin base used.
\end{construction}

Note that $p_{\alpha}^{*}\omega_{\alpha}=\chi_{\alpha}^{\otimes 2}$.  Hence we can see $p_{\alpha}^{*}\omega_{\alpha}$ vanishes with order four as a section of $(K_{\tilde{X}^{c}}(\pi^{*}D))^{\otimes 2}$ at the ramification points of $p_{\alpha}$.  Hence $\omega_{\alpha}$ vanishes with order one along the branch points of $p_{\alpha}$.

In an analogous way $\omega_{\alpha}$ vanishes to order two along the ramification locus of $\pi_{\alpha}$.

Furthermore $\tilde{X}^{c}\rightarrow \tilde{X}_{\alpha}^{c}$ is isomorphic to the $SL(2)$ cameral cover associated to the quadratic differential $\omega_{\alpha}$, where we are using the cameral construction with respect to the line bundle $(K_{\tilde{X}^{c}_{\alpha}}(D-R_{\pi_{\alpha}}))^{\otimes 2}$ on $\tilde{X}_{\alpha}^{c}$, where $R_{\pi_{\alpha}}$ is the ramification divisor of $\pi_{\alpha}:\tilde{X}_{\alpha}^{c}\rightarrow X^{c}$.

\begin{remark}
The quadratic differential $\omega_{\alpha}$ descends to a quadratic differential on $\tilde{X}\sslash Stab(\{\alpha, -\alpha\})$, where $Stab(\{\alpha, -\alpha\})\subset W$ is the subgroup preserving the unordered pair $\{\alpha,-\alpha\}$.
\end{remark}

Recall that we are considering the case where $\pi:\tilde X \to X$ has ramification points of degree 2. Locally around such a point $r\in R$, the covering map is modeled on the following simple example.

\begin{example}
\label{ex:cameral_sl2}
Let $G = SL(2)$, so that $W \cong \Z/2$. Let $X^c = \bbP^1$, $D = \{\infty\}$, so that $X = \bbA^1$. Let $x$ be a coordinate on $X$, and let $a \in \Gamma(X,\ft_{K_{X^c(D)}} \sslash  W)$ be:
\[
x \longmapsto [\pm h_\alpha\sqrt{x} dx],
\]
where square brackets denote the equivalence class under $W$, and $h_\alpha \in \ft$ is the element
\[
h_\alpha =
\left(
\begin{array}{cc}
    1 & 0 \\
    0 & -1
\end{array}
\right) .
\]
%For example, the section $a$ could arise from the Higgs bundle\footnote{For simplicity, we are using the defining representation of $SL(2)$ to regard an $SL(2)$-Higgs bundle as a rank 2 vector bundle with a trace zero endomorphism-valued 1-form.}
%$(\cO_X\oplus \cO_X, M dx)$, where:
%\[
%M = 
%\left( 
%\begin{array}{cc}
%  0   & x \\
%  1   & 0
%\end{array}
%\right).
%\]
%Note that the eigenvalues of $M$ are $\pm \sqrt{x}$.

Letting $z$ denote a choice of square root of $x$,
Diagram \ref{eq:def_cameral} becomes:
\[
\begin{tikzcd}
z\in \tilde \bbA^1 
\arrow[swap,mapsto]{d}
\arrow[mapsto]{r}{\tilde a} 
&
z h_\alpha dx \in \fr t_{K_{\bbP^1}(D)} 
\arrow[mapsto]{d}
\\
x\in \bbA^1 
\arrow[mapsto]{r}{a} 
& 
\left[\pm \sqrt{x} h_\alpha dx\right] \in \fr{t}_{ K_{\bbP^1}(D)}\sslash W
\end{tikzcd}
\]
Finally, the 1-form associated to the root $\alpha$ is:
\[  
\chi_\alpha(z) =  \alpha( z h_\alpha  dx)= 4 z^2 dz,
\]
where we have used $\alpha(h_\alpha) = 2$, and $dx = 2z dz$.
\end{example}

More generally:

\begin{lemma}
\label{lem:form_local}
Let $\pi : \tilde X \to X$ be a cameral cover, and $r \in R\subset \tilde{X}$ a ramification point of order 2. There are exactly two roots $\pm \alpha$, whose associated 1-form $\chi_{\pm \alpha}$ vanishes at $r$. There exists a local coordinate $z$ centered at $r$, such that this 1-form is:
\[
\chi_{\pm \alpha}(z) = \pm  z^2 dz.
\]
\end{lemma}

\begin{proof}
Since $\pi$ is the pullback of the projection $\fr{t}_{K_{X^{c}}(D)} \to \fr{t}_{ K_{X^{c}}(D)} \sslash W$, the image of $r$ under the map $\tilde X \to \fr{t}_{ K_{X^{c}}(D)}$ is on a root hyperplane; this determines $\pm \alpha$, as the two roots such that $s_{\pm \alpha} \in W$ fixes this hyperplane.

The lemma then reduces to the case of $SL(2)$ by considering the quadratic differential $\omega_{\alpha}$ on $\tilde{X}_{\alpha}$.  In this case $\chi_{\alpha}=-\chi_{-\alpha}$ and $\chi_{\alpha}$ vanishes with order two along the intersection of $\tilde{a}(\tilde{X})$ with $K_{X^{c}}(D)\times_{\bbG_{m}} H_{\alpha}$ as noted immediately after Construction \ref{def:form_root}.  This reduces the lemma to a classical result in the theory of quadratic differentials (see e.g. Section 6.2 of \cite{strebel1984quadratic}), which we now review. For a choice of local coordinate $u$ (vanishing at $r$) we have that
\[\chi_{\alpha}=u^{2}(a_{0}+a_{1}u+...)du,\]
with $a_{0}\neq 0$.

Setting \[b_{k}:=\frac{a_{k}}{m+k+1},\] and \[z:=(\sqrt[3]{1/3})u(b_{0}+b_{1}u+...)^{1/3},\] we have that 
\[\chi_{\alpha}=z^{2} dz,\]

and the result follows.

%For the second claim, we decompose the projection $\pi$:
%\[
%\begin{tikzcd}
%\tilde X
%\arrow[swap]{r}{p_\alpha}
%\arrow[bend left]{rr}{\pi} 
%&
%\tilde X / s_\alpha 
%\arrow[swap]{r}{\pi_\alpha}
%&
%X .
%\end{tikzcd}
%\]
%The map $\pi_\alpha$ is an unramified covering, therefore $\pi_\alpha^*\big( K_{X^c}(D)\big) \cong K_{\tilde X^c/s_\alpha}(\pi_\alpha^*D)$.
%Then note that there exists $a_\alpha$ fitting into the commutative diagram:

%\todo{We should only write down this diagrma and prove things with it once!}
%\[
%\begin{tikzcd}
%\tilde X
%\arrow{r}{\tilde a}
%\arrow{d}{p_\alpha}
%&
%\ft \otimes \pi^* K_{X^c}(D)
%\arrow{r}{\alpha \otimes 1}
%&
%\pi^* K_{X^c}(D) 
%\arrow{d}
%\arrow{r}{\cong} 
%&
%p_\alpha^* K_{\tilde X^c/s_\alpha}(\pi^*D) 
%\arrow{d}
%\\
%\tilde X/s_\alpha 
%\arrow{rr}{a_\alpha}
%&
%&
%\pi_\alpha^* K_{X^c}(D) / \langle s_\alpha \rangle
%\arrow{r}{\cong}
%&
%K_{\tilde X^c/s_\alpha}(\pi_\alpha^*D) / \langle s_\alpha \rangle 
%\end{tikzcd}
%\]
%Recall that the quadratic differential $\omega_{\alpha}$ on $\tilde{X}_{\alpha}$ of construction\ref{construction: Quadratic form alpha} exhibits $\tilde X$ as an $SL(2)$ cameral cover for $\tilde X_{\alpha}$, and that $\pi_{\alpha}^{*}\omega_{\alpha}=\chi_{\alpha}^{\otimes 2}$.  Hence after choosing a local coordinate $z$ on $U$ centered at $r$, such that $p_\alpha(z) = z^2$, we are reduced to the calculation in Example \ref{ex:cameral_sl2}.
\end{proof}

%\todo{I could replace this by noting that Lemma 4.15 is immediate from the fact that $K_{\tilde{X}}=\pi^{*}K_{X}(R)$ -- should I?}

Below, we provide a construction of a set of Stokes curves on $\tilde X^\circ$, starting from the data of a point $a \in \cA^{\diamondsuit}$ in a certain subset of the Hitchin base.
This is a straightforward generalization to arbitrary reductive algebraic $G$ of the construction of WKB spectral networks in \cite{gaiotto2013spectral} (for $SL(n)$, $GL(n)$), \cite{longhi2016ade} (type ADE), and of the Stokes graphs in e.g. \cite{bnr1982new, aoki2001exact, iwaki2014exact, honda2015virtual, takei2017wkb} (these correspond to networks for $SL(n)$, and $GL(n)$), see Warning \ref{warning: discrepancy lines} and Remark \ref{rem:root_noncanonical} for further details.
This construction does not necessarily produce a basic abstract cameral network in the sense of Definition \ref{defin:basic_abstract}, because the construction could either terminate with an error, or $\tilde{\cJ}$ could be infinite.
Basic WKB cameral networks will be defined as the output of this construction, assuming that the process does not terminate with an error, and that the intersection $\tilde{\cJ}\cap \tilde X^{\circ '}$ is finite for some subset $\tilde X^{\circ '} \subset \tilde X^\circ$, which is homotopy equivalent to $\tilde X^\circ$, see Definition \ref{def:wkb_cameral} and Remark \ref{rem:finite_network}.

\begin{construction}[WKB construction]
\label{constr:wkb_cameral}
Let $a\in \cA^{\diamondsuit}$ be a point in the subset of the Hitchin base defined in Definition \ref{definition:  cA diamondsuit}, with $\cL=K_{X^{c}}(D)$ where $\text{deg}(D)>2$ as in Convention \ref{convention:cameral}. Recall that this corresponds to a smooth cameral cover $\pi : \tilde X \to X$, with ramification of order 2. We construct from $a$ is the data of (c.f. Definition \ref{defin:basic_abstract}):
\begin{itemize}
    \item A set $\tilde \cW_{WKB}$ of oriented, semi-infinite real curves $\ell \subset \tilde X^\circ$, called \textbf{Stokes curves};
    \item for each $\ell \in \tilde \cW_{WKB}$, a root $\alpha_{\ell} \in \Phi$ called the \textbf{label} of the Stokes curve.
\end{itemize}
We define $\tilde \cW_{WKB}$ as the union $\bigcup_{k = 0}^{\infty} \tilde \cW_k$
over all iterations $k$ of the following algorithm, assuming that the algorithm does not terminate with an error.
\begin{enumerate}
    \item \label{item: initialization}\textbf{Initialization:} Due to Lemma \ref{lem:form_local}, for each $r \in R$, there are precisely two roots $\pm \alpha$ such that $\tilde a(r) \in H_\alpha$. Consider the action of $\R_+$ by scaling on the fibers of $T(\tilde X\setminus R)$, and let $V_\alpha \in \Gamma \big(\tilde X\setminus R, T(\tilde X\setminus R) / \R_+\big)$ be the unique oriented projective vector field on $\tilde X \setminus R$ such that:
    \begin{equation}
    \label{eq:diff_eq}
    \chi_\alpha(V_\alpha) \in \R_+.
    \end{equation}
    $V_\alpha$ determines a foliation of $\tilde X \setminus R$ by oriented integral curves; see Figure \ref{fig:foliation_ramif}.
    We consider only those integral curves that are incident to $r$, take their image under the map $ \tilde X \setminus R\hookrightarrow \tilde X^\circ$, and take their closure in $\tilde X^\circ$.
    
    \begin{figure}[h]
        \centering
        \includegraphics[width = 0.5\textwidth]{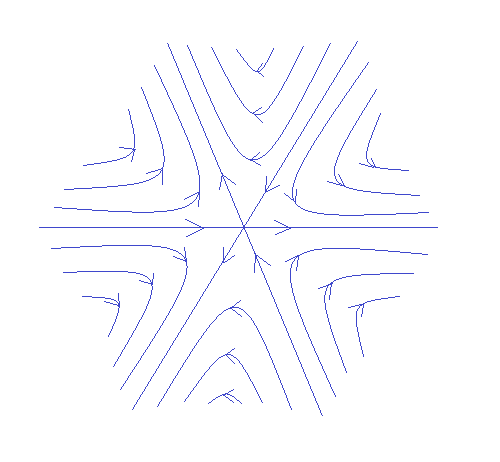}
        \caption{A schematic picture of the foliation determined by $V_\alpha$ near a ramification point fixed by $s_{\alpha}$ in the cameral cover.}
        \label{fig:foliation_ramif}
    \end{figure}
    
    The set of \textbf{primary Stokes curves} at $r \in R$ is the set $\tilde \cW_{0,r}$ of closures of the above integral curves. According to Lemma \ref{lem:wkb_local} below, there are three of them for $\alpha$, and another three for $-\alpha$. The labels of these curves are $\alpha$ and $-\alpha$, respectively.
    
    The set of all primary Stokes curves is:
    \[
    \tilde \cW_0 = \bigcup_{r \in R} \tilde \cW_{0,r}.
    \]
    
    \item \textbf{Iteration:} \label{item:new_line}
    For each iteration $k \geq 1$, do the following.
    Two or more existing Stokes curves may intersect at a point $x \in \tilde X^\circ$, which we call a \textbf{joint}. We let $\tilde \cJ_k$ denote the union of all joints which appear at iteration $k$.
    %be the union of all joints $J$, where one of the Stokes curves from iteration $k-1$ intersects a Stokes curve from an iteration $\leq k-1$.
    For each $J \in \tilde \cJ_k$, let
    $\{\alpha_1, \dots, \alpha_n\}$
    denote the labels of the incoming Stokes curves at $J$.
    Let $\gamma_j$ range over all roots in $\Conv^\N_{\alpha_1, \dots, \alpha_n}\backslash \{\alpha_{1},...,\alpha_{n}\}$. For each $j$, the 1-form $\chi_{\gamma_j}$ then determines\footnote{As in the initialization step.} $V_{\gamma_j} \in \Gamma\big(\tilde X \setminus R, T(\tilde X\setminus R) / \R_+\big)$, and the corresponding foliation by oriented, integral curves. The new Stokes curve $\ell_{x,j}$ is the unique oriented integral curve segment of $V_{\gamma_j}$ starting at $J$. We set:

    %The \textbf{new Stokes curves} at iteration $k$ and joint $x \in \tilde \cJ_k$ are $\ell_j$, labeled by $\gamma_j$, where $\gamma_j$ ranges over all roots in the restricted convex hull $\Conv^\N_{\alpha_1, \dots, \alpha_n}$. We set:
    %\[
    %\tilde \cW_{k,x} = \{\ell_j\}.
    %\]
    %there is one new Stokes curve $\ell_i$ starting at $J$, for each root $\{\beta_i\}_{i\in I_J}$ which is a positive integral linear combination:
    %\[
    %\beta_i = a_1 \alpha_{\ell_a} + a_2 \alpha_{\ell_b}, \hspace{1cm} a_1, a_2 \in \Z_+ .
    %\]
  %See subsection \ref{subsect:lie_bracket} for discussion; according to Lemma \ref{lem:brackets_vanish}, the structure of root spaces constrains $|I_J| \in \{0,1,2,4\}$.
  %In particular, due to Lemma \ref{lem:roots_commutators_vanish}, $|I_J| = 0$ if and only if the pre-Stokes factors of the incoming Stokes curves commute.
  
  %We set the label $\alpha_{\ell_i} = \beta_i$.
    \begin{align*}
    \tilde \cW_{k,x} 
    &:=
    \bigcup_{j} \ell_{x,j} ,
    \\
    \tilde \cW_k 
    &:=
    \bigcup_{x \in \tilde \cJ_k} \tilde \cW_{k,x}. 
    \end{align*}
    
    \begin{figure}[h]
    \centering
    \includegraphics[width = 0.7 \textwidth]{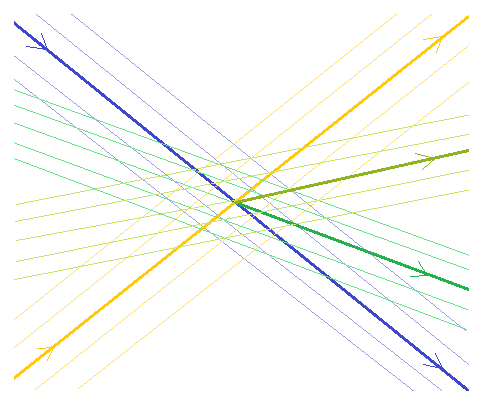}
    \caption{Stokes curves (solid) are segments of certain integral curves (transparent). When Stokes curves labeled by $\alpha$ (blue) and $\beta$ (yellow) intersect, we draw new Stokes curves for all $\gamma_i \in \Conv^\N_{\{\alpha,\beta\}}$. Here there are two new Stokes curves labeled by $\gamma_1$, $\gamma_2$ (green, khaki, respectively).  These new Stokes curves are integral curves of $V_{\gamma_1}$, $V_{\gamma_2}$ starting at the intersection point.}
    \label{fig:new_stokes}
  \end{figure}

    \item \textbf{Error:} 
    \label{item:error}
    The algorithm terminates with an error if, at some step $k$, any of the following happens:
    \begin{enumerate}
        \item \label{item: error: runs into ramification point}A Stokes curve runs into a ramification point. See Figure \ref{fig:double_wall} for an example.
        \item \label{item: error: cycle}The directed graph associated to $\tilde \cW_k$ (as in Definition \ref{defin:basic_abstract}, condition \ref{item:acyclic}) contains an oriented cycle. See Figure \ref{fig:oriented_cycle} for an diagram.
        \item \label{item: error: non convex}The roots labelling the incoming set of Stokes rays to a joint do not form a convex set, or the tangent rays of all Stokes curves entering and leaving a joint are not distinct.
        \item \label{item: error: denseness}A Stokes curve is dense in the sense that its set of limit points is not finite.  In Section \ref{subsection: denseness of individual trajectories}, specifically in Corollary \ref{corollary: behaviour on SF locus},  we show this does not occur for $a$ in the dense open subset $\cA_{SF,0}^{\diamondsuit}\subset \cA^{\diamondsuit}$.
    \end{enumerate}
    
    \begin{figure}[h]
    \centering
    \begin{minipage}{.5\textwidth}
      \centering
      \includegraphics[width = 0.9 \textwidth]{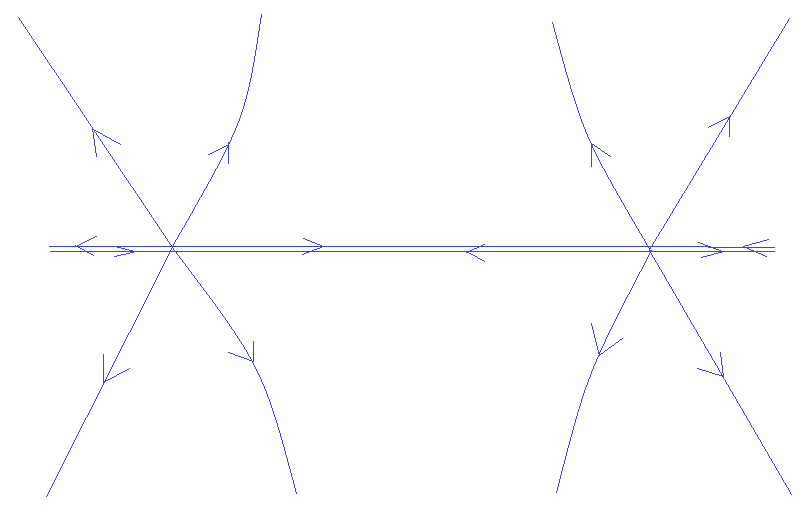}
      \caption{A Stokes curve incident to a ramification point, part of a double wall.}
      \label{fig:double_wall}
    \end{minipage}%
    \begin{minipage}{.5\textwidth}
      \centering
      \includegraphics[width = 0.9 \textwidth]{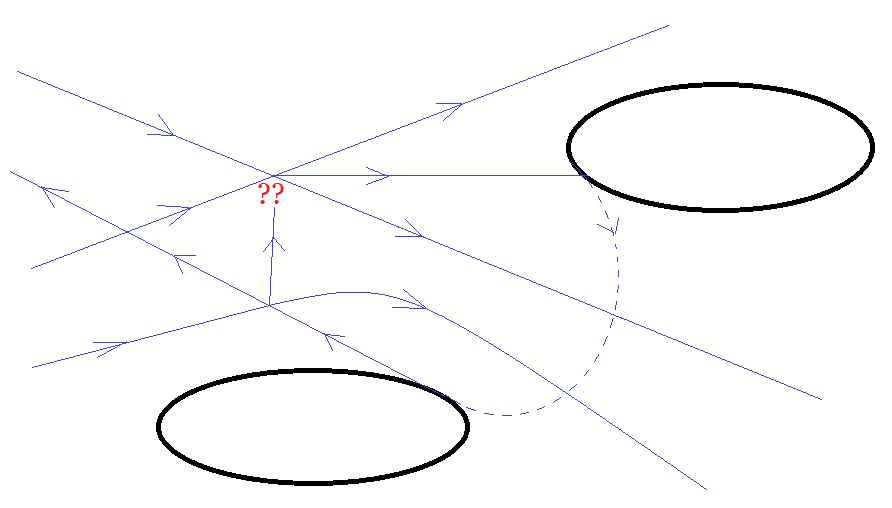}
      \caption{New Stokes curves form an oriented cycle.}
      \label{fig:oriented_cycle}
    \end{minipage}
    \end{figure}
    
    \end{enumerate}
\end{construction}

\begin{remark}
\label{rem:tangent_preimage}
The following is an alternative way to understand the definition of $V_{\alpha}$.  A section $\tilde{X}\rightarrow \ft_{ K_{\tilde{X}^{c}(D)}}$ is equivalent to a section $\tilde a^\vee\in \Hom(T\tilde{X}, \ft)$ (note that we have restricted away from the divisor $D$). To each root $\alpha$ we associate a subset $H_{\alpha}^{+}:=\{b\in \ft|\alpha(b) \in \R_+\}$, of real codimension 1. Taking the preimage of $H_{\alpha}^{+}$ under $\tilde a^\vee$ gives the oriented real projective vector field $V_{\alpha}$ on $\tilde{X}\backslash R$.  On the ramification locus the morphism $a^{\vee}$ is the zero map.
\end{remark}

\begin{warning}
\label{warning: discrepancy lines}
In a slightly more general context than the networks considered here there is one minor discrepancy between the pushforwards of the Stokes rays defined above to $X$, and those defined in the earlier works \cite{gaiotto2013spectral, longhi2016ade}.  Namely if a Stokes ray in the sense above labelled by a root $\alpha$ ends at a ramification point of $\tilde{X}^{c}\rightarrow X$ that is not fixed by $s_{\alpha}\in W$ then the associated Stokes ray on $X$ as described in \cite{gaiotto2013spectral, longhi2016ade} could continue through the associated branch point.  This discrepancy is irrelevant to the situations we consider due to us excluding networks that violate error condition \ref{item: error: runs into ramification point} of Construction \ref{constr:wkb_cameral}. 
\end{warning}

\begin{definition}
\label{def:wkb_cameral}
Let $a\in \cA^{\diamondsuit}$ be a point in the subset of the Hitchin base defined in Definition \ref{definition:  cA diamondsuit}. Assume that the WKB construction determined by $a$ does not terminate in an error. Assume, moreover, (see Remarks \ref{rem:finite_network} and \ref{remark:  Existence of generic WKB cameral networks}) that there exists a subset $X^{\circ '} \subset X^\circ$, such that:
\begin{itemize}
    \item the inclusion $X^{\circ '} \hookrightarrow X^\circ$ is a homotopy equivalence;
    \item the restriction $\tilde \cW_{WKB} ' := \tilde \cW_{WKB}|_{\tilde X^{\circ'}}$ consists of finitely many Stokes curves, intersecting in finitely many joints, where $\tilde X^{\circ '} := X^{\circ '} \times_{X^\circ} \tilde X^\circ$.
\end{itemize}
Then the \textbf{basic WKB cameral network} associated to $a \in \cA^\diamondsuit$ is the data of $X^{\circ '}$, the Stokes curves in $\tilde \cW'_{WKB}$, and their labels.
\end{definition}

\begin{remark}
\label{rem:finite_network}
In Proposition \ref{prop:condition_R_finite}, we will show that a subset $X^{\circ '}$ satisfying the conditions of Definition \ref{def:wkb_cameral} exists if we make the following assumptions on the point in the Hitchin base:
\begin{itemize}
    \item The WKB construction doesn't end with an error.
    \item The set $\tilde J$ of joints of the WKB construction doesn't have an accumulation point in $\tilde X^\circ$ (we allow an accumulation point in $\tilde{X}^{c}$).
    \item The point in the Hitchin base satisfies Condition R, introduced in Definition \ref{definition: Condition R}; this condition holds away from a real codimension 1 locus in the Hitchin base.
\end{itemize}
\end{remark}

%In the next definition, we single out a subset of WKB cameral networks, to which our constructions in Section \ref{sec: flat Donagi Gaitsgory} apply (c.f. Definition \ref{defin:basic_abstract}).

%\todo{probably remove this definition, and rewrite! this page}
%\begin{definition}
%\label{defin:wkb_generic}
%A \textbf{basic WKB cameral network} is a WKB cameral network such that:
%\begin{itemize}
%    \item For every joint $x \in \tilde \cJ$, the labels of the incoming Stokes curves are a convex subset $C_{in} \subset \Phi$.
%\end{itemize}
%\end{definition}

\begin{conjecture}
\label{conjecture: open dense set}
We conjecture that with respect to the classical topology there is a dense, open set $U\subset \cA$, such that there is a basic WKB cameral network associated to each $a\in U$.
%WKB cameral networks are determined  by some points in $\cA^{\diamondsuit}$. We conjecture that basic WKB networks in Definition \ref{defin:wkb_generic} are produced by construction \ref{constr:wkb_cameral} for a dense, open (in the classical topology)) subset $U\subset \cA^{\diamondsuit}$. 
\end{conjecture}

To prove this conjecture one would need to show that there was an open dense set for which none of the four conditions leading to an error in the WKB construction (\ref{constr:wkb_cameral}) occurred, and for which the set of joints $\tilde{J}$ does not have an accumulation point in $\tilde{X}^{\circ'}$.  As already noted the error type \ref{item: error: denseness} (of Construction \ref{constr:wkb_cameral}) does not occur on an open, dense set $\cA_{SF,0}^{\diamondsuit}\subset \cA$ (see Definition \ref{def: saddle free}) by Corollary \ref{corollary: behaviour on SF locus}.  Furthermore in Corollary \ref{corollary: behaviour on SF locus} for $a\in \cA^{\diamondsuit}_{SF,0}$ error type \ref{item: error: runs into ramification point} does not occur for primary Stokes lines We expect it to be difficult to understand the locus where error condition \ref{item: error: cycle} occurs, and the locus where $\tilde{J}$ has limit points in $\tilde{X}^{\circ}$.  We sketch an argument that a certain subset of the set $U$ is open, but not necessarily nonempty in Remark \ref{remark: open set around a }.

\begin{remark}
\label{remark:  Existence of generic WKB cameral networks}
The authors are aware of the following results in this direction of Conjecture \ref{conjecture: open dense set}:
\begin{itemize}
    \item For the group $G=SL(2)$ Lemma 4.11 of \cite{bridgeland2015quadratic} immediately implies that, as long as $D\neq 0$, the set of points of the Hitchin base that produce basic cameral networks, in the sense of Definition \ref{def:wkb_cameral}, is a dense open set in the Hitchin base $\Gamma(X^{c}, K_{X^{c}}(D)^{\otimes 2})$ with respect to the classical topology. %\cite{bridgeland2015quadratic} shiow%Furthermore in the paper \cite{haiden2014flat} the locus in the base where the spectral networks have lines joining branch points is interpreted as the locus of walls of the second kind in the moduli space of stability conditions for the Fukaya category of the surface. \todo{in fact reference proof of Lemma 4.11 of Bridgeland Smith for result that set is open.}
%Note that the condition was that one m was greater or equal to 2, which is automatic -- as here we always have quadratic poles as we get 2D as what we're twisting K_{X}^{\otimes 2} by 
%  \item For $G=SL(n)$ the construction of 
    \item There is empirical evidence for groups of type A in \cite{gaiotto2013spectral}, and for groups of type ADE in \cite{longhi2016ade, longhiloom}. In particular, \cite{longhiloom} is a piece of software for drawing spectral networks for groups of type ADE. This can be used to produce many examples of basic WKB cameral networks.  %See remark \ref{remark: open set around a } for a sketch of an argument that if $a\in \cA^{\diamondsuit}$ produces a WKB cameral network with the property that at each joint $J\in \cJ$ there are exactly two incoming lines, points sufficiently close to $a$ also produce a WKB cameral network with these properties.
\end{itemize}
%\todo{Write down the proof of this result [possibly in an appendix]}
\end{remark}

\begin{proposition}
\label{prop:basic_WKB_is_basic_abstract}
Basic WKB cameral networks (Definition \ref{def:wkb_cameral}) are examples of basic abstract cameral networks.
\end{proposition} 

\begin{proof}
There are five conditions to check in Definition \ref{defin:basic_abstract}. Acyclicity, and the local picture around joints are explicitly satisfied by the requirement that Construction \ref{constr:wkb_cameral} doesn't end with an error.  Non-denseness follows from the condition that the WKB condition does not end in an error which in particular implies that each Stokes line has a discrete set of limit points, and by the additional assumption of Definition \ref{def:wkb_cameral} that there are only finitely many Stokes curves on $X^{\circ'}$.  It remains to check the local picture around ramification points, and the equivariance. We do this in Lemma \ref{lem:wkb_local} and Lemma \ref{lem:equivariance}, respectively.
\end{proof}

\begin{lemma}
\label{lem:wkb_local}
Consider a basic WKB cameral network $\tilde \cW_{WKB}'$.
For each $r \in R$ of order 2, let $ \alpha$ be a root such that the 1-form $\chi_\alpha$ vanishes at $r$. Then, in an infinitesimal neighborhood of $r$, the Stokes curves and their labels are as in Figure \ref{fig:ramif1}.
\end{lemma}

\begin{proof}
From Lemma \ref{lem:form_local} we know that there exists a local coordinate $z$ in a neighborhood of $r$ such that $\chi_\alpha$ has the local form:
\[   \chi_\alpha(z) =   z^2  dz . \]
A tangent vector at $z = x + iy$, oriented away from $0$, is:
\[
V = x \frac{\p}{\p x} + y \frac{\p}{\p y} = z \frac{\p}{\p z} + \bar z \frac{\p}{\p \bar z} .
\]
Then: 
\[
\chi_\alpha( V) =  z^3 .
\]
Condition \ref{eq:diff_eq} is that $ z^3 \in \R_+$.
So there are three directions of Stokes curves labeled by $\alpha$ corresponding to $z=re^{i\theta}$ with $\theta$ in the set:
\[   \{ 2\pi j /3 \}_{j\in\{0,1,2\}},   \]
and three directions labeled by $-\alpha$ corresponding as above to angles:
\[   \{ \pi + 2\pi j /3 \}_{j\in\{0,1,2\}}.   \]
\end{proof}

\begin{lemma}
\label{lem:equivariance}
Basic WKB cameral networks satisfy the equivariance condition \ref{item:equivariance} from Definition \ref{defin:basic_abstract}.
\end{lemma}

\begin{proof}
The action of $W$ on $\fr t^\vee\ni \alpha$, is dual to that on $\fr{t}$. It follows that, for every $w \in W$ and $r \in R$, $\tilde a \big(w(r)\big) \in H_{w(\pm \alpha)}$. Therefore the primary Stokes curves starting at $w(r)$ are labeled by $w(\alpha)$.

By definition, $\chi_\alpha = (1 \times \alpha) \circ \tilde a$.
The $W$-equivariance of $\tilde{a}$ implies that:
\begin{align*}
    \chi_{w\alpha} &= (1\times w\alpha) \circ \tilde a\\
    &= (1\times \alpha)\circ w\tilde{a}\\
    &= (w^{-1})^{*}\chi_{\alpha} .
\end{align*}
It follows that the projective vector fields satisfy $w_*(V_\alpha) = V_{w(\alpha)}$. For the integral curves, this means $w(\ell_\alpha) = \ell_{w(\alpha)}$.

Whenever $\ell_1, \ell_2$ intersect at $J$, then $w(\ell_1), w(\ell_2)$ intersect at $w(J)$. The restricted convex hulls satisfy $\Conv^\N_{\alpha_{w(\ell_1),w(\ell_2)}} = w \big(\Conv^\N_{\alpha_{\ell_1,\ell_2}} \big)$
\end{proof}

\subsubsection{Convexity at joints}
\label{subsubsec: convexity}

A consequence of Remark \ref{rem:tangent_preimage} is the following lemma:

\begin{lemma}
\label{lem:2d_sector}
Assume that, at a joint $x \in \tilde \cJ$, there are outgoing Stokes curves labeled by $\alpha, \alpha + \beta, \beta \in \Phi$. Let $v_\alpha, v_{\alpha + \beta}, v_{\beta} \in T_x\tilde X^\circ$ denote tangent vectors to the three outgoing curves. Then there exist $c_\alpha, c_\beta \in \R_+$ such that $v_{\alpha + \beta} = c_\alpha v_\alpha + c_\beta v_\beta$.
\end{lemma}

In particular this means that the tangent vector $v_{\alpha+\beta}$ to the outgoing Stokes curve labelled by $\alpha+\beta$ is in the cone spanned by the tangent vectors $v_{\alpha}$ and $v_{\beta}$.

\begin{proof}[Proof of Lemma \ref{lem:2d_sector}]
For each $\gamma \in \{\alpha, \alpha + \beta, \beta\}$, $v_\gamma$ is defined, up to scaling by $\R_+$, by:
\[
\gamma \big( \tilde a_x^{\vee} (v_\gamma) \big) \in \R_+,
\]
with $\tilde a_x^{\vee}$ as in Remark \ref{rem:tangent_preimage}. Since the result of this lemma is independent of rescaling $v_\gamma$ by $\R_+$, we can assume without loss of generality that:
\[
\gamma \big( \tilde a_x^{\vee} (v_\gamma) \big) = 1.
\]
Moreover, since $T_x\tilde X^\circ$ is 1-dimensional over $\C$, there exist $\eta_\alpha, \eta_\beta \in \C$ such that:
\begin{align*}
v_\alpha &= \eta_\alpha v_{\alpha + \beta},
\\
v_\beta &= \eta_\beta v_{\alpha + \beta}.
\end{align*}
Then:
\begin{align*}
1 
&=
(\alpha + \beta) \big( \tilde a_x(v_{\alpha + \beta})\big)
\\
&=
\alpha \big( \tilde a_x(v_{\alpha + \beta})\big) + \beta \big( \tilde a_x(v_{\alpha + \beta})\big)
\\
&= 
\eta_\alpha^{-1} \alpha \big( \tilde a_x(v_{\alpha})\big)
+ \eta_\beta^{-1} \beta \big( \tilde a_x(v_{\beta})\big)
\\
&= 
\eta_\alpha^{-1} + \eta_\beta^{-1}.
\end{align*}
Consequently, $\bar \eta_\alpha^{-1} + \bar \eta_\beta^{-1} = 1$. Then, using the fact that $\bar \eta \eta = |\eta|^2$, for all $\eta \in \C$:
\begin{align*}
\frac{1}{|\eta_\alpha|^2} v_\alpha 
+ \frac{1}{|\eta_\alpha|^2} v_\beta 
&=
\frac{1}{|\eta_\alpha|^2} \eta_\alpha v_{\alpha +\beta}
+ \frac{1}{|\eta_\alpha|^2} \eta_\beta v_{\alpha +\beta} 
\\
&= 
(\bar \eta_\alpha^{-1} + \bar \eta_\beta^{-1}) v_{\alpha + \beta} 
\\
&=
v_{\alpha + \beta}.
\end{align*}
So we can take $c_\alpha = 1/|\eta_\alpha|^2$ and $c_\beta = 1/ |\eta_\beta|^2$.
\end{proof}

\subsection{Examples of Spectral networks}

In the section we describe the relation between $SL(2)$ spectral networks and trajectories of a quadratic differential in Example \ref{example: SL2 and quadratic differentials}.  Then in Example \ref{ex:bnr} we revisit the example of \cite{bnr1982new} where new Stokes curves were first introduced.

\begin{example}[$G=SL(2)$ and Quadratic differentials]
\label{example: SL2 and quadratic differentials}
For $G=SL(2)$, so that $\ft \cong \C$ and $W = \Z/2$. The Hitchin base is:
\[
\cA
=
\Gamma\big(X^{c}, \ft_{K_{X^{c}}(D)} \sslash W\big) 
\cong 
\Gamma\big(X^c, (K_{X^c}(D))^{\otimes 2} \big).
\]
We call an element $\omega \in \Gamma\big(X^{c}, (K_{X^{c}}(D))^{\otimes 2}\big)$ a quadratic differential.

The commutative diagram that defines the cameral cover of $X$ associated to $a$;
\[
\begin{tikzcd}
\tilde X^{c}
\arrow[swap]{d}{\pi} \arrow{r}{\tilde a}
& 
\ft_{ K_{X^c(D)}}
\arrow{d}
\\
X^{c}
\arrow{r}{a}
&
\ft_{ K_{X^c(D)}} \sslash W
\end{tikzcd}
\]
provides a section $\tilde a \in \Gamma\big(\tilde X^{c}, K_{X^c}(D)\times_{\bbG_{m}}\ft\big)$. If we locally (in the classical topology) pick a section of $s$ of $\tilde{X}$ the pullback $s^{*}(\tilde a)$ on $X$ only depends up to sign on the section $s$ chosen, and $\big(s^*(\tilde a)\big)^{\otimes 2} = \omega$.

Each $\omega$, determines a foliation of $X$ by curves $\gamma : \bbR_{t}\rightarrow X$ defined by the condition:
\begin{equation}
\label{eq:foliation_qd}
\frac{d}{ds}\int^{s}_{t=t_0}\pm \sqrt{\omega}(\gamma(t))\in \bbR  .
\end{equation}
These curves don't come with a canonical parametrization, because condition \ref{eq:foliation_qd} determines the tangent vectors of $\gamma$ only up to a scaling action of $\R^*$. These curves also do not have a canonical orientation.  The leaves of this foliation are called \textbf{trajectories of the quadratic differential} $\omega$. (See the classical text \cite{strebel1984quadratic}, as well as the more recent works \cite{bridgeland2015quadratic, iwaki2014exact, iwaki2015exact}.) 

Since $\big(s^*(\tilde a)\big)^{\otimes 2} = \omega$, the trajectories of $\omega$ which start from branch points of $\pi$ agree with the Stokes curves in the WKB construction associated to $a$, after we forget the orientation and label of the latter.

Examples of trajectories of a quadratic differential can be found in figure \ref{figure: various quadratic differentails}.
\begin{figure}[h]
    \centering
    \begin{minipage}{.33\textwidth}
      \centering
      \includegraphics[scale = 0.3]{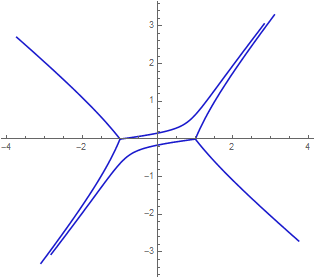}
      \label{fig:qd-}
    \end{minipage}%
    \begin{minipage}{.33\textwidth}
      \centering
      \includegraphics[scale = 0.3]{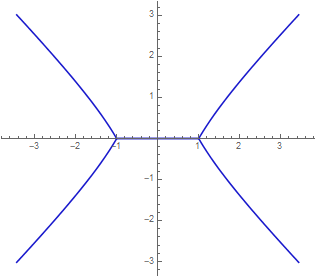}
      \label{fig:qd0}
    \end{minipage}%
    \begin{minipage}{.33\textwidth}
      \centering
      \includegraphics[scale = 0.3]{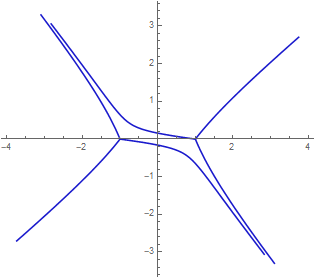}
      \label{fig:qd+}
    \end{minipage}
    \caption{Trajectories of the quadratic differentials $\omega = e^{i\pi/5}(1-z^2)dz^{\otimes 2}$, $\omega = (1-z^2)dz^{\otimes 2}$ and $\omega = e^{-i\pi/5}(1-z^2)dz^{\otimes 2}$, respectively.}
    \label{figure: various quadratic differentails}
    \end{figure}
\end{example}

\begin{example}[$SL(3)$, BNR example \cite{bnr1982new}.  See Figure \ref{fig:bnr_network}]
\label{ex:bnr}
Recall from Example \ref{eg:convex_hull_typeA} that, for Lie algebras of type A, given any two roots $\alpha, \beta$, the restricted convex hull satisfies:
\[
\Conv^\N_{\alpha,\beta} 
\subset 
\{\alpha, \alpha+\beta, \alpha\} ,
\]
(see Figure \ref{fig:a2root}).
As such, there is at most one new Stokes curve generated at each joint where two Stokes lines intersect.

Consider $G = SL_3(\C)$, which has the root system $A_2$. 
Consider $X = \bbP^1$ with a single puncture at $\infty$. We consider the following characteristic polynomial as a point in a Hitchin base for $X$:
\[  P(z) =  \lambda^3 + 3\lambda + 2i z  = 0.  \]
The discriminant of $P(z)$ vanishes at $z = \pm 1$, which are the branch points of the associated cameral cover. We will choose a $W$-framing, away from two branch cuts starting at $z=1$ (and $z=-1$) and tending to $\infty$ along a path slightly above the locus where $z$ is real and  negative (respectively real and positive).   With respect to the $W$-framing we choose let $s_{\alpha_{1}}\in W$ be the element of $W$ fixing the ramification point\footnote{That is to say if we take the ramification point corresponding to the identity branch of the cameral cover, this point is fixed by $s_{\alpha_{1}}$.} at $z=1$, and $s_{\alpha_{2}}$ fixes the ramification point at $z=-1$.  This creates 3 primary Stokes curves starting from $z=1$, labeled by $\pm \alpha_1$, and three starting from $z=-1$, labeled by $\pm \alpha_2$, where we are using the chosen framing to label the curves by roots as explained in Remark \ref{rem:root_noncanonical}. As can be seen in Figure \ref{fig:bnr_network}, there are two intersections of primary Stokes curves. If we choose the $W$-framing such that the two primary Stokes lines in the upper half of Figure \ref{fig:bnr_network} are labelled by $\alpha_{1}$ and $\alpha_{2}$ then the new lines generated at the intersection are labeled by $\pm (\alpha_1 + \alpha_2)$ with respect to our chosen $W$-framing.

\begin{figure}[h]
\label{fig:bnr_network}
\includegraphics[scale=0.5]{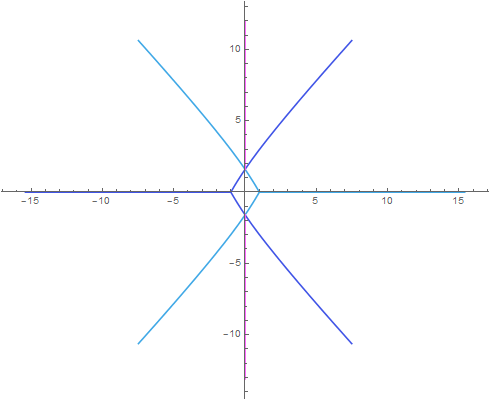}
\caption{The BNR spectral network. Light blue is $\pm \alpha_1$, dark blue is $\pm \alpha_2$, purple is $\pm(\alpha_1 + \alpha_2)$.}
\end{figure}

This example first appeared in \cite{bnr1982new}, which was, to the best of our knowledge, the first source to mention new Stokes curves at the intersections of old ones. The context in \cite{bnr1982new} was the WKB analysis of solutions to differential equations.
\end{example}

%Next, we describe an example of spectral network of type B2. Note the appearance of two new Stokes curves, rather than just one, at certain joints. This phenomenon is characteristic of  Lie algebras of type BCF.

%\begin{example}
%\label{ex:so5}
%Consider $G = SO(5)$, with root system $B_2$ depicted in Figure \ref{fig:b2root}. With $\alpha$ (short root) and $\beta$ (long root) as in Figure \ref{fig:b2root}:
%\[
%\Conv^\N_{\alpha,\beta}
%=
%\{ \alpha, 2\alpha+\beta, \alpha+\beta, \beta \} .
%\]

%Consider $X = \bbP^1$ with a single puncture at $\infty$. We consider the following characteristic polynomial for Higgs bundles on $X$:
%\[
%P(\lambda, z) = \lambda^2 (- 6 - \lambda^2) - (z^4 - 1). 
%\]
%The roots are:
%\[
%\lambda_{1,2,3,4} = \pm \sqrt{-3 \pm \sqrt{10- z^4}}.
%\]
%This gives 4 branch points where $z^4 = 10$, and 4 more where $z^4 = 1$. We have 3 primary Stokes curves starting from each branch point, and an avalanche of new Stokes curves at various intersections. \todo{Fix the picture and add it.}

%\begin{figure}[h]
%\caption{A spectral network of type B2.}
%\label{fig:b2_network}
%\includegraphics[scale=0.3]{B2_spectral.png}
%\end{figure}

%\end{example}

%\todo{IMPORTANT: Include the (rigorous version) of the proof sketched in email of 23rd August "(Almost) [Analytic] open non-denseness of a spectral network"}

\subsection{Trajectories near $D\subset X^{c}$}
\label{subsec: trajectories near D}

In this subsection we introduce a condition (Definition \ref{definition: Condition R}) 
(which we call Condition R) on points in the Hitchin base 
$\cA:= \Gamma(X^{c}, \ft_{K_{X^{c}}(D)}\sslash W)$, 
which allows us to consider WKB spectral networks with only a finite number of Stokes curves.
In general, as long as Stokes curves labeled by different roots are asymptotic to the same $d \in D$, they can have infinitely many intersections, each of which generates new Stokes curves. See Figure \ref{fig:joints_accumulate}, Remark \ref{remark: infinitely many new Stokes curves}, and the discussion in Section 5.5 of \cite{gaiotto2013spectral}.

\begin{figure}[h]
    \centering
    \includegraphics[width = 0.5 \textwidth]{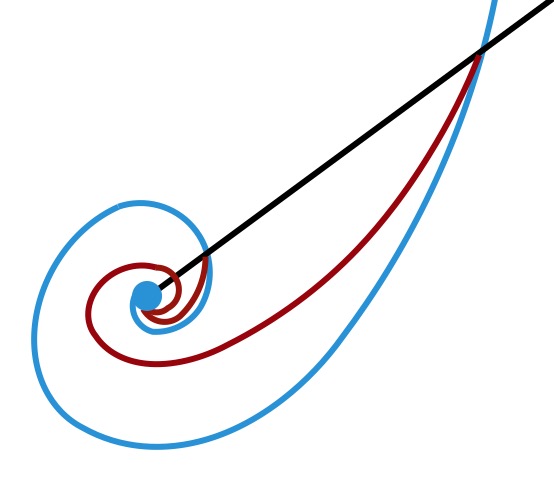}
    \caption{Two Stokes curves (blue and black) whose intersections generate infinitely many new Stokes curves (red).  We have truncated the diagram (with the large blue circle) so that only a finite number of intersections can be seen.}
    \label{fig:joints_accumulate}
\end{figure}

By restricting the spectral network to an appropriate subset $X^\circ \setminus \coprod_{d\in D} B_d$, we will be free to ignore this behaviour, under the assumption that the set of joints $\cJ$ is discrete in $X^{\circ}$.

\begin{definition}
\label{def:residue_hitchin_point}
Let $a\in \Gamma(X^{c}, \ft_{K_{X^{c}}(D)}\sslash W)$ be a point in the Hitchin base. In an open set $U_{d}\ni d$, for each $d\in D$, we can write:
\[
a(z) = r_{d}(z)\frac{dz}{z},
\]
where $z$ is a local coordinate around $d$ in $U_{d}$, and $r_{d}$ is a function $r_d: U_{d}\rightarrow \ft\sslash W$. We call any lift $\tilde r_d$ of $r_d(0)$ to $\ft$ a \textbf{residue} of $a$ at $d$.
\end{definition}

%\todo{make an explicit statement what this is for SL 2! via local equation as is used in following lemma}

\begin{figure}[h]
    \centering
    \begin{minipage}{.33\textwidth}
      \centering
      \includegraphics[width = 0.9\textwidth]{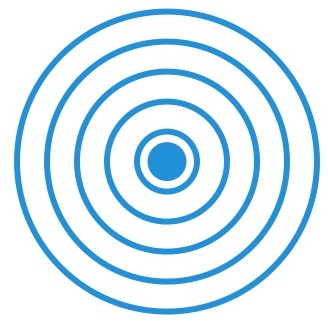}
      \caption{$\tilde r_d \in i\R$}
      \label{fig:residue_imag}
    \end{minipage}%
    \begin{minipage}{.33\textwidth}
      \centering
      \includegraphics[width = 0.9\textwidth]{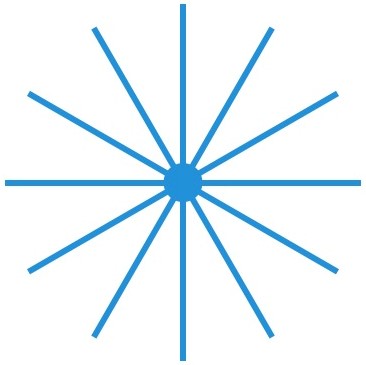}
      \caption{$\tilde r_d \in \R$}
      \label{fig:residue_real}
    \end{minipage}%
    \begin{minipage}{.33\textwidth}
      \centering
      \includegraphics[width = 0.9\textwidth]{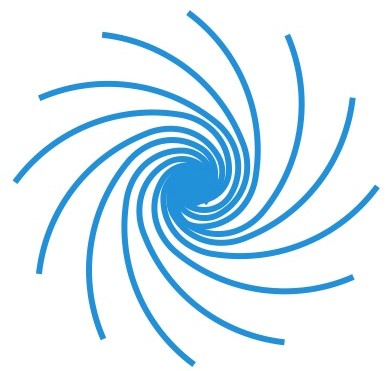}
      \caption{$\tilde r_d \in \C \setminus (i\R\cup \R)$}
      \label{fig:residue_generic}
    \end{minipage}
\end{figure}

Recall from Example \ref{example: SL2 and quadratic differentials} the connection between points in the Hitchin base for $SL(2)$ and quadratic differentials. We use the residue of Definition \ref{def:residue_hitchin_point} to understand the behaviour of the trajectories of the quadratic differential near $d \in D$.  It is well known (see e.g. Section 6.2 of \cite{strebel1984quadratic}) that there is a neighbourhood $B_{d}$ of a point $d\in D$ for which the trajectories look like one of the Figures \ref{fig:residue_imag} -- \ref{fig:residue_generic}, depending on whether $\alpha(\tilde r_d)$ is in $\R$, $i\R$ or $\C \setminus (\R \cup i\R)$, respectively.  This means that in the case that $\alpha(\tilde r_d) \in \C \setminus i \R$, each trajectory of $\omega$ in $B_{d}$ has one end that converges to $d$.  Because we will be working with multiple trajectories simultaneously we need the following slightly stronger result:

\begin{lemma}[Asymptotics of trajectories of quadratic differentials]
\label{lem:qd_at_pole}
In the case of $G = SL(2)$, let $\tilde r_d$ be a non-zero residue of $a \in \Gamma(X^c, \ft_{K_{X^c}(D)}\sslash W)$ at $d\in D$. Let $\omega \in \Gamma(X^c, (K_{X^c}(D))^{\otimes 2})$ be the quadratic differential associated to $a$, as in Example \ref{example: SL2 and quadratic differentials}. Let $z$ be a local coordinate around $d\in D$ which vanishes at $d$. For any real number $r>0$ we define the set $B_{d,r,z}:=\{x\in X||z(x)|<r\}\subset X^{c}$, which also depends on the choice of coordinate $z$.

If $\tilde{r}_{d}\in \bbC\backslash i\bbR$ then there is a real number $r_{0}>0$ (depending on $z$), such that for all $r<r_{0}$ the open set $B_{d,r,z} \subset X^c$, has the property that all the trajectories of $\omega$ that enter $B_{d,r,z}$ have one direction that converges to $d$, and never leaves the open set $B_{d,r,z}$.
\end{lemma}

\begin{proof}
By hypothesis, $a$ has a simple pole at $d$, so square roots of $\omega$ have the following expansion around $d$:
\[
\pm \sqrt{\omega} = \pm \tilde r_d \frac{dz}{z} + \sum_{i=0}^\infty c_i z^i dz .
\]
Recall from Equation \ref{eq:foliation_qd} that trajectories of $\omega$ are integral curves of the real projective vector field $V_\omega$ which satisfies:
\[
\pm \sqrt{\omega}(V_\omega) \in \R.
\]
Let $V_\omega^{\text{asymp}}$ be the projective vector field which satisfies:
\begin{equation}
\label{eq:v_asymp}
\pm \tilde r_d \frac{dz}{z} (V_\omega^{\text{asymp}}) \in \R.
\end{equation}
Equation \ref{eq:v_asymp} implies:
\[
V^{\text{asymp}}_\omega = \pm (\tilde r_d)^{-1} z \p_z .
\]

The trajectories of $V_{\omega}^{\text{asymp}}$ are logarithmic spirals, which can be parameterized by $t\in \bbR$ using the equation $T(t):=\exp((ic+t)/\tilde{r}_{d})$ for some constant $c\in \bbR$.  In particular this means that the angle at which they intersect each circle $|z|=r$ does not depend on $r$.  These trajectories are plotted in Figures \ref{fig:residue_real} and \ref{fig:residue_generic}.

Then, using the fact that $\pm \tilde r_d \frac{dz}{z}$ is the dominant term in $\pm \sqrt{\omega}$ as $z \to 0$, $V_\omega$ is asymptotic to $V_\omega^{\text{asymp}}$ in a sufficiently small open set $B_d \ni d$.  This means that the difference in the direction of $V_{\omega}$ and $V_\omega^{\text{asymp}}$ can be made arbitrarily small by selecting sufficiently small values of $r$.  The trajectories of $V_{\omega}^{\text{asymp}}$ have the desired property.  Hence as these trajectories are logarithmic spirals, any projective vector field $V'$ such that the difference in direction between $V'$ and $V_{\omega}^{\text{asymp}}$ is sufficiently small near $d$ would also have this property.  Hence $V_\omega$ has this property, that is to say there exists $r_{0}$ with the property that for all $r<r_{0}$ the open set $B_{d,r,z} \subset X^c$, has the property that all the trajectories of $\omega$ that enter $B_{d,r,z}$ has one direction that converges to $d$, and never leaves the open set $B_{d,r,z}$.

%It remains to show that the integral curves of $V_\omega^{\text{asymp}}$ are as shown in figures \ref{fig:residue_imag} - \ref{fig:residue_generic}.

%Therefore the phase of $\alpha_{SL(2)}(\tilde r_d)$ determines $V_\omega^{\text{asymp}}$, and we obtain the trichotomy in the Figures \ref{fig:residue_imag} - \ref{fig:residue_generic}.

%Then, if we choose any sufficiently small open set $B_d \ni d$, any trajectory of $\omega$ within $B_d$ is asymptotic to the curves drawn in the figures.
\end{proof}

Motivated by Lemma \ref{lem:qd_at_pole}, we define:

\begin{definition}[Condition R]
\label{definition: Condition R}
We say a point $a\in \Gamma(X^{c}, \ft_{K_{X^{c}}(D)}\sslash W)$ satisfies \textbf{condition R} if for each $d \in D$, the residue $\tilde r_d$ (as above) satisfies:
\[
\tilde r_d\in \ft \setminus \left(\bigcup_{\alpha \in \Phi}\alpha^{-1}(i\bbR)\right).
\]
This condition does not depend on the choice of lift $\tilde{r}_d$.

We denote the set of $a\in \cA^{\diamondsuit}$ satisfying condition R by $\cA^{\diamondsuit}_{R}$.
\end{definition}

The locus where Condition R is not satisfied is real codimension one in $\cA^{\diamondsuit}$, so this condition is generically satisfied in the Hitchin base.  The following proposition can be seen as formalizing an empirical observation of \cite{gaiotto2013spectral}.

\begin{proposition}
\label{prop: open sets around D}
Assume that $a\in \cA^{\diamondsuit}_{R}$ (Definition \ref{definition: Condition R}), For each $d\in D$, there exists an open disc $B_{d}$, with $d\in B_d \subset X^{c}$, such that any trajectory $\ell$ produced by the WKB construction (\ref{constr:wkb_cameral}) entering $\pi^{-1}(B_d)$, or produced as a new Stokes curve starting inside $\pi^{-1}(B_d)$, 
%associated with any root $\alpha$ (as specified in Construction \ref{constr:wkb_cameral})
will not leave the disc $\pi^{-1}(B_d)$.
\end{proposition}

\begin{proof}
%\todo{remove definition of the quadratic differentials and back reference!!}
Firstly, condition R implies that the residue $r(0)\in \ft\sslash W$ does not lie on the hyperplane $\left(\cup_{\alpha}H_{\alpha}\right)\sslash W$.  Hence the covering $\tilde{X}^{c}\rightarrow X^{c}$ is trivial when restricted to some neighbourhood $V_{d}$ of $d\in D$. Then $\pi^{*}K_{X^{c}}|_{\pi^{-1}(V_{d})}\cong K_{\tilde{X}^{c}}|_{\pi^{-1}(V_{d})}$. 

%For each $\alpha \in \Phi$, we decompose $\pi$ as a sequence of covering maps:
%\[
%\begin{tikzcd}
%\tilde{X} 
%\arrow[swap]{r}{p_\alpha}
%\arrow[bend left]{rr}{\pi}
%&
%\tilde X/s_\alpha 
%\arrow[swap]{r}{\pi_\alpha}
%&
%X
%\end{tikzcd}
%\]
%We also have $\pi_{\alpha}^{*}K_{X^{c}}|_{\pi_{\alpha}^{-1}(V_{d})}\cong K_{\tilde{X}^{c}/s_{\alpha}}|_{\pi_{\alpha}^{-1}(V_{d})}$ (as $\tilde{X}^{c}/{s_{\alpha}}\rightarrow \tilde{X}^{c}$ is also locally trivial).
%Then $a$ determines a map\todo{make it clear that the divisor $D$ is also being pulled back}: 
%\[
%\begin{tikzcd}
%\tilde{X}^c/s_{\alpha}|_{\pi_{\alpha}^{-1}(V_{d})}
%\arrow{r}
%&
%\ft/s_{\alpha}\otimes \pi_{\alpha}^{*}K_{X^c}(D)|_{\pi_{\alpha}^{-1}(V_{d})} 
%\arrow{d}{\alpha}
%&
%\\
%&
%\bbC/\{\pm 1\}\otimes \pi_{\alpha}^{*}K_{X^c}(D)|_{\pi_{\alpha}^{-1}(V_{d})}
%\arrow{r}{\cong}
%&
%\bbC/\{\pm 1\} \otimes K_{\tilde{X}^{c}/s_{\alpha}}(\pi_{\alpha}^{-1}(D)).
%\end{tikzcd}
%\]

Recall from Construction \ref{construction: Quadratic form alpha} we have a quadratic differential on each $\tilde{X}_{\alpha}=\tilde X^c / s_\alpha$, that is to say a section of $\big(K_{\tilde{X}^{c}/s_{\alpha}}(\pi_{\alpha}^{-1}(D))\big)^{\otimes 2}$. The (unoriented) foliation on $\tilde{X}$ associated to $\alpha$ is the inverse image by $p_\alpha$ of the foliation on $\tilde{X}_{\alpha}$ associated to the quadratic differential.
We choose a local coordinate $z$ of $X^{c}$ in a neighbourhood of $d$ which vanishes at $d$.  By pullback along $\pi_{\alpha}$ this gives a local coordinate at each preimage of $d$ in $\tilde{X}_{\alpha}^{c}$.  These are the coordinates we will use when we apply Lemma \ref{lem:qd_at_pole}, and we will denote them by $\tilde{z}$.

Lemma \ref{lem:qd_at_pole} provides, for each $\alpha \in \Phi$ and each lift $\tilde{d}$ of $d$ to $\tilde{X}_{\alpha}^{c}$, a real number $r_{0}>0$, such that the set $B_{\tilde{d},\tilde{z},r,\alpha}$ (we add $\alpha$ to the notation of Lemma \ref{lem:qd_at_pole}) has the desired property for Stokes curves labelled by $\alpha$, for every positive real number $r<r_{0}$.  We also pick $r_{0}$ sufficiently small such that $p_{\alpha}$ is unramified on $B_{\tilde{d},\tilde{z},r_{0},\alpha}$, and $\pi_{\alpha}$ is single valued on $B_{\tilde{d},\tilde{z},r,\alpha}$.  Define $B_{d}:=\cap_{\tilde{d}, \alpha} \pi_{\alpha}(B_{\tilde{d},\tilde{z},r_{0}/2,\alpha})$ where we take the intersection over all roots $\Phi$, and all lifts $\tilde{d}$ of $d$ to $\tilde{X}_{\alpha}$.  The set $B_{d}$ has the property that every preimage of it in $\tilde{X}_{\alpha}$ (for any root $\alpha$) is of the form $B_{\tilde{d},\tilde{z},r,\alpha}$ for some positive real $r<r_{0}$.  

The set $B_{d}$ hence has the property that any trajectory of the vector field $\chi_{\alpha}$ (of Construction \ref{def:form_root} for any root $\alpha$) entering $\pi^{-1}(B_{d})$ has the property one direction of this trajectory converges to $d$ and never leaves $\pi^{-1}(B_{d})$.

%We define $B_d = \bigcap_{\alpha \in \Phi} \pi_\alpha(B_{d,\alpha})$

%n open subset $\pi_\alpha^{-1}(d) \in B_{d,\alpha} \subset \pi^{-1}_\alpha(V_d)$, such that the Stokes curves in $p_\alpha^{-1}(B_{d,\alpha})$ labeled by $\alpha$ are asymptotic to the curves in Figures \ref{fig:residue_imag} - \ref{fig:residue_generic}. Letting $B_d = \bigcap_{\alpha \in \Phi} \pi_\alpha(B_{d,\alpha})$, and using Condition R, all Stokes curves in $\pi^{-1}(B_d)$ have an end asymptotic to $d$.

It remains to discuss the orientation of Stokes curves. We work inductively on the total order on $\tilde \cJ$ that the acyclicity assumption in the WKB construction guarantees.
\begin{itemize}
    \item Base step: For the Stokes curves entering $B_d$, the orientation is necessarily towards $d$.
    \item Inductive step: At each $x \in \tilde \cJ \cap \pi^{-1}(B_d)$, the inductive hypothesis says that the incoming Stokes curves are oriented towards $d$. By Lemma \ref{lem:2d_sector}, this is also the case for the new Stokes curve.
\end{itemize} 
\end{proof}

%\todo{Additional details on quadratic differentials, + address the issue raised in my email to myself of Dec 29th about how if alpha curves go to the pole, - alpha curves must go away from it.}

\begin{proposition}
\label{prop:condition_R_finite}
For a point in the Hitchin base satisfying condition $R$, such that the WKB construction $\tilde \cW_{WKB}$ for $a$ does not end with an error, and the joints of $\tilde \cW_{WKB}$ do not have accumulation points in $\tilde X^\circ$, the restriction:
\[
\tilde \cW_{WKB}' := \tilde \cW_{WKB}|_{X^\circ \setminus \cup_{d\in D}B_{d}}
\]
consists of finitely many Stokes curves, intersecting at finitely many joints.
\end{proposition}

\begin{proof}
If there are infinitely many Stokes curves in $X^\circ \setminus \cup_{d\in D}B_{d}$, there must be infinitely many joints $x\in \tilde \cJ\cap \left(X^\circ \setminus \cup_{d\in D}B_{d}\right)$ to produce them. These joints must then have an accumulation point. This contradicts the non-denseness assumption.
\end{proof}

\begin{remark}
\label{remark: infinitely many new Stokes curves}
%Note that the considerations in the proof of Proposition \ref{prop: open sets around D} show why the phenomena observed in \cite{gaiotto2013spectral} of two Stokes lines, corresponding to different roots given a choice of trivialization, approaching $d\in D$ produce infinitely many new Stokes lines is the generic behaviour that occurs.   This is these Stokes lines are asymptotic to logarithmic spirals, which in a dense open set of the base $\cA^{\diamonsuit}$ will have different angles at which they intersect the circles where a local coordinate $z$ has fixed absolute value $|z|$.  Hence these lines will intersect infinitely many times, and may produce infinitely many new Stokes lines if 
In \cite{gaiotto2013spectral} it was observed that two Stokes lines approaching a point $d\in D$ can produce infinitely many new Stokes lines.  The considerations in the proof of Proposition \ref{prop: open sets around D} show why this happens.  If two Stokes lines approaching $d\in D$ have different labels then they will each be asymptotic to logarithmic spirals.  When these asymptotic spirals have different angles of intersection with the circles $|z|=r$ they will intersect infinitely many times. 
\end{remark}

\begin{remark}
\label{remark: open set around a }
Using the above proposition one can show that the property of $a\in \cA^{\diamondsuit}_{R}$ defining a basic WKB cameral network such that at all joints there are precisely two incoming lines, is open in the classical topology.  This is because small modifications of the point $a\in \cA^{\diamondsuit}$ will not change the topology of a cameral network on $X^{\circ '}$ with this property, where we have chosen sets $B_{d}$ that satisfy the specified requirements not only for $a$, but for a small open containing $a$. However recall that following Remark \ref{remark:  Existence of generic WKB cameral networks} this open set is only known to be non-empty in some examples.  Furthermore the restriction that there are precisely two incoming lines at each joint is a significant restriction.  It is not clear that we should expect the set of $a\in \cA$ such that the network has these properties to be non-empty in general.
\end{remark}

\subsection{Individual Trajectories}
\label{subsection: denseness of individual trajectories}

In this section we review some well known properties of quadratic differentials following \cite{strebel1984quadratic, bridgeland2015quadratic}. We then use the relation to Stokes curves to deduce that, with respect to the classical topology, there is an open, dense set $\cA^{\diamondsuit}_{SF,0}\subset \cA^{\diamondsuit}$ (Definition \ref{def: saddle free}) for which each Stokes curve has a finite limit set, and hence error condition \ref{item: error: denseness} in Construction \ref{constr:wkb_cameral} does not occur when we start from a point in this set. % Furthermore the end point of all initial Stokes curves are infinite critical points, and hence because all ramification points of $\pi_{\alpha}$ are finite critical points of $\omega_{\alpha}$ as noted after Construction \ref{construction: Quadratic form alpha}, the error condition \ref{item: error: runs into ramification point} does not occur for initial lines in networks associated to a point in this set. 

We start by reviewing the classification of trajectories of a meromorphic quadratic differential $\omega$ (as defined in Example \ref{example: SL2 and quadratic differentials}).

\begin{definition}(See e.g. Section 3.4 in \cite{bridgeland2015quadratic})
\label{definition: types trajectories}
For $\omega$ a meromorphic quadratic differential on a compact Riemann surface $X^{c}$ we define:
\begin{itemize}
    \item The set $Crit_{<\infty}(\omega)$ of \textbf{finite critical points} of $\omega$ as the set of poles of $\omega$ of order 1, and zeros of $\omega$ of any order.
    \item The set $Crit_{\infty}(\omega)$ of \textbf{infinite critical points} of $\omega$ as the set of poles of $\omega$ of order $\geq 2$.
\end{itemize}
We call a trajectory of $\omega$:
\begin{enumerate}
    \item A \textbf{saddle trajectory} if the trajectory tends to a finite critical point in each direction (these two points may be the same).
    \item A \textbf{separating trajectory} if the trajectory tends to a finite critical point in one direction, and an infinite critical point in the other direction.
    \item A \textbf{closed trajectory} if it is a closed curve.
    \item A \textbf{recurrent trajectory} if it is recurrent in either direction.  By a direction being recurrent we mean that there is a parametrization of a ray of the trajectory by a map $T:[0,\infty)\rightarrow X$ such that $T([0,\infty))\subset \lim_{a\rightarrow \infty}\overline{T([a,\infty))}$.
    \item A \textbf{generic trajectory} if the trajectory tends to an infinite critical point in each direction (these two points may be the same).
\end{enumerate}
\end{definition}

Note that a trajectory can not pass through a critical point.

%\todo{why does recurrent trajectory include the case of spiralling in to a limit cycle?? It doesn't but these don't exist if there is at least one pole}

\begin{proposition}(See Sections 9-11 of \cite{strebel1984quadratic} for proof)
\label{proposition : types of trajectories}
Every trajectory of a meromorphic quadratic differential $\omega$ on a compact Riemann surface $X^{c}$ is one of the five types of Definition \ref{definition: types trajectories}.
\end{proposition}

\begin{example}
\label{eg:net_separating_generic_trajectories}
Consider the meromorphic quadratic differential on $\bbP^1$:
\begin{equation}
\label{eq:separating_generic_trajectories}
\omega(x) = \frac{x(x-4)}{(x-1)^2(x-2)^2(x-3)^2} dx \otimes dx .
\end{equation}
This has simple zeros at $0, 4$ and double poles at $1, 2, 3$. 
The trajectory structure of $\omega$ on $\bbP^1\backslash \infty$, is sketched in Figure \ref{fig:separating_generic_trajectories}.

The solid curves are separating trajectories; there are three of them for each simple zero. They subdivide $\bbP^1$ into 3 regions, each of which is filled by generic trajectories having both ends tend to a double pole.

\begin{figure}[h]
    \centering
    \includegraphics{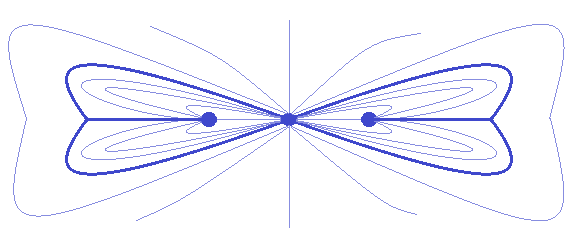}
    \caption{Separating trajectories (solid) and generic trajectories (transparent), for the quadratic differential from Equation \ref{eq:separating_generic_trajectories}. The three blue circles are the double poles.}
    \label{fig:separating_generic_trajectories}
\end{figure}
\end{example}

\begin{definition}
\label{definition: Condition 0}
We denote  by $\cA^{\diamondsuit}_{0}$ the set of $a\in \cA^{\diamondsuit}$ such that the residues (Definition \ref{def:residue_hitchin_point}) satisfy $\alpha(\tilde r_d)\neq 0$ for all $d\in D$ and $\alpha\in \Phi$.
\end{definition}

\begin{example}
\label{example: finite and infinite critical points of omega alpha}
Consider the quadratic differential $\omega_{\alpha}$ on $\tilde{X}^{c}_{\alpha}$ of Construction \ref{construction: Quadratic form alpha}, associated to $a\in \cA^{\diamondsuit}_{0}\subset \cA^{\diamondsuit}$.  Then the preimage $\pi_{\alpha}^{*}(D)$ (using the notation of Equation \ref{equation: factoring pi}) gives the infinite critical points of $\omega$.  By the discussion after Construction \ref{construction: Quadratic form alpha} the branch points of $p_{\alpha}$ and the ramification points of $\omega_{\alpha}$ are the finite critical points of $\omega_{\alpha}$. The union of these points is simply the preimage of the branch locus of $\pi:\tilde{X}^{c}\rightarrow X^{c}$ under the map $\pi_{\alpha}$.
\end{example}

\begin{definition}
\label{def: saddle free}
We call a quadratic differential \textbf{saddle--free} if if has no saddle trajectories.  We call a point $a\in \cA^{\diamondsuit}$ \textbf{saddle--free} if for all roots $\alpha\in \Phi$ the quadratic differential $\omega_{\alpha}$ of Construction \ref{construction: Quadratic form alpha} is saddle-free. 

We denote by $\cA_{SF}^{\diamondsuit}\subset \cA^{\diamondsuit}$ the subset consisting of the points $a\in \cA^{\diamondsuit}$ that are saddle free.

We denote $\cA_{SF,0}^{\diamondsuit}:=\cA_{SF}^{\diamondsuit}\cap \cA_{0}^{\diamondsuit}$.
\end{definition}

\begin{remark}
The WKB cameral network associated to a saddle--free point of the Hitchin base can still have paths between ramification points, provided that these ramification points are labelled by different pairs of roots.
\end{remark}

The following gives a condition for there to be no recurrent or closed trajectories:

\begin{proposition}[Sections 9-11 of \cite{strebel1984quadratic}; Lemma 3.1 of \cite{bridgeland2015quadratic}; Section 6.3 of \cite{gaiotto2013wallHitchin}]
\label{prop: no recurrent or closed}
Let $\omega$ be a saddle-free meromorphic quadratic differential with at least one finite critical points, and at least one infinite critical point.  Then no trajectories of $\omega$ are recurrent or closed.
\end{proposition}

By Proposition \ref{proposition : types of trajectories} this means that all trajectories of $\omega$ are either separating or generic.

%Note that Bridgeland--Smith have buried some additional conditions in their statement -- however do not use them in the proof

\begin{definition}[See \S 2.5 of \cite{bridgeland2015quadratic}\footnote{Our $n$ corresponds to the $n$-tuplet $m=(2,...,2)$ in \emph{loc. cit.}}]
\label{definition: Quad(g,n)}
Let $g,n\in \mathbb{N}$ with $2g-2+n>0$.  

Informally we describe $Quad(g,n)$ as the moduli stack of pairs $(X,\omega)$ of Riemann surfaces $X$ with $n$ marked (unlabelled) points together with a meromorphic quadratic differential $\omega$ with a pole of order $2$ at each marked point, and only simple zeros.

More formally let $\cM_{g,n}$ be the moduli stack of Riemann surfaces of genus $g$ with $n$ marked points.  Let $\cH_{g,n}\rightarrow \cM_{g,n}/S_{n}$ denote the vector bundle with fiber over $(X, \{p_{1},...,p_{n}\})$ being $H^{0}(X, K_{X}^{\otimes 2}(2p_{1}+...+2p_{n}))$.

Define $Quad(g,n)\hookrightarrow \cH_{g,n}$ to be the Zariski open substack consisting of sections with simple zeros located away from $\{p_{1},...,p_{n}\}$, and having at least one zero.  We consider this as a Deligne--Mumford stack in the $C^{\infty}$-setting, or as a differential orbifold.
\end{definition}

The group $S^{1}$ acts on $Quad(g,n)$, with $\theta\in S^{1}$ acting by the mapping $(X,\omega)\mapsto (X, e^{i\theta}\omega)$. 

\begin{proposition}[Lemma 4.11 of \cite{bridgeland2015quadratic}]
\label{prop: SF dense in each S1 orbit}
For $n\geq 1$ the subset of saddle free differentials $Quad(g,n)^{SF}\subset Quad(g,n)$ is open and dense in $Quad(g,n)$ in the classical topology.  Furthermore the intersection of each $S^{1}$-orbit with $Quad(g,n)^{SF}$ is open and dense in the $S^{1}$-orbit.
\end{proposition}

%In Matei's thesis Proposition 4.2.8 reference Lemma 4.11 of Bridgeland--Smith.  Lemma 4.11 of Bridgeland--Smith DOES NOT STATE DENSE!!  However it notes that in each S^{1} orbit the set of points where there are saddle trajectories is countable.  So we have an open complement of a countable set -- hence we are dense. Note that the condition of Bridgeland smith that one m i is greater than or equal to 2 is satisfied. 

See Theorem 1.4 of \cite{aulicino2018cantor} for an analogous statement in the case of higher order poles.

\begin{proposition}
\label{prop: saddle-free points}
The subset $\cA^{\diamondsuit}_{SF,0}$ of saddle-free points of $\cA^{\diamondsuit}_{0}$ (for a line bundle $K_{X^{c}}(D)$, $D$ a reduced divisor with $deg(D)>2$) is open and dense in the classical topology.

Hence we also have that it's intersection with $\cA^{\diamondsuit}_{R}$ is open and dense in $\cA$ with respect to the classical topology.
\end{proposition}

%\todo{There is one major problem in Matei's proof -- Matei considers the S1 action on $\omega$, but not how it changes $\tilde{X}_{\alpha}$.  Note we will also have to prove at least one pole!}

\begin{proof}
Let $\xi_{\alpha}$ be the map $\cA^{\diamondsuit}_{0}\rightarrow Quad(g_{\alpha}, n)$ where $g_{\alpha}$ is the genus of $\tilde{X}_{\alpha,a}^{c}=\tilde{X}_{a}^{c}/s_{\alpha}$ for $a\in \cA^{\diamondsuit}_{0}$ (note that this is independent of $a$), and $n$ the degree of $\pi_{\alpha}^{*}D$ (see Equation \ref{equation: factoring pi} for notation), given by
\[\cA^{\diamondsuit}_{0} \ni a\mapsto (\tilde{X}_{\alpha, a}, \omega_{\alpha,a}),\]
where $\omega_{\alpha, a}$ was defined in Construction \ref{construction: Quadratic form alpha}.

Consider the scaling action of $\bbC^{*}$ on $\ft$, this induces a weighted action of $\bbC^{*}$ on $\ft\sslash W$, which gives an action of $S^{1}\hookrightarrow \bbC^{*}$ on $\cA=\Gamma(X^{c}, \ft_{K_{X^{c}}(D)}\sslash W)$.  This action preserves the isomorphism class of $\tilde{X}_{a}$, and $\tilde{X}_{\alpha,a}$.  This action preserves $\cA^{\diamondsuit}_{0}$, in fact $\cA^{\diamondsuit}_{0}=S^{1}\bullet \cA^{\diamondsuit}_{R}$.  Clearly the map $\xi_{\alpha}$ is equivariant with respect to the given $S^{1}$ action on $\cA^{\diamondsuit}$, and the square of the action of $S^{1}$ on $Quad(g,n)$ specified after Definition \ref{definition: Quad(g,n)}.  That is to say for $e^{i\theta}\in S^{1}$, and $a\in \cA^{\diamondsuit}$ we have 
\[\xi_{\alpha}(e^{i\theta}\cdot a)=e^{2i\theta}\cdot \xi_{\alpha}(a),\]
or equivalently
\[\omega_{\alpha, e^{i\theta}a}=e^{2i\theta}\cdot \omega_{\alpha, a}.\]
Hence by Proposition \ref{prop: SF dense in each S1 orbit} a dense open subset of the image of $\xi_{\alpha}$ consists of saddle--free quadratic differentials.

By definition the locus of saddle free differentials in $\cA^{\diamondsuit}$ is given by \[\cA^{\diamondsuit}_{SF,0}:=\cap_{\alpha\in \Phi}\xi_{\alpha}^{-1}(Quad(g,n)^{SF}).\]
Hence there is an open dense subset $\cA^{\diamondsuit}_{SF,0}\subset \cA^{\diamondsuit}_{0}$ consisting of saddle--free differentials.
\end{proof}

Each Stokes curve in the WKB construction (\ref{constr:wkb_cameral}) forms part of the pullback to $\tilde{X}$ of a trajectory of a quadratic differential on $\tilde{X}_{\alpha}$.  

\begin{lemma}
\label{lemma: omega has finite critical point}
Let $X^{c}$ be a connected compact Riemann surface, and $D$ a divisor with $\text{deg}(D)>2$.  Then for $a\in \cA^{\diamondsuit}_{0}$ and any root $\alpha$ the quadratic form $\omega_{\alpha,a}$ has at least one finite critical point.  
\end{lemma}

\begin{proof}
It is enough to show that the form $\chi_{\alpha}$ of Construction \ref{def:form_root} has a zero.  This is immediate because $\text{deg}(K_{\tilde{X^{c}}}(\pi^{-1}(D)))\geq \text{deg}(K_{X^{c}}(D))>0$.
\end{proof}

It is immediate that $\omega_{\alpha}$ has at least one infinite critical point, because the preimage of $D$ in $\tilde{X}_{\alpha}$ consists of infinite critical points for $a\in \cA^{\diamondsuit}_{0}$.  An alternative approach to showing Lemma \ref{lemma: omega has finite critical point} in the case of a simply laced group is by using that the cameral cover is connected, as is proved under these assumptions in e.g. Proposition 4.5.4 of \cite{ngo2010lemme}.

\begin{corollary}
\label{corollary: behaviour on SF locus}
For $a\in \cA_{SF,0}^{\diamondsuit}$ (which is a dense open set of the Hitchin base for a connected compact Riemann surface $X^{c}$, the line bundle $K_{X^{c}}(D)$, and $\text{deg}(D)>2$) a Stokes curve labelled by $\alpha$ in the WKB construction (\ref{constr:wkb_cameral}) for $a$ has one of the following behaviours::
\begin{itemize}
\item It is a primary Stokes line, starts at the pullback to $\tilde{X}$ of a finite critical point of $\omega_{\alpha}$ (a ramification point), and tends to the pullback to $\tilde{X}^{c}$ of an infinite critical point of $\omega_{\alpha}$ (That is to say a point in $\pi^{-1}(D)$).
\item It starts at a joint and tends to the pullback to $\tilde{X}^{c}$ of an infinite critical point of $\omega_{\alpha}$ (That is to say a point in $\pi^{-1}(D)$).
\item It starts at a joint and tends to the pullback to $\tilde{X}$ of a finite critical point of $\omega_{\alpha}$, that is to say a ramification point of $\pi:\tilde{X}\rightarrow X$.
 %   \item Starts at either a joint or the pullback to $\alpha$ of a finite critical point of $\omega_{\alpha}$ (corresponding to a ramification point of $\tilde{X}$).
  %  \item Ends at the pullback to $\tilde{X}$ of an infinite critical point of $\omega_{\alpha}$ on $\tilde{X}_{\alpha}$.
\end{itemize}
%The limit points of this Stokes curve\footnote{When projected from $\tilde{X}^{\circ}$ to $\tilde{X}$.} are a not necessarily proper subset of the start point and the end point.
It is hence the case that the image in $X^{c}$ of a Stokes curve constructed in the WKB Construction (\ref{constr:wkb_cameral}) for $a\in \cA^{\diamondsuit}_{SF,0}$ has one of the following behaviours:
\begin{itemize}
    \item It is a primary Stokes curve, starts at a branch point (of $\tilde{X}_{a}\rightarrow X$), and tends to a point in $D$.
    \item It starts at a joint and tends to a point in $D$.
    \item It starts at a joint and tends to a branch point.
\end{itemize}
\end{corollary}

\begin{proof}
Each Stokes curve of $a\in \cA^{\diamondsuit}_{SF,0}$ corresponds to a trajectory of the saddle-free quadratic differential $\omega_{\alpha,a}$ for some root $\alpha$.  By Lemma \ref{lemma: omega has finite critical point} $\omega_{\alpha,a}$ has at least one finite critical point. As noted after this lemma, it also has at least one infinite critical point.  Hence by Proposition \ref{proposition : types of trajectories} and Proposition \ref{prop: no recurrent or closed} this trajectory is either separating or generic.  The result follows.
\end{proof}

\section{Non-Abelianization}
\label{sec: flat Donagi Gaitsgory}

In this section we introduce the objects and moduli spaces involved in non-abelianization, and provide the non-abelianization construction.  In Section \ref{subsec: N-shifted moduli spaces} we introduce $N$-shifted weakly $W$-equivariant $T$-local systems on $\tilde{X}^{\circ}$, and show their equivalence to $N$-local systems on $X^{\circ}$ corresponding to the cameral cover $\tilde{X}^{\circ}$.  In Section \ref{subsec: alpha monodromy} we introduce a restriction on monodromy which we call the $S$-monodromy condition.  Finally in Section \ref{subsec: Non abelianiazation via Cameral Networks} we provide the non-abelianization construction (Construction \ref{thm:bulk_map_exists}).

\subsection{$N$-shifting and $T$-local systems on cameral covers}
\label{subsec: N-shifted moduli spaces}

In the case $G=GL(n)$, a line bundle on an unramified $n$-sheeted spectral cover $\overline{X}\rightarrow X$ describes an $N$-bundle over $X$.  
In \cite{donagi2002gerbe}, an analogue of this is given for arbitrary groups:  an $N$-bundle on $X$ is the same as an ``$N$-shifted weakly $W$-equivariant'' $T$-bundle on an unramified cameral cover $\tilde{X}\rightarrow X$.% a version of this for local systems is made precise in Theorem \ref{thm:flat_DG}. 
%The precise links to the spectral description are provided in Section \ref{sec: Spectral Decomposition}. 

In this section we introduce these ideas in the framework of local systems, where the (simple unramified case of the) work \cite{donagi2002gerbe} applies with no significant modifications resulting in Theorem \ref{thm:flat_DG}.  In this setting, for a weakly $W$-equivariant $T$-local system (see Definition \ref{def: weakly W-equivariant}) on $\tilde{X}^{\circ}$ to admit the structure of an $N$-shifted local system on $X^{\circ}$ is a non-trivial restriction, as shown in Example \ref{example: Consequences for T monodromy}. 
  
The setting we work in is that of a map $\tilde{X}^{\circ} \rightarrow 
X^{\circ}$, which is a $W$-bundle (as in convention \ref{convention:blowup}).  It is not a problem that the real blow up is not a scheme, because we will consider \emph{topological objects} on $X^{\circ}$, $\tilde{X}^{\circ}$, namely local systems with respect to the classical topology.

We first introduce the moduli spaces in terms of $N$-local systems.  Informally, we define $\Loc_{N}^{\tilde{X}^{\circ}}(X^{\circ})$ as the moduli stack of $N$-local systems  $\cE$, such that the associated $W$-bundle $\cE/T$ is equipped with an isomorphism to $\tilde{X}^{\circ}$.  More formally:

\begin{definition}
\label{defin: N local systems corresponding to a cameral cover}
We define
\[
\Loc_{N}^{\tilde{X}^{\circ}}(X^{\circ})
:=
\Loc_{N}(X^{\circ})\times_{\Loc_{W}(X^{\circ})}\{\tilde{X}^{\circ}\}.
\]

Recall that $\Loc_{N}(X^{\circ})\cong \Map_{dSt}(X^{\circ}_{B}, BN)$, (we are here considering $X^{\circ}$ as a topological space) and the map to $\Loc_{W}(X^{\circ})$ is given by postcomposing with the map $BN\rightarrow BW$ that comes from quotienting $N$ by $T$.
\end{definition}

There is a clearly a map $Loc_{N}^{\tilde{X}^{\circ}}(X^{\circ})\rightarrow Loc_{T}(X^{\circ})$ coming from mapping an $N$ local system $\cE_{N}\rightarrow X^{\circ}$ to the $T$-local system $\cE_{N}\rightarrow \cE_{N}/T\cong \tilde{X}^{\circ}$.

The condition that a $T$-local system comes from an $N$-local system imposes non-trivial constraints on the $T$-local system, as can be seen in the following example:

\begin{example}[Consequences for $T$-monodromy]
\label{example: Consequences for T monodromy}
Consider the case where $G=SL(2)$, $X=\bbA^{1}_{z}$, and $\tilde{X}=\bbA^{1}_{y}$, with the map $\tilde{X}\rightarrow X$ given by $z\mapsto z^{2}$.

Then $\Loc_{N}^{\tilde{X}^{\circ}}(X^{\circ})\xrightarrow{\cong} (nT)/N$, where $n\in N$ is a representative of the $T$-coset $N \setminus T$, and the map takes a local system to its monodromy around $0$ (specified up to conjugation by elements of $N$).

Let $x_{1}, x_{2}$ be the preimages of $x\in X$  in $\tilde{X}$.   Suppose that $\cL$ is a $T$-local system on $\tilde{X}$, which comes from an $N$-local system.   Then the monodromy of $\cL$ considered as a $T$-local system on $\tilde{X}$ must be $-1$, because for all $g\in Im(nT)$, $g^{2}=-Id$.

%Now consider a $T$-bundle $\cL$ on $\tilde{X}$, equipped with a trivialization at $x_{1}$ and $x_{2}$.  

%Consider the $T^{2}\rtimes W$ monodromy $\rho$ of $\cL$ (considered as a local system on $X$) on going around $0$ in $X$. 

%If we have a map $N\xrightarrow{\gamma} \Aut_{X}(\cL)|_{x} \subset GL(\cL|_{x_{1}}\oplus \cL|_{x_{2}})$.

%Firstly note that the image of $N$ can be described by
%\[im(T)=T_{SL_{\cL|_{x_{1}}\oplus \cL|_{x_{2}}}}\hookrightarrow SL(\cL|_{x_{1}}\oplus \cL|_{x_{2}})\hookrightarrow GL(\cL|_{x_{1}}\oplus \cL|_{x_{2}})\]
%\todo{what do SL, GL mean? I think this should be replaced with appropriate subsets of $\Aut(\cL)$}
%that is, using the weak $W$-equivariance, and the a basis of $\cL|_{x_{1}}\oplus \cL|_{x_{2}}$, corresponding to a basis for $\cL_{x_{1}}$ and a basis of $\cL_{x_{2}}$, we can write
%\[im(T)= \{\left(\begin{matrix} a & 0\\ 0 & a^{-1}\end{matrix}\right)\in End(\cL|_{x_{1}}\oplus \cL|_{x_{2}})\}\].

%Again using the weak $W$-equivariance, and this basis we can describe 

%\[Im(wT)= \{\left(\begin{matrix} 0 & a\\ -a^{-1} & 0\end{matrix}\right)\in End(\cL|_{x_{1}}\oplus \cL|_{x_{2}})\}\].

%Using Proposition \ref{prop: N local type requirement} we see that for any $p\in \cL|_{x}$

%\[\rho(p)= pg \mathrm{for} g\in Im(wT).\]

%This means that the monodromy of $\cL$ considered as a $T$-local system on $\tilde{X}$ must be $-1$, because for all $g\in Im(wT)$, $g^{2}=-Id$.

% Note that it looks like I am using the naive version of imaging monodromy does not depend on base point -- however I am not!

This corresponds to the condition on monodromy of the $\bbG_{m}$ local systems considered on the spectral cover in \cite{gaiotto2013spectral},  and in Proposition \ref{proposition : Reason for alpha monodromy condition}.
\end{example}

We will now develop a description of the moduli of $N$-local systems on $X^{\circ}$ with associated $W$-local system $\tilde{X}^{\circ}$ purely in terms of $T$-local systems on $\tilde{X}^{\circ}$ equipped with additional structure. 
  
\begin{definition}[Modification of a definition in \S 5 of \cite{donagi2002gerbe}]
\label{def: Automorpshisms of T system  on X}  
Let $\tilde{X}^{\circ}\rightarrow 
X^{\circ}$ be as above.  Let $\cL$ be a $T$-local system on $\tilde{X}^{\circ}$.  We define:
\[
\Aut_{X^{\circ}}(\cL) 
= 
\left\{(w,\phi) | w\in W, \phi : w^*\cL \overset{\cong}{\to} \cL \right\}.
\]

This sheafifies over (the classical topology of) $X^{\circ}$, as a locally constant sheaf $\mathcal{A}ut_{X^{\circ}}(\cL)$.
\end{definition}

\begin{definition}[Weakly $W$-Equivariant $T$-local systems on $\tilde{X}^\circ$; cf. Definition 5.7 of \cite{donagi2002gerbe}]
\label{def: weakly W-equivariant}
We call a $T$-local system $\cL\rightarrow \tilde{X}^{\circ}$ \textbf{weakly $W$-equivariant} if the morphism:
\[
\Aut_{X^{\circ}}(\cL) \rightarrow W
\]
is surjective.
\end{definition}

The sheaf $\mathcal{A}ut_{X^{\circ}}(\cL)$ is part of a sequence of sheaves (on $X^{\circ}$ with respect to the classical topology), which is exact if $\cL$ is weakly $W$-equivariant:

\begin{equation}
\label{equation: Automorphism short exact sequence}
1\rightarrow \mathcal{H}om_{lc}(\tilde{X}^{\circ}, T)\rightarrow \mathcal{A}ut_{X^{\circ}}(\cL) \rightarrow W\rightarrow 1,
\end{equation}
where $\mathcal{H}om_{lc}(\tilde{X}^{\circ}, T)$ refers to the sheaf of locally constant morphisms $\tilde{X}^{\circ}\rightarrow T$ (with global sections denoted by $\Hom_{lc}(\tilde{X}^{\circ},T)$).  Defining this sequence used the fact that $T$ is abelian, and hence we can globally identify automorphisms of $\cL$ on $\tilde{X}^{\circ}$ with locally constant maps $\tilde{X}^{\circ}\rightarrow T$.

\begin{remark}
\label{rem:W-equivariant}
The short exact sequence \ref{equation: Automorphism short exact sequence} splits if and only if the weakly $W$-equivariant $T$-local system is $W$-equivariant.
\end{remark}

\begin{definition}[$N$-shifted Weakly $W$-equivariant $T$-local system, cf. \cite{donagi2002gerbe}]
\label{defn: N shifted Weakly W equivariant}
An \textbf{$N$-shifted weakly $W$-equivariant $T$-local system} on $\tilde{X}^{\circ}$ is a $T$-local system $\cL\rightarrow \tilde{X}^{\circ}$ together with a map $\gamma:N\rightarrow \Aut_{X^{\circ}}(\cL)$ of $\infty$-groups in derived stacks (in the non-derived setting we can see this as a map of group schemes) making the following diagram commute:
\begin{equation}
\label{diag:n_twisted}
\begin{tikzcd}
1 \arrow{r}\arrow{d} & T \arrow{r} \arrow{d}{\Delta} & N \arrow{r} \arrow{d}{\gamma} & W \arrow{r} \arrow{d}{Id} & 1\arrow{d} \\
1 \arrow{r} & \Hom_{lc}(\tilde{X}^{\circ},T) \arrow{r} & \Aut_{X^{\circ}}(\cL) \arrow{r} & W \arrow{r} & 1,
\end{tikzcd}
\end{equation}
where $\Delta$ is the diagonal map taking an element $t\in T$ to the corresponding constant map in $\Hom_{lc}(\tilde{X}^{\circ},T)$.

\end{definition}

%To define the moduli stack of $N$-shifted $T$-local systems we first introduce the mapping stack of group stacks:

%We note that when not working in the derived setting one can simply see this as parametrizing pairs of morphisms 

We are not able to directly realize the moduli space of $N$-shifted $T$-local systems as a stack.  Instead we describe it as a functor. 

%\begin{definition}
%\label{definition: mapping stack of group stacks}
%\todo{This definition is massively sloppy -- what do I mean by a group object in an %$\infty$-category -- and what do I mean by this very naive definition of group homomorphism}
%For group objects $G,H\in Grp(dSt_{/Y})$ (for some derived stack $Y$) we define the stack of %group homomorphisms $Hom_{Grp(dSt/Y)}(G,H)$ as the equalizer:
%\[
%Hom_{Grp(dSt/Y)}(G,H) \rightarrow Hom_{dSt/Y}(G,H) \rightrightarrows Hom_{dSt/Y}(G\times_{X} %G, H)
%\]
%where the morphisms $Hom_{dSt/Y}(G,H) \rightrightarrows Hom_{dSt/Y}(G\times_{X} G, H)$ are %given by $\phi\mapsto (\phi\circ m_{G})$, and $\phi \mapsto (m_{H}\circ %(\phi\times_{Id_{X}}\phi))$, for $m_{G}$ and $m_{H}$ the group multiplications of $G$ and $H$ %respectively, and $Hom_{dSt/Y}(G,H)$ the derived mapping stack (see \S 19.1 of %\cite{lurie20spectral} and \cite{toen2008homotopical}).
%\end{definition}

\begin{definition}
\label{defn: Moduli N shifted weakly W equivariant}
We now define  $\Loc_{T}^{N}(\tilde{X}^{\circ})$ to be the moduli functor of $N$-shifted weakly $W$-equivariant $T$-local systems on $\tilde{X}^{\circ}$.

Let $S\in dSch$ be a derived scheme, given an $S$-point $s:S\rightarrow Loc_{T}(\tilde{X}^{\circ})$, denote by $\Hom^{N-Shift}(N, \Aut_{X^{\circ}}(s^{*}\cL))$ the fiber product 
\[\Scale[0.84]{Hom^{inv}_{\infty-Grp(dSt_{/S})}(N, Aut_{X^{\circ}}(s^{*}\cL))\times_{\Hom_{\infty-Grp(dSt_{/S})}^{inv}(T,\Hom_{lc}(\tilde{X}^{\circ}))\times Hom_{\infty-Grp(dSt_{/S})}^{inv}(W,W)} \{(\Delta,Id)\},}\]
where $\cL$ is the universal $T$-local system over $\Loc_{T}^{N}(\tilde{X}^{\circ})\times X^{\circ}$, and the $inv$ superscript refers to invertible morphisms.  We are using $\infty-Grp(\cC)$ to denote the category of infinity groups in a given $\infty$-category $\cC$, as introduced in Section 7.2.2 of \cite{lurie2009higher}.

We define a functor $dSch^{op}\rightarrow \cS$ (where $\cS$ denotes the $\infty$-category of Kan complexes), informally by 
\[\Loc_{T}^{N}(\tilde{X}^{\circ})(S)=\{(s, \rho)|s\in Loc_{T}(\tilde{X}^{\circ})(S), \rho\in \Hom^{N-Shift}(N, \Aut_{X^{\circ}}(s^{*}\cL))\}.\]

More formally there is a functor from the $\infty$-groupoid $Loc_{T}(\tilde{X}^{\circ})(S)$ to $\cS$ taking $s\in Loc_{T}(\tilde{X}^{\circ})(S)$ to the space of morphisms $\Hom^{N-Shift}(N, \Aut_{X^{\circ}}(s^{*}\cL))$.  By the Gr\"{o}thendieck construction this corresponds to a fibration $F(S)\rightarrow Loc_{T}(\tilde{X}^{\circ})(S)$.  We define the functor  $\Loc_{T}^{N}(\tilde{X}^{\circ})$ to be the functor mapping $S\in dSch^{op}$ to the space $F(S)$.  

%More formally, let $\cL\rightarrow \Loc_{T}(\tilde{X}^{\circ})\times (\tilde{X}^{\circ})_{B}$ be the universal object.  We then have an group object $\Aut_{X^{\circ}}(\cL)$ in the category of stacks over $\Loc_{T}(\tilde{X}^{\circ})\times X^\circ_B$.  
%We then define the stack of $N$-shifted weakly $W$-equivariant $T$-local systems on $\tilde{X}^{\circ}$ as the stack of sections $\Gamma(X^\circ_B, \cF)$, where $\cF$ is the fiber product in the diagram:
%\todo{Make this diagram fit!!}
%\[
%\Loc_{T}^{N}(\tilde{X}^{\circ}):=\Gamma(X_{B}, \Hom_{dSt/\Loc_{T}(\tilde{X}^{\circ})\times X_{B}}(N, \Aut_{X^{\circ}}(\cL))
%\times_{\Hom_{dSt/\Loc_{T}(\tilde{X}^{\circ})\times X_{B}}(T,\Hom_{lc}(\tilde{X^{\circ}}, T))\times \Hom_{dSt/\Loc_{T}(\tilde{X}^{\circ})\times X_{B}}(W,W)}
%\{diag, Id\} )
%\]
%\[
%\begin{tikzcd}
%\cF\arrow{d} \arrow{r} 
%& 
%\Hom_{Grp(dSt/\Loc_{T}(\tilde{X}^{\circ})\times X^\circ_{B})}\big(N, \Aut_{X^{\circ}}(\cL)\big)
%\arrow{d}
%\\
%\{\diag, \id\}
%\arrow{r} 
%& 
%\Hom_{Grp(dSt/\Loc_{T}(\tilde{X}^{\circ})\times %X^\circ_{B})}\big(T,\Hom_{lc}(\tilde{X^{\circ}}, T)\big)\times \Hom_{Grp(dSt/\Loc_{T}(\tilde{X}^{\circ})\times X^\circ_{B})}(W,W) 
%\end{tikzcd}
%\]
%Above, $N$, $T$, $W$ refer to the constant stacks over $\Loc_{T}(\tilde{X}^{\circ})\times X_{B}$ with the given group as fiber.
\end{definition}

When not working in the derived setting, one can simply use groups rather than $\infty$-groups.

We will show that $\Loc_{T}^{N}(X^{\circ})$ is in fact a stack by providing an isomorphism to the moduli stack of $N$-local systems associated to $\tilde{X}^{\circ}$ that we introduced earlier.

\begin{theorem}[$N$-local systems as $N$-shifted weakly $W$-equivariant $T$ local systems, minor modification of \cite{donagi2002gerbe}]
\label{thm:flat_DG}
Consider the unramified map $\pi^{\circ} : \tilde X^{\circ} \to X^{\circ}$, where we have removed the branch points and their preimages. Then there is an equivalence of categories between:
\begin{enumerate}
\item $N$-shifted, weakly $W$-equivariant,  $T$-local systems $\cL$ on $\tilde X^\circ$;
\item $N$-local systems $\cE$ on $X^\circ$, equipped with an isomorphism $\cE/T \cong \tilde X^\circ$ of $W$-bundles on $X^\circ$.
\end{enumerate}
This gives an isomorphism of functors;
\begin{equation}
    \label{eq: Map of Moduli spaces}
   \Loc_{N}^{\tilde{X}^{\circ}}(X^{\circ}) \cong \Loc_{T}^{N}(\tilde{X}^{\circ}),
\end{equation}
which in particular realizes $\Loc_{T}^{N}(X^{\circ})$ as a stack.
\end{theorem}

\begin{proof}
%[Proof of Theorem \ref{thm:flat_DG}]
Let $\cE\rightarrow X^{\circ}$ be a $N$-local system on $X^{\circ}$, then $\cE\rightarrow \cE/T\cong \tilde{X}^{\circ}$ is a $T$-local system on $\tilde{X}^{\circ}$ (the $T$-action is specified by $T\hookrightarrow N$). The action of $N$ on $\cE$ gives a map $N\xrightarrow{\gamma} \Aut_{X^{\circ}}(\cE)$.  This fits into a commutative diagram of the form of diagram \ref{diag:n_twisted}.

This has an inverse map, in that if we have a weakly $W$-equivariant, $N$-shifted $T$-local system $\cL$ on $\tilde{X}^{\circ}$, then the map $N\rightarrow \Gamma(X,\Aut_{X}(\cL))$ allows us to consider $Tot(\cL)\rightarrow X^{\circ}$ as an $N$-local system.  Clearly we have an isomorphism of $W$-bundles; $Tot(\cL)/T\cong \tilde{X}^{\circ}$.

Clearly the above constructions work in families and hence give the stated isomorphism of functors when applied to the universal families.
\end{proof}

\begin{remark}
\label{remark: N-twisting at a point}
For $X$ connected providing a map $N\xrightarrow{\gamma} \Aut_{X^{\circ}}(\cL)$ is equivalent to, for some $x\in X$, providing a map 
\[
N\rightarrow \mathcal{A}ut_{X^{\circ}}(\cL)|_{x}^{\pi_{1}(X^{\circ},x)}.
\]
\end{remark}

\subsection{The S-monodromy condition}
\label{subsec: alpha monodromy}%\todo{In Section 4 change notation for the compactification}

We introduce a condition on the monodromy of an $N$-local system, which we call the S-monodromy condition. This corresponds to a natural condition on the monodromy in the case of the spectral cover description, as explained in Proposition \ref{proposition : Reason for alpha monodromy condition} and the following paragraph.

Consider the case of $G=SL(n)$ and a ramified spectral cover $\overline{X}\rightarrow X$.  Let $R_{\rho}\subset \overline{X}$ be the ramification locus, and $P\subset X$ the branch locus.  Unlike the case of a cameral cover, the preimage of the branch locus is not in general the ramification locus.  Because of this, rank one local systems on $\overline{X}\backslash R_{\rho}$ do \emph{not} correspond to arbitrary $N$-local systems on $X\backslash P$, but rather a subset of these satisfying a restriction on the monodromy around the branch locus $P$.  In this section we formalize a version of this restriction for $N$-local systems, calling it the S-monodromy condition.  We show that this corresponds precisely to the above mentioned condition coming from spectral covers in Proposition \ref{proposition : Reason for alpha monodromy condition}.

%We will first define the $S$-monodromy condition for $N$-lsocal system.  Recall that by lemma \ref{lem:commutation_n} we have that for $\cO\in \Phi$ a $W$-orbit of roots, the subscheme $\cup_{\alpha\in \cO}n_{\alpha}T_{\alpha}\subset N$ is preserved by the action of $N$ by conjugation. 

We now define the $S$-monodromy condition for $N$-local systems.  Fix a smooth cameral cover $\tilde{X}\rightarrow X$, and let $\tilde{X}^{\circ}\rightarrow X^{\circ}$ be the associated unramified cover of Convention \ref{convention:blowup}.  For each $p\in P$ let $\Lambda_{p}\subset \Phi$ be the subset of roots such that the $W$-conjugacy class corresponding to the monodromy around $S_{1}^{p}$ of $\tilde{X}^{\circ}$ is given by $\{s_{\alpha}|\alpha\in \Lambda_{p}\}$.

\begin{definition}
\label{def:alpha_monodromy_N}Let $\cE_{N}\rightarrow X^{\circ}$, be an $N$-local system, equipped with an isomorphism of $W$-bundles, $\cE_{N}/T\xrightarrow{\cong}\tilde{X}^{\circ}$.  We say that $\cE_{N}$ satisfies the \textbf{S-monodromy condition} if for each $p\in P$, the monodromy of $\cE_{N}$ around $S_{p}^1$ is contained in $\bigcup_{\alpha\in \Lambda_{p}}n_{\alpha}T_{\alpha}$ (with respect to any choice of trivialization).
\end{definition}

\begin{definition}
\label{definition: Moduli N bundles S monodromy}
We define the moduli space of $N$-local system corresponding to $\tilde{X}^{\circ}$ satisfying the S-monodromy condition by:
\[ \Loc_{N}^{\tilde{X}^{\circ},S}(X^{\circ}):=\Loc_{N}^{\tilde{X}^{\circ}}(X^{\circ})
\times_{\left(\prod_{p\in P}N/N\right)}
\prod_{p\in P}\Bigg(\bigg(\coprod_{\alpha\in \Lambda_{p}}n_{\alpha}T^{\alpha}\bigg)/N\Bigg),\]
where the map $\Loc_{N}^{\tilde{X}^{\circ}}(X^{\circ})\rightarrow \prod_{p\in P}N/N$ comes from restricting to the boundary circles $S^{1}_{p}$, $p\in P$.
\end{definition}

There is not quite such a simple description of the $S$-monodromy condition in terms of $N$-shifted weakly $W$-equivariant $T$-local systems, however we can still describe the moduli space, fairly simply.

\begin{definition}[The S-monodromy condition on $N$-shifted, weakly $W$-equivariant, $T$ local systems]
\label{alpha monodromy T bundle side}
We say that an $N$-shifted, weakly $W$-equivariant, $T$-local system satisfies the \textbf{S-monodromy condition}, if the $N$-local system produced associated to it under the map of Theorem \ref{thm:flat_DG} satisfied the $S$-monodromy condition for $N$-local systems of definition \ref{def:alpha_monodromy_N}.
\end{definition}

Note that if $n\in n_{\alpha}T^{\alpha}$ we have that $n^{2}=I_{\alpha}(-Id)$.  

\begin{lemma}
\label{lemma: connected component of N/N}
The subscheme $\coprod_{\alpha\in \Lambda_{p}}(n_{\alpha}T^{\alpha})/N$ is a connected component of the scheme $\{n\in N|n^{2}=I_{\alpha}(-Id)\text{ for }\alpha\in \Lambda_{p}\}/N=(N\times_{T}\{I_{\alpha}(-Id)\})/N$ (where the map $N\rightarrow T$ is given by $n\mapsto n^{2}$).
\end{lemma}

\begin{proof}
The map $T_{\alpha}\times T_{Ker(\alpha)}\rightarrow T$ has a finite kernel.  We can hence write any element in the $N$ such that $q(N)=s_{\alpha}$, where $q$ is the map $N\xrightarrow{q} W$, can be written in finitely many different ways as $n=n_{\alpha}t_{\alpha}t_{Ker(\alpha)}$ with $t_{\alpha}\in T_{\alpha}$, and $t_{Ker(\alpha)}\in T_{Ker(\alpha)}$.  If $n^{2}=I_{\alpha}(-Id)$ we have that 
\begin{align*}
    (n_{\alpha}t_{\alpha}t_{Ker(\alpha)})^{2}&= (n_{\alpha}t_{\alpha})^{2}t_{Ker(\alpha)}^{2}\\
    &= I_{\alpha}(-Id)t_{Ker(\alpha)}^{2}.
\end{align*}
The result follows, as there are a discrete set of possible values of $t_{Ker(\alpha)}$.
\end{proof}

\begin{lemma}
\label{lemma: N local systems with S monodromy connected}
If $deg(D)>0$ then the stack $Loc_{N}^{\tilde{X}^{\circ}, S}(X^{\circ})$ is connected.
\end{lemma}

\begin{proof}
Pick a base point $x\in X^{\circ}$ and a set of generators $\{\alpha_{i}, \beta_{i}, \gamma_{j}\}_{1\leq i \leq g, 1\leq j \leq \text{deg}(D)+\text{deg}(P)}$ of $\pi_{1}(X^{\circ})$ of $\pi_{1}(X,x)$ as in Equation \ref{equation: Relation fundamental group}, with the property that the generators $\gamma_{i}$, $\text{deg}(D)<i\leq \text{deg}(D)+\text{deg}(P)$ correspond to loops around points $p\in P$.

For each choice of trivialization of $\tilde{X}|_{x}$ and an element $\alpha\in \pi_{1}(X^{\circ},x)$ we gain an element $w_{\alpha}\in W$ corresponding to the monodromy around $\alpha$ with respect to this trivialization.  We let $t_{l}$ parameterize the $W$-orbit of such $(2g+\text{deg}(D)+\text{deg}(P))$-tuples of elements in $W$, and we denote the corresponding tuple by \[(w_{t_{l},\alpha_{1}}, w_{t_{l},\beta{1}},...,w_{t_{l},\alpha_{g}}, w_{t_{l},\beta{g}},w_{t_{l},\gamma{1}},..., w_{t_{l},\gamma_{deg(D)+deg(P)}}). \] 
For each $\gamma_{i}$, $\text{deg}(D)<i\leq \text{deg}(D)+\text{deg}(P)$ pick a root $\alpha_{t_{l},\gamma_{i}}$, such that $w_{t_{l},\gamma_{i}}=s_{\alpha_{t_{l},\gamma_{i}}}$.

For each $w\in W$ pick an element $n_{w}\in q^{-1}(w)\subset N$.  We then have that the truncation of $\Loc_{N}^{\tilde{X}^{\circ},S}(X^{\circ})$ is given by
\[t_{0}(\Loc_{N}^{\tilde{X}^{\circ},S}(X^{\circ}))=\left( \coprod_{t_{l}} Y_{t_{l}}\right)/N,\]
where $Y_{t_{l}}\hookrightarrow T^{2g+\text{deg}(D)}\times T_{SL(2)}^{\text{deg}(P)}$ is cut out by the equation 
\[Id=\prod_{i=1}^{2g}[n_{w_{t_{l},\alpha_{i}}}t_{\alpha_{i}}, n_{w_{t_{l},\beta_{i}}}t_{\beta_{i}}]\prod_{i=1}^{deg(D)}n_{w_{t_{l},\gamma_{i}}}t_{\gamma_{i}}\prod_{i=deg(D)+1}^{deg(D)+deg(P)}n_{\alpha_{t_{l},\gamma_{i}}}I_{\alpha}(t_{\gamma_{i}})\]

It is clear that the map $Y_{t_{l}}\rightarrow T^{2g+\text{deg}(D)-1}\times T_{SL(2)}^{\text{deg}(P)}$ given by forgetting the element $\gamma_{\text{deg}(D)}$ is an isomorphism.  Hence $Y_{t_{l}}$ is connected, hence $Loc_{N}^{\tilde{X}^{\circ}}(X^{\circ})$ is connected.

\end{proof}

We can thus see the moduli space of $N$-shifted weakly $W$-equivariant $T$-local systems satisfying the $S$-monodromy condition as a connected component of the moduli space of $N$-shifted weakly $W$-equivariant $T$-local systems that have monodromy $I_{\alpha}(-Id)$ around the ramification points where the root $\alpha$ vanishes.

\begin{definition}
\label{def: locT-1}
We define the stack  
\[\Loc_{T}^{N, -1}(\tilde{X}^{\circ}):=\Loc_{T}^{N}(\tilde{X}^{\circ})\times_{\prod_{r\in R}T/T} \prod_{r\in R} \{ I_{\alpha(r)}(-Id)\}/T,\]
where the map to $\prod_{r\in R}T/T$ comes from restricting the $T$-local system to $S^{1}_{r}$, and $\alpha(r)$ is either of the two roots such that $s_{\alpha(r)}$ fixes the point $r$.
\end{definition}

\begin{theorem}
\label{theorem: connected component of LocT-1}
For $\deg(D)>0$ the isomorphism of stacks $Loc_{N}^{\tilde{X}^{\circ}}(X^{\circ})\xrightarrow{\cong} Loc_{T}^{N}(\tilde{X}^{\circ})$ restricted to $Loc_{N}^{\tilde{X}^{\circ},S}(X^{\circ})$ induces an isomorphism of stacks between $Loc_{N}^{\tilde{X}^{\circ},S}(X^{\circ})$ and a connected component of $Loc_{T}^{N,-1}(\tilde{X}^{\circ})$.
\end{theorem}

\begin{proof}
This follows from Lemmas \ref{lemma: connected component of N/N}, \ref{lemma: N local systems with S monodromy connected} and Theorem \ref{thm:flat_DG}.
\end{proof}

\begin{definition}
\label{moduli alpha monodromy T bundle side}
For $\deg(D)>0$ we define $\Loc_{T}^{N,S}(\tilde{X}^{\circ})$ to be the connected component of $\Loc_{T}^{N, -1}(\tilde{X}^{\circ})$ that is isomorphic to $Loc_{N}^{\tilde{X}^{\circ},S}(X^{\circ})$.
 \end{definition}

It is now tautological that for $\deg(D)>0$ there is a commutative diagram:
\begin{equation}
\label{equation: commutative diagram S monodromy for both T bundles and N bundles}
\begin{tikzcd}
\Loc_{T}^{N, S}(\tilde{X}^{\circ}) \arrow{r} \arrow{d}{\cong} &  \Loc_{T}^{N}(\tilde{X}^{\circ}) \arrow{d}{\cong}\\
\Loc_{N}^{\tilde{X}, S}(X^{\circ}) \arrow{r} & \Loc^{\tilde{X}}_{N}(X^{\circ})
\end{tikzcd}
\end{equation}
with both vertical arrows being isomorphisms.

\subsection{The Non-abelianization Map}
\label{subsec: Non abelianiazation via Cameral Networks}

The main result of this paper is:

\begin{construction}
\label{thm:bulk_map_exists}
%The data of a basic WKB spectral network, which is associated to a point in the Hitchin base satisfying Condition R (definition \ref{definition: Condition R}), determines a morphism of algebraic stacks:
The data of a basic abstract spectral network determines a morphism of algebraic stacks:
\[
\begin{tikzcd}
\Loc_N^{\tilde X^\circ,S}(X^\circ)
\arrow{r}{\nonab} &
\Loc_G(X).
\end{tikzcd}
\]
\end{construction}

Before we provide details for this construction we will provide an informal overview.

Firstly we replace $X^\circ$ with the homotopy equivalent subset $X^{\circ '}$ (see Definition \ref{def:wkb_cameral}), in order to work with finitely many Stokes curves. We start with an $N$-local system on $X^{\circ '}$ (corresponding to a cameral cover $\tilde{X}^{\circ '}$), which we induce to a $G$-local system on $X^{\circ '}$.  We will ``cut and reglue" this $G$-local system (in the sense of Example \ref{ex: modifying local systems on S1}) along $\cW\backslash \cJ$ to provide a new $G$-local system on $X^{\circ '}\backslash \cJ$.  We will show that this extends to a $G$-local system on $X$.

In order to do this we need to provide locally constant automorphisms of the $G$-local system along each connected component of $\cW\backslash \cJ$.  This is done using the filtration on the connected components of $\cW\backslash \cJ$ provided at the start of Construction \ref{construction of swkb}. For the components in $F_{0}(\cW\backslash \cJ)$ these are provided by in Construction \ref{construction of swkb}, Lemma \ref{lemma: monodromy to Stokes factors}, and Lemma \ref{lemma: N equivariant map to Stokes factors initial}.  For components not in $F_{0}(\cW\backslash \cJ)$ these automorphisms are provided by Lemma \ref{lem:assignment_stokes_newlines_equivariant}.

The referenced lemmas show that the resultant $G$-local system extends to $X$, because the monodromy around each joint $J\in \cJ$, and the monodromy around each $S^{1}_{p}$ ($p\in P$) is the identity.

We now move on to the formal construction of the non-abelianization morphism.
Let $\cW$ be a basic abstract cameral network as in Definition \ref{defin:basic_abstract}.
%which is associated to a point in the Hitchin base satisfying Condition R (definition \ref{definition: Condition R}).  
%By proposition \ref{prop: open sets around D} there exists an open set $B=\cup_{d\in D}B_{d}\subset X$ (we also consider it as a subset of $X^{\circ}$), that homotopy retracts onto $D$, such that we can work with $X^{\circ '}:=X^{\circ}\backslash B$.  
%By assumption, $\cW$ comes with the data of a subset $X^{\circ '} \subset X^\circ$, on which the (finitely many) Stokes curves live. Let 
%$\cW':=\cW\cap X^{\circ '}$, $X':=X\backslash B$, $\cJ':=\cJ\cap X'$, and 
Let $\tilde{X}^{\circ '}:=\tilde{X}^{\circ}\times_{X^{\circ}}X^{\circ '}$. Let $\cW\backslash \cJ=\cup_{i\in I} c_{i}$ be the decomposition of $\cW\backslash \cJ$ into connected components.
%\todo{To be clear: is the convention that we use $\ell$ for Stokes curves, and $l$ for Stokes curve segments? Add this to notation section.}

\begin{remark}
\label{rem:passing_to_'_iso}
Because $X^{\circ '}$ and $X^\circ$ are assumed to be homotopy equivalent, there are isomorphisms $\Loc_{G}(X^{\circ '})\cong \Loc_{G}(X^{\circ})$, and $\Loc_{N}^{\tilde{X}^{\circ '}, S}(X^{\circ '})\cong \Loc_{N}^{\tilde{X}^{\circ},S}(X^{\circ})$.
\end{remark}

\begin{definition}
\label{defn: stack G bundles with automorphisms}
Informally, we define the stack:
\[
\Aut_{\cW, G}(X^{\circ '})
\]
to be the moduli stack parametrizing $G$-local systems $\cE$ on $X^{\circ '}$ together with a section $S_{i}$ of the locally constant sheaf $\Aut(\cE)|_{c_{i}}$ for each $i\in I$.

More rigorously, let $un: \Loc_{G}(X^{\circ '})\times X^{\circ '}_{B}\rightarrow BG$ be the map corresponding to the universal local system. Let $A=\Hom_{dSt/BG}(\pt, \pt)$ be the automorphisms of the universal $G$-bundle on $BG$.  We then define $\Aut_{\cW, G}(X^{\circ '})$ as the fiber product in the diagram:
\[
\begin{tikzcd}
\Aut_{\cW,G}(X^{\circ'}) \arrow{d}\arrow{r} & \Gamma
\Big(\coprod_{i\in I}(c_{i})_{B}\times \Hom\big(\coprod_{i\in I}(c_{i})_{B}, BG\big), un^{*}A\Big)\arrow{d}\\
\Hom\big(X^{\circ '}_{B}, BG\big)
\arrow{r}
&
\Hom\big(\coprod_{i\in I}(c_{i})_{B}, BG\big)
\end{tikzcd}
\]
where: $\Gamma\left(\coprod_{i\in I}(c_{i})_{B}\times \Hom\big(\coprod_{i\in I}(c_{i})_{B}, BG\big), un^{*}A\right)$ is the moduli stack of sections of $un^{*}A\rightarrow \coprod_{i\in I}(c_{i})_{B}\times \Hom\big(\coprod_{i\in I}(c_{i})_{B}, BG\big)$.
\end{definition}

There is a forgetful map $\Aut_{\cW, G}(X^{\circ'})\rightarrow \Loc_{G}(X^{\circ'})$.

\begin{definition}
\label{def: Regluing stack with N structure}
We define the stack:
\[
\Aut_{\cW, N, G}^{\tilde{X}^{\circ'}}(X^{\circ'}):=\Loc_{N}^{\tilde{X}^{\circ'}}(X^{\circ'})\times_{\Loc_{G}(X^{\circ'})}\Aut_{\cW, G}(X^{\circ'}).
\]
where the map $\Loc_{N}^{\tilde{X}^{\circ'}}(X^{\circ'})\rightarrow\Loc_{G}(X^{\circ'})$ induces a $G$-local system on $X^{\circ'}$ from an $N$-local system on $X^{\circ'}$.
\end{definition}

Let $X^{\circ'}\backslash \cW=\cup_{k\in \cK} V_{k}$ be the decomposition of $X^{\circ'}\backslash \cW$ into connected components. Let $V_{k}'=\overline{V}_k\backslash (\overline{V}_k\cap \cJ)$ where $\cJ$ is the set of joints of the network as defined in Definition \ref{defin:basic_abstract}, and $\overline{V}_k$ is the closure of $V_k\hookrightarrow X^{\circ'}$.  Let $p:\coprod_{k\in \cK}V_{k}'\rightarrow X^{\circ '}$ be the inclusions.  It is essential that we here use the disjoint union.

For each $i\in I$ there are two distinct topological spaces $V_{k_{1}(i)}'$ and $V_{k_{2}(i)}'$, such that $c_{i}\backslash (c_{i}\cap \cJ)\in p(V_{k_{1}(i)}), p(V_{k_{2}(i)})$.  We define $c_{i,1}^{'}$ and $c_{i,2}^{'}$ to be the lifts of $c_{i}\backslash (c_{i}\cap \cJ)$ to $V_{k_{1}(i)}'$ and $V_{k_{2}(i)}'$ respectively.  The orientations on $X$ and $c_{i}$ give a preferred normal direction to $c_{i}$. We choose $V_{k_{1}(i)}'$ and $V_{k_{2}(i)}'$ such that this normal direction is directed from $V_{k_{1}(i)}'$ to $V_{k_{2}(i)}'$.

\begin{definition}[Regluing Map]
\label{defn: regluing map}
We define the regluing map 
\begin{equation}
\label{eq: Reglue map}
\reglue:  
\Aut_{\cW, G}(X^{\circ '})
\rightarrow 
\Loc_{G}(X^{\circ'} \backslash \cJ)
\end{equation}
as follows. Over $\Aut_{\cW,G}(X^{\circ '})$ we can construct the local system $\cE'$ as:
\[
\cE':=(p^{*}\cE)/\left\{p^{*}\cE|_{c_{i,1}^{'}}\xrightarrow{S_{i}} p^{*}\cE|_{c_{i,2}^{'}}\right\}.
\]
This gives a local system on $\left(\coprod_{k} V_{k}'\right)/\left\{c_{i,1}^{'} \sim c_{i,2}^{'}\right\}\cong X^{\circ}\backslash \cJ$, and hence by the universal property defines the map in Equation \ref{eq: Reglue map}.

Informally, we can see the regluing map as regluing a local system $\cE$ by the data of the automorphisms $S_{i}$ over $c_{i}$.
\end{definition}

\begin{proof}[Details for Construction \ref{thm:bulk_map_exists}]
We will provide a map $\Loc_{N}^{\tilde{X}^{\circ '}, S}(X^{\circ '})\to \Loc_{G}(X)$, which by the isomorphisms from Remark \ref{rem:passing_to_'_iso} is equivalent to a map 
\[
\begin{tikzcd}
\Loc_N^{\tilde X^\circ,S}(X^\circ)
\arrow{r}{\nonab} &
\Loc_G(X).
\end{tikzcd}
\]

Inducing $N$-local systems to $G$-local systems gives a map $ind: \Loc_N^{\tilde X^{\circ '},S}(X^{\circ '})\rightarrow \Loc_{G}(X^{\circ '})$.
Given a basic abstract cameral network, Construction \ref{construction of swkb} below gives a lift  $s_{WKB}: \Loc_N^{\tilde X^{\circ '},S}(X^{\circ '})\rightarrow \Aut_{\cW, G}(X^{\circ '})$ of $ind$, so that the following diagram commutes:
\[
\begin{tikzcd}
 & \Aut_{\cW, G}(X^{\circ '})\arrow{d}{forget}\\
 \Loc_N^{\tilde X^{\circ '},S}(X^{\circ '}) \arrow{r}{ind} \arrow{ru}{s_{WKB}} & \Loc_{G}(X^{\circ '}).
\end{tikzcd}
\]

%\todo{No: the below is NOT good enough!!!}
%We first describe $s_{WKB}$ on Stokes lines emerging from branch points, as coming from the composition of the maps 
%\begin{align*}\Loc_N^{\tilde X^{\circ '},\alpha}(X^{\circ '}) &\rightarrow \left(\coprod_{\alpha}n_{\alpha}T^{\alpha}\right)/N\\
%&\rightarrow  \left(\coprod U_{\alpha}\right)/N\times \left(\coprod U_{-\alpha}\right)/N \times \left(\coprod U_{\alpha}\right)/N\\
%&\rightarrow Aut(\cE_{N}^{un}\times_{N}G)|_{l_{1}}\times Aut(\cE_{N}^{un}\times_{N}G)|_{l_{2}}\times Aut(\cE_{N}^{un}\times_{N}G)|_{l_{3}}.\end{align*}

%Here the first map is restricting to the boundary of $\tilde{X}^{\circ '}$ near a branch point $p\in P$, here we note that $\left(\coprod_{\alpha}n_{\alpha}T^{\alpha}\right)/N$

%\todo{I am now explaining $s_{WKB}$ more carefully below!}

Lemmas \ref{lemma: monodromy to Stokes factors} and \ref{lem:assignment_stokes_newlines_equivariant} below show that the map $\reglue \circ s_{WKB}$ factors as shown in the following commutative diagram, which defines the map $\nonab$:
\[
\begin{tikzcd}
\Aut_{\cW, G}(X^{\circ '}) \arrow{r}{\reglue} & \Loc_{G}(X^{\circ '}\backslash \cJ)
\\
\Loc_N^{\tilde X^{\circ '},S}(X^{\circ '})\arrow{u}{s_{WKB}} &
\\
\Loc_N^{\tilde X^{\circ},S}(X^{\circ}) \arrow{u}{\cong} \arrow[dashed]{r}{\nonab} & \Loc_{G}(X) \arrow[hookrightarrow]{uu}
\end{tikzcd}
\]
\end{proof}

%\todo{Excise some large fragment of below!!!!}
%Let us first give an outline of the construction.
%Given $\cE$, an object of $\Loc_N^{\tilde X^\circ,\alpha}(X^\circ)$, we first form the induced $G$-bundle $\cE \times_N G$. We cut $X^{\circ}$ along every line $\ell \in \cW$, and call the resulting connected components $\{X_i\}_{i\in I}$. The index set $I$ may be infinite, but as long as the lines in $\cW$ are nowhere dense on $X$, this is not a problem for our construction. Let $\cE_i$ denote the restriction of $\cE \times_N G$ to $X_i$. In particular, we have:
%\begin{equation}
%\label{eq:bundle_cut}
%\cE \times_N G = \bigsqcup_{i\in I} \cE_i / \big\{\cE_i|_{\ell} \sim \cE_j|_{\ell} \text{ whenever } \ell \text{ separates } X_i, X_j \big\}.
%\end{equation}

%We want to modify the equivalence relation used in \ref{eq:bundle_cut}. For each $\ell \in \cW$, we construct a flat section $T_\ell$ of $\Aut(\cE \times_N G)|_{\ell}$ -- note that $\ell$ is contractible, so there are no issues about the existence of this flat section. Then we define:
%\begin{equation}
%\label{eq:bundle_paste}
%\nonab(\cE)  = \bigsqcup_{i\in I} \cE_i / \big\{\cE_i|_{\ell} \sim T_{\ell}(\cE_j|_{\ell}) \text{ whenever } \ell \text{ separates } X_i, X_j, \text{ and } X_i \text{ is to the left of } X_j \big\}.
%\end{equation}

%\todo{Rephrase the above?  Yes in particular we don't have $X_{i}'s =X\backslash \cW$, because we need the lines of $\cW$ to be in the $X_{i}$.}
%\todo{How about just defining $X_i$ as the closure of the connected components?}

%\todo{rephrase the sketch below in terms of unipotent groups}

We first construct $s_{WKB}$:

\begin{construction}[Construction of $s_{WKB}$]
\label{construction of swkb}
Recall the acyclicity assumption in Definition \ref{defin:basic_abstract}, which provides a total order on the set of joints $\cJ$, as proved in Lemma \ref{lem:topological_ordering}. We can thus define an increasing filtration $F_\bullet (\cW \setminus \cJ)$, such that:
\begin{itemize}
    \item $F_0 (\cW)$ contains those Stokes curve segments starting from the boundary circles (which correspond to the branch points).
    \item $F_i (\cW)$ contains those Stokes curve segments starting at joints with index $\leq i$. 
\end{itemize}

We first construct a map 
\[
\Loc_{N}^{\tilde{X}^{\circ'},S}(X^{\circ'})
\overset{s_{WKB}^0}{\rightarrow} 
\Aut^{\tilde X^{\circ '}}_{F_0(\cW),N, G}(X^{\circ '}).
\]

Let us first consider the case of primary lines.  Consider the case of a boundary circle $S^1_{p}$.  We have identifications $\Aut_{F_0(\cW_{1})\cap S^1_{p},N, G}(S^1_{p})\cong [(N\times G^{3})/N]$, and $\Aut_{F_0(\cW_{1})\cap S^1_{p}, G}(S^1_{p})\cong [G^{4}/G]$, where $N$ and $G$ are acting by simultaneous conjugation. This is because if we fix a trivialization of the $G$ bundle at a point, we can identify the space with $G^{4}$, where three copies of $G$ come from the specified automorphisms\footnote{Where we identify these with elements of $G$ using the parallel transport of the trivialization in the direction specified by the orientation of $S^{1}_{p}$.}, and one comes from the monodromy around $S^1_{p}$.

For $\Lambda \subset \Phi$ an orbit of the $W$ action on roots, we have $N$-equivariant maps (from Lemma \ref{lemma: N equivariant map to Stokes factors initial}): 
\[
\left(\coprod_{\alpha \in \Lambda}n_{\alpha}T_{\alpha}\right) 
\overset{S_{\pm}}{\longrightarrow} 
\coprod_{\alpha \in \Lambda} U_{\pm \alpha}.
\]
This induces a map:
\[
\left[ \left(\coprod_{\alpha\in \Lambda}n_{\alpha}T_{\alpha}\right) \middle/N\right]
\rightarrow 
\left[\left(\coprod_{\alpha\in \Lambda}n_{\alpha}T_{\alpha}\times U_{\alpha}\times U_{ -\alpha}\times U_{\alpha}\right)\middle/N\right]
\rightarrow
[N\times G^{3}/N].
%= 
%\Aut_{F_{1}(\cW)\cap S^1_{p},N G}(S^1_{p}).
\]
Recall that $[N\times G^{3}/N]\cong \Aut_{F_{1}(\cW)\cap S^1_{p},N G}(S^1_{p})$.  The way that branches of the cameral cover correspond to choices of roots in $\Lambda$ is explained in Section \ref{subsubsection: Stokes factors}.

Doing this on each boundary circle, and inducing from an $N$-local system to a $G$-local system on $X$ provides the desired map:
\begin{equation}
\label{eq:first_swkb}
\Loc_N^{\tilde X^{\circ '},\alpha}(X^{\circ '})
\overset{s^0_{WKB}}{\rightarrow} 
\Aut^{\tilde X^{\circ '}}_{F_0(\cW),N, G}(X^{\circ'}).
\end{equation}

We now lift this to a map $s_{WKB}$.
We work inductively, by assuming we are given a map $s_{WKB}^{i}$ as in the below diagram, such that the Stokes factors $S_{c_{i}}$ are in $U_{\alpha}$ when we trivialize the $N$-local system in such a way that the line segment $c_{i}$ is labelled by $\alpha$. We can then construct a map $s_{WKB}^{i+1}$ such that the diagram below commutes.
\[
\begin{tikzcd}
 & 
 \Aut^{\tilde X^{\circ '}}_{F_{i+1}(\cW),N, G}(X^{\circ'}) 
 \arrow{d}{forget}
 \\
\Loc_N^{\tilde X^{\circ '},S}(X^{\circ '})
\arrow{ru}{s_{WKB}^{i+1}}
\arrow[swap]{r}{s_{WKB}^{i}}
& 
\Aut^{\tilde X^{\circ '}}_{F_{i}(\cW), N, G}(X^{\circ'})
\end{tikzcd}
\]
The base case $s_{WKB}^0$ is provided by the map in Equation \ref{eq:first_swkb}. The inductive step of obtaining $s_{WKB}^{i+1}$ from $s_{WKB}^i$ works by picking a trivialization of the $N$-local system at $\cJ$, and applying the $N$-equivariant map of Equation \ref{equation: New Stokes lines} in Lemma \ref{lem:assignment_stokes_newlines_equivariant} . This provides the Stokes factors $S_{l}$, for all Stokes curves generated at the $i+1^{\text{th}}$ joint, and hence the lift $s_{WKB}^{i+1}$.  

Because the number of curves in $\cW$ is finite, there is some integer $N$, such that $F_{N}(\cW)=\cW$.  We then define $s_{WKB}$ as the composition:
\[
\Loc_N^{\tilde X^{\circ '},S}(X^{\circ '})
\xrightarrow{s_{WKB}^{N}}
\Aut^{\tilde X^{\circ '}}_{F_{N}(\cW),N, G}(X^{\circ'})
\xrightarrow{forget}
\Aut^{\tilde X^{\circ '}}_{\cW, G}(X^{\circ'}).
\]
\end{construction}

\subsubsection{Stokes factors for primary curves}
\label{subsubsection: Stokes factors}

We now state and provide some lemmas used in the previous construction.

%\todo{Rephrase the construction in terms of map of stacks that gives regluing data - work in progress. For now naive version, using local trivialization of cameral cover.}

Fix a branch point $p\in P$, and let $x_p$ be the marked point on the boundary circle of $X^\circ$ corresponding to $p$. For the following construction, fix a trivialization $\phi_p : \tilde X^\circ|_{x_p} \cong W$. Under this identification, the monodromy of $\cE$ around $S^1_p$ takes values in:
\[
[n_\alpha \cdot T^{\alpha} / T],
\]
for some $\alpha$ where we have chosen a Chevalley basis, and we recall $n_{\alpha}$ was introduced in Equation \ref{eq:preferred_lift}, and $T_{\alpha}$ was introduced immediately after Equation \ref{eq:decomp_lie_algebra}.  In fact there is some ambiguity, this statement is true for a pair of roots $\pm \alpha$.  Consider the path around $S^{1}_{p}$ starting at $x_{p}$ and in the direction specified by the orientation of $X$.  The lift of this path starting at $\phi_{p}^{-1}(1_{W})$ ($1_{W}$ denoting the identity of $W$) crosses three Stokes curves.  We choose $\alpha$ such that the labels of these three Stokes curves are (in the order crossed) $\alpha$, $-\alpha$ and $\alpha$.

\begin{definition}
\label{def: monodromy to Stokes factors}
We define the morphisms:
\begin{align}
\begin{split}
n_\alpha \cdot T^\alpha 
&\overset{S_{\pm \alpha}}{\longrightarrow}
U_{\pm \alpha}
\\
n_\alpha t_{\alpha} 
&\longmapsto
(u_{\pm \alpha} ) ,
\end{split}
\end{align}
where:
\[
u_{\pm \alpha} 
= 
\Ad_{t_{\alpha}^{-1/2}}\exp( - e_{\pm \alpha}).
\]
Here $e_{\pm \alpha}$ are Chevalley basis elements, and the choice of square root $t_{\alpha}^{1/2}$ does not affect the adjoint action $\Ad_{t_{\alpha}^{-1/2}}$. 
\end{definition}

Remark \ref{rem:chevalley_basis_independence} explains that the morphisms in Definition \ref{def: monodromy to Stokes factors} do not depend on the choice of Chevalley basis.

Informally, $S_{\pm \alpha}$ send the monodromy of an $N$-local system $\cE$ to the Stokes factors that will be used in regluing. The factors $S_{\pm \alpha}$ have the following properties:

\begin{lemma}
\label{lemma: monodromy to Stokes factors}
For all $t_\alpha \in T^\alpha$, and $u_{\pm \alpha}$ as in Definition \ref{def: monodromy to Stokes factors}, the following relation holds in $G$:
\begin{equation}
\label{eq:monodromy_cancels}
n_\alpha t_\alpha u_\alpha u_{-\alpha} u_\alpha = \id .
\end{equation}

Moreover, $S_{\pm \alpha}$ are equivariant with respect to the adjoint action of $T$. As such, they descend to morphisms of stacks:
\[
[n_\alpha \cdot T^\alpha / T]
\longrightarrow
[U_{\pm \alpha} / T] ,
\]
where the action of $T$ is by conjugation.
\end{lemma}

\begin{proof}
In light of Lemma \ref{lem:commutation_t_alpha} applied to a square root $t_\alpha^{1/2}$:
\[
n_\alpha t_\alpha^{1/2} n_\alpha^{-1} = t_\alpha^{-1/2},
\]
the relation $n_\alpha t_\alpha  = \Ad_{t_\alpha^{-1/2}} n_\alpha$ holds. Then we can write the maps $S_{\pm \alpha}$ in the manifestly $T^\alpha$-equivariant form:
\begin{equation}
\label{eq:map_to_stokes_t_equivariant}
\Ad_{t_{\alpha}^{-1/2}} n_\alpha 
\mapsto
\Ad_{t_{\alpha}^{-1/2}}\exp( - e_{\pm \alpha}).
\end{equation}
Due to Lemma \ref{lem:t_ker_comm}, the adjoint action of $T^{\Ker \alpha}$ on $e_{\pm \alpha}$ and $n_\alpha$ is trivial, so that $S_{\pm \alpha}$ are actually $T$-equivariant.

It remains to demonstrate Equation \ref{eq:monodromy_cancels} holds:
\begin{align*}
n_\alpha t_\alpha u_\alpha u_{-\alpha} u_\alpha
=&
\Ad_{t_{\alpha}^{-1/2}} n_\alpha
\Ad_{t_{\alpha}^{-1/2}}\exp( - e_{ \alpha})
\Ad_{t_{\alpha}^{-1/2}}\exp( - e_{- \alpha})
\Ad_{t_{\alpha}^{-1/2}}\exp( - e_{ \alpha})
\\
=&
\Ad_{t_{\alpha}^{-1/2}} \Big(
n_\alpha \exp(-e_{\alpha}) \exp(-e_{-\alpha}) \exp(-e_{\alpha})
\Big)
\\
=&
\Ad_{t_\alpha^{-1/2}} \big(n_\alpha n_\alpha^{-1}\big)
\\
=&
\id
\end{align*}
In the calculation above, the third equality follows from Lemma \ref{lem:triple_product}.
\end{proof}

\begin{remark}
\label{rem:chevalley_basis_independence}
Definition \ref{def: monodromy to Stokes factors} uses a choice of Chevalley basis. Due to the $T$-equivariance proved in Lemma \ref{lemma: monodromy to Stokes factors}, different choices of Chevalley basis produce the same map.% However, this statement would be false if we allowed an arbitrary (vector space) basis compatible with the decomposition $\fg = \ft \oplus \bigoplus_{\alpha \in \Phi} \fg_\alpha$.

Concretely, since root spaces are 1-dimensional, and since all scalar multiples of $e_{\pm \alpha}$ can be obtained from the adjoint action of $T_\alpha$ on $e_{\pm \alpha}$ (see Lemma \ref{lem:adj_scalar}), in a new basis we have:
\begin{align*}
    e'_\alpha &= \ad_{t_1}e_\alpha 
    \\
    e'_{-\alpha} &= \ad_{t_2}e_{-\alpha}
\end{align*}
for some $t_1, t_2 \in T_\alpha$. Now, if we assume that $\ad_{t_1}e_{-\alpha} = \ad_{t_2}e_{-\alpha}$, then the relation:
\[
n_\alpha = \exp(e_\alpha) \exp(e_{-\alpha}) \exp(e_\alpha)
\]
shows that $n'_\alpha = \ad_{t_1}(n_\alpha)$. Thus, $T$-equivariance implies that the definition of $S_{\pm \alpha}$ using the new basis:
\[
\Ad_t n'_\alpha \mapsto \Ad_t \exp(e'_{\pm \alpha})
\]
agrees with the definition of $S_{\pm \alpha}$.

A Chevalley basis is constrained to satisfy:
\[
[e_\alpha, e_{-\alpha}] = -h_\alpha ,
\]
where $h_\alpha$ is fixed. Therefore, in two distinct Chevalley bases, we must have $[e_\alpha, e_{-\alpha}] = [e'_\alpha, e'_{-\alpha}]$. This forces $\ad_{t_1}e_{-\alpha} = \ad_{t_2}e_{-\alpha}$, justifying the assumption we had made in the previous paragraph.

%Note that the requirement we use a Chevalley basis is essential, for an arbitrary basis of tthere would be no reason for $\ad_{t_1}$ and $\ad_{t_2}$ to be related, and different bases for $\fg$ would give different maps in Definition \ref{def: monodromy to Stokes factors}.
\end{remark}

Next, we remove the trivialization $\phi_p : \tilde X^\circ|_{x_p} \cong W$. The monodromy of $\cE$ around $S^1_p$ takes values in:
\[
\left[\left(
\coprod_{\alpha} n_\alpha T^\alpha
\right)
\middle/ N \right],
\]
where the disjoint union is over roots in the $W$-orbit. In this setting, we can refine Lemma \ref{lemma: monodromy to Stokes factors} as follows.

\begin{lemma}
\label{lemma: N equivariant map to Stokes factors initial}
Let $\Lambda$ denote an orbit of the $W$ action on $\Phi$. The maps $S_{\pm \alpha}$ of definition \ref{def: monodromy to Stokes factors}, for $\alpha \in \Lambda$, give rise to two $N$-equivariant maps:
\begin{equation}
    \label{equation: N-equivariant map to Stokes factors initial}
    S_{\pm} : \left(\coprod_{\alpha \in \Lambda} n_{\alpha}T^{\alpha}\right)
    \rightarrow 
    \coprod_{\alpha \in \Lambda} U_{\pm \alpha}.
\end{equation}
As such, there are induced maps of stacks:
\[
\left[\left(\coprod_{\alpha \in \Lambda} n_{\alpha}T^{\alpha}\right)\middle/N\right]
\rightarrow 
\left[\left(\coprod_{\alpha \in \Lambda} U_{\alpha}\right)\middle/N \right].
\]
\end{lemma}

\begin{proof}
In light of Lemma \ref{lem:commutation_t_alpha}, for every $\alpha \in \Phi$ and $t_\alpha \in T^\alpha$, we have $n_\alpha t_\alpha = \Ad_{t_\alpha^{-1/2}} n_\alpha$, This means that we can parametrize $n_\alpha T^\alpha$ by $\Ad_{t_\alpha^{-1/2}} n_\alpha$, where $t_\alpha$ ranges over $T^\alpha$. Moreover, we can parametrize $\coprod_\alpha n_\alpha T^\alpha$ by $\Ad_n n_\alpha$, where $\alpha$ is fixed and $n$ ranges over $N$.

Let $[n]$ denote the image of $n$ under the projection to $W$, and $\alpha' = [n](\alpha)$. Then, according to Lemma \ref{lem:commutation_n}, there exists $t_0 \in T_{\alpha'}$ such that:
\begin{itemize}
    \item $\ad_{n_0} (e_{\pm \alpha}) = \ad_{t_0} (e_{\pm \alpha'})$;
    \item $\Ad_{n_0} (n_\alpha) = \Ad_{t_0}(n_{\alpha'})$.
\end{itemize}

Then the map is:
\[
\Ad_{n} n_\alpha 
\overset{\text{Lem. } \ref{lem:commutation_n}}{=} \Ad_{t_0}(n_{\alpha'})
\overset{\text{Eq. } \ref{eq:map_to_stokes_t_equivariant}}{\longmapsto}
\Ad_{t_0} \exp(-e_{\alpha'} )
\overset{\text{Lem. } \ref{lem:commutation_n}}{=}
\Ad_n \exp(- e_\alpha ) ,
\]
which is manifestly $N$-equivariant.
\end{proof}

\subsubsection{Stokes factors for new Stokes curves}
\label{sect:nonab_new_stokes}

Consider a basic abstract cameral network $\tilde \cW$, as introduced in Definition \ref{defin:basic_abstract}, and let $\cW$ be the induced basic spectral network. We work at a joint $x$ of $\cW$, and want to prove that the Stokes factors for incoming Stokes curves uniquely determine those for outgoing Stokes curves. 

Due to Definition \ref{defin:basic_abstract}, the tangent vectors to the Stokes curves at $x$, together with the labels $C_{in}$, $C_{out}$ of incoming and outgoing curves and a choice of branch of the cameral cover at $x$,  provide the data of an undecorated 2D scattering diagram (see Definition \ref{def:scattering_diagram}). In the presence of a trivialization
$\phi_x : \cE_x \cong N$, we can identify the tuple of Stokes factors for incoming curves with elements $(u_\gamma)_{\gamma \in C_{in}} \in \prod_{\gamma \in C_{in}} U_\gamma$.
Then, according to Theorem \ref{thm:assignment_stokes_intersection}, the scattering diagram has a unique solution, i.e. there is a unique morphism of schemes:
\begin{align*}
\prod_{\gamma \in C_{in}} U_{\gamma} 
&\longrightarrow 
\prod_{\gamma \in C_{out}} U_{\gamma}
\\
(u_\gamma)_{\gamma \in C_{in}} 
&\longmapsto
(u'_\gamma)_{\gamma \in C_{out}}
\end{align*}
such that the product of all $u$ and $u'$, taken in clockwise order around the joint $x$, is the identity. 

Then we define the outgoing Stokes factors as the preimages $(u'_\gamma)_{\gamma \in C_{out}}$ under the identification:
\[
\Aut(\cE_x \times_N G) \cong G
\]
determined by the trivialization $\phi_x$.

It remains to show that this definition does not depend on $\phi_x$. This follows from the next lemma.

\begin{lemma}
\label{lem:assignment_stokes_newlines_equivariant}
The morphism:
\begin{align*}
\prod_{\gamma \in C_{in}} U_{\gamma} 
&\longrightarrow 
\prod_{\gamma \in C_{out}} U_{\gamma}
\\
(u_\gamma)_{\gamma \in C_{in}} 
&\longmapsto
(u'_\gamma)_{\gamma \in C_{out}}
\end{align*}
from Theorem \ref{thm:assignment_stokes_intersection} is $N$-equivariant with respect to the adjoint action of $N$ acting diagonally. As a consequence, it descends to a morphism of stacks:
\begin{equation}
\label{equation: New Stokes lines}
\left[\left(
\coprod
\prod_{\gamma \in C_{in}} U_{\gamma}
\right) \middle/ N\right]
\to 
\left[\left(
\coprod
\prod_{\gamma \in C_{out}} U_{\gamma}
\right) \middle/ N\right],
\end{equation}
where the disjoint union is taken over the orbit of the diagonal action of $W$.
\end{lemma}

\begin{proof}
Theorem \ref{thm:assignment_stokes_intersection} asserts that there is a unique tuple $(u'_\gamma)_{\gamma \in C_{out}}$ such that the following equation holds:
\[
\overrightarrow{\prod}_{\gamma \in C_{in}\coprod C_{out}} u_\gamma^{\pm 1} = \id ,
\]
where the product is ordered clockwise around the joint, and the exponent is $-1$ for incoming curves and $+1$ for outgoing curves. Applying $\Ad_n$ to this equation, and using the uniqueness of the solution, the result follows.
\end{proof}

\subsubsection{Borel Structures along $D$}
\label{subsubsec: Borel structures arbitrary G}

As outlined for $G=SL(n)$ or $GL(n)$ in \cite{gaiotto2013spectral}, and made precise for $SL(2)$ in \cite{hollands2016spectral} the nonabelianization map constructed in this section, can be enhanced to a map to the moduli of $G$-local system on $X$ together with a reduction of structure to a $B$-local system around\footnote{More precisely we can say on the intersection of $X$ with contractible neighbourhoods in $X^{c}$ of each $d\in D$.} $d\in D$ for $B$ a Borel subgroup of $G$, when we are working with a basic WKB Spectral network associated to $a\in \cA_{R}^{\diamondsuit}$.  

Let $a\in \cA^{\diamondsuit}_{R}$ be such that there is a basic spectral network $\cW$ associated to it, and let $d\in D$.  Pick a point $d\neq x_{d}\in B_{d}$ (the set defined in Proposition \ref{prop: open sets around D}) for each $d\in D$.  Pick a trivialization at each $x_{d}$, corresponding to a point $\tilde{x}_{d}\in \tilde{X}_{a}|_{x_{d}}$.  Let $\tilde{B_{d}}$ be the connected component of the preimage of $B_{d}$ containing $\tilde{x}_{d}$.   Suppose that we have an $N$-local system $E_{N}$ on $X^{\circ}$ with associated cameral cover $\tilde{X}_{a}^{\circ}$.  Let $E_{T}$ be the associated $T$-local system on $\tilde{X}_{a}^{\circ}$.  We have that
\[E_{N}|_{x_{d}}\times_{N}G\cong (E_{T}|_{\tilde{x}_{d}}\times_{T}N)\times_{N}G\cong E_{T}|_{\tilde{x}_{d}}\times_{T}G.\]

%Let $\tilde{r}_{d}\in \ft$ be the lift of the residue corresponding to the branch of the cameral cover containing $\tilde{x}_{d}$ (see Definition \ref{def:residue_hitchin_point}).  
Define $\Phi_{+}$ to be the set of roots $\alpha\in \Phi$, such that any trajectories of $V_{\alpha}$ starting in $B_{d}$ end at $d\in D$.  Define $\Phi_{-}=-\Phi_{+}$.  The inductive step in the proof of Proposition \ref{prop: open sets around D} shows that $\Phi_{+}$ is closed under addition.  Condition R and the fact that $V_{-\alpha}$ is equivalent to $V_{\alpha}$ with the orientation reversed, means that this gives a partition $\Phi=\Phi_{+}\cup \Phi_{-}$.  Hence by Corollary 1 of Proposition 20 in \S 1.7 of Chapter VI of \cite{bourbaki2002Lie} we have that $\Phi=\Phi_{+}\cup \Phi_{-}$ is a polarization of $\Phi$.  This polarization, together with the maximal torus $T\subset G$ specifies a Borel subgroup $T\subset B\subset G$.  Furthermore the Stokes curves that enter $B_{d}$ are labelled (on the sheet of $\tilde{X}$ specified by $\tilde{x}_{d}$) by a subset of $\Phi_{+}$.

We hence have that the procedure of ``cutting and regluing'' preserves the reduction of structure of $E_{G}=E_{N}\times_{N}G$ to $B$ on $B_{d}$ given by \[E_{B}:=E_{T}|_{\tilde{B}_{d}}\times_{T}B\hookrightarrow E_{T}|_{\tilde{B}_{d}}\times_{T}G\cong E_{N}|_{B_{d}}\times_{N}G.\]  
This Borel structure does not depend on the choice of trivialization $\tilde{x}_{d}$.  This is because if $\tilde{x}_{d,1}$ and $\tilde{x}_{d,2}$ are two such choices, we have $w\cdot (\tilde{x}_{d,1})=\tilde{x}_{d,2}$ for some $w\in W$.  The two lifts of the residue $\tilde{r}_{d,1}$ and $\tilde{r}_{d,2}$ are then related by the equation $w(\tilde{r}_{d,1})=\tilde{r}_{d,2}$.  We hence have that the two choices of polarization are related by $w({\Phi_{+}}_{1})={\Phi_{+}}_{2}$.

Let $\tilde{B_{d,}}_{1}$ and $\tilde{B_{d,}}_{2}$ be the two branches of $\tilde{X}|_{B_{d}\cap X}$  containing $\tilde{x}_{d,1}$ and $\tilde{x}_{d,2}$ respectively.  Let $B_{1}$ and $B_{2}$ denote the Borels corresponding to $T$ and the polarizations ${\Phi_{+}}_{1}$ and ${\Phi_{+}}_{2}$ respectively.  Noting that for any $n\in q^{-1}(w)$ we have an isomorphism $w^{-1}(E_{T}|_{B_{d,1}})\cong E_{T}|_{B_{d,2}}$, and that $w$ is acting both by pullback, and by modifying the action of $T$ it is clear that this induces an isomorphism
\[w^{*}(E_{T}|_{B_{d,2}})\times_{T}B_{2}\cong (E_{T}|_{B_{d,1}})\times_{T}B_{1}.\]

This shows that the Borel structure does not depend on the choice of the lifts $\tilde{x}_{d}$.  Clearly this construction works in families, and as such allows one to modify the nonabelianization construction (for basic spectral networks arising from $a\in \cA^{\diamondsuit}_{R}$) to give a $G$-local system on $X$ together with a reduction of structure to a Borel around $D$ (in the precise sense specified previously).

\begin{remark}
This procedure is slightly clearer in the case of $G=GL(n)$ or $G=SL(n)$.  Here the labels of the Stokes lines on the curve $X$ (see section \ref{sec: Spectral Decomposition} for more details on the relation between the spectral cover and the cameral cover) can be seen as corresponding to ordered pairs of branches of the spectral curve. We can also interpret the projective oriented vector fields on the cameral curve around $\pi^{-1}(x_{d})$ as descending to projective oriented vector fields on the curve $B_{d}$ labelled by ordered pairs of branches of the spectral curve.  As such Proposition \ref{prop: open sets around D} can be seen as specifying a subset of ordered pairs of branches of the spectral curve.  These are the labels for the projective oriented vector fields such that all trajectories starting in $B_{d}$ terminate at $d$.  This subset has the property that it is precisely the set of pairs of the form $(b_{1}, b_{2})$, where $b_{1}<_{R}b_{2}$ for some total order $<_{R}$ of the branches of the spectral curve restricted to $B_{d}$.  The ordering of these branches gives a Borel structure on the pushforward of a local system of one dimensional vector spaces on the spectral cover, which is preserved under the ``cutting and regluing'' operation.

%As such for $a\in \cA_{R}^{\diamondsuit}$, the polarization attached to each branch of the cameral cover at $d\in D$ is equivalent to an ordering of the branches of the spectral cover restricted to $d\in D$.  The Stokes lines in $B_{d}$ are then labelled by pairs of branches $(b_{1},b_{2})$ where $b_{1}<b_{2}$.  The ordering of these branches gives a Borel structure on the pushforward of a local system of one dimensional vector spaces on the spectral cover, which is preserved under the ``cutting and regluing'' operation.
% The inequality is correct because ultimately it comes down to a convention of which of the two filtrations we choose!
\end{remark}

\section{Spectral Decomposition and Non-abelianization}
\label{sec: Spectral Decomposition}

In this section we describe some of the spaces and morphisms involved
in non-abelianization in terms of spectral rather than cameral covers.  We hope that these descriptions, particularly those in Section \ref{subsec: explicit descrtions for defining representations of classical groups} will be clearer to people not familiar with cameral covers.  In Section \ref{subsec: Spectral covers and associated local systems} we review the construction in \cite{donagi1993decomposition} of a curve we call the non-embedded spectral cover associated to a cameral cover $\tilde{X}$ and a representation $\rho$ that is closely related to the spectral cover that would be associated to the cameral cover via the representation $\rho$ as explained in Remark \ref{remark: spectral cover and non embedded spectral cover}.  We show how $N$-local systems with associated $W$-local system $\tilde{X}\backslash R$ give rise to local systems of vector spaces on the associated non-embedded spectral covers\footnote{With the preimages of branch points removed.}.  In Section \ref{sec: exponential path rules} we describe the relationship between the construction of non-abelianization in Section \ref{subsec: Non abelianiazation via Cameral Networks} and the path detour rules of \cite{gaiotto2013spectral, longhi2016ade}.  %In section \ref{subsec: faithful N representations} we describe how the moduli stack of $N$-shifted weakly $W$-equivariant $T$-local systems on an unramified cameral cover, and the moduli stack of those satisfying the $S$-monodromy condition are equivalent to stacks of $\bbG_{m}$-local systems with extra data on spectral covers associated to faithful representations of $N$.  
Finally in Section \ref{subsec: explicit descrtions for defining representations of classical groups} for classical groups we provide explicit descriptions of the moduli spaces of the relevant $N$-local systems in terms of moduli spaces of line bundles on spectral covers equipped with extra structure.  These descriptions are analogous to the descriptions of Hitchin fibers in \cite{hitchin1987stable, donagi1993decomposition, donagi1995spectral}.

\subsection{Spectral Covers, Cameral covers, and Local systems}
\label{subsec: Spectral covers and associated local systems}

Let $\rho: G\rightarrow GL(V)$ be a representation of $G$.  This induces a map from local systems, and Higgs bundles for the group $G$ to local systems and Higgs bundles for the group $GL(V)$.  We show that to a cameral cover $\tilde{X}$, the representation associates a spectral cover $\overline{X}_{\rho}$.  This breaks up as a union of curves, corresponding to the $W$-orbits of the weights of the representation $\rho$.  We then show that an $N$-local system with associated $W$-local system $\tilde{X}\backslash R$ induces a local system of vector spaces on each of these curves.

\subsubsection{The associated spectral cover}

Consider a maximal torus $T\hookrightarrow G$.  This gives a weight space decomposition $V=\oplus_{\omega\in \Omega_{\rho}} V_{\omega}$ of the representation $V$, where $\Omega_{\rho}$ parametrizes the weights of the representation $\rho$.  There is a $W$-action on the weights.

\begin{definition}
\label{defn: non-embedded spectral cover}
We define the \textbf{associated non-embedded spectral cover} of a smooth cameral cover $\tilde{X}$ as 

\[\overline{X}_{\rho}^{ne}:=\tilde{X}\times_{W}\Omega_{\rho}\]
\end{definition}

This is equipped with a map $\pi_{\rho}: \overline{X}_{\rho}^{ne}\rightarrow X$.  We denote the ramification locus of this map by $R_{\rho}\subset \overline{X}_{\rho}^{ne}$.

\begin{remark}
\label{remark: spectral cover and non embedded spectral cover}
If we take a Higgs $G$-bundle $\varphi\in \Gamma(X, ad(E)\otimes \cL)$ with reduced cameral cover $\tilde{X}$, then the spectral cover $\overline{X}\hookrightarrow Tot(\cL)$ associated to $\rho(\varphi)$ has the property that $\overline{X}_{red}$ and $\overline{X}_{\rho}^{ne}$ are birational.   The difference between $\overline{X}$ and $\overline{X}_{\rho}^{ne}$ is due to the phenomenon of accidental singularities in the sense of \cite{donagi1993decomposition}, differences at the ramification points of the cover, and the fact that the weight spaces may not be one dimensional.  The cases of the classical groups and defining representations are reviewed in Section \ref{subsec: explicit descrtions for defining representations of classical groups}, cf. \cite{donagi1993decomposition, hitchin1987stable}.
\end{remark}

The set of weights $\Omega_{\rho}$ is a union of $W$-orbits of weights, which we write as $\Omega_{\rho}=\coprod_{o\in or}\Omega^{o}_{\rho}$, where the disjoint unit is over the set $or$ of $W$-orbits of weights of $\rho$.  Defining $\overline{X}_{\rho}^{ne,o}:=\tilde{X}\times_{w}\Omega^{o}_{\rho}$ we have the decomposition 
\begin{equation}
\label{equation: decomposition into W orbits}
\overline{X}_{\rho}^{ne}:=\coprod_{o\in or}\overline{X}_{\rho}^{ne, o}.\end{equation}
We call each $\overline{X}_{\rho}^{ne, o}$ a non-embedded spectral cover.

We also introduce the following variants:

\begin{definition}
\label{definition: blown up spectral covers}
We define:
\[
\overline{X}^{ne, \circ_{P}}_{\rho}:=\tilde{X}^{\circ}\times_{W}\Omega_{\rho}.
\]
\[
\overline{X}^{ne, \circ_{R}}_{\rho}:=ReBl_{R_{\rho}}(\overline{X}^{ne}_{\rho}),
\]
where this refers to the oriented real blow up of $\overline{X}^{ne}_{\rho}$ along the ramification locus $R_{\rho}$ of  $\overline{X}^{ne}_{\rho}\rightarrow X$.
\end{definition}

The space $\overline{X}^{ne, \circ_{P}}_{\rho}$ is equipped with a map to $X^{\circ}$.  The space $\overline{X}^{ne, \circ_{R}}_{\rho}$ is not equipped with such a map, but does have a map $\overline{X}^{ne, \circ_{R}}_{\rho}\rightarrow X$.

\subsubsection{The associated local systems}
\label{subsubsec: The associated local system}

Not only does $N$ act on $V=\oplus_{\omega\in \Omega_{\rho}}V_{\omega}$ via $\rho$, but we further have an action on the disjoint union of weight spaces:
\[
N \acts \coprod_{\omega\in \Omega_{\rho}}V_\omega.
\]
This action lifts the action $W\acts \Omega_{\rho}$. More precisely, consider the quotient map $q:N\rightarrow W$; then for each $w\in W$, and $\omega\in \Omega_{\rho}$ there is a unique map $\bar \rho_{w,\omega} : q^{-1}(w) \rightarrow \Hom(V_{\omega}, V_{w(\omega)})$ which makes the following diagram commute:

\begin{equation*}
\begin{tikzcd}
V \arrow{r}{\rho(n)} & V 
\\
V_{\omega}\arrow[hook]{u} \arrow{r}{\bar \rho_{w,\omega} (n)} &
V_{w(\omega)} .\arrow[hook]{u}
\end{tikzcd}
\end{equation*}
As a consequence, $\bar \rho$ identifies composition of morphisms with the product in $N$, in the sense that $\bar{\rho}_{w_{1},w_{2}(\omega)}(n_{1})\bar{\rho}_{w_{2},\omega}(n_{2})=\bar{\rho}_{w_{1}w_{2},\omega}(n_{1}n_{2})$, for $w_{1},w_{2}\in W$, and $n_{i}\in q^{-1}(w_{i})$.

We define $\LocVect(\overline{X})$ to be the moduli space of local systems of vector spaces on $\overline{X}$, where we do not require the dimension of the local system to be the same on different components of $\overline{X}$.  This is an ind-stack, but in practice we will only be considering a fixed component of it.

\begin{construction}
\label{construction: associated local system of vector bundles on spectral}
Consider the unramified map $\tilde{X}^{\circ} \rightarrow X^{\circ}$.  We have a morphism
\[Sp_{\rho}: \Loc_{N}^{\tilde{X}^{\circ}}(X^{\circ})\rightarrow \LocVect(\overline{X}_{\rho}^{ne,\circ_{P}})\]
given on points by 
\[E_{N}\xmapsto{Sp_{\rho}} E_{N}\times_{N}\left(\coprod_{\omega\in \Omega_{\rho}}V_{\omega}\right).\]
The right hand side is a local system on $\overline{X}_{\rho}^{ne,\circ_{P}}$ because the map $\coprod_{\omega\in \Omega_{\rho}}V_{\omega}\rightarrow \Omega_{\rho}$ induces a map 
\[
E_{N}\times_{N}\left(\coprod_{\omega\in \Omega_{\rho}}V_{\omega}\right)\rightarrow \overline{X}_{\rho}^{ne,\circ_{P}},
\] 
and this has the structure of a local system.  The map on moduli spaces is given by doing this to the universal $N$-local system.
\end{construction}

\begin{remark}
It is clear that in general the maps $Sp_{\rho}$ are not surjective on a component of $\LocVect(\overline{X}_{\rho}^{ne,\circ_{P}})$.  For classical groups and the defining representation we describe the image of this map, and how to interpret the S-monodromy condition in terms of spectral data in Section \ref{subsec: explicit descrtions for defining representations of classical groups}.
\end{remark}

\subsection{Path Detour Rules}
\label{sec: exponential path rules}

In this section we describe the path detour rules of \cite{gaiotto2013spectral, longhi2016ade}, which provide an alternative version of non-abelianization involving locally constant vector bundles and $\bbG_{m}$-local systems on the non-embedded spectral cover associated to a minuscule representation of a simply laced reductive algebraic group $G$. By ``simply laced'' we mean that the Lie algebra of $G$ is a direct sum of abelian and simply laced Lie algebras.  %See \todo{reference a remark for why we use these vs. path detour rules.} 

We then show that, for a minuscule representation, the path detour rules are compatible with the non-abelianization we defined in Section \ref{sec: flat Donagi Gaitsgory} as is formulated precisely in Theorem \ref{conj: exp path rules minuscule}.

Recall:

\begin{definition}
A \textbf{minuscule representation} of $G$ is an irreducible representation $\rho$ such that the Weyl group of $G$ acts transitively on the weights of $\rho$.
\end{definition}

%\begin{warning}
%Some authors use the term minuscule for the condition that the Weyl group acts transitively on the \emph{non-zero} weights of $\rho$.  The results of this section do not hold for this definition. 
%\end{warning}

\begin{construction}[Path Detour Non-abelianization, essentially that of \cite{gaiotto2013spectral, longhi2016ade}]
\label{construction: exponential path rule non abelianization}
Pick a simply laced reductive algebraic group $G$.  Pick a basic abstract spectral network on $X$, with the property that:
\begin{itemize}
    \item \label{item: Condition for path detour nonabelianization} At each joint if $\{\alpha_{1},...,\alpha_{n}\}$ are the roots labelling the incoming Stokes curves (with respect to some trivialization of the cameral cover at the joint), there is at most one root $\beta\in \Conv^{\bbN}_{\alpha_{1},..,\alpha_{n}}$ such that $\beta$ is not contained in a one dimensional face of $\Conv^{\bbN}_{\alpha_{1},..,\alpha_{n}}$.
\end{itemize}
Pick a minuscule representation $G\rightarrow GL(V)$.

The strategy is as follows.  Recall that Construction \ref{thm:bulk_map_exists} proceeded by composing the morphisms
\[
\Loc_{N}^{\tilde{X}^{\circ}, S}(X^{\circ})
\xrightarrow{s_{WKB}}
\Aut_{\cW, G}(X^{\circ '})
\xrightarrow{reglue} 
\Loc_{G}(X\backslash (P\cup \cJ))
\]
and noting that this factors through $\Loc_{G}(X)$.  Here, for a fixed minuscule representation $\rho: G\rightarrow GL(V)$ we provide a map $s_{PD}:\Loc_{\bbG_{m}}(\overline{X}_{\rho}^{ne, \circ_{R}})\rightarrow \Aut_{GL(V), \cW}(X^{\circ '})$. We then define a map as the composition $nonab_{PD}':= reglue\circ s_{PD}$, where we are using the regluing map of Equation \ref{eq: Reglue map}.  We will then define the nonabelianization map $nonab_{PD}$ to be a modification of this map.

Let $\Loc_{\bbG_{m}}^{fr}(\overline{X}_{\rho}^{ne, \circ_{R}})$ be the finite unramified cover of $\Loc_{\bbG_{m}}(\overline{X}_{\rho}^{ne, \circ_{R}})$ such that sheets correspond to trivializations of the cameral cover at the points $x_{p}\in S^{1}_{p}\subset X^{\circ}$  (as introduced in Convention \ref{convention:blowup}).

We now define a map $s_{PD}': \Loc_{\bbG_{m}}^{fr}(\overline{X}_{\rho}^{ne, \circ_{R}})\rightarrow \Aut_{\cW, GL(V)}(X^{\circ '})$.  Firstly we get a $GL(V)$-local system $\cE$ on $X^{\circ }$ by pushing forward the one dimensional local system on $\overline{X}_{\rho}^{ne,\circ_{R}}$, restricted away from $\pi_{\rho}^{-1}(P)$, to a local system on $X\backslash P$ which is homotopic to $X^{\circ '}$.

Firstly we describe the Stokes factors associated to primary lines, that is to say, using the notation of Section \ref{subsec: Non abelianiazation via Cameral Networks}, we provide a map:
\[
\Loc_{\bbG_{m}}^{fr}(\overline{X}_{\rho}^{ne, \circ_{R}})\rightarrow \Aut_{F_{0}(\cW), GL(V)}(X^{\circ '}).
\]

The trivializations of the cameral cover at each $x_p$ identify the monodromy of the cameral cover around $p\in P$ with some $s_{\alpha}\in W$.  The associated spectral cover $\overline{X}_{\rho}^{ne,\circ_{R}}|_{x}:=\tilde{X}\times_{W}\Omega_{\rho}|_{x}$ (for $x\notin P$) then has branches $x'\in \overline{X}_{\rho}^{ne,\circ_{R}}|_{x}$ labelled by weights which we denote by $l(x')\in \Omega_{\rho}$.

We parallel transport this labelling around the boundary circle $S^1_p$, in the direction specified by the orientation of $S^1_p$, induced from the orientation of $X^\circ$. The trivialization at $x_{p}$ also determines a labeling by a root associated to each primary Stokes curve starting at $p$.  We let $\gamma_{p}$ be the path starting and ending at $x_{p}$ that generates the fundamental group of $S^{1}_{p}$ and is in the direction opposite to that specified by the orientation of $S^{1}_{p}$ (which is induced by the orientation of $X$). Let $\cE$ be a $\bbG_{m}$-bundle on $\overline{X}_{\rho}^{ne, \circ_{R}}$.  We will write $(\pi_{\rho})_{*}\cE$ for the local system of vector spaces on $X\backslash P$ (or the homotopic $X^{\circ}$) given by the pushforward of the associated one dimensional local system on  $\overline{X}_{\rho}^{ne}\backslash \pi^{-1}(P)$.  For each pair of distinct preimages $x_{1}, x_{2}\in \overline{X}_{\rho}^{ne,\circ_{R}}|_{x}$ of $x_{s}\in X$ on a Stokes line, such that $s_{\alpha}(x_{1})=x_{2}$, with the property that $l(x_{1})-l(x_{2})=\alpha$ we let $d_{1,2}: \cE|_{x_{1}}\rightarrow \cE|_{x_{2}})$ be the morphism given by parallel transport along the lift of the path $\gamma_{p}$ to $\overline{X}_{\rho}^{ne,\circ_{R}}$ which starts at $x_{1}$ and finished at $x_{2}$. We define:
\begin{equation}
\label{eq: path detour Stokes Factors}
S_{PD} := \exp\left(\sum d_{i,j}\right)\in GL((\pi_{\rho})_{*}\cE|_{x_{s}}) ,
\end{equation}
where we are summing over pairs of branches of $\overline{X}^{ne,\circ_{R}}_{\rho}$ satisfying the above condition.  We are abusing notation to refer to $d_{i,j}$ as an element of $GL(((\pi_{\rho})_{*}\cE)|_{x_{s}})$ -- we are referring to the composition $((\pi_{\rho})_{*}\cE)|_{x_{s}}\rightarrow \cE|_{x_{1}}\xrightarrow{d_{1,2}} \cE|_{x_{2}}\hookrightarrow ((\pi_{\rho})_{*}\cE)|_{x_{s}}$ where the first map is the projection coming from the description $((\pi_{\rho})_{*}\cE)|_{x_{s}}=\oplus_{x\in \pi_{\rho}^{-1}(x_{s})}\cE|_{x}$.  Clearly this makes sense for families.%\todo{precisely where did you use the label of $\alpha$ or $-\alpha$? A.  You used it in saying l(x_{1})-l(x_{2})=+\alpha rather than  -\alpha}

Consider a joint $J\in \cJ$ with incoming Stokes curves labelled by (with respect to a choice of trivialization of $\tilde{X}$ at $J$) $\alpha_{1},...,\alpha_{n}$.  For each Stokes curve labelled by $\beta$ that lies in a one dimensional face of $\Conv^{\bbN}_{\alpha_{1},...,\alpha_{n}}$, we must have $\beta\in \{\alpha_{1},...,\alpha_{n}\}$, and as such we can make the automorphism attached to this curve equal to the parallel transport of the automorphism associated to the incoming curve labelled by $\beta$.  By our assumption on joints there is at most one outgoing Stokes curve such that its label $\beta$ is not contained in a one dimensional face of $\Conv^{\bbN}_{\alpha_{1},...,\alpha_{n}}$, hence we can assign to this Stokes curve the unique automorphism such that upon regluing by these automorphisms the monodromy around $J$ will be trivial.  We can interpret this as corresponding to the parallel transport in the ``reglued" local system along the path that goes around the other side of $J$.  One can iteratively apply this procedure to interpret this element as the sum of monodromies, and exponentials of monodromies, of the original local systems along certain paths, see Figure \ref{Figure: Path Detour Nonabelianization 1}.  This provides a lift

\[
\begin{tikzcd}
 & 
 \Aut^{\tilde X^{\circ '}}_{F_{i+1}(\cW), GL(V)}(X^{\circ'}) 
 \arrow{d}{forget}
 \\
\Loc_{\bbG_{m}}(\overline{X}_{\rho}^{ne, \circ_{R}})
\arrow[dashed]{ru}
\arrow[swap]{r}
& 
\Aut^{\tilde X^{\circ '}}_{F_{i}(\cW), GL(V)}(X^{\circ'})
\end{tikzcd}
\]
and thus allows construction of the map $s_{PD}$ as in Construction \ref{construction of swkb} of $s_{WKB}$.

%\todo{There's an equivariance statement here; more details?}
The map $s_{PD}'$ factors through the map
\[
\Loc_{\bbG_{m}}^{fr}(\overline{X}_{\rho}^{ne, \circ_{R}})\rightarrow \Loc_{\bbG_{m}}(\overline{X}_{\rho}^{ne, \circ_{R}}).
\]
This is because if with respect to one trivialization two branches are labelled by weights $\chi_{1}$, $\chi_{2}$, with $\chi_{1}-\chi_{2}=\alpha$ for some root $\alpha$, then in another trivialization related to the original trivialization by $w\in W$, we have that $w(\chi_{1})-w(\chi_{2})=w(\alpha)$.   Hence the choice of trivialization does not affect the map $s_{PD}$.   Thus this gives a map $ \Loc_{\bbG_{m}}(\overline{X}_{\rho}^{ne, \circ_{R}})\xrightarrow{s_{PD}} \Aut_{GL(V), \cW}(X^{\circ '})$.

This gives a map 
\[nonab_{PD}':=reglue\circ s_{PD}: \Loc_{\bbG_{m}}(\overline{X}_{\rho}^{ne, \circ_{R}})\rightarrow Loc_{GL(V)}(X^{\circ}).\]
Denote by:
\[\Loc_{\bbG_{m}}^{-1}(\overline{X}^{ne,\circ_{R}}_{\rho}):=\Loc_{\bbG_{m}}(\overline{X}^{ne,\circ_{R}}_{\rho})\times_{(\bbG_{m}/\bbG_{m})^{\# R_{\rho}}}(-1/\bbG_{m})^{\# R_{\rho}},\]
the moduli space of $\bbG_{m}$-local systems on $\overline{X}^{ne,\circ_{R}}_{\rho}$ with monodromy $-1$ around each ramification point $r\in R_{\rho}$.

Lemma \ref{lemma: path detour rule factors} shows the restriction of the map $nonab_{PD}'$ to $\Loc_{\bbG_{m}}^{-1}(\overline{X}^{ne,\circ_{R}}_{\rho})$ factors through $\Loc_{G}(X)$, and we denote the map $\Loc_{\bbG_{m}}^{-1}(\overline{X}^{ne,\circ_{R}}_{\rho})\rightarrow \Loc_{G}(X)$ by $nonab_{PD}$ (see Equation \ref{equation: nonabPD map factors}).
\end{construction}

\begin{figure}
    \centering
    \def\svgwidth{200pt}
      %% Creator: Inkscape 1.0.1 (0767f8302a, 2020-10-17), www.inkscape.org
%% PDF/EPS/PS + LaTeX output extension by Johan Engelen, 2010
%% Accompanies image file 'DiagramPD1.pdf' (pdf, eps, ps)
%%
%% To include the image in your LaTeX document, write
%%   \input{<filename>.pdf_tex}
%%  instead of
%%   \includegraphics{<filename>.pdf}
%% To scale the image, write
%%   \def\svgwidth{<desired width>}
%%   \input{<filename>.pdf_tex}
%%  instead of
%%   \includegraphics[width=<desired width>]{<filename>.pdf}
%%
%% Images with a different path to the parent latex file can
%% be accessed with the `import' package (which may need to be
%% installed) using
%%   \usepackage{import}
%% in the preamble, and then including the image with
%%   \import{<path to file>}{<filename>.pdf_tex}
%% Alternatively, one can specify
%%   \graphicspath{{<path to file>/}}
%% 
%% For more information, please see info/svg-inkscape on CTAN:
%%   http://tug.ctan.org/tex-archive/info/svg-inkscape
%%
\begingroup%
  \makeatletter%
  \providecommand\color[2][]{%
    \errmessage{(Inkscape) Color is used for the text in Inkscape, but the package 'color.sty' is not loaded}%
    \renewcommand\color[2][]{}%
  }%
  \providecommand\transparent[1]{%
    \errmessage{(Inkscape) Transparency is used (non-zero) for the text in Inkscape, but the package 'transparent.sty' is not loaded}%
    \renewcommand\transparent[1]{}%
  }%
  \providecommand\rotatebox[2]{#2}%
  \newcommand*\fsize{\dimexpr\f@size pt\relax}%
  \newcommand*\lineheight[1]{\fontsize{\fsize}{#1\fsize}\selectfont}%
  \ifx\svgwidth\undefined%
    \setlength{\unitlength}{595.27559055bp}%
    \ifx\svgscale\undefined%
      \relax%
    \else%
      \setlength{\unitlength}{\unitlength * \real{\svgscale}}%
    \fi%
  \else%
    \setlength{\unitlength}{\svgwidth}%
  \fi%
  \global\let\svgwidth\undefined%
  \global\let\svgscale\undefined%
  \makeatother%
  \begin{picture}(1,1)%
    \lineheight{1}%
    \setlength\tabcolsep{0pt}%
    \put(0,0){\includegraphics[width=\unitlength,page=1]{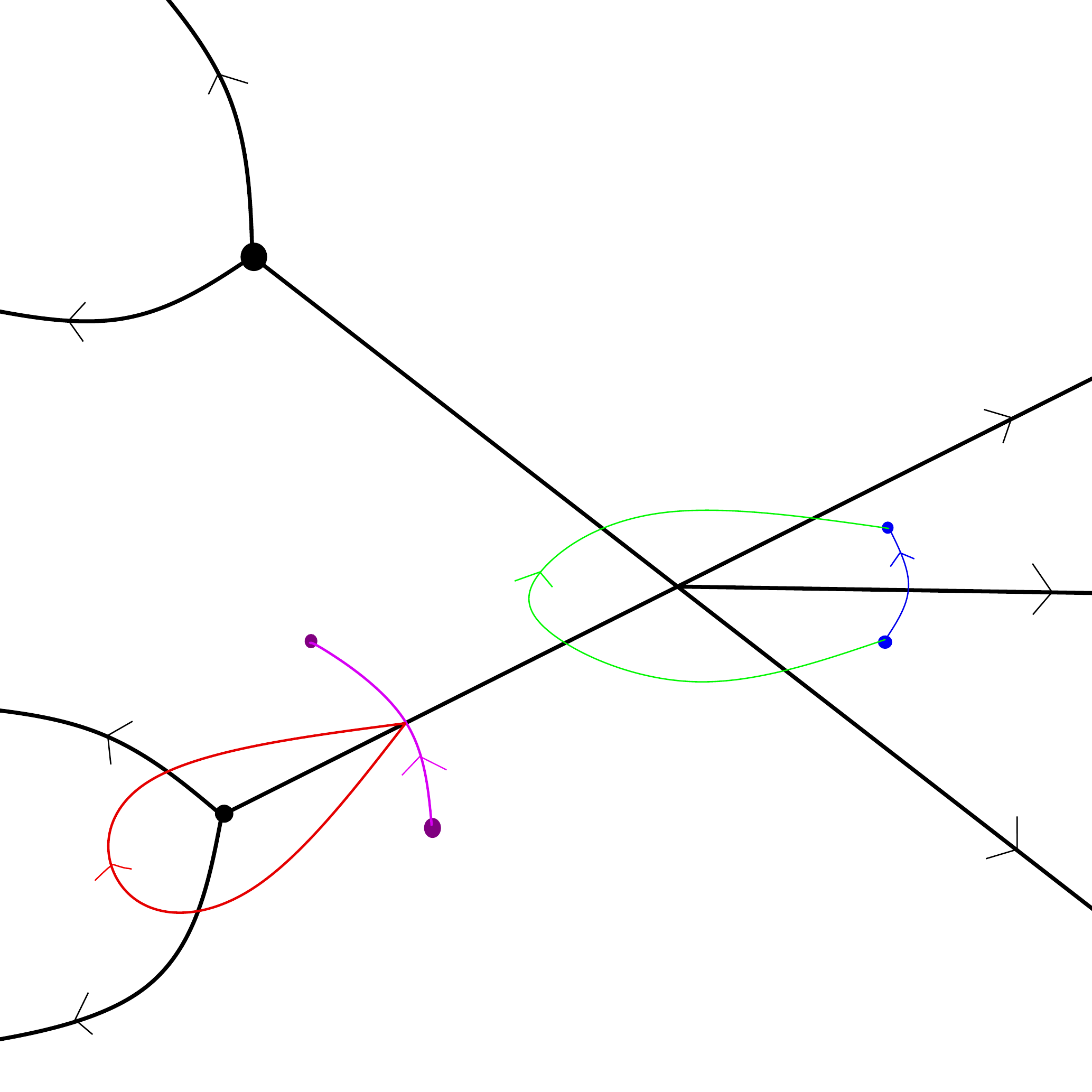}}%
    \put(0.44139422,0.29627373){\color[rgb]{0,0,0}\makebox(0,0)[lt]{\lineheight{1.25}\smash{\begin{tabular}[t]{l}$\gamma_{1}$\end{tabular}}}}%
    \put(0.35199901,0.47293189){\color[rgb]{0,0,0}\makebox(0,0)[lt]{\lineheight{1.25}\smash{\begin{tabular}[t]{l}$\gamma_{2}$\end{tabular}}}}%
    \put(0.85618555,0.49882573){\color[rgb]{0,0,0}\makebox(0,0)[lt]{\lineheight{1.25}\smash{\begin{tabular}[t]{l}$\gamma_{3}$\end{tabular}}}}%
    \put(0.00894026,0.18443807){\color[rgb]{0,0,0}\makebox(0,0)[lt]{\lineheight{1.25}\smash{\begin{tabular}[t]{l}$\gamma_{4}$\end{tabular}}}}%
  \end{picture}%
\endgroup%

    \caption{Interpreting monodromy along paths on $X$.  The parallel transport along $\gamma_{1}$ (purple) is modified along the intersection with the spectral network by the exponential of the monodromy of certain lifts to the spectral cover of the path $\gamma_{4}$.  The parallel transport along $\gamma_{3}$ is replaced by the modified parallel transport along $\gamma_{2}$ (green).  The modification of the parallel transport along $\gamma_{2}$ is analogous to the modification of $\gamma_{1}$.}
    \label{Figure: Path Detour Nonabelianization 1}
\end{figure}
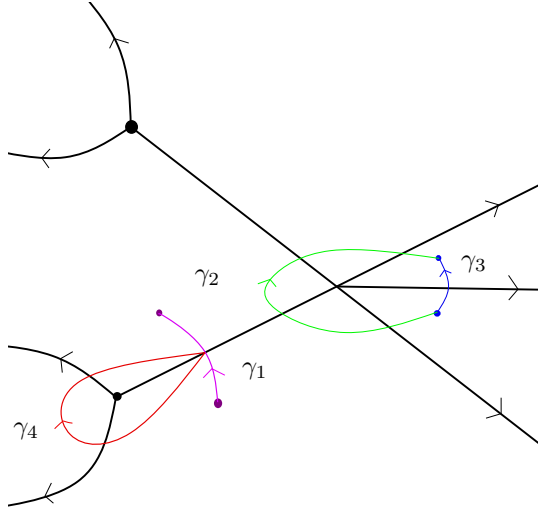

See Remark \ref{remark: link to Longhi--Park path detour} for the precise link to the path detour rules of \cite{gaiotto2013spectral} and \cite{longhi2016ade}, which provide a slightly different description of this procedure.

\begin{remark}
We can modify the above construction to deal with local systems of vector spaces with additional structure;  see Section  \ref{subsec: equivalence for GL(n) and SL(n)}.

\end{remark}

\begin{remark}
The simply laced condition, and the restriction on the joints appearing in the spectral network are required as it is not clear how to define the Stokes automorphisms for new Stokes lines in terms of monodromy of paths without these conditions.
\end{remark}

\begin{lemma}
\label{lemma: path detour rule factors}
The restriction of the map $nonab_{PD}'$ of Construction \ref{construction: exponential path rule non abelianization} to local systems with monodromy $-1$ around each ramification point $r\in R$ factors through $\Loc_{G}(X)$, forming the following commutative diagram which defines the map $nonab_{PD}$;
\begin{equation}
\label{equation: nonabPD map factors}
\begin{tikzcd}
\Loc_{\bbG_{m}}(\overline{X}^{ne,\circ_{R}}_{\rho}) \arrow{r}{nonab_{PD}'} & \Loc_{GL(V)}(X^{\circ})\\
\Loc_{\bbG_{m}}^{-1}(\overline{X}^{ne,\circ_{R}}_{\rho})\arrow[hook]{u}\arrow{r}{nonab_{PD}} & \Loc_{GL(V)}(X)\arrow[hook]{u}
\end{tikzcd}
\end{equation}
where the vertical arrows are inclusions.
\end{lemma}

We prove this lemma using the following lemma about reducing minuscule representations to $SL(2)$ representations.

\begin{lemma}
\label{lemma:  Properties of sl2 representations}
Let $G$ be a simply laced reductive algebraic group.

Let $c: SL(2)\xrightarrow{I_{\alpha}} G \xrightarrow{\rho} GL(V)$ be the composition of an $SL(2)$-triple associated to a root $\alpha$ with a minuscule representation $\rho$.  This gives a finite dimensional representation $c$ of $SL(2)$.  

The representation $c$ is a direct sum of one and two dimensional irreducible representations of $SL(2)$.
\end{lemma}

\begin{proof}
We first consider the case that the group $G$ is simple.  In this case by Corollary 6.6.6 of \cite{green2013combinatorics} the weights $\lambda$ of the Lie algebra representation $Lie(\rho)$ have the property that $s_{\alpha}(\lambda)\in \{\lambda, \lambda+\alpha, \lambda-\alpha\}$. 

%\todo{IS THIS TRUE: Noting that $\alpha$ is the image under the map $c$ of the root of $SL_{2}$} this means that an $SL_{2}$ representation, $c$ is the direct sum of one and two dimensional representations of $SL(2)$, as all other irreducible representations violate the above property.

Identifying $\Lambda_{char}\hookrightarrow \ft_{SL(2)}$ with $\bbZ\hookrightarrow \bbC$ identifies the maps $\ft_{G}^{\vee}\rightarrow \ft_{SL(2)}^{\vee}$ with the map $\ft_{G}^{\vee}\xrightarrow{\alpha^{\vee}}\ft_{SL(2)}^{\vee}$ given by evaluation on the coroot $\alpha^{\vee}$.  This is because the dual maps are $\ft_{SL(2)}\rightarrow \ft_{G}$, $h_{\alpha, SL(2)}\mapsto h_{\alpha}$, and $\bbC\rightarrow \ft_{G}$, $1\mapsto \alpha^{\vee}$, and $h_{\alpha}=\alpha^{\vee}$.  Now $\alpha^{\vee}(\alpha)=2$, hence this is saying that the weights $\lambda_{\mathfrak{sl}(2)}$ of the $\mathfrak{sl}(2)$ representation $c$ satisfy $s_{\alpha}(\lambda_{\mathfrak{sl}(2)})\in \{\lambda, \lambda+2, \lambda-2\}$.  Hence we have that the $SL(2)$ representation $\rho \circ I_{\alpha}$ is a sum of one and two dimensional representations of $SL(2)$.

%\todo{IMPORTANT: Make it clear that we can use the product structure of roots and weyl group to reduce to the factors!}

A general reductive group can be written as a product $(T_{1}\times \Pi_{i}G_{i})/Z$, where $T_{1}$ is commutative, $G_{i}$ is a simple factor for each $i$, and $Z$ is finite.  The Weyl group does not depend on $Z$ because it is determined by the Lie algebra, and is hence the product $W=\Pi W_{i}$ where $W_{i}$ is the Weyl group of $G_{i}$.  The highest weight classification of representations of Lie algebras then makes it clear that a miniscule representation of $G$ must lift to a representation $V_{T_{1}}\otimes\bigotimes_{i} V_{i}$ of $(T_{1}\times \Pi_{i}G_{i})$, where $V_{i}$ is miniscule representation of $G_{i}$ for each $i$. and $V_{T_{1}}$ is a one dimensional representation of $T_{1}$.  The result then follows immediately for the case where the group $Z$ is trivial, because any root $\alpha$ of $(T_{1}\times \Pi_{i}G_{i})$ corresponds to a root of $G_{i}$ for some $i$, and hence we still have $s_{\alpha}(\lambda)\in \{\lambda, \lambda+\alpha, \lambda-\alpha\}$.  Now suppose the result holds for $T_{1}\times \Pi_{i}G_{i}$, it immediately follows for $T_{1}\times \Pi_{i}G_{i}/Z$, upon observing that the result only depends on the associated Lie algebra representation.
%\todo{MI: I don't understand the second paragraph. But I can prove that the claim follows from the first paragraph, so I'm happy.}
\end{proof}

\begin{proof}[Proof of Lemma \ref{lemma: path detour rule factors}]
We only need to show that for local systems in $\Loc_{\bbG_{m}}^{-1}(\overline{X}^{ne,\circ_{R}}_{\rho})$, Construction \ref{construction: exponential path rule non abelianization} produces a local system on $X^{\circ}$ with trivial monodromy around the preimage of each branch point under the map $X^{\circ}\rightarrow X$.  

Lemma \ref{lemma:  Properties of sl2 representations} shows that we need only consider a single branch point for $G=SL(2)$.  In this case the result follows by direct computation.
\end{proof}

One way to do this computation is by using the relation to the nonabelianization of Section \ref{subsec: Non abelianiazation via Cameral Networks} expressed in Theorem \ref{conj: exp path rules minuscule}.

\begin{lemma}
\label{lemma: getting local systems with -1 monodromy }
For $\rho$ a miniscule representation, the restriction of the map $Sp_{\rho}$ to $\Loc_{N}^{\tilde{X}^{\circ}, S}(X^{\circ})$ factors through $\Loc_{\bbG_{m}}^{-1}(\overline{X}^{ne,\circ_{R}}_{\rho})$, as shown in the following commutative diagram:
\[
\begin{tikzcd}
\Loc_{N}^{\tilde{X}^{\circ}, S}(X^{\circ}) \arrow{rr}{Sp_{\rho}|_{\Loc_{N}^{\tilde{X}^{\circ}, S}(X^{\circ})}} \arrow{rd} & & \Loc_{\bbG_{m}}(\overline{X}^{ne,\circ_{R}}_{\rho})\\
& \Loc_{\bbG_{m}}^{-1}(\overline{X}^{ne,\circ_{R}}_{\rho}) \arrow{ru} & 
\end{tikzcd}
\]
\end{lemma}

\begin{proof}
Pick a trivialization such that the monodromy of the $N$-local system around a branch point is $n_{\alpha}$.  We then have that the monodromy of the $\bbG_{m}$-local system on $\overline{X}^{ne,\circ_{R}}_{\rho}$ (considered as a local system on $X$) is given by $c(\left(\begin{matrix}0 & 1 \\ -1 & 0\end{matrix}\right))$.  The representation $c$ of Lemma \ref{lemma:  Properties of sl2 representations} breaks up into a direct sum of one dimensional representations (corresponding to branches of the spectral cover which are ramified at this point), and two dimensional irreducible representations (corresponding to the ramified branches of the spectral cover). Hence the monodromy of the $\bbG_{m}$-local system (on $\overline{X}^{ne,\circ_{R}}_{\rho}$) around the ramification points is -1, because
\[\left(\begin{matrix}0 & 1 \\ -1 & 0\end{matrix}\right)^{2}=-Id.\]

%Lemma \ref{lemma:  Properties of sl2 representations} shows that this reduces immediately to the case of the defining representation of $GL_{2}$ \todo{be a bit clearer about this, and I think this isn't quite the reduction I should say!}.  This is treated in proposition \ref{prop: alpha monodromy GL(n)}.

\end{proof}

\begin{theorem}[Compatibility between non-abelianization and path detour non-abelianization]
\label{conj: exp path rules minuscule}
Let $G$ be a simply laced reductive algebraic group.  Let $\rho:G\rightarrow GL(V)$ be a minuscule representation.  The following diagram commutes, where we have denoted $Sp_{\rho}^{S}:=Sp_{\rho}|_{\Loc_{N}^{\tilde{X}^{\circ}, S}(X^{\circ})}$:
\begin{equation}
\label{eq:nonab_pathrule_commute}
\begin{tikzcd}
\Loc_{N}^{\tilde{X}^{\circ}, S}(X^{\circ}) \arrow{r}{Sp_{\rho}^{S}} \arrow{d}{nonab} &  \Loc_{\bbG_{m}}^{-1}(\overline{X}_{\rho}^{ne, \circ_{R}}) \arrow{d}{nonab_{PD}}\\
\Loc_{G}(X)\arrow{r}{\rho} & \Loc_{GL(V)}(X)
\end{tikzcd}
\end{equation}

\end{theorem}

%\todo{Wait on -- there is a major gap here -- why are are we using R rather than inverse image of P -- work this out!!!!!}

\begin{proof}[Proof of Theorem \ref{conj: exp path rules minuscule}]
We need to show that the automorphisms associated to each Stokes line are the same under the nonabelianization of Construction \ref{thm:bulk_map_exists}, and the exponential path rule version of non-abelianization are equal.

Firstly we consider primary Stokes lines.  We only need to consider one branch point.  We start with an $N$-local system $\cE$ on $X^{\circ}$, with associated $W$-local system $\tilde{X}^{\circ}$, which satisfies the $S$-monodromy condition. We use a path $\gamma_{p}$ with base point $x\in \gamma_{p}$ around a boundary circle $S_{p}^{1}\subset X^{\circ}$ (as in Construction \ref{construction: exponential path rule non abelianization}, where the direction is opposite to that specified by the orientation of $X$).  We pick a trivialization of $\cE|_{x}$, such that the monodromy around $\gamma_{p}$ is given by $n_{\alpha}^{-1}=I_\alpha(n_{\alpha, SL(2)}^{-1})$ (for some root $\alpha$).

We then have the monodromy of the associated $GL(V)$ local system $\rho(\cE)$ is given by 
\[\rho(n_{\alpha}^{-1}):\oplus_{\omega \in \Omega_{V}}V_{\omega}\rightarrow \oplus_{\omega\in \Omega_{V}}V_{\omega},\]
where $w$ ranges over the weights of the representation $V=\oplus_{\omega}V_{\omega}$.  By reducing the representation to an $SL(2)$ representation following Lemma \ref{lemma:  Properties of sl2 representations} we will have that the monodromy fixes the branches corresponding to the weight $0$ of $SL(2)$, while swapping the branches corresponding to the two weights of each irreducible two dimensional sub-representation.

We then have that 
\[\rho(n_{\alpha}^{-1})=\rho(-e_{\alpha})+\rho(-e_{-\alpha})+Id_{\oplus_{\omega| \omega(\alpha)=0}V_{\omega}},\]
which gives the identification $\rho(-e_{\alpha})=\sum d_{i,j}$ where we are summing over the pairs of branches specified immediately before Equation \ref{eq: path detour Stokes Factors}.

Hence 
\begin{equation}
\label{equation: path detour factor as exponential}S_{PD}=exp(\sum d_{i,j})=exp(\rho(-e_{\alpha}))=\rho(exp(-e_{\alpha})),\end{equation} for the line labelled by $\alpha$ with respect to the parallel transport of the chosen trivialization at $x_{p}$.   Similarly for the line labelled by $-\alpha$ we have agreement between the automorphisms specified as in Section \ref{subsubsection: Stokes factors}, and those specified in Construction \ref{construction: exponential path rule non abelianization}.

Hence for primary rays the exponential path detour rule, agrees with the assignment of Stokes factors in Construction \ref{construction of swkb}, Lemma \ref{lemma: monodromy to Stokes factors}, and Lemma \ref{lemma: N equivariant map to Stokes factors initial}.  For non-primary Stokes rays Lemma \ref{lem:assignment_stokes_newlines_equivariant} shows that the Stokes factors assigned by that lemma agree with those of the exponential path detour rules.  The result follows.
\end{proof}

\begin{remark}[Link to path detour rules of \cite{gaiotto2013spectral} and \cite{longhi2016ade}]
\label{remark: link to Longhi--Park path detour}
For $G$ of type $ADE$, and a minuscule representation $\rho$ this agrees with the path detour rules of \cite{longhi2016ade}, which in turn agrees with that of \cite{gaiotto2013spectral} for $G=SL(n)$, and $\rho$ the defining representation.

The path detour rules suggested in \cite{longhi2016ade} assigned to a primary Stokes line the automorphism 
\begin{align*}
S_{\alpha} 
=& 
Id+\sum d_{i,j} =\rho\big(I_{\alpha}(Id-e)\big) =\rho\Big(I_{\alpha}\big(\exp(-e)\big)\Big) =\exp\Big(\rho\big(I_{\alpha}(-e)\big)\Big)=
\\
=&
\exp\big(\sum d_{i,j}\big),
\end{align*}
where $e=\left(\begin{matrix}0 & 1\\ 0 & 0\end{matrix}\right)$ is the Chevalley basis element of $\mathfrak{sl}_{2}$, and we are using the results of Lemma \ref{lemma:  Properties of sl2 representations}, and the proof of Theorem \ref{conj: exp path rules minuscule}.  The term on the right is the automorphism assigned to the line by Construction \ref{construction: exponential path rule non abelianization}. 

Furthermore as in \cite{gaiotto2013spectral} we can also interpret the Stokes factor when we cross a primary Stokes line in the opposite direction to that specified by the orientation of $X$ in terms of the parallel transport along certain paths.  We have that
\[
S_{\alpha}^{-1}=\rho\big(\exp(e_{\alpha})\big)^{-1}=\rho\big(\exp(e_{\alpha})^{-1}\big)=\rho\big(\exp(-e_{\alpha})\big).
\]
At the Stokes lines for each pair of distinct preimages $x_{1}, x_{2}\in \overline{X}_{\rho}^{ne,\circ_{R}}|_{x}$ of $x_{s}$ on a Stokes line, interchanged by $s_{\alpha}$, with the property that $l(x_{1})-l(x_{2})=\alpha$ we let $f_{1,2}: \cE|_{x_{1}}\rightarrow \cE|_{x_{2}}$ be the morphism given by parallel transport along the lift of the path $\overline{\gamma_{p}}$ (ie. $\gamma_{p}$ with orientation reversed) to $\overline{X}_{\rho}^{ne,\circ_{R}}$ which starts at $x_{1}$ and finishes at $x_{2}$.
Hence we can interpret $S_{\alpha}^{-1}$ in terms of parallel transport along paths via the equation
\[S_{\alpha}^{-1}=Id+\sum_{l,k}f_{l,k},\]
%\todo{Check this precise statement!}
where the sum is over all pairs of distinct preimages specified above).  See Figure \ref{figure: Path Detour Rules 2} for a diagram.

\begin{figure}[h]
    \centering
    \def\svgwidth{100pt}
  %% Creator: Inkscape 1.0.1 (0767f8302a, 2020-10-17), www.inkscape.org
%% PDF/EPS/PS + LaTeX output extension by Johan Engelen, 2010
%% Accompanies image file 'DiagramPD211.pdf' (pdf, eps, ps)
%%
%% To include the image in your LaTeX document, write
%%   \input{<filename>.pdf_tex}
%%  instead of
%%   \includegraphics{<filename>.pdf}
%% To scale the image, write
%%   \def\svgwidth{<desired width>}
%%   \input{<filename>.pdf_tex}
%%  instead of
%%   \includegraphics[width=<desired width>]{<filename>.pdf}
%%
%% Images with a different path to the parent latex file can
%% be accessed with the `import' package (which may need to be
%% installed) using
%%   \usepackage{import}
%% in the preamble, and then including the image with
%%   \import{<path to file>}{<filename>.pdf_tex}
%% Alternatively, one can specify
%%   \graphicspath{{<path to file>/}}
%% 
%% For more information, please see info/svg-inkscape on CTAN:
%%   http://tug.ctan.org/tex-archive/info/svg-inkscape
%%
\begingroup%
  \makeatletter%
  \providecommand\color[2][]{%
    \errmessage{(Inkscape) Color is used for the text in Inkscape, but the package 'color.sty' is not loaded}%
    \renewcommand\color[2][]{}%
  }%
  \providecommand\transparent[1]{%
    \errmessage{(Inkscape) Transparency is used (non-zero) for the text in Inkscape, but the package 'transparent.sty' is not loaded}%
    \renewcommand\transparent[1]{}%
  }%
  \providecommand\rotatebox[2]{#2}%
  \newcommand*\fsize{\dimexpr\f@size pt\relax}%
  \newcommand*\lineheight[1]{\fontsize{\fsize}{#1\fsize}\selectfont}%
  \ifx\svgwidth\undefined%
    \setlength{\unitlength}{595.27559055bp}%
    \ifx\svgscale\undefined%
      \relax%
    \else%
      \setlength{\unitlength}{\unitlength * \real{\svgscale}}%
    \fi%
  \else%
    \setlength{\unitlength}{\svgwidth}%
  \fi%
  \global\let\svgwidth\undefined%
  \global\let\svgscale\undefined%
  \makeatother%
  \begin{picture}(1,1)%
    \lineheight{1}%
    \setlength\tabcolsep{0pt}%
    \put(0,0){\includegraphics[width=\unitlength,page=1]{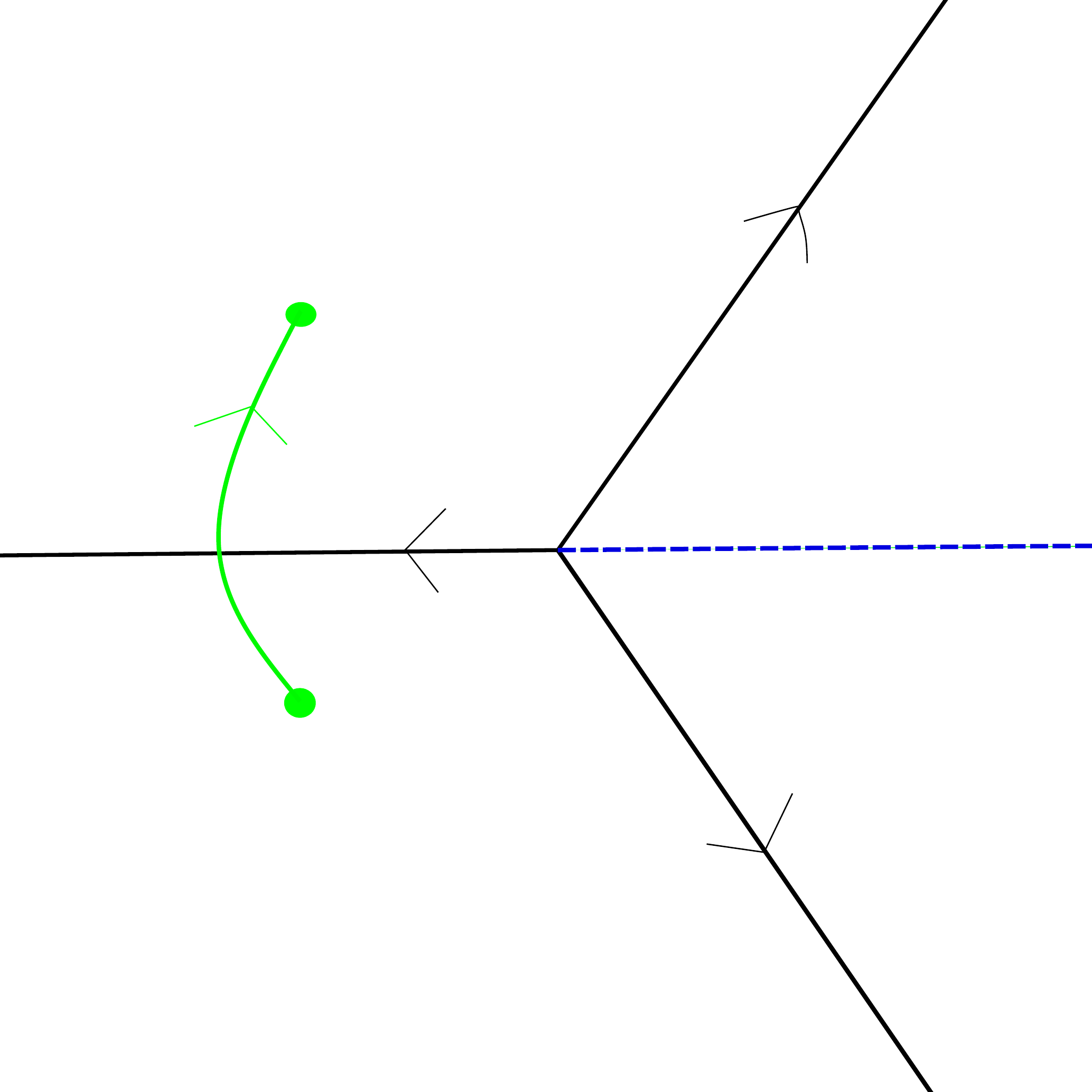}}%
    \put(0.83893912,0.20457219){\color[rgb]{0,0,0}\makebox(0,0)[lt]{\lineheight{1.25}\smash{\begin{tabular}[t]{l}$21$\end{tabular}}}}%
    \put(0.81972068,0.77463003){\color[rgb]{0,0,0}\makebox(0,0)[lt]{\lineheight{1.25}\smash{\begin{tabular}[t]{l}$21$\end{tabular}}}}%
    \put(0.04517432,0.56288187){\color[rgb]{0,0,0}\makebox(0,0)[lt]{\lineheight{1.25}\smash{\begin{tabular}[t]{l}$12$\end{tabular}}}}%
  \end{picture}%
\endgroup%

  \caption{Above we have a path crossing a line of an $SL(2)$ spectral network on $X$.  We have chosen a trivialization of the spectral cover away from a branch cut (dotted line), and this has allowed us to label the Stokes curves by roots, or equivalently by the ordered pairs $12$ or $21$.  The modified parallel transport along the path corresponds to the sum of the parallel transport along the three paths on the spectral cover shown below.  We have drawn the image of these paths in $X$, together with a labelling of the end points by the sheet of the spectral cover the end point lies on.}
    \begin{minipage}{.33\textwidth}
      \centering
     \def\svgwidth{100pt}
  %% Creator: Inkscape 1.0.1 (0767f8302a, 2020-10-17), www.inkscape.org
%% PDF/EPS/PS + LaTeX output extension by Johan Engelen, 2010
%% Accompanies image file 'DiagramPD212.pdf' (pdf, eps, ps)
%%
%% To include the image in your LaTeX document, write
%%   \input{<filename>.pdf_tex}
%%  instead of
%%   \includegraphics{<filename>.pdf}
%% To scale the image, write
%%   \def\svgwidth{<desired width>}
%%   \input{<filename>.pdf_tex}
%%  instead of
%%   \includegraphics[width=<desired width>]{<filename>.pdf}
%%
%% Images with a different path to the parent latex file can
%% be accessed with the `import' package (which may need to be
%% installed) using
%%   \usepackage{import}
%% in the preamble, and then including the image with
%%   \import{<path to file>}{<filename>.pdf_tex}
%% Alternatively, one can specify
%%   \graphicspath{{<path to file>/}}
%% 
%% For more information, please see info/svg-inkscape on CTAN:
%%   http://tug.ctan.org/tex-archive/info/svg-inkscape
%%
\begingroup%
  \makeatletter%
  \providecommand\color[2][]{%
    \errmessage{(Inkscape) Color is used for the text in Inkscape, but the package 'color.sty' is not loaded}%
    \renewcommand\color[2][]{}%
  }%
  \providecommand\transparent[1]{%
    \errmessage{(Inkscape) Transparency is used (non-zero) for the text in Inkscape, but the package 'transparent.sty' is not loaded}%
    \renewcommand\transparent[1]{}%
  }%
  \providecommand\rotatebox[2]{#2}%
  \newcommand*\fsize{\dimexpr\f@size pt\relax}%
  \newcommand*\lineheight[1]{\fontsize{\fsize}{#1\fsize}\selectfont}%
  \ifx\svgwidth\undefined%
    \setlength{\unitlength}{595.27559055bp}%
    \ifx\svgscale\undefined%
      \relax%
    \else%
      \setlength{\unitlength}{\unitlength * \real{\svgscale}}%
    \fi%
  \else%
    \setlength{\unitlength}{\svgwidth}%
  \fi%
  \global\let\svgwidth\undefined%
  \global\let\svgscale\undefined%
  \makeatother%
  \begin{picture}(1,1)%
    \lineheight{1}%
    \setlength\tabcolsep{0pt}%
    \put(0,0){\includegraphics[width=\unitlength,page=1]{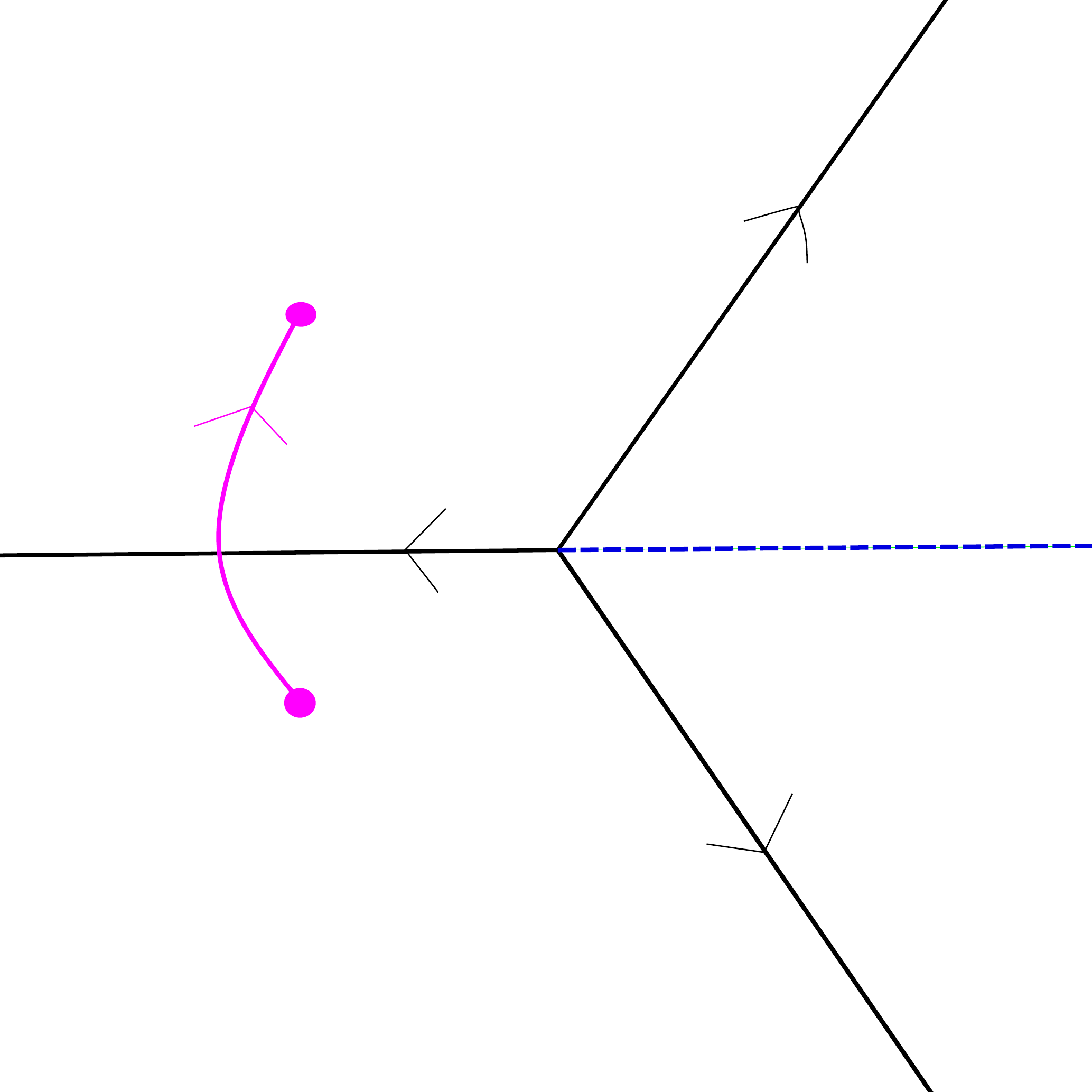}}%
    \put(0.83893912,0.20457219){\color[rgb]{0,0,0}\makebox(0,0)[lt]{\lineheight{1.25}\smash{\begin{tabular}[t]{l}$21$\end{tabular}}}}%
    \put(0.81972068,0.77463003){\color[rgb]{0,0,0}\makebox(0,0)[lt]{\lineheight{1.25}\smash{\begin{tabular}[t]{l}$21$\end{tabular}}}}%
    \put(0.04517432,0.56288187){\color[rgb]{0,0,0}\makebox(0,0)[lt]{\lineheight{1.25}\smash{\begin{tabular}[t]{l}$12$\end{tabular}}}}%
    \put(0.27004384,0.24212617){\color[rgb]{0,0,0}\makebox(0,0)[lt]{\lineheight{1.25}\smash{\begin{tabular}[t]{l}$1$\end{tabular}}}}%
    \put(0.28818087,0.74651041){\color[rgb]{0,0,0}\makebox(0,0)[lt]{\lineheight{1.25}\smash{\begin{tabular}[t]{l}$1$\end{tabular}}}}%
  \end{picture}%
\endgroup%

      \label{fig:path 1}
    \end{minipage}%
    \begin{minipage}{.33\textwidth}
      \centering
      \def\svgwidth{100pt}
  %% Creator: Inkscape 1.0.1 (0767f8302a, 2020-10-17), www.inkscape.org
%% PDF/EPS/PS + LaTeX output extension by Johan Engelen, 2010
%% Accompanies image file 'DiagramPD212.pdf' (pdf, eps, ps)
%%
%% To include the image in your LaTeX document, write
%%   \input{<filename>.pdf_tex}
%%  instead of
%%   \includegraphics{<filename>.pdf}
%% To scale the image, write
%%   \def\svgwidth{<desired width>}
%%   \input{<filename>.pdf_tex}
%%  instead of
%%   \includegraphics[width=<desired width>]{<filename>.pdf}
%%
%% Images with a different path to the parent latex file can
%% be accessed with the `import' package (which may need to be
%% installed) using
%%   \usepackage{import}
%% in the preamble, and then including the image with
%%   \import{<path to file>}{<filename>.pdf_tex}
%% Alternatively, one can specify
%%   \graphicspath{{<path to file>/}}
%% 
%% For more information, please see info/svg-inkscape on CTAN:
%%   http://tug.ctan.org/tex-archive/info/svg-inkscape
%%
\begingroup%
  \makeatletter%
  \providecommand\color[2][]{%
    \errmessage{(Inkscape) Color is used for the text in Inkscape, but the package 'color.sty' is not loaded}%
    \renewcommand\color[2][]{}%
  }%
  \providecommand\transparent[1]{%
    \errmessage{(Inkscape) Transparency is used (non-zero) for the text in Inkscape, but the package 'transparent.sty' is not loaded}%
    \renewcommand\transparent[1]{}%
  }%
  \providecommand\rotatebox[2]{#2}%
  \newcommand*\fsize{\dimexpr\f@size pt\relax}%
  \newcommand*\lineheight[1]{\fontsize{\fsize}{#1\fsize}\selectfont}%
  \ifx\svgwidth\undefined%
    \setlength{\unitlength}{595.27559055bp}%
    \ifx\svgscale\undefined%
      \relax%
    \else%
      \setlength{\unitlength}{\unitlength * \real{\svgscale}}%
    \fi%
  \else%
    \setlength{\unitlength}{\svgwidth}%
  \fi%
  \global\let\svgwidth\undefined%
  \global\let\svgscale\undefined%
  \makeatother%
  \begin{picture}(1,1)%
    \lineheight{1}%
    \setlength\tabcolsep{0pt}%
    \put(0,0){\includegraphics[width=\unitlength,page=1]{DiagramPD212.pdf}}%
    \put(0.83893912,0.20457219){\color[rgb]{0,0,0}\makebox(0,0)[lt]{\lineheight{1.25}\smash{\begin{tabular}[t]{l}$21$\end{tabular}}}}%
    \put(0.81972068,0.77463003){\color[rgb]{0,0,0}\makebox(0,0)[lt]{\lineheight{1.25}\smash{\begin{tabular}[t]{l}$21$\end{tabular}}}}%
    \put(0.04517432,0.56288187){\color[rgb]{0,0,0}\makebox(0,0)[lt]{\lineheight{1.25}\smash{\begin{tabular}[t]{l}$12$\end{tabular}}}}%
    \put(0.27004384,0.24212617){\color[rgb]{0,0,0}\makebox(0,0)[lt]{\lineheight{1.25}\smash{\begin{tabular}[t]{l}$2$\end{tabular}}}}%
    \put(0.28818087,0.74651041){\color[rgb]{0,0,0}\makebox(0,0)[lt]{\lineheight{1.25}\smash{\begin{tabular}[t]{l}$2$\end{tabular}}}}%
  \end{picture}%
\endgroup%

      \label{fig:path 2}
    \end{minipage}%
    \begin{minipage}{.33\textwidth}
      \centering
      \def\svgwidth{100pt}
  %% Creator: Inkscape 1.0.1 (0767f8302a, 2020-10-17), www.inkscape.org
%% PDF/EPS/PS + LaTeX output extension by Johan Engelen, 2010
%% Accompanies image file 'DiagramPD214.pdf' (pdf, eps, ps)
%%
%% To include the image in your LaTeX document, write
%%   \input{<filename>.pdf_tex}
%%  instead of
%%   \includegraphics{<filename>.pdf}
%% To scale the image, write
%%   \def\svgwidth{<desired width>}
%%   \input{<filename>.pdf_tex}
%%  instead of
%%   \includegraphics[width=<desired width>]{<filename>.pdf}
%%
%% Images with a different path to the parent latex file can
%% be accessed with the `import' package (which may need to be
%% installed) using
%%   \usepackage{import}
%% in the preamble, and then including the image with
%%   \import{<path to file>}{<filename>.pdf_tex}
%% Alternatively, one can specify
%%   \graphicspath{{<path to file>/}}
%% 
%% For more information, please see info/svg-inkscape on CTAN:
%%   http://tug.ctan.org/tex-archive/info/svg-inkscape
%%
\begingroup%
  \makeatletter%
  \providecommand\color[2][]{%
    \errmessage{(Inkscape) Color is used for the text in Inkscape, but the package 'color.sty' is not loaded}%
    \renewcommand\color[2][]{}%
  }%
  \providecommand\transparent[1]{%
    \errmessage{(Inkscape) Transparency is used (non-zero) for the text in Inkscape, but the package 'transparent.sty' is not loaded}%
    \renewcommand\transparent[1]{}%
  }%
  \providecommand\rotatebox[2]{#2}%
  \newcommand*\fsize{\dimexpr\f@size pt\relax}%
  \newcommand*\lineheight[1]{\fontsize{\fsize}{#1\fsize}\selectfont}%
  \ifx\svgwidth\undefined%
    \setlength{\unitlength}{595.27559055bp}%
    \ifx\svgscale\undefined%
      \relax%
    \else%
      \setlength{\unitlength}{\unitlength * \real{\svgscale}}%
    \fi%
  \else%
    \setlength{\unitlength}{\svgwidth}%
  \fi%
  \global\let\svgwidth\undefined%
  \global\let\svgscale\undefined%
  \makeatother%
  \begin{picture}(1,1)%
    \lineheight{1}%
    \setlength\tabcolsep{0pt}%
    \put(0,0){\includegraphics[width=\unitlength,page=1]{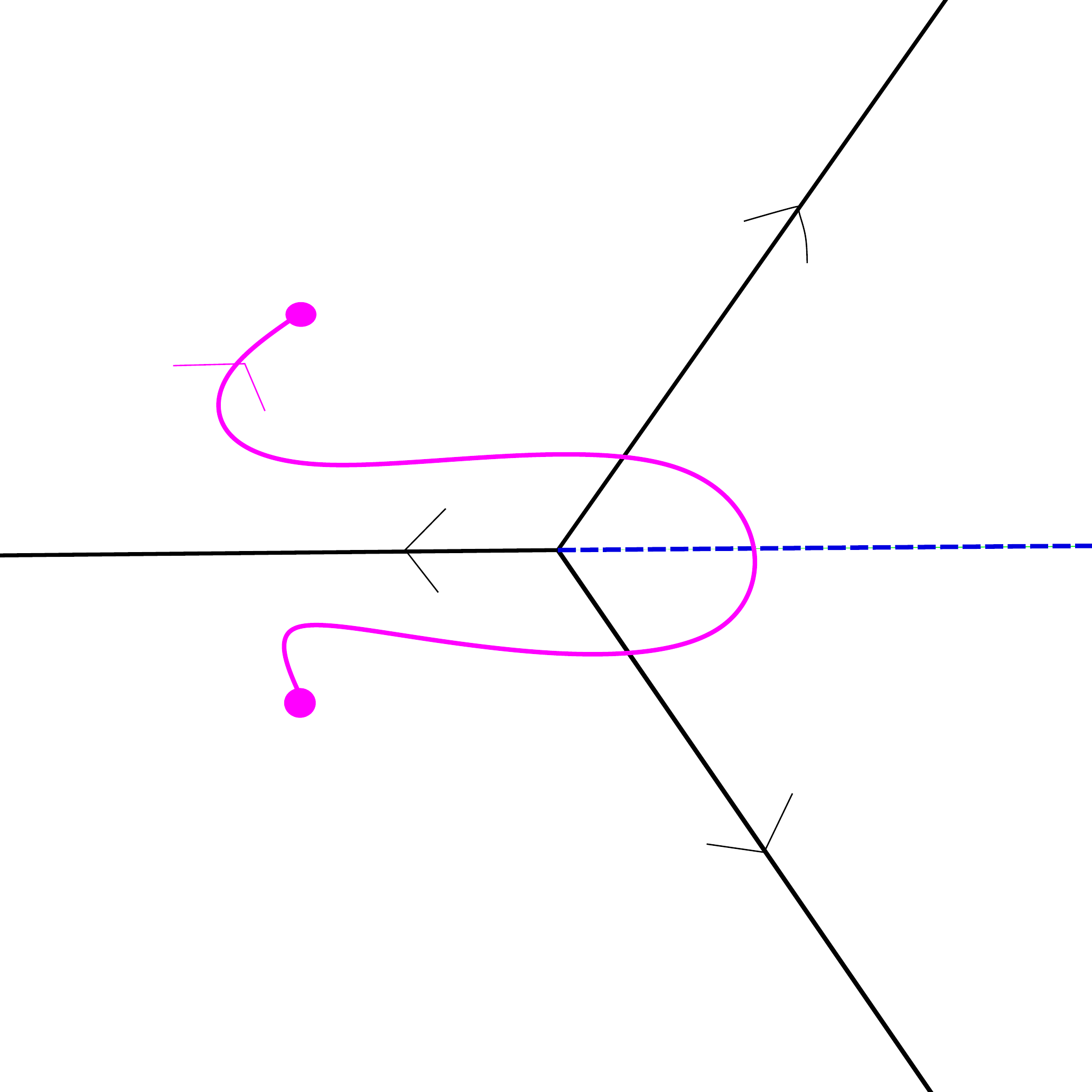}}%
    \put(0.84670744,0.11798931){\color[rgb]{0,0,0}\makebox(0,0)[lt]{\lineheight{1.25}\smash{\begin{tabular}[t]{l}$21$\end{tabular}}}}%
    \put(0.84132719,0.8455317){\color[rgb]{0,0,0}\makebox(0,0)[lt]{\lineheight{1.25}\smash{\begin{tabular}[t]{l}$21$\end{tabular}}}}%
    \put(0.03945904,0.54490197){\color[rgb]{0,0,0}\makebox(0,0)[lt]{\lineheight{1.25}\smash{\begin{tabular}[t]{l}$12$\end{tabular}}}}%
    \put(0.27161117,0.27571622){\color[rgb]{0,0,0}\makebox(0,0)[lt]{\lineheight{1.25}\smash{\begin{tabular}[t]{l}$1$\end{tabular}}}}%
    \put(0.28927943,0.77974376){\color[rgb]{0,0,0}\makebox(0,0)[lt]{\lineheight{1.25}\smash{\begin{tabular}[t]{l}$2$\end{tabular}}}}%
  \end{picture}%
\endgroup%

      \label{fig:path 12}
    \end{minipage}
  
    \label{figure: Path Detour Rules 2}
\end{figure}

%For $G=SL(n)$ and a basic WKB spectral network the exponential path rules agree with the path detour rules of \cite{gaiotto2013spectral}.  Recall that the path detour rules assigned to a primary Stokes line the automorphism $S_{\alpha}=Id + d_{1,2}$.  We then have the equality $Id+d_{1,2}=Id+e_{\alpha}=exp(e_{\alpha})$.  \todo{change notation for $detour$} (we are using Theorem \ref{conj: exp path rules minuscule} to make this statement).\todo{clarify the relation between $d_{1,2}$ and $e_{\alpha}$!  This involves some choices!}

%This does \emph{not} agree with the proposed modifications of the path detour rules in \cite{longhi2016ade}.  The proposed path detour rules of \emph{loc cit.} give a Stokes factor for a Stokes line labelled by $\alpha$ in some local trivialization that is not in the image under $\rho$ of the group \todo{name for the group} $U_{\alpha}:=exp(\fg_{\alpha})$.  Following the work of Boalch \cite{boalch2002g, boalch2001stokes} we expect Stokes factors in the groups $U_{\alpha}$.

%\todo{Is this true for $D_{n}$,  $E7$, $E6$  -- either confirm or find a root with $e_{\alpha}^{2}\neq 0$.  Wait: Path detour is correct for miniscule representations! -- it follows from Lemma 6.14.}
\end{remark}

%\todo{Include the fact that it is not a compatibility map but can be made an equality for defining representations of $GL(n)$, $SL(n)$ and $SO(2n)$. And in principal one should be able to use S6.3 to do this for E6, and E7.}

\subsubsection{Path Detour Non-abelianization with additional structure}
\label{subsec: equivalence for GL(n) and SL(n)}
For a specific group and a specific minisule representation we can modify $nonab_{PD}$ to apply to local systems with extra structure, of the sort considered in Section \ref{subsec: explicit descrtions for defining representations of classical groups} for defining representations of classical groups.  In this section we will use the results of Section \ref{subsec: explicit descrtions for defining representations of classical groups}.

We will here sketch this for $GL(n)$, $SL(n)$, and $SO(2n)$ for their defining representations.  This provides a precise relation to the nonabelianization of \cite{gaiotto2013spectral} for $GL(n)$ and $SL(n)$.  For $GL(n)$ this identifies $nonab$ and $nonab^{PD}$, while for $SL(n)$ and $SO(2n)$ it identifies $nonab$ with modifications $nonab^{SL(n)}_{PD}$ and $nonab^{SO(2n)}_{PD}$ of $nonab_{PD}$.  We expect analogous statements could be made for other groups and miniscule representations.

For $GL(n)$ Proposition \ref{prop: alpha monodromy GL(n)} shows that the commutative diagram of Theorem \ref{conj: exp path rules minuscule} collapses to the identification of maps via the following commutative diagram:
\begin{equation}
\label{eq:nonab_pathrule_identify_GL(n)}
\begin{tikzcd}
\Loc_{N}^{\tilde{X}^{\circ}, S}(X^{\circ}) \arrow{rr}{\cong} \arrow{rd}{nonab} & & \Loc_{\bbG_{m}}^{-1}(\overline{X}_{\rho}^{ne, \circ_{R}}) \arrow{ld}{nonab_{PD}}\\
 & \Loc_{GL(n)}(X) &
\end{tikzcd}
\end{equation}

For $G=SL(n)$ we can define $nonab_{PD}^{SL(n)}$ as the modification of $nonab_{PD}$ that maps local systems equipped with a trivialization of the determinant of the pushforward, that is points of $\Loc_{\bbG_{m}}(\overline{X}^{ne,\circ_{R}}_{\rho})\times_{\Loc_{\bbG_{m}}(X^{\circ})}\{\ \underline{\bbC}_{X^{\circ}}\}$ (see Proposition \ref{proposition : Reason for alpha monodromy condition} for a precise definition) to $SL(n)$ local systems.  This exists because $\Loc_{SL(n)}(X)\cong \Loc_{GL(n)}(X)\times_{\Loc_{GL(1)}}(X)\{\underline{\bbC}_{X}\}$, where the map $\Loc_{GL(n)}(X)\rightarrow \Loc_{GL(1)}(X)$ maps a local system $\cE$ to its determinant local system $\cE^{\wedge n}$, and all the operations in construction \ref{construction: exponential path rule non abelianization} for $SL(n)$ and its defining representation preserve this extra data.  Hence in this setting Equation \ref{eq:nonab_pathrule_commute} can be modified (using Proposition \ref{proposition : Reason for alpha monodromy condition}) to the identification of $nonab$ and $nonab_{PD}^{SL(n)}$ given by;
\begin{equation}
\label{eq:nonab_pathrule_identify_SL(n)}
\begin{tikzcd}
\Loc_{N}^{\tilde{X}^{\circ}, S}(X^{\circ}) \arrow{r}{\cong} \arrow{d}{nonab} & \Loc_{\bbG_{m}}^{-1}(\overline{X}_{\rho}^{ne, \circ_{R}}) \arrow{d}{nonab_{PD}^{SL(n)}}\\
\Loc_{SL(n)}\arrow{r}{\cong} & \Loc_{GL(n)}(X)\times_{\Loc_{GL(1)}(X)}\{\underline{\bbC}_{X}\}. 
\end{tikzcd}
\end{equation}

Similarly for $SO(2n)$ we can define $nonab_{PD}^{SO(2n)}$ as the modification of the $nonab_{PD}$ that preserves the structure of both a trivialization of the associated determinant bundle, and of a symmetric bilinear form on the bundle.  This then gives a commutative diagram   

\begin{equation}
\label{equation:  path detour SO(2n)}
\begin{tikzcd}
\Loc_{N}^{\tilde{X}^{\circ}, S}(X^{\circ})
 \arrow{rr}{\cong}\arrow{rd}{nonab} & & \Loc_{\bbG_{m}}^{enh, SO(2n)}(\overline{X}^{\circ})\arrow{ld}{nonab_{PD}^{SO(2n)}}
\\
 &  Loc_{SO(2n)}(X) &\\
\end{tikzcd}
\end{equation}
where we define $\Loc_{\bbG_{m}}^{enh, SO(2n)}(\overline{X}^{ne,\circ_{R}})$ to be the fiber product
\[ \Loc_{\bbG_{m}}(\overline{X}^{ne,\circ_{R}})
\times_{\Loc_{\bbG_{m}}(\overline{X}^{ne,\circ_{R}}/i))\times (\bbG_{m}/\bbG_{m})^{\# R_{\rho,1}}}\left(\big\{\cM_{\overline{X}^{ne,\circ_{R}}/i}\big\}\times (-1/\bbG_{m})^{\# R_{\rho,1}}\right)\]
from Proposition \ref{prop: spectral cameral SO(2n+1) with alpha} and we are using the notation of Section \ref{subsubsec: SO(2n)}.

\subsection{Explicit descriptions for defining representations of classical groups}
\label{subsec: explicit descrtions for defining representations of classical groups}

%\todo{Clear up:  when we are pushing forward using pi we need to use 1 dimensional vector spaces, not prinicpal c star local systems.  When we construct the N bundles we need to use the prinipcal local systems.  We need to be clear about what we're using where}

In this section we use spectral covers to provide explicit descriptions of the moduli space of $N$-local systems, with associated $W$-local system $\tilde{X}^{\circ}$ and satisfying the S-monodromy condition, for the classical groups $G=SL(n)$, $GL(n)$, $Sp(2n)$, $SO(2n)$, and $SO(2n+1)$.  Except for the S-monodromy condition this is analogous to calculations in \cite{hitchin1987stable, donagi1993decomposition}.   We can use these descriptions to replace the moduli space of $N$-local systems used in Construction \ref{thm:bulk_map_exists} with moduli spaces of certain $\bbG_{m}$ local systems on spectral covers.

%These are closely related to the computations in section \ref{subsec: faithful N representations}.  One important difference is that in some of these descriptions it is more convenient to use the wedge product rather than the tensor product of weight spaces.  This changes how we deal with changing the order of a set of characters, and means we use a modification of the lines bundles corresponding to $\cL_{K}$ in equation \ref{equation: The Local system we want!}.  Some additional simplifications include avoiding the modification of the moduli of local systems in definition \ref{definition: Loc'}, and removing the need to provide reductions of structure of the cameral cover at points $x_{p}$ ($p\in P$).  We provide more explicit proofs which are independent of section \ref{subsec: faithful N representations}.

Similarly to the case of the non-embedded spectral cover, we will write $\overline{X}^{\circ_{P}}=ReBl_{\pi^{-1}(P)}(\overline{X})$ for the oriented real blow up of the spectral cover at the preimage of the branch locus, and $\overline{X}^{\circ_{R}}=ReBl_{R_{\rho}}(\overline{X})$ for the oriented blow up along the ramification locus $R_{\rho}$ of $\overline{X}\rightarrow X$.

\subsubsection{The cases $G=SL(n)$, $G=GL(n)$.}

In these cases the Weyl group is $S_{n}$.  Hence by the inclusion of $n-1$ elements into $n$ elements, we get inclusions $j_{i}: S_{n-1}\hookrightarrow S_{n}$ for $i=1,...,n$.  Quotienting by the image of $j_i$ gives $n$ maps 
\[
\pi_{i}: \tilde{X}\rightarrow \overline{X}_{\rho}^{ne}
\]
where $\rho$ is the defining representation.

\begin{proposition}[\cite{donagi1993decomposition}]
There is an isomorphism $\overline{X}_{\rho}^{ne}\cong \overline{X}$ where $\overline{X}_{\rho}^{ne}$ is the non-embedded spectral cover associated to the defining representation of $SL(n)$ or $GL(n)$, and a cameral cover $\tilde{X}$ associated to a point $a \in \cA$ the Hitchin base.  The right hand side $\overline{X}$ is the spectral cover associated to $a\in \cA$ (see Definition \ref{defn: spectral cover}). 
\end{proposition}

\begin{proof}
Follows immediately from identifying the map $\ft\rightarrow \ft\sslash W$ as mapping an ordered set of $n$ complex numbers, to the coefficients of the polynomial having these numbers as roots.  Choose an identification $\ft\cong \bbC^{n}$ that identifies $W$ with $S_{n}$.  A projection $\bbC^{n}\rightarrow \bbC$ induces a map $\ft_{\cL}\rightarrow Tot(\cL)$ that identifies $\overline{X}^{ne}:=\tilde{X}/S_{n-1}\cong \overline{X}$.
\end{proof}

Suppose that $\cL$ is the locally constant system of one dimensional vector spaces associated to a $\bbG_{m}$-local system on $\overline{X}_{\rho}^{\circ_{P}}$.
The product of the local systems $\pi_{i}^{*}\cL$ specifies a $T_{GL(n)}$-local system on $\tilde{X}^{\circ}$ that we denote by $\oplus_{i=1}^{n}\pi_{i}^{*}\cL$. 

\begin{warning}
The map $\cL\mapsto \oplus_{i=1}^{n}\pi_{i}^{*}\cL$ is \emph{not} the correct map to use if one wants to consider the relation between spectral and cameral descriptions of Higgs bundles in the case where ramification is present.
\end{warning}

For $G=GL(n)$ the short exact sequence
\[1\rightarrow T_{GL(n)}\rightarrow N_{GL(n)}\rightarrow W_{GL(n)}\cong S_{n}\rightarrow 1\] splits.  
Furthermore $\oplus_{i=1}^{n}\pi_{i}^{*}\cL$ is clearly $W$-equivariant.  Hence $\oplus_{i=1}^{n}\pi_{i}^{*}\cL$ is a $W$-equivariant $T$-local system on $\tilde{X}^{\circ}$, and hence an $N_{GL(n)}\cong T_{GL(n)}\rtimes S_{n}$ shifted $T_{GL(n)}$ local system on $X^{\circ}$.

Since the above construction can be done in families, and the identification of Theorem \ref{thm:flat_DG} is the inverse of Construction \ref{construction: associated local system of vector bundles on spectral}, we obtain the following result:

\begin{proposition}[Spectral and Cameral descriptions of local systems for $GL(n)$, without S-monodromy condition. cf. \cite{donagi2002gerbe, donagi1993decomposition}]
\label{prop: spectral cameral GL(n)}
There is an isomorphism of stacks 
\[
\Loc_{N}^{\tilde{X}^{\circ}}(X^{\circ}) \xrightarrow{\cong}  \Loc_{\bbG_{m}}\big(\overline{X}^{\circ_{P}}\big),
\]
where $\overline{X}\cong \overline{X}_{\rho}^{ne}$ is the spectral cover corresponding to the cameral cover $\tilde{X}$.
The morphism is given by Construction \ref{construction: associated local system of vector bundles on spectral}.\footnote{While in general construction \ref{construction: associated local system of vector bundles on spectral} provides a local system of vector spaces, for this representation the vector spaces are one dimensional.}  The inverse is given by 
\[
\cL\mapsto \oplus_{i}\pi_{i}^{*}\cL,
\]
and Theorem \ref{thm:flat_DG}.
\end{proposition}

Before considering the S-monodromy condition at ramification points we will consider the analogous relation for $G=SL(n)$.

\begin{proposition}[Spectral and Cameral descriptions of local systems for $SL(n)$, without S-monodromy condition. cf. \cite{donagi1993decomposition}]
\label{prop: SL(n) without alpha momodromy}
Construction \ref{construction: associated local system of vector bundles on spectral} factors through an isomorphism of stacks 
\[
\Loc_{N_{SL(n)}}^{\tilde{X}^{\circ}}(X^{\circ})\xrightarrow{\cong} \Loc_{\bbG_{m}}\big(\overline{X}^{\circ_{P}}\big)\times_{\Loc_{\bbG_{m}}(X^{\circ})}\{\underline{\bbC}_{X^{\circ}}\}
\]
where the map $\Loc_{\bbG_{m}}(\overline{X}^{\circ_{P}})\rightarrow \Loc_{\bbG_{m}}(X^{\circ})$ corresponds to the map on one dimensional local systems $\cL\mapsto \det(\pi_{*}\cL)=(\pi_{*}\cL)^{\wedge n}$, and $\underline{\bbC}_{X^{\circ}}$ is the constant local system on $X^{\circ}$.
\end{proposition}

%\begin{warning}
%As we are using the wedge product to describe a local system on $X^{\circ}$, rather than the tensor product as in section \ref{subsubsec: enhanced moduuli space unramified spectral cover}, this description is slightly different to that provided in Theorem \ref{Claim: Description of N bundles in spectral terms}.
%\end{warning}

\begin{proof}
Firstly it is clear that the inclusion $N_{SL(n)}\hookrightarrow N_{GL(n)}$ gives rise to a commutative diagram, which defines the morphism $Sp_{\rho}^{enh}$ in this diagram:
\[
\begin{tikzcd}
\Loc_{N_{GL(n)}}^{\tilde{X}^{\circ}}(X^{\circ}) \arrow{r}{\cong} & \Loc_{\bbG_{m}}\big(\overline{X}^{\circ_{P}}\big)
\\
\Loc_{N_{SL(n)}}^{\tilde{X}^{\circ}}(X^{\circ})\arrow{r}{Sp_{\rho}^{enh}} \arrow{u} & \Loc_{\bbG_{m}}\big(\overline{X}^{\circ_{P}}\big)\times_{\Loc_{\bbG_{m}}(X^{\circ})}\{\bbC_{X}|_{X^{\circ}}\}\arrow{u}
\end{tikzcd}
\]
We have used that $\wedge^{n}\bbC^{n}$ is trivial as a representation of $G$, and hence in particular the induced $S_{n}$-action is trivial.  %To make a connection to section \ref{subsec: faithful N representations} this is saying that the line bundle $\cL_{K}$ defined in equation \ref{equation: The Local system we want!} is here trivial.

It is clear that $Sp_{\rho}$ factors through $Sp_{\rho}^{enh}$.  It remains to show that the morphism $Sp_{\rho}^{enh}$ is an isomorphism of stacks.  We do this by constructing an inverse, as in the $GL(n)$ case.

Consider applying the map $\cL\rightarrow \oplus_{i=1}^{n}\pi_{i}^{*}\cL$, to a line bundle $\cL$ corresponding to a point of $\Loc_{\bbG_{m}}\big(\overline{X}^{\circ_{P}}\big)\times_{\Loc_{\bbG_{m}}(X^{\circ})}\{\bbC_{X}|_{X^{\circ}}\}$.  In this case we have the extra data of a trivialization $\otimes_{i=1}^{n}\pi_{i}^{*}\cL\xrightarrow{\cong}\underline{\bbC}_{\tilde{X}^{\circ}}$. %\todo{Explain why this is the determinant bundle! -- actually this is very clear!} 

Note that $\otimes_{i=1}^{n}\pi_{i}^{*}\cL$ is the image of $\oplus_{i=1}^n \pi_{i}^{*}\cL$ in the map $\Loc_{T_{GL(n)}}\rightarrow \Loc_{\bbG_{m}}$ induced by the product map $T\cong (\bbG_{m})^{n}\rightarrow \bbG_{m}$.  This fits into a short exact sequence:
\[
1\rightarrow T_{SL(n)}\rightarrow T_{GL(n)}\rightarrow \bbG_{m}\rightarrow 1,
\]
providing an identification:
\[
\Loc_{T_{SL(n)}}(\tilde{X}^{\circ})\cong \Loc_{T_{GL(n)}}(\tilde{X}^{\circ})\times_{\Loc_{\bbG_{m}}(\tilde{X}^{\circ})}\{\underline{\bbC}_{\tilde{X}^{\circ}}\}.
\]
Hence $\oplus_{i=1}^{n}\pi_{i}^{*}\cL$ determines a $T_{SL(n)}$-local system $\cE_{T_{SL(n)}}\rightarrow \tilde{X}$.  Furthermore if we denote $\cE_{T_{GL(n)}}$ the associated $T_{GL(n)}$ local system (ie. $\oplus_{i=1}^{n}\pi_{i}^{*}\cL$), we have a commutative diagram:
\[
\begin{tikzcd}
N_{SL(n)}\arrow{r}\arrow{d} & N_{GL(n)}\arrow{d}\\
\Aut_{X^{\circ}}(\cE_{T_{SL(n)}})\arrow{r} & \Aut_{X^{\circ}}(\cE_{T_{GL(n)}}),
\end{tikzcd}
\]
formed by the factoring of the map 
\[N_{SL(n)}\rightarrow N_{GL(n)}\rightarrow \Aut_{X}(\cE_{T_{GL(n)}}).\]
This provides the inverse of the map of Construction \ref{construction: associated local system of vector bundles on spectral}, which completes the proof.
\end{proof}

We now consider the effect of the S-monodromy condition. On the spectral curve side, this gives local systems which extend from $\bar X^{\circ_{P}}$ to $\bar X^{\circ_{R}}$.

\begin{proposition}[Spectral and Cameral descriptions of local systems for $SL(n)$ with S-monodromy condition.]
\label{proposition : Reason for alpha monodromy condition}
Let $\tilde{X}\rightarrow X$ be a smooth cameral cover. %\todo{as in? -- conglomerate assumptions somewhere}.

There is an isomorphism of stacks
\[
\Loc_{N_{SL(n)}}^{\tilde{X}^{\circ},S}(X^{\circ})\xrightarrow{\cong} \Loc_{\bbG_{m}}(\overline{X}^{\circ_{R}})\times_{\Loc_{\bbG_{m}}(X^{\circ})}\{\ \underline{\bbC}_{X^{\circ}}\}
\]
provided by Construction \ref{construction: associated local system of vector bundles on spectral}.
\end{proposition}

\begin{proof}
Consider a boundary circle $S^{1}_{p}$, for $p\in P$.  We then have that on the right hand side this is equivalent to restricting to $N$-local systems whose monodromy is of the form:
\[
\left(\begin{array}{cc|c} 0 & a & 0\\ -1/a & 0 & 0 \\ \hline 0 & 0 & Id_{n-2}\end{array}\right)
\]
where we are using a basis for the fiber at $p$ corresponding to the branches of the spectral cover.

The non-identity part of this corresponds to the two branches which are ramified above $p\in P$.  By the assumption that the cameral cover is smooth there is precisely one pair of branches which are ramified.

This restriction is the S-monodromy condition, and was in fact the motivation for the definition of the S-monodromy condition.
\end{proof}

\begin{proposition}[Spectral and Cameral descriptions of local systems for $GL(n)$ with S-monodromy condition]
\label{prop: alpha monodromy GL(n)}
Construction \ref{construction: associated local system of vector bundles on spectral} gives an isomorphism of stacks:
\[
\Loc_{N_{GL(n)}}^{\tilde{X}^{\circ}, S}(X^{\circ})\xrightarrow{\cong}\Loc_{\bbG_{m}}(\overline{X}^{\circ_{R}})\times_{(\bbG_{m}/\bbG_{m})^{\# R_{\rho}}}(\{-1\}/\bbG_{m})^{\# R_{\rho}},
\]
where the map 
\[
\Loc_{\bbG_{m}}(\overline{X}^{\circ_{R}})\rightarrow (\bbG_{m}/\bbG_{m})^{\# R_{\rho}}
\] is given by restricting to the preimage of the ramification divisor $R_{\rho}$ under the map $\overline{X}^{\circ_{R}}\rightarrow \overline{X}$.
\end{proposition}
 
 \begin{proof}
Under the map $N_{SL(2)}\xhookrightarrow{i} N_{GL(2)}$, we can identify the image in $N_{GL(2)}$ of the non-identity coset of $T_{SL(2)}$ in $N_{SL(2)}$ as elements of the non-identity coset of $T_{GL(2)}$ in $N_{GL(2)}$ whose square is $-Id$.
 
That is to say
\[i(N_{SL(2)}\backslash T_{SL(2)})=\{a\in N_{GL(2)}\backslash T_{GL(2)}|a^{2}=-Id\},\]
where in both cases the symbol ``$\backslash$'' refers to the complement rather than the quotient.

The result follows by applying this to the preimage of the ramification divisor $R_{\rho}$ under the map $\overline{X}^{\circ_{R}}\rightarrow \overline{X}$, and applying the considerations of the proof of Proposition \ref{proposition : Reason for alpha monodromy condition}.
 \end{proof}
 
\subsubsection{The case $G=Sp(2n)$}

Here $T\cong \bbG_{m}^{n}$, and $W\cong S_{n}\ltimes \{\pm 1\}^{n}$. %\todo{check Weyl group} 

Firstly we can identify the non-embedded spectral covers and the spectral cover in this case.  The weights of the defining representation consist of a single $W$-orbit.

\begin{proposition}[Non-embedded and embedded spectral covers for $Sp(2n)$, \cite{donagi1993decomposition}]
\label{prop: Non embedded and embedded spectral covers for Sp2n}
There is an isomorphism $\overline{X}_{\rho}^{ne}\cong \overline{X}$ where $\overline{X}_{\rho}^{ne}$ is the non-embedded spectral cover associated to the defining representation of $Sp(2n)$, and a cameral cover $\tilde{X}$ associated to a point $a \in \cA$ the Hitchin base.  The right hand side $\overline{X}$ is simply the spectral cover associated to $a\in \cA$.
\end{proposition}

\begin{proof}
Identify $\ft\cong \bbC^{n}$ in a way that identifies the Weyl group with $S_{n}\ltimes \{\pm 1\}^{n}$.  Projecting onto one copy of $\bbC$ gives a map $\ft_{\cL}\rightarrow Tot(\cL)$ that provides the required identification.
\end{proof}

There is an involution $i:\overline{X}\rightarrow \overline{X}$, coming from the involution $x\mapsto -x$ in $\bbC$.  The induced involution of $Tot(\cL)$ preserves the spectral curve $\overline{X}\hookrightarrow Tot(\cL)$ because if $\lambda$ is an eigenvalue of $Sp(2n)$, $-\lambda$ is also an eigenvalue.  This induces involutions of $\overline{X}^{\circ_{P}}$ and $\overline{X}^{\circ_{R}}$ that we also denote by $i$.

We can consider the quotient map $\overline{X}^{\circ_{P}}\xrightarrow{p_{i}} \overline{X}^{\circ_{P}}/i$.

\begin{proposition}[Spectral and Cameral Descriptions for $Sp(2n)$ without S-monodromy condition. Cf. \cite{hitchin1987stable, donagi1993decomposition}]
\label{prop: spectral cameral SP(2n) no alpha}
Construction \ref{construction: associated local system of vector bundles on spectral} can be enhanced to an isomorphism of stacks
\begin{equation}
\label{equation:  map for Sp(2n) no boundary}
\begin{tikzcd}
\Loc_{N}^{\tilde{X}^{\circ}}(X^{\circ})
\arrow{d}{\cong}
\\
\Loc_{\bbG_{m}}\big(\overline{X}^{\circ_{P}}\big)\times_{\Loc_{\bbG_{m}}(\ \overline{X}^{\circ_{P}}/i)}\{\underline{\bbC}_{\overline{X}^{\circ_{P}}/i}\},
\end{tikzcd}
\end{equation}
where the map \[
\Loc_{\bbG_{m}}(\overline{X}^{\circ_{P}})\rightarrow \Loc_{\bbG_{m}}(\overline{X}^{\circ_{P}}/i)
\]
corresponds to the map $\cL\mapsto \det((p_{i})_{*}\cL)$ (again thinking of the local system of one dimensional vector spaces associated to a $\bbG_{m}$ local system). %The local system $\tilde{X}^{\circ}\times_{\{Stab(\{\alpha, -\alpha\})\}}(V_{\alpha}\otimes V_{-\alpha})$ comes from seeing $\tilde{X}^{\circ}\rightarrow \overline{X}^{\circ_{P}}/i$ as an $Stab(\{\alpha, -\alpha\})$-bundle (where  $Stab(\{\alpha, -\alpha\})\subset W$ is the subgroup of elements that preserve the set $\{\alpha, -\alpha\}$, rather than preserving individual elements of the set), and taking the associated line bundle corresponding to the action of $Stab(\{\alpha, -\alpha\})$ on the tensor product of weight spaces $V_{\alpha}\otimes V_{-\alpha}$.
\end{proposition}

\begin{proof}

We can identify $\tilde{X}^{\circ}\times_{Stab(\{\alpha, -\alpha\})}(V_{\alpha}\wedge V_{-\alpha})\cong \underline{\bbC}_{\overline{X}^{ \circ_{P}}/i}$ (here we consider $\tilde{X}^{\circ}\rightarrow \overline{X}^{\circ_{P}}/i$ as an $Stab(\{\alpha, -\alpha\})$-bundle) because $Stab(\{\alpha, -\alpha\})\subset W$, the stabilizer of the unordered set of the pair of weights $\{\alpha, -\alpha\}$ of the defining representation of $Sp(2n)$, acts trivially on $V_{\alpha}\wedge V_{-\alpha}$.  %This is a modification of construction \ref{equation: The Local system we want!}, where we are using a wedge product rather than a tensor product.

It is automatic that construction \ref{construction: associated local system of vector bundles on spectral} factors through a map as in Equation \ref{equation:  map for Sp(2n) no boundary}. %\todo{EXPLAIN THIS-- actually I think this is clear!} 

We now construct an inverse.  Let $\cL\rightarrow \overline{X}^{\circ_{P}}$ be a locally constant line bundle on $\overline{X}^{\circ P}$ corresponding to a point of $\Loc_{\bbG_{m}}\big(\overline{X}^{\circ_{P}}\big)\times_{\Loc_{\bbG_{m}}(\ \overline{X}^{\circ_{P}}/i)}\{\underline{\bbC}_{\overline{X}^{\circ_{P}}/i}\}$.  %Let $\cL_{K}$ denote the line bundle $\tilde{X}^{\circ}\times_{\{Stab(\{\alpha, -\alpha\})\}}(V_{\alpha}\otimes V_{-\alpha})$ on $\overline{X}^{\circ_{P}}/i)$.

Consider the set $A:=\coprod_{-n\leq i\leq n, i\neq 0}\bbC_{i}$ consisting of the disjoint union of $2n$ copies of $\bbC$.  We can identify homomorphisms from $A\rightarrow A$ that consist of;
\begin{itemize}
    \item a permutation $\sigma$ of $\{-n,...,-1,1,...,n\}$,
    \item linear maps $\bbC_{i}\rightarrow \bbC_{\sigma(i)}$ for $i\in\{-n,...,-1,1,...,n\}$,
\end{itemize}
with $N_{GL(2n)}$ -- the normalizer of the torus in $GL(2n)$. 

Consider the isomorphisms $\bbC_{i}\wedge \bbC_{-i}\xrightarrow{\cong} \bbC$ that sends $1\wedge 1$ to $1$ for $i>0$ (and $1\wedge 1\mapsto -1$ for $i<0$). 

Consider the subset of these homomorphisms where:
\begin{itemize}
    \item The permutation satisfies $\sigma\in S_{n}\ltimes \{\pm 1\}^{n}$ (identified as a subgroup of $S_{2n}$ via $W_{Sp(2n)}\hookrightarrow W_{GL(2n)}$).
    \item For all $1\leq i\leq n$ the following diagram commutes:
\[
\begin{tikzcd}
\bbC_{i}\wedge \bbC_{-i} \arrow{r}{\cong}\arrow{d} &  \bbC\arrow{d}{Id}\\
\bbC_{\sigma(i)}\wedge \bbC_{\sigma(-i)} \arrow{r}{\cong} & \bbC
\end{tikzcd}
\]
\end{itemize}
We can identify these with $N_{Sp(2n)}$.

Consider the locally constant sheaf on $X^{\circ}$ associated to $\cL\in \Loc_{\bbG_{m}}\big(\overline{X}^{\circ_{P}}\big)\times_{\Loc_{\bbG_{m}}(\ \overline{X}^{\circ_{P}}/i)}\{\underline{\bbC}_{\overline{X}^{\circ_{P}}/i}\}$ denoted by
$\Hom^{Sp}(A\times X^{\circ},\cL)$, where over a contractible open set $U\subset X^{\circ}$ an element $b\in \Hom^{Sp}(A\times X^{\circ},\cL)$ consists of the data of:
\begin{itemize}
    \item A morphism $b_{1}:\{-n,...,-1,1,...,n\}\times U\rightarrow \overline{X}^{\circ_{P}}|_{U}$ of sheaves of sets over $U$ that is surjective and intertwines $i$ and the involution sending $j\mapsto -j$.  Denote by $b_{1,K}$ the map $\{1,...,n\}\times U\rightarrow (X^{\circ, P}/i)|_{U}$ induced by $b_{1}$
    \item Locally constant isomorphisms $\underline{\bbC_{j}}\xrightarrow{b_{2,j}}\cL|_{b_{1}(j\times U)}$ (considered as sheaves on $U$), such that the following diagram commutes:
\[
\begin{tikzcd}
\cL|_{b_{1}(\{j\}\times U)}\wedge \cL|_{b_{1}(\{-j\}\times U)} \arrow{r}{\cong} & \underline{\bbC}_{\overline{X}^{\circ_{P}}/i}|_{b_{1,K}\{j\}\times U}\\
\underline{\bbC_{j}}_{U}\wedge \underline{\bbC_{-j}}_{U} \arrow{r} \arrow{u}{b_{2,j}\wedge b_{2,-j}} & \underline{\bbC}_{U}\arrow{u}{\cong}
\end{tikzcd}
\]
\end{itemize}

Then $\Hom^{Sp}(A\times X^{\circ}, \cL)$ is a $N_{Sp(2n)}$ local system, where the $N_{Sp(2n)}$ action comes from precomposing with the homomorphisms $A\rightarrow A$ specified above.  Furthermore it is clear that the associated $W$ bundle is $\tilde{X}^{\circ}$, and that the map from $\cL$ to $\Hom^{Sp}(A\times X^{\circ},\cL)$ is an inverse to  that of Equation \ref{equation:  map for Sp(2n) no boundary}.
%\todo{Do one run over the following argument}

%We have $n$ distinguished inclusions:
%\[%
%S_{n-1}\ltimes \{\pm 1\}^{n-1}\hookrightarrow S_{n}\ltimes \{\pm 1\}^{n}=W_{Sp(2n)}.
%\]
%Quotienting by each of these gives $n$ maps $\pi_{i}: \tilde{X}\rightarrow \overline{X}$.
%Given a $\bbG_{m}$-local system $\cL\rightarrow \overline{X}$, we can construct a $T$-local system $\cE_{T}\rightarrow \tilde{X}$ as: 
%\[
%\cE_{T}:=\oplus_{i=1}^{n} \pi_i^*\cL_{i}.
%\]

%\todo{I am NOT happy with the rest of this.  What you need to do is }

%\todo{Define $Aut_{X}(\cL)$ as the set of automorphisms of the spectral cover and line bundle, that preserve the involution, and the trivialization of the induced line bundle on $\overline{X}/i$}

%We define $Aut_{X}(\cL)$ as the sheaf (over $X$) of groups of pairs $\{(s, \gamma)|s:\overline{X}\cong \overline{X}, s\textrm{ preserves intertwines with inovolution }i: \overline{X}\rightarrow \overline{X}, \gamma:\cL\cong s^{*}\cL \}$ There is a short exact sequence 
%\[1\rightarrow \pi_{*}Aut_{\overline{X}^{\circ_{P}}}(\cL)\rightarrow  Aut_{X^{\circ}}(\cL)\rightarrow W\rightarrow 1.\]

%We now need to construct a morphism $N_{Sp(2n)}\rightarrow \Aut_{X\backslash P}(\cE_{T})$.  Note that there is a morphism: 
%\[
%\Aut_{\overline{X}^{\circ_{P}}}(\cL)\rightarrow \Aut_{X^{\circ}}(\cE_{T}),
%\]
%and hence a morphism:
%\[
%N_{Sp(2n)}\rightarrow \Aut_{\overline{X}^{\circ_{P}}}(\cL)\rightarrow \Aut_{X^{\circ}}(\cE_{N}).
%\]
%This is clearly injective, hence it is an isomorphism.
\end{proof}

We now consider the case where we impose the S-monodromy condition.

\subsubsection{Description of Branch points and associated divisors for $Sp(2n)$}
\label{subsubsection: Describing branch points sp(2n)}

Under the assumption the cameral cover is smooth there are two different types of branch points, in the sense that the local geometry of the spectral cover differs.   The first is on the locus where we have a eigenvalue $\lambda_{i}=0=-\lambda_{i}$.  The other is where we have $\lambda_{i}=\lambda_{j}\neq 0$, for $i\neq j$.  Hence we also have $-\lambda_{i}=-\lambda_{j}$.  Hence there are two ramification points in $\overline{X}$ over the corresponding point $p\in P$ where this behaviour occurs.  

%\todo{pictures of the two types of branch points, with blurb saying "$\alpha$-condition imposes/doesn't impose an extra restriction in this case"}

Denote by $R_{\rho, 0}$ the subdivisor of the ramification divisor (note that the ramification divisor is effective) corresponding to the first type of branch point.  Denote by $R_{\rho, 1}$ a reduced subdivisor of the ramification divisor such that $R_{\rho,1}\coprod i(R_{\rho, 1})$ is the subdivisor of the ramification divisor corresponding to branch points of the second type, and such that as sets $R_{\rho,1}\cap i(R_{\rho,1})=\emptyset$.  This means that $R_{\rho}=R_{\rho,0}\coprod R_{\rho,1}\coprod i(R_{\rho,1})$.

\begin{proposition}[Spectral and Cameral Descriptions for $Sp(2n)$ with S-monodromy condition.]
\label{propL Spectral and Cameral descriptions for Sp(2n) with monodromy condition}
For $G=Sp(2n)$ there is an isomorphism of stacks, induced by Construction \ref{construction: associated local system of vector bundles on spectral}:
\[
\begin{tikzcd}
\Loc_{N_{Sp(2n)}}^{\tilde{X}, S}(X^{\circ})
\arrow{d}{\cong}
\\
\Loc_{\bbG_{m}}(\overline{X}^{\circ_{R}})\times_{\Loc_{\bbG_{m}}(\overline{X}^{\circ_{R}}/i)\times (\bbG_{m}/\bbG_{m})^{\# R_{\rho,1}}}
\left(\big\{\underline{\bbC}_{\overline{X}^{\circ_{R}}/i}\big\}\times (\{-1\}/\bbG_{m})^{\# R_{\rho,1}}\right)
\end{tikzcd}
\]
where the map to $(\bbG_{m}/\bbG_{m})^{R_{\rho,1}}$ corresponds to restricting to the preimage of the divisor $R_{\rho,1}$ under the map $\overline{X}^{\circ_{R}}\rightarrow \overline{X}$.
\end{proposition}

\begin{proof}
By considering the $SL(2)\cong Sp(2)$ case we see that the monodromy around a ramification point corresponding to the first type of branch point (corresponding to the divisor $R_{\rho,0}$) is automatically satisfied, and does not need to be imposed as is done in Proposition \ref{proposition : Reason for alpha monodromy condition}. % Hence for this type of branch point the $SL(n)$ version of this result (proposition \ref{proposition : Reason for alpha monodromy condition}) shows that this proposition is correct when only branch points of the first type are considered.

In the case of branch points of the second type, restricting to the one of the two ramification points in the preimage which is in $R_{\rho,1}$ we see that we need to impose that the monodromy around this point is $-1$, for exactly the same reasons as in the $GL(n)$ case (Proposition \ref{prop: alpha monodromy GL(n)}).  It is clear that if we impose this condition as in the proposition at $r\in R_{\rho, 1}$, we automatically get the imposition of this condition at $i(r)$.  The result follows.
\end{proof}

\subsubsection{The case $G=SO(2n)$}
\label{subsubsec: SO(2n)}Recall that for $G=SO(2n)$, a maximal torus is $T\cong \bbG_{m}^{n}$, and the Weyl group is $W\cong S_{n}\ltimes H_{n}$, where $H_{n}$ is the kernel in the short exact sequence
\[1\rightarrow H_{n}\rightarrow \{\pm 1\}^{n}\xrightarrow{p} \{\pm 1\}\rightarrow 1,\]
where the map $p$ corresponds to taking the product.

The spectral cover and the embedded spectral cover are here different \cite{donagi1993decomposition} (cf. \cite{hitchin1987stable}).  The underlying reason is that the coefficients of the characteristic polynomial do not give a basis of $SO(2n)$-invariant polynomials on $\mathfrak{so}(2n)$.  One can resolve this by replacing the determinant with the Pfaffian, to get a basis of invariant polynomials.

\begin{proposition}[Non embedded and embedded spectral covers for $SO(2n)$ \cite{donagi1993decomposition}, \cite{hitchin1987stable}]
\label{prop: spectral covers for SO(2n)}
For $G=SO(2n)$ and a smooth cameral cover $\tilde{X}$, the spectral cover $\overline{X}$ associated to the same point in the Hitchin base has singularities on the intersection of the zero section of $Tot(\cL)$ and the vanishing locus of the Pfaffian.%\todo{What about if the intersection with the vanishing locus of the Pfaffian is a fat point?} -- would contradict smoothness of the cameral cover.

There is a map $\overline{X}^{ne}\rightarrow \overline{X}$ which corresponds to the normalization of these singularities (at the intersection of $\overline{X}$ with the zero section of $\text{Tot}(\cL)$). 

%\todo{Write a more careful proof -- note in particular that the precise argument in Donagi is Wrong}
\end{proposition}

\begin{proof}
We first consider the case $n>1$ where the Weyl group acts transitively on the weights.  Pick an identification $\ft\cong \bbC^{n}$ that identifies the action of the Weyl group with the action of $S_{n}\ltimes H_{n}$.  The projection $p:\ft\rightarrow \bbC$ to one copy of $\bbC$ gives a map $\ft_{\cL}\rightarrow \cL$, which maps $\tilde{X}\rightarrow \overline{X}$.  This factors as $\tilde{X}\rightarrow \tilde{X}/(S_{n-1}\ltimes H_{n-1})\rightarrow \overline{X}$.

Consider the locus where the Pfaffian vanishes.  The locus in $\ft$ where the Pfaffian vanishes corresponds to the ramification locus of the map $\ft\sslash W\rightarrow \ft\sslash (S_{n}\ltimes \{\pm 1\}^{n})$.  As the characteristic polynomial is given by $p(\lambda)=\lambda^{2n}+a_{2}\lambda^{2n-2}+....+a_{2n-2}\lambda^{2}+p_{f}^{2}$, where $p_{f}$ corresponds to the Pfaffian, we have that the spectral curve is singular on the intersection of the zero section of $Tot(\cL)$ with the vanishing locus of the Pfaffian \cite{hitchin1987stable}.

A smooth cameral cover can either intersect the locus where the Pfaffian vanishes at a point which does not lie on a reflection hyperplane $H_{\alpha}$, or can intersect the locus where the Pfaffian vanishes at a point which also lies on a single root hyperplane.  The restriction that the cameral cover is smooth means that in either case the spectral cover has a singularity at this point in $x$ (at the intersection of the spectral cover with the zero section of $\text{Tot}(\cL)$), and the map from the non-embedded spectral cover $\tilde{X}/(S_{n-1}\ltimes H_{n-1})\rightarrow \overline{X}$ is locally the resolution of this singularity.  Away from the vanishing locus of the Pfaffian we get an isomorphism $\tilde{X}/(S_{n-1}\ltimes H_{n-1})\rightarrow \overline{X}$.  See \cite{donagi1993decomposition} for more details on these accidental singularities.
%\todo{This is FALSE!}Noting that a smooth cameral cover intersects the locus where the Pfaffian vanishes transversely, we have that the cameral cover at these points is non-singular, and the map $\tilde{X}/(S_{n-1}\ltimes H_{n-1})$ hence corresponds to resolving the nodal singularities on the intersection of the zero section of $Tot(\cL)$ with the vanishing locus of the Pfaffian.  Away from this locus we get an isomorphism $\tilde{X}/(S_{n-1}\ltimes H_{n-1})\rightarrow \overline{X}$.  See \cite{donagi1993decomposition} for more details on these accidental singularities.

For the case $n=1$ we have that the Weyl group is trivial, and $\overline{X}^{ne}=X\coprod X$.  The result follows immediately. 

%Note that on the locus where the Pfaffian vanishes to first order, the determinant of the characteristic polynomial vanishes to second order.  This means that (as the constant coefficient of the characteristic polynomial is zero), we have that the characteristic polynomial is singular here (see \cite{hitchin1987stable}). 

%Away from the zero section $X\hookrightarrow Tot(\cL)$ we have that this gives an isomorphism $\overline{X}^{ne}$ and $\overline{X}$.  However along the zero section $\overline{X}^{ne}$ is non singular while $\overline{X}$ has nodal singularities.  The result follows.  See \cite{donagi1993decomposition} for more discussion about when these ``accidental" singularities occur in spectral covers.
\end{proof}

As in the case of $Sp(2n)$ the involution $\bbC\rightarrow \bbC$, $x\mapsto -x$ induces an involution $Tot(\cL)\rightarrow Tot(\cL)$.  This involution preserves the spectral curve and hence descends to an involution $i:\overline{X}\rightarrow \overline{X}$.  This induces involutions of  $\overline{X}^{ne}$, $\overline{X}^{ne, \circ_{R}}$, and $\overline{X}^{ne, \circ_{P}}$ that we also denote by $i$. We have a map $p_{i}: \overline{X}^{ne,\circ_{P}}\rightarrow \overline{X}^{ne,\circ_{P}}/i$, and a map $\pi_{i}:\overline{X}^{ne,\circ_{P}}/i\rightarrow X$

\begin{proposition}[Spectral and Cameral Descriptions for $SO(2n)$, $n>1$, without S-monodromy condition]
\label{proposition: Spectral Cameral SO(2n) without S monodromy condition}
For $G=SO(2n)$, $n>1$, Construction \ref{construction: associated local system of vector bundles on spectral} can be enhanced to an isomorphism of stacks
\begin{equation}
\label{equation:  map for SO(2n) no boundary}
\begin{tikzcd}
\Loc_{N}^{\tilde{X}^{\circ}}(X^{\circ})
\arrow{d}{\cong}
\\
\Loc_{\bbG_{m}}(\overline{X}^{ne, \circ_{P}})
\times_{\Loc_{\bbG_{m}}(\overline{X}^{ne,\circ_{P}}/i)}
\big\{\cM_{\overline{X}^{ne,\circ_{P}}/i}\big\},
\end{tikzcd}
\end{equation}
where the map 
\[
\Loc_{\bbG_{m}}(\overline{X}^{ne,\circ_{P}})\rightarrow
\Loc_{\bbG_{m}}(\overline{X}^{ne,\circ_{P}}/i))
\]
corresponds to the map (of local systems of one dimensional vector spaces) $\cL\mapsto \det\big((p_{i})_{*}\cL\big)$, and the local system $\cM_{\overline{X}^{ne,\circ_{P}}/i}$ on $\overline{X}^{ne,\circ_{P}}/i$ is defined by $\cM_{\overline{X}^{ne,\circ_{P}}/i}:=\tilde{X}^
{\circ}\times_{Stab(\{\alpha,-\alpha\})}(\bbC)$, where $\alpha$ is one of the projections $\alpha:\ft\rightarrow \bbC$ used in the proof of Proposition \ref{prop: spectral covers for SO(2n)}, and $Stab(\{\alpha,-\alpha\})$  acts on $\bbC$ via the map $Stab(\{\alpha,-\alpha\})\rightarrow \{\pm 1\}$ with kernel $Stab(\alpha)$.

%For the construction of this map we fix an isomorphism $(\pi_{i})_{*}\cM_{\overline{X}^{ne,\circ_{P}}/i}\cong \underline{\bbC}_{X^{\circ}}$, such that at each point $x\in X^{\circ}$ this is induced by some trivialization of the $\bbZ/2$-bundle $\tilde{X}^{\circ}|_{x}\rightarrow \tilde{X}^{\circ}|_{x}/i$.   This trivialization is needed in order to make the correspondence of Equation \ref{equation:  map for SO(2n) no boundary} canonical.
%\todo{You need to explain what alpha is here -- presumably it's a coordinate function as before?}
\end{proposition}

\begin{proof}
%Define the $\bbG_{m}$-local system $\cM_{\overline{X}^{\circ_{P}}/i}$ on $\overline{X}^{ne,\circ_{P}}/i$ by $\cM_{\overline{X}^{\circ_{P}}/i}:=\overline{X}^{ne,\circ_{P}}\times_{\bbZ/2}(\bbC)$, where $\bbZ/2$ acts on $\bbC$ by the isomorphism $\bbZ/2\cong \{\pm 1\}$.  

Consider the $\bbZ/2$ bundle on $X^{\circ}$, such that it's fiber at $x\in X$ is the tensor product over $\bbZ/2$ of the bundles $\overline{X}^{ne,\circ_{P}}|_{x'}$ for each lift $x'$ of $x$ to $\overline{X}^{ne,\circ_{P}}/i$.  This is canonically trivial, as it is isomorphic to $\tilde{X}\times_{W} \{\pm 1\}$, with $W$ acting trivially.  More precisely we can pick the identity section as corresponding to the collection of all the positive weights.

This induces an isomorphism $((\pi_{i})_{*}\cM_{\overline{X}^{ne,\circ_{P}}/i})^{\wedge n}\cong \underline{\bbC}_{X^{\circ}}$.

%Firstly note that the $\bbZ/2$-bundle formed as the tensor product, over the abelian group $\bbZ/2$, of the $\bbZ/2$

%Firstly note that such an isomorphism $(\pi_{i})_{*}\cM_{\overline{X}^{ne,\circ_{P}}/i}\cong \underline{\bbC}_{X}$ always exists, because the monodromy group of $\tilde{X}$ is in $S_{n}\ltimes H_{n}$ (as opposed to $S_{n}\ltimes \{\pm 1\}^{n}$).

 Secondly we note that $\tilde{X}^{\circ}\times_{Stab(\{\alpha,-\alpha\})}(V_{\alpha}\wedge V_{-\alpha})\cong \cM_{\overline{X}^{ne,\circ_{P}}/i}$, where $\alpha$ is one of the weights of the defining representation of $SO(2n)$, and we consider $\tilde{X}^{\circ}\rightarrow \overline{X}^{
ne, \circ_{P}}/i$ as a $Stab(\{\alpha,-\alpha\})$-bundle.  This is because $Stab(\{\alpha,-\alpha\})$ acts by $\pm 1$ on $V_{\alpha}\wedge V_{-\alpha}$ (with $Stab(\alpha)$ acting trivially).  Hence it is clear that the morphism of Construction \ref{construction: associated local system of vector bundles on spectral} factors through a morphism of the form of Equation \ref{equation:  map for SO(2n) no boundary}, where we pick the isomorphism $\tilde{X}\times_{Stab\{\alpha,-\alpha\}}\{\pm 1\}\cong \cM_{\overline{X}^{ne,\circ_{P}}/i}$ so that this is compatible with the chosen identification in the following sense:  Let $\cL$ be the local system of one dimensional vector spaces on $\overline{X}^{\circ}$ produced by Construction  \ref{construction: associated local system of vector bundles on spectral}.  We then can pick an isomorphism $((p_{i})_{*}\cL)^{\wedge 2}\rightarrow \cM_{\overline{X}^{\circ_{P}}/i}$, such that the following diagram commutes:
\begin{equation*}
\begin{tikzcd}
((\pi_{i})_{*}(((p_{i})_{*}\cL)^{\wedge 2}))^{\wedge n}\arrow{r}{\cong}\arrow{d}{\cong} & ((\pi_{i})_{*}\cM_{\overline{X}^{ne,\circ_{P}}/i})^{\wedge n}\arrow{d}{\cong}\\
(\pi_{*}\cL)^{\wedge 2n} \arrow{r}{\cong} & \underline{\bbC}_{X} ,
\end{tikzcd}
\end{equation*}
where the top arrow comes from the isomorphism picked, the right arrow comes from the chosen isomorphism $(\pi_{i})_{*}\cM_{\overline{X}^{ne,\circ_{P}}/i}\cong \underline{\bbC}_{X}$, and the bottom arrow is immediate as we started with an $N_{SO(2n)}$-local system.

Consider the set $B:=\coprod_{-n\leq i\leq n, i\neq 0}\bbC_{i}$ consisting of the disjoint union of $2n$ copies of $\bbC$.  Consider homomorphisms from $B\rightarrow B$ that consist of the data:
\begin{itemize}
    \item A permutation $\sigma$ of $\{-n,...,-1,1,...,n\}$. 
    \item Linear maps $\bbC_{i}\rightarrow \bbC_{\sigma(i)}$ for $i\in \{-n,...,-1,1,...,n\}$.
\end{itemize}
   We can identify these homomorphisms with $N_{GL(2n)}$ -- the normalizer of the torus in $GL(2n)$.  
   
Consider the isomorphisms $\bbC_{i}\wedge \bbC_{-i}\xrightarrow{\cong} \bbC$ that sends $1\wedge 1$ to $1$ for $i>0$ (and $1\wedge 1\mapsto -1$ for $i<0$). Consider the subset of these homomorphisms where:

\begin{itemize}
    \item The permutation $\sigma$ satisfies $\sigma\in S_{n}\ltimes H_{n}$ (identified as a subgroup of $S_{2n}$ via $W_{SO(2n)}\hookrightarrow W_{GL(2n)}$).
    \item For all $1\leq i\leq n$, the following diagram commutes;
\[
\begin{tikzcd}
\bbC_{i}\wedge \bbC_{-i} \arrow{r}{\cong}\arrow{d} &  \bbC\arrow{d}{\pm 1}\\
\bbC_{\sigma(i)}\wedge \bbC_{\sigma(-i)} \arrow{r}{\cong} & \bbC
\end{tikzcd}
\]
where the map $\bbC\rightarrow \bbC$ is $+1$ if $\sigma(i)>0$, and $-1$ otherwise. % We get the overall factor is in SO and not O by the condition that the element in the Weyl group is in the subgroup H_{n} of $(\pm 1)^{n}$.
\end{itemize}
%\[
%\begin{tikzcd}
%\wedge^{2n}(\oplus_{-n\leq i \leq n, i\neq 0}\bbC_{i})\arrow{rd}\arrow{dd} & \\
% & \bbC\\
%\wedge^{2n}(\oplus_{-n\leq i \leq n, i\neq 0}\bbC_{\sigma(i)}) \arrow{ru} & \\
%\end{tikzcd}
%\]

We can identify the homomorphisms satisfying these conditions with $N_{SO(2n)}$.

Consider the locally constant sheaf on $X^{\circ}$ associated to $\cL\in \Loc_{\bbG_{m}}(\overline{X}^{ne, \circ_{P}})
\times_{\Loc_{\bbG_{m}}(\overline{X}^{ne,\circ_{P}}/i)}
\big\{\cM_{\overline{X}^{ne,\circ_{P}}/i}\big\}$ denoted by
$\Hom^{SO(2n)}(B\times X^{\circ},\cL)$, where over a contractible open set $U\subset X^{\circ}$ an element $b\in \Hom^{SO(2n)}(B\times X^{\circ},\cL)$ consists of the data of;
\begin{itemize}
    \item A morphism $b_{1}: \{-n,...,-1,1,...,n\}\times U\rightarrow \overline{X}^{ne, \circ_{P}}|_{U}$ of sheaves of sets over $U$ that is surjective and intertwines $i$ and the involution sending $j\mapsto -j$.  Denote by $b_{1,K}$ the map $\{1,...,n\}\times U\rightarrow (X^{\circ_P}/i)|_{U}$ induced by $b_{1}$
    \item Locally constant isomorphisms $\underline{\bbC_{j}}\xrightarrow{b_{2,j}}\cL|_{b_{1}(j\times U)}$ (considered as sheaves on $U$)
    %\item Morphisms $b_{3,j}:\cM_{\overline{X}^{ne, \circ_{P}}/i}|_{b_{1,K}^{-1}(\{j\}\times U)}\rightarrow \underline{\bbC}_{U}$, that are induced from a morphism of principal $\bbZ/2$-bundles on $b_{1,K}^{-1}(\{j\}\times U)\subset \overline{X}^{ne,\circ_{P}}/i$ from $\overline{X}^{ne,\circ_{P}}|_{b_{1,K}^{-1}(\{j\}\times U)}\xrightarrow{\cong} (b_{1,K}^{-1}(\{j\}\times U)) \times \bbZ/2$
\end{itemize} 
such that;
\begin{itemize}
    \item The following diagram commutes:

\[
\begin{tikzcd}
\cL|_{b_{1}(\{j\}\times U)}\wedge \cL|_{b_{1}(\{-j\}\times U)} \arrow{r} & \cM_{\overline{X}^{ne, \circ_{P}}/i}|_{b_{1,K}\{j\}\times U}\\
\underline{\bbC_{j}}_{U}\wedge \underline{\bbC_{-j}}_{U} \arrow{r} \arrow{u}{b_{2,j}\wedge b_{2,-j}}& \underline{\bbC}_{U}\arrow{u}{b_{3,j}}
\end{tikzcd}
\]
where the morphism $b_{3,j}:\underline{\bbC}_{U}\rightarrow \cM_{\overline{X}^{ne, \circ_{P}}/i}|_{b_{1,K}(\{j\}\times U)}$, is the morphism induced from the morphism of principal $\bbZ/2$-bundles on $b_{1,K}(\{j\}\times U)\subset \overline{X}^{ne,\circ_{P}}/i$ given by $\overline{X}^{ne,\circ_{P}}|_{b_{1,K}(\{j\}\times U)}\xleftarrow{\cong} (b_{1,K}(\{j\}\times U)) \times \{j,-j\}\cong (b_{1,K}^{-1}(\{j\}\times U)) \times \bbZ/2$ (where in the final map we identify $j$ with the identity $1_{\bbZ/2}$).
\item We have an isomorphism $\bigwedge^{2n}\left(\oplus_{-n<j<n, j\neq 0}b_{2,j}\right)=Id$, where we are using the identification $(\pi_{*}\cL)^{\wedge 2n}\cong ((\pi_{i})_{*}\cM_{\overline{X}^{ne,\circ_{P}}/i})^{\wedge n}\cong \bbC$.
\end{itemize}

%\todo{More seriously cut down from $N_{SO(2n+1)}$ to $N_{SO(2n)}$}
%\todo{1.  Notation $\pi_{K}$ is S3 notation -- at the very least explain what I'm doing!.  Secondly using the trivialization of $\cL_{K}$ does not seem to make clear why how swapping the order of $\alpha,-\alpha$ introduces a negative sign!}
%\[
%\begin{tikzcd}
%\wedge^{2n}(\pi_{*}\cL|_{U}) \arrow{r}{\cong} \arrow{d} & \wedge^{n}(\pi_{K*}\cL_{K}|_{U})\arrow{d}\\
%\wedge^{2n}(\oplus_{-n\leq i \leq n; i\neq 0}\underline{\bbC_{i}}_{U}) \arrow{r} & \underline{\bbC}_{U}
%\end{tikzcd}
%\]

Then $\Hom^{SO(2n)}(B\times X^{\circ},\cL)$ is an $N_{SO(2n)}$ local system, where the $N_{SO(2n)}$ action comes from precomposing with the homomorphisms $B\rightarrow B$ specified above.  The map $\cL\mapsto \Hom^{SO(2n)}(B\times X,\cL)$ is then the inverse to the map of Equation \ref{equation:  map for SO(2n) no boundary}.
\end{proof}

We now need to consider the S-monodromy condition.  The proof is the same as that for $Sp(2n)$.

We define the divisor $R_{\rho,1}\subset \overline{X}^{ne}$ as in Section \ref{subsubsection: Describing branch points sp(2n)}.  Unlike in Section \ref{subsubsection: Describing branch points sp(2n)} we do not need to define $R_{\rho,0}$, because as $\overline{X}^{ne}\rightarrow \overline{X}$ is the (algebraic) blow up $\overline{X}$ along this locus, there is not a corresponding ramification locus in $\overline{X}^{ne}$ (under the assumption of smoothness of the cameral cover).

\begin{proposition}[Spectral and Cameral Descriptions for $SO(2n)$, $n>1$, with S-monodromy condition.]
\label{propL Spectral and Cameral descriptions for SO(2n) with monodromy condition}
For $G=SO(2n)$, $n>1$, there is an isomorphism of stacks, induced by Construction \ref{construction: associated local system of vector bundles on spectral}, 
\[
\begin{tikzcd}
\Loc_{N_{SO(2n)}}^{\tilde{X}^{\circ}, S}(X^{\circ})
\arrow{d}{\cong}
\\
\Loc_{\bbG_{m}}(\overline{X}^{ne,\circ_{R}})
\times_{\Loc_{\bbG_{m}}(\overline{X}^{ne,\circ_{R}}/i))\times (\bbG_{m}/\bbG_{m})^{\# R_{\rho,1}}}\left(\big\{\cM_{\overline{X}^{ne,\circ_{R}}/i}\big\}\times (-1/\bbG_{m})^{\# R_{\rho,1}}\right)
\end{tikzcd}
\]
where the map to $(\bbG_{m}/\bbG_{m})^{R_{\rho,1}}$ corresponds to the preimage of $R_{\rho,1}$ under the morphism $\overline{X}^{ne,\circ_{R}}\rightarrow \overline{X}^{ne}$, and $\cM_{\overline{X}^{ne,\circ_{R}}/i}:=\overline{X}^{ne,\circ_{R}}\times_{\{\pm 1\}}\bbC$, where we consider $\overline{X}^{ne,\circ_{R}}\rightarrow \overline{X}^{ne,\circ_{R}}/i$ as a $\{\pm 1\}$-bundle.  Note that $\cM_{\overline{X}^{ne,\circ_{R}}/i}|_{\overline{X}^{ne,\circ_{P}}/i}\cong \cM_{\overline{X}^{ne,\circ_{P}}/i}$ (defined in Proposition \ref{proposition: Spectral Cameral SO(2n) without S monodromy condition}).
\end{proposition}

\begin{proof}
The analysis is identical to that of Proposition \ref{propL Spectral and Cameral descriptions for Sp(2n) with monodromy condition}, with the exception of the fact we do not need to consider $R_{\rho,0}$ as mentioned above.
\end{proof}

\begin{example}[$G=SO(2)$]
We now consider the case where $G=SO(2)\cong \bbG_{m}$.  Hence in this case the spectral description is simply the equivalence 
\[\Loc_{\bbG_{m}}(X)\cong (\Loc_{\bbG_{m}}(X)\times \Loc_{\bbG_{m}}(X))\times_{\Loc_{\bbG_{m}}(X)}\{\underline{\bbC}_{X}\},\]
where the map $\Loc_{\bbG_{m}}(X)\times \Loc_{\bbG_{m}}(X)\rightarrow \Loc_{\bbG_{m}}(X)$ corresponds to the map $(\cL_{1}, \cL_{2})\mapsto \cL_{1}\otimes \cL_{2}$.
\end{example}

\subsubsection{The case $G=SO(2n+1)$} Recall that for $G=SO(2n+1)$, a maximal torus is $T\cong \bbG_{m}^{n}$, and the Weyl group is $W\cong S_{n}\ltimes \{\pm 1\}^{n}$.

The weights of the defining representation of $SO(2n+1)$ form two $W$-orbits.  There is a partial normalization of the spectral cover $\overline{X}_{0}\coprod \overline{X}_{1}\rightarrow \overline{X}_{\rho}$, where $\overline{X}_{0}$ corresponds to the eigenvalue $0$ and corresponds to the zero section $X\hookrightarrow Tot(\cL)$.

We can identify the disjoint union of the two $W$-orbits of the spectral with the non-embedded spectral cover:

\begin{proposition}
There is an isomorphism
\[\overline{X}_{1} \coprod \overline{X}_{0} \xrightarrow{\cong} \overline{X}_{\rho}^{ne}.\]
\end{proposition}

\begin{proof}
We use the decomposition of Equation \ref{equation: decomposition into W orbits}.  Denoting the orbits by $0$, and $1$ it is immediate that $\overline{X}^{ne, 0}\cong \overline{X}_{0}$.  The identification of $\overline{X}^{ne,1}\cong \overline{X}_{1}$ is identical to the proof of Proposition \ref{prop: Non embedded and embedded spectral covers for Sp2n}.
\end{proof}

There is an involution $i:\overline{X}^{ne}_{\rho}\rightarrow \overline{X}^{ne}_{\rho}$, coming from the involution $\bbC\rightarrow \bbC$ ($x\mapsto -x$), because the eigenvalues of a matrix in $SO(2n+1)$ are always of the form $\{\lambda_{1}, -\lambda_{1}, \lambda_{2}, -\lambda_{2},...,\lambda_{n},-\lambda_{n},0\}$.  This involution fixes $\overline{X}_{0}$.  We can hence reuse the arguments of the case for $Sp(2n)$ to get the result.  

Let $R_{\rho}=R_{a}\coprod R_{b}$ where $R_{a}\subset \overline{X}_{1}$, and $R_{b}\subset \overline{X}_{0}$ be the ramification locus of the map $\overline{X}^{ne}_{\rho}\rightarrow X$.  We denote by $\overline{X}_{1}^{\circ_{R}}$ and $\overline{X}_{0}^{\circ_{R}}$ the oriented real blow ups of $\overline{X}_{1}$ and $\overline{X}_{0}$ at $R_{a}$ and $R_{b}$ respectively. As in the $Sp(2n)$ case we have two types of branch points, the first corresponding to the locus where an eigenvalue (and hence two eigenvalues) are zero $\lambda_{i}=-\lambda_{i}$ for some $i$, and the second where two eigenvalues are equal and non-zero, informally; $\lambda_{i}=\lambda_{j}\neq 0$ for $i\neq j$.  As in the $Sp(2n)$ case we decompose $R_{a}=R_{\rho,0}\coprod R_{\rho,1}\coprod i(R_{\rho,1})$.

We denote by $p_{i}$ the following map corresponding to quotienting by $i$: $p_{i}: \overline{X}^{ne,\circ_{R}}\rightarrow \overline{X}^{ne,\circ_{R}}/i=\overline{X}_{0}^{\circ}\coprod \overline{X}_{1}^{\circ}/i$.  This is a $2:1$ cover over $\overline{X}_{1}^{\circ}/i$, and an isomorphism on $\overline{X}_{0}^{\circ}$.

\begin{proposition}[Spectral and Cameral Descriptions for $SO(2n+1)$ with S-monodromy condition. Cf. \cite{hitchin1987stable, donagi1993decomposition}]
\label{prop: spectral cameral SO(2n+1) with alpha}
For $G=SO(2n+1)$ there is an isomorphism of stacks, induced by Construction \ref{construction: associated local system of vector bundles on spectral}: 
\[
\begin{tikzcd}
\Loc_{N_{SO(2n+1)}}^{\tilde{X}^{\circ}, S}(X^{\circ})
\arrow{d}{\cong}
\\
\Loc_{\bbG_{m}}(\overline{X}^{ne,\circ_{R}})\times_{Loc_{\bbG_{m}}(\overline{X}^{ne,\circ_{R}}/i)\times (\bbG_{m}/\bbG_{m})^{\# R_{\rho, 1}}}\left(\{\cM\}\times (-1/\bbG_{m})^{\# R_{\rho,1}}\right),
\end{tikzcd}
\]
where the map 
%\textcolor{orange}{We may need to replace the constant sheaf on $\overline{X}/i$.  This is a question: how does $s_{\alpha}$ act on $V_{\alpha}\otimes V_{-\alpha}$ -- I don't know this!}

\[
\Loc_{\bbG_{m}}(\overline{X}^{ne,\circ_{R}})\rightarrow \Loc_{\bbG_{m}}(\overline{X}^{ne,\circ_{R}}/i)
\]
corresponds to the map (on local systems of one dimensional vector spaces) $\cL\mapsto \det((p_{i})_{*}\cL)$, and the map to $(\bbG_{m}/\bbG_{m})^{\# R_{\rho, 1}}$ corresponds to taking the monodromy around each point in $R_{\rho, 1}$.

The $\bbG_{m}$-local system $\cM$ is defined by $\cM|_{\overline{X}_{0}^{\circ_{R}}}:= \tilde{X}^{\circ}\times_{W}V_{0}$, where $V_{0}$ is the representation of $W$ corresponding to the zero weight space of the defining representation of $SO(2n+1)$, and $\cL|_{\overline{X}_{1}^{\circ_{R}}/i}:=\tilde{X}^
{\circ}\times_{Stab(\{\alpha,-\alpha\})}(\bbC)$, where $Stab(\{\alpha,-\alpha\})$  acts on $\bbC$ via the map $Stab(\{\alpha,-\alpha\})\rightarrow \{\pm 1\}$ with kernel $Stab(\alpha)$.%where this (modulo the difference in Weyl groups) should be interpreted in the same way as in proposition \ref{prop: spectral cameral SP(2n) no alpha}.
\end{proposition}

\begin{proof}
The description of $\cM|_{\overline{X}_{1}^{\circ_{R}}/i}$ is for the exact same reason as in the $SO(2n)$ case of Propositions \ref{proposition: Spectral Cameral SO(2n) without S monodromy condition} and \ref{propL Spectral and Cameral descriptions for SO(2n) with monodromy condition}.

%Note that $\underline{\bbC}|_{\overline{X}_{1}^{\circ_{R}}/i}\cong \tilde{X}^{\circ}\times_{Stab(\{\alpha,-\alpha\})}(V_{\alpha}\otimes V_{-\alpha})$, where we are seeing $\tilde{X}^{\circ}\rightarrow \overline{X}_{1}^{\circ_{P}}$ as a $Stab(\{\alpha,-\alpha\})$-bundle, where $\alpha$ is a non-zero weight of the defining representation of $SO(2n+1)$.

%\todo{This is no longer acceptable -- Here I should be using tensors, whereas in $Sp(2n)$ case I am using wedges!  Check that the argument for $S$-monodromy along $R_{\rho,0}$ is NOT affected.}

The analysis of $\overline{X}_{1}^{\circ_{R}}$ is now identical to that in Propositions \ref{propL Spectral and Cameral descriptions for Sp(2n) with monodromy condition} and \ref{prop: spectral cameral SP(2n) no alpha}.  The local system on $\overline{X}_{0}^{\circ_{R}}\cong X^{\circ}$ is then given by:
\[\tilde{X}^{\circ}\times_{W}V_{0},\]
where $V_{0}$ is the representation of $W$ on the zero-weight space of the defining representation of $SO(2n+1)$.  This is the representation 
\[S_{n}\ltimes \{\pm 1\}^{n}\rightarrow \{\pm 1\}\acts \bbC,\]
where the map corresponds to taking the product of the $\pm 1$ factors, and the action on $\C$ is by multiplication.

Consider the set $C:=\coprod_{-n\leq i\leq n, i\neq 0}\bbC_{i}$ consisting of the disjoint union of $2n$ copies of $\bbC$.  Consider homomorphisms from $C\rightarrow C$ that consist of the data:
\begin{itemize}
    \item A permutation $\sigma$ of $\{-n,...,-1,1,...,n\}$. 
    \item Linear maps $\bbC_{i}\rightarrow \bbC_{\sigma(i)}$ for $i\in \{-n,...,-1,1,...,n\}$.
\end{itemize}
   
Consider the isomorphisms $\bbC_{i}\wedge \bbC_{-i}\xrightarrow{\cong} \bbC$ that sends $1\wedge 1$ to $1$ for $i>0$ (and $1\wedge 1\mapsto -1$ for $i<0$).  Consider the subset of these homomorphisms where

\begin{itemize}
    \item The permutation $\sigma$ satisfies $\sigma\in S_{n}\ltimes \{\pm 1\}^{n}$ (identified as a subgroup of $S_{2n}$ via the map $W_{SO(2n+1)}\hookrightarrow W_{GL(2n+1)}$ factoring through $W_{GL(2n)}$).
    \item For all $1\leq i\leq n$ the following diagrams commutes;
\[
\begin{tikzcd}
\bbC_{i}\wedge \bbC_{-i} \arrow{r}{\cong}\arrow{d} &  \bbC\arrow{d}{\pm 1}\\
\bbC_{\sigma(i)}\wedge \bbC_{\sigma(-i)} \arrow{r}{\cong} & \bbC
\end{tikzcd}
\]
where the map $\bbC\rightarrow \bbC$ is $+1$ if $\sigma(i)>0$, and $-1$ otherwise.
\end{itemize}
%\[
%\begin{tikzcd}
%\wedge^{2n}(\oplus_{-n\leq i \leq n, i\neq 0}\bbC_{i})\arrow{rd}\arrow{dd} & \\
% & \bbC\\
%\wedge^{2n}(\oplus_{-n\leq i \leq n, i\neq 0}\bbC_{\sigma(i)}) \arrow{ru} & \\
%\end{tikzcd}
%\]

We can identify the homomorphisms satisfying these conditions with $N_{SO(2n+1)}$.

We define a locally constant sheaf on $X^{\circ}$ associated to $\cL\in \Loc_{\bbG_{m}}(\overline{X}_{1}^{\circ_{P}})\times_{Loc_{\bbG_{m}}(\overline{X}^{ne,\circ_{P}}/i)}\{\cM\}$ which we denote by
$\Hom^{SO(2n+1)}(C\times X^{\circ},\cL|_{\overline{X}^{ne,\circ_{P}}_{1}})$.  Over a contractible open set $U\subset X^{\circ}$ an element $b\in \Hom^{SO(2n)}(C\times X^{\circ},\cL)$ is defined to consist of the data of;
\begin{itemize}
    \item A morphism $b_{1}:\{-n,...,-1,1,...,n\}\times U\rightarrow \overline{X}^{\circ_{P}}|_{U}$ of sheaves of sets over $U$ that is surjective and intertwines $i$ and the involution sending $j\mapsto -j$.  Denote by $b_{1,K}$ the map from $\{1,...,n\}\times U\rightarrow (X^{\circ_P}/i)|_{U}$ induced by $b_{1}$.
    \item Locally constant isomorphisms $\underline{\bbC_{j}}\xrightarrow{b_{2,j}}\cL|_{b_{1}(j\times X)}$ (considered as sheaves on $U$),
    %\item Morphisms $b_{3,j}:\cM_{\overline{X}^{\circ_{P}}/i}|_{b_{1,K}^{-1}(\{j\}\times U)}\rightarrow \bbC_{U}$, that are induced from a morphism of principal $\bbZ/2$-bundles on $b_{1,K}^{-1}(\{j\}\times U)\subset \overline{X}^{\circ_{P}}/i$ from $\overline{X}^{\circ_{P}}|_{b_{1,K}^{-1}(\{j\}\times U)}\xrightarrow{\cong} (b_{1,K}^{-1}(\{j\}\times U)) \times \bbZ/2$
\end{itemize} 
such that;
\begin{itemize}
    \item The following diagram commutes:

\[
\begin{tikzcd}
\cL|_{b_{1}(\{j\}\times U)}\wedge \cL|_{b_{1}(\{-j\}\times U)} \arrow{r} & \cM_{\overline{X}^{\circ_{P}}/i}|_{b_{1,K}\{j\}\times U}\\
\underline{\bbC_{j}}_{U}\wedge \underline{\bbC_{-j}}_{U} \arrow{r} \arrow{u}{b_{2,j}\wedge b_{2,-j}} & \underline{\bbC}_{U}\arrow{u}{b_{3,j}}
\end{tikzcd}
\]
where the morphism $b_{3,j}:\cM_{\overline{X}^{ne, \circ_{P}}/i}|_{b_{1,K}(\{j\}\times U)}\rightarrow \underline{\bbC}_{U}$, is the morphism induced from the morphism of principal $\bbZ/2$-bundles on $b_{1,K}(\{j\}\times U)\subset \overline{X}^{ne,\circ_{P}}/i$ given by $\overline{X}^{ne,\circ_{P}}|_{b_{1,K}(\{j\}\times U)}\xleftarrow{\cong} (b_{1,K}(\{j\}\times U)) \times \{j,-j\}\cong (b_{1,K}(\{j\}\times U)) \times \bbZ/2$ (where in the final map we identify $j$ with the identity $1_{\bbZ/2}$).
\end{itemize}

Then $\Hom^{SO(2n+1)}(C\times X^{\circ}, \cL|_{\overline{X}^{ne,\circ_{R}}_{1}})$ is an $N_{SO(2n+1)}$ local system, where the $N_{SO(2n+1)}$ action comes from precomposing with the homomorphisms $C\rightarrow C$ specified above.  The map $\cL\mapsto \Hom^{SO(2n+1)}(C\times X,\cL)$ is then the inverse to the map in Proposition \ref{prop: spectral cameral SO(2n+1) with alpha}.

The consideration of the $S$-monodromy condition at $R_{\rho,1}$ is entirely analogous to the $SO(2n)$ and $Sp(2n)$ cases in the proof of Propositions \ref{propL Spectral and Cameral descriptions for Sp(2n) with monodromy condition} and \ref{propL Spectral and Cameral descriptions for SO(2n) with monodromy condition}.  The consideration at $R_{\rho,0}$ is analogous to that of $Sp(2n)$  in Proposition \ref{propL Spectral and Cameral descriptions for Sp(2n) with monodromy condition} with the minor difference that it reduces to the case of $SO(3)$ rather than $SL(2)$.  The restriction on the monodromy of the local systems on $\overline{X}_{1}$ and $\overline{X}_{0}$ ensure that the $S$-monodromy condition is satisfied.  The result follows.
\end{proof}

\begin{appendices}
\section{Explicit computations for planar root systems}
\label{sect:planar}

In this appendix we list a few explicit calculations, which exemplify the results of Section \ref{sect:prelim} in the case of planar root systems.  While we use the logarithm and the exponential we only apply these functions to unipotent and nilpotent elements respectively.  Hence all results hold for algebraic groups. Recall that there are only four planar root systems; they are depicted in Figures \ref{fig:a1root}-\ref{fig:g2root}.

\begin{lemma}
\label{lem:convex_hull_planar}
Choose a polarization of the planar root system, and let $\alpha, \beta$ be the simple roots determined by the polarization. If their lengths differ, let $\alpha$ be the shorter root. Then the restricted convex hull (see Definition \ref{def:convex_hulls}) is given by:
\begin{enumerate}
    \item $\Conv^{\bbN}_{\alpha, \beta} = \{\alpha, \beta\}$ in the $A_1 \times A_1$ case;
    \item $\Conv^{\bbN}_{\alpha, \beta} = \{\alpha, \alpha + \beta, \beta\}$ in the $A_2$ case;
    \item $\Conv^{\bbN}_{\alpha, \beta} = \{\alpha, 2\alpha + \beta, \alpha + \beta, \beta\}$ in the $B_2$ case;
    \item $\Conv^{\bbN}_{\alpha, \beta} = \{\alpha, 3\alpha + \beta, 2\alpha + \beta, 3\alpha + 2\beta, \alpha + \beta, \beta\}$ in the $G_2$ case.
\end{enumerate}
\end{lemma}

\begin{proof}
Obvious from Figures \ref{fig:a1root}-\ref{fig:g2root}; our convention on $\alpha$ and $\beta$ agrees with the notation in the figures.
\end{proof}

Consequently, when two Stokes curves labeled by $\alpha$ and $\beta$ intersect, the number of new Stokes curves produced is either 0, 1, 2, or 4 (respectively).

\begin{lemma}
\label{lem:swap2}
In the setting of Lemma \ref{lem:convex_hull_planar}, let two Stokes curves labeled by $\alpha$ and $\beta$ intersect.  Consider the morphism of Corollary \ref{cor:stokes_sector}, which maps incoming Stokes factors to outgoing Stokes factors:
\[
U_\alpha \times U_\beta 
\longrightarrow 
\prod_{\gamma \in \Conv^\N_{\alpha,\beta}} U_\gamma.
\]
Then, letting $x \in \fu_\alpha$, $y \in \fu_\beta$, and using the short-hand notation $[x,y]^{[n]}:= \Big[x, \big[x, \dots,[x,y]\big]\Big]$, and the notation $\exp(a)=e^{a}$, the morphism
has one of the following explicit forms:
\begin{enumerate}
    \item In the $A_1 \times A_1$ case:
    \begin{equation}
    \begin{tikzcd}
    (e^x, e^y)
    \arrow[mapsto]{r}
    &
    (e^y,e^x).
    \end{tikzcd}
    \end{equation}
    
    \item In the $A_2$ case:
    \begin{equation}
    \begin{tikzcd}
    (e^x, e^y)
    \arrow[mapsto]{r}
    &
    \big(e^y,
    \exp{[x,y]},
    e^x
    \big).
    \end{tikzcd}
    \end{equation}
    
    \item In the $B_2$ case:
    \begin{equation}
    \begin{tikzcd}
    (e^x, e^y)
    \arrow[mapsto]{r}
    &
    \bigg(e^y,
    \exp{[x,y]},
    \exp\left(\frac{1}{2} [x,y]^{[2]}\right),
    e^x
    \bigg).
    \end{tikzcd}
    \end{equation}
    
    \item In the $G_2$ case:
    \begin{equation}
    \begin{tikzcd}
    (e^x, e^y)
    \arrow[mapsto]{d}
    \\
    \Bigg(e^y,
    \exp{[x,y]},
    \exp\left(\frac{1}{6} \big[[x,y]^{[2]},[x,y]\big]\right),
    \exp\left(\frac{1}{2} [x,y]^{[2]}\right),
    \exp\left(\frac{1}{6} [x,y]^{[3]}\right),
    e^x
    \Bigg).
    \end{tikzcd}
    \end{equation}
\end{enumerate}

More generally if the two Stokes curves are labelled by roots $\alpha$ and $\beta$ with the property that $\Conv^{\bbN}_{\alpha, \beta} \subset \{\alpha, 3\alpha + \beta, 2\alpha + \beta, 3\alpha + 2\beta, \alpha + \beta, \beta\}$, then the image of $(e^x, e^y)$ is the subtuple of 
\[\Bigg(e^y,
    \exp{[x,y]},
    \exp\left(\frac{1}{6} \big[[x,y]^{[2]},[x,y]\big]\right),
    \exp\left(\frac{1}{2} [x,y]^{[2]}\right),
    \exp\left(\frac{1}{6} [x,y]^{[3]}\right),
    e^x
    \Bigg),\]
corresponding to the subset $\Conv^{\bbN}_{\alpha, \beta} \subset \{\alpha, 3\alpha + \beta, 2\alpha + \beta, 3\alpha + 2\beta, \alpha + \beta, \beta\}$.
\end{lemma}

\begin{proof}
Using Lemma \ref{lem:roots_commutators_vanish}, and the explicit description of $\Conv^\N_{\alpha,\beta}$ from Lemma \ref{lem:convex_hull_planar}, it's clear that the given elements live in the correct one-parameter subgroups $U_\gamma$, for $\gamma \in \Conv^{\N}_{\alpha, \beta}$. It suffices, then, to prove that:
\begin{equation}
\label{eq:ugly_formula_to_prove}
e^x e^y = e^y
    \exp{[x,y]}
    \exp\left(\frac{1}{6} \big[[x,y]^{[2]},[x,y]\big]\right)
    \exp\left(\frac{1}{2} [x,y]^{[2]}\right)
    \exp\left(\frac{1}{6} [x,y]^{[3]}\right)
    e^x,
\end{equation}
where we allow the possibility that some of the exponents are zero, in order to treat all four cases simultaneously.

Up to order 5 in $x, y\in \fg$, the Baker-Campbell-Hausdorff formula gives:
\begin{align*}
\log(e^xe^y)
&=
x + y + \frac{1}{2}[x,y] + \frac{1}{12}\left(\big[x,[x,y]\big] + \big[y,[y,x]\big] \right)
-\frac{1}{24} \Big[x,\big[y,[x,y]\big]\Big] -
\\
&- \frac{1}{720}\left( \bigg[y,\Big[y,\big[y,[y,x]\big]\Big]\bigg]
+ \bigg[x,\Big[x,\big[x,[x,y]\big]\Big]\bigg]\right) +
\\
&+ \frac{1}{360}\left( \bigg[y,\Big[x,\big[x,[x,y]\big]\Big]\bigg]
+ \bigg[x,\Big[y,\big[y,[y,x]\big]\Big]\bigg]\right) +
\\
&+ \frac{1}{720}\left( \bigg[y,\Big[x,\big[y,[x,y]\big]\Big]\bigg]
+ \bigg[x,\Big[y,\big[x,[y,x]\big]\Big]\bigg]\right) + 
\\
&+ \dots
\end{align*}
See, for example, Theorem 2 in II.6.6 of \cite{bourbakiLie1-3} for the general combinatorial formula, originally due to Dynkin.

Now let $x \in \fu_\alpha$, $y\in \fu_\beta$. Lemma \ref{lem:convex_hull_planar} ensures that the only nonzero terms are:
\begin{equation}
\label{eq:BCH0}
\log(e^x e^y) 
=
x + y + \frac{1}{2}[x,y] + \frac{1}{12}[x,y]^{[2]} + \frac{1}{360}\big[[x,y]^{[2]},[x,y] \big].
\end{equation}

Similarly:
\begin{align}
\begin{split}
\label{eq:BCH1}
\log\left( 
e^y
\exp{[x,y]}
\exp\left(\frac{1}{6} \big[[x,y]^{[2]},[x,y]\big]\right)
\right)
&=
y + [x,y] + \frac{1}{6} \big[[x,y]^{[2]},[x,y]\big],
\end{split}
\\
\begin{split}
\label{eq:BCH2}
\log \left(
\exp\left(\frac{1}{2} [x,y]^{[2]}\right)
\exp\left(\frac{1}{6} [x,y]^{[3]}\right)
e^x
\right)
&=
\frac{1}{2}[x,y]^{[2]}
+ \frac{1}{6} [x,y]^{[3]}
+ x
+ \frac{1}{4}\big[[x,y]^{[2]},x \big]
\\
&=
x + \frac{1}{2}[x,y]^{[2]}
- \frac{1}{12} [x,y]^{[3]}.
\end{split}
\end{align}

Let $z$, $w$ denote the right-hand sides of \ref{eq:BCH1} and \ref{eq:BCH2}, respectively. Then we need to prove that $e^xe^y = e^{z}e^{w}$. The nonzero iterated Lie brackets of $z$, $w$ are:
\begin{align*}
[z,w] &= -[x,y] - [x,y]^{[2]} - \frac{7}{12}\big[ [x,y]^{[2]},[x,y] \big],
\\
[z,w]^{[2]} &= \big[ [x,y]^{[2]},[x,y] \big],
\\
[w,z]^{[2]} &= [x,y]^{[2]} + [x,y]^{[3]} + \frac{1}{2} \big[ [x,y]^{[2]},[x,y] \big],
\\
\big[ [w,z]^{[2]},[w,z] \big] &= \big[ [x,y]^{[2]},[x,y] \big].
\end{align*}

To calculate $[z,w]$ we used the Jacobi identity to show $[y,[x,y]^{[3]}]=[[x,y]^{[2]},[x,y]]$.

Then:
\begin{align*}
\log(e^{z}e^{w})
&=
z + w + \frac{1}{2} [z,w] + \frac{1}{12} [z,w]^{[2]} + \frac{1}{12} [w,z]^{[2]} + \frac{1}{360} \big[ [w,z]^{[2]},[w,z] \big]
\\
&= 
x + y + \frac{1}{2}[x,y] + \frac{1}{12}[x,y]^{[2]} + \frac{1}{360}\big[[x,y]^{[2]},[x,y] \big].
\end{align*}
This agrees with $\log(e^xe^y)$, as computed in \ref{eq:BCH0}.
\end{proof}

%This is the version of Matei's computation where I have corrected the sign!

\begin{lemma}
\label{lemma commuting alpha plus beta past alpha}
Let $\gamma$ and $\delta$ be roots such that $\Conv^{\bbN}_{\gamma, \delta}\subset \{\gamma, 2\gamma+\delta, \gamma+\delta, \gamma+2\delta, \delta\}$.  Consider the morphism of Corollary \ref{cor:stokes_sector}, which maps incoming Stokes factors to outgoing Stokes factors:
\[
U_\alpha \times U_\beta 
\longrightarrow 
\prod_{\gamma \in \Conv^\N_{\alpha,\beta}} U_\gamma.
\]
Letting $x \in \fu_\delta$, $y \in \fu_\gamma$, and using the short-hand notation from Lemma \ref{lem:swap2}, the morphism
is given by mapping $(e^{x},e^{y})$ to the subtuple of
\[\left(e^{y}, \exp{\left(\frac{1}{2}[y,x]^{[2]}\right)}, \exp{\left([x,y]\right)}, \exp{\left(\frac{1}{2}[x,y]^{[2]}\right)},e^{x}\right)\]
corresponding to the subset $\Conv^{\bbN}_{\gamma, \delta}\subset \{\gamma, 2\gamma+\delta, \gamma+\delta, \gamma+2\delta, \delta\}$.
\end{lemma}

This is shown by a computation that is essentially identical to that for Lemma \ref{lem:swap2}.  As such we omit the calculation.

\begin{remark}
An example of the situation of Lemma \ref{lemma commuting alpha plus beta past alpha} is given by the $G2$ root system with $\gamma=\alpha+\beta$ and $\delta=\alpha$, where $\alpha$ and $\beta$ are as in Lemma \ref{lem:convex_hull_planar}.
\end{remark}

Consider, now, an intersection of $k$ curves, labeled by a maximal planar, convex, set of roots, as in Figure \ref{fig:rays_reverse}.

\begin{figure}[h]
    \centering
    \includegraphics[width = 0.5 \textwidth]{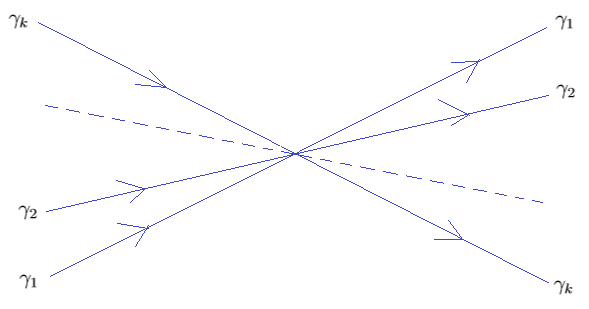}
    \caption{Intersection corresponding to maximal convex set of roots.}
    \label{fig:rays_reverse}
\end{figure}

Let $\gamma_1, \dots, \gamma_k$ denote the order on $\Conv^{\N}_{\alpha, \beta}$ induced by the order of the incoming Stokes curves. Then the outgoing curves have the reverse order. So, in this case, Corollary \ref{cor:stokes_sector} gives a morphism which commutes with the multiplication maps:
\begin{equation}
\label{diag:reverse}
\begin{tikzcd}
U_{\gamma_1} \times \dots \times U_{\gamma_k}
\arrow[swap]{d}{\text{reverse}}
\arrow{r}
&
U_{\Conv^\N_{\alpha,\beta}}
\\
U_{\gamma_k} \times \dots \times U_{\gamma_1} 
\arrow{ur}
&
\end{tikzcd}
\end{equation}

We now explain how to obtain explicit formulas for the morphism $reverse$ in Diagram \ref{diag:reverse}, by repeated application of Lemmas \ref{lem:swap2} and \ref{lemma commuting alpha plus beta past alpha}.

We use the fact that, in the planar case we are dealing with, the order on outgoing Stokes curves is convex  as shown in section \ref{subsubsec: convexity}. This means that, if there are three outgoing curves labeled by $\gamma, \delta, \gamma + \delta$, then the one labeled by $\gamma + \delta$ lies within the sector bounded by the curves labeled by $\gamma$ and $\delta$ (see Lemma \ref{lem:2d_sector}). There are only two total orders on $\Conv^\N_{\alpha,\beta}$ satisfying this convexity property: they correspond to an order $(\gamma_1, \dots, \gamma_k)$ and its reverse $(\gamma_k, \dots, \gamma_1)$, as in Diagram \ref{diag:reverse}. Since the orders on $\Conv^\N_{\alpha,\beta}$ used on the RHS of Lemmas \ref{lem:swap2} and \ref{lemma commuting alpha plus beta past alpha} are convex, the right hand side of the applicable lemma must coincide with one of these two. We assume that it is $(\gamma_k, \dots, \gamma_1)$; otherwise, we would describe $\text{reverse}^{-1}$ instead.

We build the morphism $\text{reverse}$ as a composition of ``twisted transpositions'' and ``contractions'', as defined below.

We define a twisted transposition to be an application of Lemma \ref{lem:swap2} or Lemma \ref{lemma commuting alpha plus beta past alpha}: we replace a term of the form $U_{\gamma_i} \times U_{\gamma_j}$ (with $i<j$) with the tuple provided by the RHS of Lemma \ref{lem:swap2} or Lemma \ref{lemma commuting alpha plus beta past alpha}. The requirement that $i<j$ ensures that the resulting tuple is ordered correctly. 

We define a contraction to be a multiplication map:
\[
U_{\gamma_i} \times U_{\gamma_i} \to U_{\gamma_i}.
\]

It's clear that, using finitely many twisted transpositions and contractions, we obtain a morphism which commutes with the multiplication maps as in diagram \ref{diag:reverse}. This must agree with the morphism $reverse$, due to the uniqueness statement in Corollary \ref{cor:stokes_sector}. 

\begin{example}[Cecotti--Vafa Wall Crossing Formula]
Consider a Lie algebra $\fr g$ of ADE type. The restriction of the root system to the plane spanned by $\alpha$ and $\beta$ is a root system of type $A_1 \times A_1$ or $A_2$. In the first case, $k=2$, and we have $\text{reverse}(e^{X_\alpha},e^{X_\beta}) = (e^{X_\beta},e^{X_\alpha})$, is simply the transposition. 
In the second case, $k=3$ and the commutativity of the diagram \ref{diag:reverse} is expressed by:
\begin{equation}
\label{eq:CV_A2}
\exp(X_{\alpha}) \exp(X_{\alpha + \beta}) \exp(X_{\beta})
= 
\exp\big(X_{\beta}\big) 
\exp\big(X_{\alpha + \beta} + [X_\alpha,X_\beta]\big) 
\exp\big(X_\alpha \big).
\end{equation}
\end{example}

\begin{example}[B2 analogue of Cecotti--Vafa Wall Crossing Formula]
In the B2 case, $k=4$ and we have:
\begin{align}
\begin{split}
\label{eq:CV_B2}
\exp&(X_\alpha) \exp(X_{2\alpha+\beta}) \exp(X_{\alpha + \beta}) \exp(X_\beta)
\\
=&
\exp\big(X_\beta\big) \times 
\\
&
\exp\big(X_{\alpha + \beta} + [X_\alpha, X_\beta]\big)  \times
\\
&
\exp\left(X_{2\alpha + \beta} + [X_{\alpha},X_{\alpha + \beta}] + \frac{1}{2}[X_\alpha,X_\beta]^{[2]}\right) \times
\\
&
\exp\big(X_\alpha\big).
\end{split}
\end{align}
\end{example}

\begin{example}[G2 analogue of Cecotti--Vafa Wall Crossing Formula]
In the G2 case, $k=6$ and we have:
\begin{align}
\begin{split}
\label{eq:CV_G2}
\exp&(X_\alpha) \exp(X_{3\alpha + \beta}) \exp(X_{2\alpha + \beta}) \exp(X_{3\alpha + 2\beta}) \exp(X_{\alpha + \beta}) \exp(X_{\beta})
\\
=&
\exp\big(X_{\beta}\big) \times
\\
 &
\exp\big(X_{\alpha + \beta} + [X_\alpha,X_\beta]\big)\times
\\
 &
\exp\left(
X_{3\alpha + 2\beta} + [X_{3\alpha + \beta},X_\beta] + [X_{2\alpha + \beta},X_{\alpha + \beta}] 
+ \frac{1}{2} \big[ [X_\alpha,X_\beta]^{[2]},X_{\alpha + \beta} \big]+\right.
\\
& \hspace{7mm}
\left. \frac{1}{6} \big[ [X_\alpha,X_\beta],[X_\alpha,X_\beta]^{[2]}\big]+\frac{1}{2}[X_{\alpha+\beta},X_{\alpha}]^{[2]} 
\right) \times
\\
&
\exp\left(X_{2\alpha + \beta} + [X_\alpha, X_{\alpha + \beta}]
+ \frac{1}{2} [X_\alpha,X_\beta]^{[2]}\right) \times
\\
&
\exp\left(X_{3\alpha + \beta} + [X_\alpha,X_{2\alpha + \beta}] + \frac{1}{6}[X_\alpha,X_\beta]^{[3]}+\frac{1}{2}[X_{\alpha},X_{\alpha+\beta}]^{[2]}\right) \times
\\
 &
\exp\big(X_\alpha\big).
\end{split}
\end{align}

\end{example}

\end{appendices}

\bibliographystyle{amsplain}
\bibliography{bib}

\end{document}

%------------------------------------------------------------------------------
% End of journal.tex
%------------------------------------------------------------------------------